\documentclass{amsart}

\usepackage{enumerate}
\usepackage[latin1]{inputenc}
\usepackage{dsfont}
\usepackage{amssymb,amsthm,amsmath}
\usepackage{lscape}
\allowdisplaybreaks

\input{xy}
\xyoption{all}

\usepackage{mathrsfs}

\theoremstyle{plain}

\newtheorem{prop}{Proposition}[section]
\newtheorem{lemma}[prop]{Lemma}

\newtheorem{thm}[prop]{Theorem}
\newtheorem{theorem}{Theorem}

\newtheorem{coroll}[theorem]{Corollary}
\newtheorem{cor}[prop]{Corollary}

\theoremstyle{definition}
\newtheorem{conv}[prop]{Convention}

\newtheorem{defi}[prop]{Definition}

\newtheorem{example}[prop]{Example}

\theoremstyle{remark}
\newtheorem{remark}[prop]{Remark}

\makeatletter
\@namedef{subjclassname@2020}{%
  \textup{2020} Mathematics Subject Classification}
\makeatother

\newcommand{\funddom}{\mathcal{F}}
\newcommand{\ini}{\mathrm{ini}}
\DeclareMathOperator{\GL}{GL}
\DeclareMathOperator{\SL}{SL}
\DeclareMathOperator{\PSL}{PSL}
\DeclareMathOperator{\PGamma}{P\Gamma}
\DeclareMathOperator{\Ima}{Im}
\DeclareMathOperator{\Rea}{Re}
\DeclareMathOperator{\sgn}{sgn}
\newcommand\N{\mathbb{N}}

\newcommand\R{\mathbb{R}}
\newcommand\Z{\mathbb{Z}}
\newcommand\C{\mathbb{C}}
\newcommand{\h}{\mathbb{H}}
\newcommand{\mc}[1]{\mathcal #1}
\newcommand{\mf}[1]{\mathfrak #1}
\newcommand{\wt}{\widetilde}
\newcommand{\wh}{\widehat}
\newcommand{\eps}{\varepsilon}
\newcommand{\mat}[4]{\begin{pmatrix} #1&#2\\#3&#4\end{pmatrix}}
\newcommand{\bmat}[4]{\begin{bmatrix} #1&#2\\#3&#4\end{bmatrix}}
\newcommand{\textmat}[4]{\left(\begin{smallmatrix} #1&#2 \\ #3&#4
\end{smallmatrix}\right)}
\newcommand{\textbmat}[4]{\left[\begin{smallmatrix} #1&#2 \\ #3&#4
\end{smallmatrix}\right]}

\newcommand{\downmapsto}{\rotatebox[origin=c]{-90}{$\scriptstyle\mapsto$}
\mkern2mu}

\usepackage{mathtools}
\usepackage[nospace,noadjust]{cite}
\usepackage{enumitem}
\usepackage{verbatim}
\usepackage{colonequals}
\usepackage{stmaryrd}
\usepackage{amssymb}
\usepackage{xcolor} 
\usepackage{color,soul}
\usepackage[scr=boondoxo]{mathalfa}
\usepackage[backgroundcolor=pink!50!white,linecolor=white,shadow]{todonotes}
\usepackage{graphicx}

\newcommand*{\logeq}{\ratio\Longleftrightarrow}
\newtheorem{propsperty}{Property}

\theoremstyle{definition}
\newtheorem{algo}[prop]{Algorithm}

\newtheorem{lemdef}[prop]{Lemma and Definition}

\usepackage{tikz}
\usetikzlibrary{%
	intersections,
	math,
	automata,
	matrix,
	arrows,
	positioning,
	quotes,
	bending%
	}
\usetikzlibrary{cd, shapes, patterns, shapes.geometric, shapes.symbols, 
shapes.arrows, shapes.multipart, shapes.callouts, shapes.misc, 
decorations.pathreplacing}
\usetikzlibrary{calc}
\tikzset{
    >=stealth',
    pil/.style={
           ->,
           thick,
           shorten <=2pt,
           shorten >=2pt,}
}

\newcounter{Fig}
\setul{0.5ex}{0.3ex}
\setulcolor{red}

\DeclareRobustCommand{\SkipTocEntry}[5]{}

\makeatletter
\providecommand*\xrightarrowtriangle[2][]{%
  \ext@arrow 0055{\arrowfill@\relbar\relbar\rightarrowtriangle}{#1}{#2}}
\makeatother

\usepackage{imakeidx}
\indexsetup{level=\part*,toclevel=part}
\makeindex[name=defs,title={Index of terminology}]
\makeindex[name=symbols,columns=3,title={List of notations}]
\makeindex

\usepackage[colorlinks,breaklinks]{hyperref}

\global\long\def\defset#1#2{\left\{  #1\;\middle|\;#2\right\}  }
\global\long\def\dif#1{\operatorname{d}\!#1}
\global\long\def\i{\mathrm{i}}
\global\long\def\H{\mathbb{H}}
\global\long\def\R{\mathbb{R}}
\global\long\def\Z{\mathbb{Z}}

\global\long\def\C{\mathbb{C}}
\global\long\def\N{\mathbb{N}}
\global\long\def\PSLR{\mathrm{PSL}_2(\R)}
\global\long\def\x{\mathrm{x}}
\global\long\def\hyp{\mathrm{h}}
\global\long\def\id{\operatorname{id}}
\global\long\def\tr#1{\operatorname{tr}(#1)}
\global\long\def\TO#1{\mathcal{L}_{#1}}
\global\long\def\fTO#1{\widetilde{\mathcal{L}}_{#1}}
\global\long\def\Return{\mathscr{R}}
\global\long\def\act{\boldsymbol{.}}
\global\long\def\CrSc{\widehat{\operatorname{C}}}
\global\long\def\BrU{\mathrm{C}}
\global\long\def\BrS{\mathcal{C}}
\global\long\def\Cs#1{\operatorname{C}_{#1}}
\global\long\def\Css#1#2{\operatorname{C}_{#1}^{(#2)}}
\global\long\def\Csr#1{\wt{\operatorname{C}}_{#1}}

\global\long\def\Csacc#1{\operatorname{C}_{#1}^{\,\acc}}
\global\long\def\Iset#1{I_{#1,\st}}
\global\long\def\Ired#1{\mathcal{I}_{#1}}
\global\long\def\Jset#1{J_{#1,\st}}

\global\long\def\Plussp#1{\mathrm{H}_+(#1)}
\global\long\def\Minussp#1{\mathrm{H}_-(#1)}
\global\long\def\PMsp#1{\mathrm{H}_\pm(#1)}
\global\long\def\MPsp#1{\mathrm{H}_\mp(#1)}

\global\long\def\base#1{\operatorname{bp}(#1)}
\global\long\def\fund{\mathcal{F}}

\global\long\def\iso#1{\mathrm{I}(#1)}

\global\long\def\intiso#1{\mathrm{int\,}\iso{#1}}
\global\long\def\extiso#1{\mathrm{ext\,}\iso{#1}}

\global\long\def\UTB#1{\operatorname{S}\!#1}
\global\long\def\quod#1#2{#1\diagdown#2}
\global\long\def\Orbi{\mathbb{X}}
\global\long\def\GeoFlow{\widehat{\Phi}}
\global\long\def\cyc#1#2{\operatorname{cyc}_{#1}(#2)}
\global\long\def\cycstar#1{\operatorname{cyc}^*_{#1}}
\global\long\def\cycset#1#2{\operatorname{Cyc}_{#1,#2}}
\global\long\def\cycstarset#1{\operatorname{Cyc}^*_{#1}}
\global\long\def\cycnostarset#1{\operatorname{Cyc}_{#1}}
\global\long\def\cycnext#1{g_{#1}}
\global\long\def\cycstarnext#1{g^*(#1)}
\global\long\def\cycnostarnext#1{g(#1)}
\global\long\def\cyctrans#1#2{u_{#1,#2}}
\global\long\def\cycstartrans#1{u_{#1}}
\global\long\def\point#1#2{#2_{#1}}
\global\long\def\Per{\operatorname{Per}}

\global\long\def\st{\operatorname{st}}
\global\long\def\T#1{\mathscr{T}(#1)}
\global\long\def\abs#1{\left|#1\right|}
\global\long\def\geo{\mathscr{g}}
\global\long\def\Geo{\mathscr{G}}
\global\long\def\Ball#1#2{\mathrm{B}_{#1}(#2)}
\global\long\def\Vanish{\mathrm{Van}}
\global\long\def\retime#1{t^+_{#1}}
\global\long\def\pretime#1{t^-_{#1}}
\global\long\def\ittime#1{\mathrm{t}_{#1}}
\global\long\def\ittrans#1{\mathrm{g}_{#1}}
\global\long\def\itindex#1{\mathrm{k}_{#1}}
\global\long\def\Trans#1#2#3{\mathcal{G}_{#1}(#2,#3)}
\global\long\def\Transs#1#2#3{\wt{\mathcal{G}}_{#1}(#2,#3)}
\global\long\def\Transss#1#2#3#4{\mathcal{G}_{#1}^{(#2)}(#3,#4)}

\global\long\def\Transprime#1#2#3{\mathcal{G}'_{#1}(#2,#3)}
\global\long\def\Past#1#2#3{\mathcal{V}_{#1}(#2,#3)}

\global\long\def\Heir#1#2{H_{#1}(#2)}

\global\long\def\acc{\mathrm{acc}}
\global\long\def\red{\mathrm{red}}
\global\long\def\rred{\mathrm{\textcolor{red}{red}}}
\global\long\def\eX{\mathrm{X}}
\global\long\def\eY{\mathrm{Y}}
\global\long\def\eR{\mathrm{R}}
\global\long\def\eZ{\mathrm{Z}}
\global\long\def\eW{\mathrm{W}}

\global\long\def\Index{\wh{A}}
\global\long\def\gbound#1{\geo#1}
\global\long\def\ReturnGraph#1{\mathrm{RG}_{#1}}

\global\long\def\ram#1{\operatorname{ram}(#1)}
\global\long\def\Ram#1{\mathrm{Ram}_{#1}}
\global\long\def\fixp#1#2{\operatorname{f}_{#1}(#2)}
\global\long\def\Stab#1#2{\operatorname{Stab}_{#1}(#2)}
\global\long\def\standh#1{\mathrm{h}_{#1}}
\global\long\def\countit#1{\varphi(#1)}
\global\long\def\Unit{\operatorname{U}}
\global\long\def\ct{\operatorname{ct}}
\global\long\def\Att#1#2{\operatorname{Att}_{#1}(#2)}
\global\long\def\nbAtt#1{\operatorname{Att}_{#1}}
\global\long\def\IAtt#1#2{I(\Att{#1}{#2})}
\global\long\def\conv#1#2{\operatorname{conv}_{#1}(#2)}
\global\long\def\slow{\mathrm{slow}}
\global\long\def\fast{\mathrm{fast}}
\global\long\def\Fct{\operatorname{Fct}}
\global\long\def\edge#1{\xrightarrowtriangle{#1}}
\global\long\def\conv{\mathrm{conv}}
\global\long\def\prA{\pi_{\!A}}
\global\long\def\prZ{\pi_{\eZ}}

\newcommand{\algofin}{\hfill\rotatebox{45}{\textsquare}}

\begin{document}

\title[Strict transfer operator approaches]{Selberg zeta functions, cuspidal accelerations, and existence of strict transfer operator approaches}
\author[A.\@ Pohl]{Anke Pohl}
\address{Anke Pohl, University of Bremen, Department~3 - Mathematics, Institute for Dynamical Systems, Bibliothekstr.~5, 28359~Bremen, Germany}
\email{apohl@uni-bremen.de}
\author[P.\@ Wabnitz]{Paul Wabnitz}
\address{Paul Wabnitz, University of Bremen, Department~3 - Mathematics, Institute for Dynamical Systems, Bibliothekstr.~5, 28359~Bremen, Germany}
\email{pwabnitz@uni-bremen.de}

\begin{abstract} 
For geometrically finite non-compact developable hyperbolic orbisurfaces (including those of infinite volume), we provide transfer operator families whose Fredholm determinants are identical to the Selberg zeta function. Our proof yields an algorithmic and uniform construction. This construction is initiated with an externally provided cross section for the geodesic flow on the considered orbisurface that yields a highly faithful, but non-uniformly expanding discrete dynamical system modelling the geodesic flow. Through a number of algorithmic steps of reduction, extension, translation, induction and acceleration, we turn this cross section into one that yields a still highly faithful, but now uniformly expanding discrete dynamical system. The arising transfer operator family is nuclear of order zero on suitable Banach spaces. In addition, finite-dimensional twists with non-expanding cusp monodromy can be included.
\end{abstract}

\maketitle
\tableofcontents{}
\vspace*{\fill}

\section{Introduction}\label{SEC:intro}

The geometric and spectral properties of hyperbolic orbisurfaces\footnote{We use the term ``(developable) hyperbolic orbisurfaces'' to refer to all 
two-dimensional good connected hyperbolic Riemannian orbifolds. All orbifolds considered in this article will be assumed to be good (i.e., developable), for which reason we will omit throughout ``developable'' and abbreviate to ``hyperbolic orbisurfaces.'' In other words, using the notation introduced further below, we call all spaces of the form~$\quod{\Gamma}{\H}$ ``hyperbolic orbisurfaces.'' We emphasize that this notion also includes those quotient spaces~$\quod{\Gamma}{\H}$ for which the Fuchsian group~$\Gamma$ has torsion. Thus, it includes the spaces~$\quod{\Gamma}{\H}$ that are not hyperbolic surfaces in the original sense, i.e., which are not Riemannian manifolds.}
and their interdependence are of great interest throughout mathematics and physics, and beyond. For investigations, we can therefore take advantage of insights, knowledge and methods from various areas.
Each of the existing methods has its particular strengths and weaknesses; the improvement of methods is an ongoing process. With this article we contribute to the further development of one of these methods: transfer operator techniques. Our main motivation is to provide representations of Selberg zeta functions by Fredholm determinants of families of transfer operators that are particularly well-structured and have an accompanying slow transfer operator (see further below for terminology and details). Further motivation stems from other applications that these particular transfer operators have, of which we will indicate one below. 

\addtocontents{toc}{\SkipTocEntry}
\subsection*{Motivation} 

Before presenting our main results, we provide now a more detailed account 
on our motivation. The Selberg zeta function of a geometrically finite hyperbolic orbisurface is a dynamical zeta function that yields an intimate 
relation between the geodesic length spectrum and the resonances of this 
orbisurface.
We refer to~\eqref{eq:szf_infprod_intro} for a precise formula of this zeta function.
Since its introduction by Selberg~\cite{Selberg} in the 1950s, it is a 
highly valued object when investigating resonances (or vice versa, when 
investigating periodic geodesics using knowledge on resonances). Selberg zeta 
functions are now available for all geometrically finite hyperbolic orbisurfaces, and their theory is already quite extensive. Nevertheless, new contributions are made regularly, as we will do with this article. 

Another family of objects that allows us to establish relations between 
geodesics and resonances, even resonant states, of hyperbolic orbisurfaces are 
transfer operators associated to suitable discretizations of the geodesic flow 
on hyperbolic orbisurfaces. The theory of these (discrete-time or Ruelle-type) 
transfer operators is much younger and less developed than the theory of the 
Selberg zeta function. Its use in the spectral theory of hyperbolic orbisurfaces 
goes back to Ruelle~\cite{Ruelle_zeta, Ruelle_dynzeta, Ruelle_thermo}, 
Mayer~\cite{Mayer_zeta, Mayer_thermo, Mayer_thermoPSL}, Fried~\cite{Fried_zetafunctionsI} 
and Pollicott~\cite{Pollicott} and has made a great leap forward in recent 
years due to the efforts of many different researchers. (See further below for more details and more extensive references. For completeness we remark that there is also a theory of continuous-time transfer operators. However, we will not discuss it
here.) 

While the Selberg zeta function focuses on the \emph{static geometry} of hyperbolic 
orbisurfaces (namely, the lengths of periodic geodesics), transfer operators take 
advantage of the \emph{dynamics} of the geodesic flow (namely, the paths of periodic 
geodesics). Accordingly, investigations of spectral properties of hyperbolic 
surfaces by means of transfer operators on the one hand and by means of Selberg 
zeta functions on the other hand are complementary. A number of particularly 
good results have been achieved by combining both approaches; further below we 
will list a few examples. In these cases the Selberg zeta function of the 
considered hyperbolic orbisurface is represented as a Fredholm determinant of a 
family of transfer operators that are particularly well-structured. 

With this article we will significantly extend the realm of hyperbolic orbisurfaces for which such well-structured transfer operators exist. Our results will show 
that such transfer operators exist for a huge class of non-compact geometrically finite hyperbolic orbisurfaces with at least one periodic geodesic, including a huge class of infinite-area hyperbolic orbisurfaces. We will not only 
establish existence but indeed present a construction of such transfer operators. We refer to the discussion of our main results further below for more detailed information on the realm of our results. 

To be more detailed regarding our motivation, we use the upper half plane model 
\[
 \H =\defset{z=x+\i y\in\C}{y=\Ima z > 0}\,,\qquad \dif s^2_z = \frac{\dif x^2 + 
\dif y^2}{y^2}\,,
\]
of the hyperbolic plane, fix a discrete, finitely generated group~$\Gamma$ of 
orientation-preserving Riemannian isometries of~$\H$ (thus, $\Gamma$ is a 
Fuchsian group) with at least one hyperbolic element, and consider the geometrically finite hyperbolic orbisurface
\[
\Orbi\coloneqq\quod{\Gamma}{\H}\,.
\]
The \emph{spectral entities} of~$\Orbi$ that motivate our investigations are the Laplace resonances and resonant states of~$\Orbi$. For a more detailed exposition, we let $\Delta_\Orbi$ denote the (positive) Laplacian on~$\Orbi$. In the local coordinates of~$\H$ it is given by
\[
 \Delta_\Orbi = -y^2 \big(\partial_x^2 + \partial_y^2 \big)\,.
\]
Its resolvent 
\[
 R_\Orbi(s) \coloneqq \big( \Delta_\Orbi - s(1-s) \big)^{-1}\colon L^2(\Orbi) 
\to H^2(\Orbi)
\]
is defined for all $s\in\C$ with $\Rea s > \tfrac12$ for which $s(1-s)$ is not 
an $L^2$-eigenvalue of~$\Delta_\Orbi$. Here, $L^2(\Orbi)$ denotes the space of (equivalence classes) of square-Lebesgue-integrable functions, and $H^2(\Orbi)$ the Sobolev space. The family of restricted operators
\[
 R_\Orbi(s)\colon C^\infty_c(\Orbi) \to C^\infty(\Orbi)\,,\qquad \Rea s \gg 1\,,
\]
extends in $s$ to a meromorphic family of operators
\[
 R_\Orbi(s) \colon L^2_{\textnormal{comp}}(\Orbi) \to 
H^2_{\textnormal{loc}}(\Orbi)
\]
on all of~$\C$, see~\cite{Mazzeo_Melrose, GZ_upper_bounds}, where $L^2_{\textnormal{comp}}(\Orbi)$ denotes the space of the compactly supported elements of~$L^2(\Orbi)$, and $H^2_{\textnormal{loc}}(\Orbi)$ is the space of functions that are locally in~$H^2$. The \emph{resonances} of~$\Orbi$ are the poles of the map $s\mapsto R_\Orbi(s)$, and the \emph{resonant states with spectral parameter~$s$} are the generalized eigenfunctions of~$\Delta_\Orbi$ at the resonance~$s$. We let~$\mc R_\Orbi$ 
denote the multiset of the resonances of~$\Orbi$, that is, the set of resonances with multiplicities.

The resonances of~$\Orbi$ are closely related to the geodesics of~$\Orbi$, as 
can be shown with several different methods, thereby establishing a mathematically rigorous incarnation of the classical-quantum correspondence principle from physics. The two methods that motivate this article---the Selberg zeta functions and the transfer operator techniques---show such relations by exploiting the properties of \emph{periodic geodesics}. In order to provide a more detailed account, we denote by~$L_\Orbi$ the \emph{primitive geodesic length spectrum} of~$\Orbi$, that is, the multiset of the lengths of all primitive periodic geodesics of~$\Orbi$. Further we denote the unit tangent bundle of~$\Orbi$ by~$\UTB\Orbi$, and by~$T_\Orbi$ the \emph{trapped set}, that is the set of unit tangent vectors $v\in \UTB\Orbi$ for which the geodesic determined by~$v$ stays in some compact subset of~$\UTB\Orbi$. We let 
$\nu$ denote the Hausdorff dimension of~$T_\Orbi$ and set $\delta \coloneqq 
(\nu-1)/2$. (We remark that~$\delta$ equals the Hausdorff dimension of the 
limit set of the fundamental group of~$\Orbi$. For details we refer 
to~\cite[Section~14.8]{Borthwick_book}.) The \emph{Selberg zeta function}~$Z_\Orbi$ of~$\Orbi$ is given by the infinite product (Euler product)
\begin{equation}\label{eq:szf_infprod_intro}
 Z_\Orbi(s) = \prod_{\ell\in L_\Orbi} \prod_{k=0}^\infty \left(1 - 
e^{-(s+k)\ell}\right)
\end{equation}
for $s\in\C$ with $\Rea s > \delta$, and by the meromorphic continuation of the 
infinite product in~\eqref{eq:szf_infprod_intro} to all of~$\C$ (see, e.~g., 
\cite{Selberg, Venkov_book, Borthwick_book}). The multiset of the zeros 
of~$Z_\Orbi$ is almost equal to the multiset~$\mc R_\Orbi$ of resonances 
of~$\Orbi$ as every zero of~$Z_\Orbi$ is a resonance or a so-called topological zero, and almost every resonance of~$\Orbi$ is a zero of~$Z_\Orbi$, including (almost) matching multiplicities. The set of topological zeros is well-understood and rather sparse; it is contained in $\tfrac12(1-\N_0)$. Also the resonances of~$\Orbi$ that do not arise as zeros of~$Z_\Orbi$ are well-understood; these are caused by exceptional $L^2$-eigenvalues that give rise to two resonances of~$\Orbi$. For details we refer, e.g., to the textbooks~\cite{Venkov_book, Borthwick_book}. For completeness we mention that if the Fuchsian group~$\Gamma$ does not contain any hyperbolic elements, or in other words, if $\Orbi$ does not have any periodic geodesics, then the infinite product in~\eqref{eq:szf_infprod_intro} is void and hence equals~$1$. We will here omit the discussion of these degenerate situations.

There exists a huge variety of methods, approaches and tools that allow us to 
enhance our understanding of the resonances of~$\Orbi$ by taking advantage of 
their interpretation as zeros of the Selberg zeta function~$Z_\Orbi$. Among these are one-parameter families of transfer operators~$\{\TO s^\Orbi\}_{s\in\C}$ that arise from well-chosen discretizations of the geodesic flow on~$\Orbi$ and that allow a representation of the Selberg zeta function~$Z_\Orbi$ as a Fredholm determinant:
\begin{equation}\label{eq:szf_fredholm_intro}
 Z_\Orbi(s) = \det\left( 1 - \TO s^\Orbi \right)\,.
\end{equation}
The identity~\eqref{eq:szf_fredholm_intro} converts the search for zeros of~$Z_\Orbi$ (and hence for resonances of~$\Orbi$) to a question on the existence of eigenfunctions with eigenvalue~$1$ of the transfer operator~$\TO s^\Orbi$ with parameter~$s$. \textit{A priori}, the latter seems to be the more involved problem. However, this conversion allows us to take advantage of, e.g., functional-analytic properties of~$\TO s^\Orbi$ for investigations. This is exactly the reason why a relation between~$Z_\Orbi$ and transfer operators as in~\eqref{eq:szf_fredholm_intro} is rather powerful, as the list of sample results below indicates. However, such a relation is not yet known for all hyperbolic orbisurfaces. 

In such studies of resonances, sometimes also a \emph{twisted variant} of the 
Selberg zeta function and of~\eqref{eq:szf_fredholm_intro} features. For its 
definition we let $\chi\colon\Gamma\to\Unit(V)$ be a unitary representation 
(the \emph{twist}) of the Fuchsian group~$\Gamma$ on a finite-dimensional unitary vector space~$V$. Further we denote the set of prime periodic unit speed 
geodesics on~$\Orbi$ by $P_\Orbi$, where we consider geodesics as identical if 
they differ only in their time parametrization. For $\widehat\gamma\in P_\Orbi$ we let $\ell(\widehat\gamma)$ denote the length of~$\widehat\gamma$, and we let 
$h_{\widehat\gamma}$ be the certain specific hyperbolic element of~$\Gamma$ that 
characterizes~$\widehat\gamma$ (see Section~\ref{SUBSEC:geodsurfaces}). The \emph{$\chi$-twisted Selberg zeta function}~$Z_{\Orbi,\chi}$ is then the infinite product
\begin{equation}\label{eq:szf_twist_intro}
 Z_{\Orbi,\chi}(s) \coloneqq \prod_{\widehat\gamma\in P_\Orbi} \prod_{k=0}^\infty \det\left(1 - \chi(h_{\widehat\gamma})e^{-(s+k)\ell(\widehat\gamma)}\right)\,,
\end{equation}
which converges for $\Rea s > \delta$, and its meromorphic continuation to all 
of~$\C$. For some hyperbolic orbisurfaces, a family of twisted transfer 
operators~$\{\TO {s,\chi}^\Orbi\}$ is available that allows us to find the 
twisted analogue of~\eqref{eq:szf_fredholm_intro}:
\begin{equation}\label{eq:szf_fredholm_twist_intro}
 Z_{\Orbi,\chi}(s) = \det\left(1-\TO {s,\chi}^\Orbi\right)\,.
\end{equation}
We list a few examples that make crucial use of~\eqref{eq:szf_fredholm_intro} 
or~\eqref{eq:szf_fredholm_twist_intro}.
\begin{itemize}
\item Prime geodesic theorems, also including error terms~\cite{PP_annals, 
PP_asterisque, Pollicott_Sharp_experr, Pollicott_Sharp_err, Naud_expanding, 
Oh_Winter} (some of these works use a variant of~\eqref{eq:szf_fredholm_intro} 
or~\eqref{eq:szf_fredholm_twist_intro} or are for other spaces, but are 
nevertheless good examples).
\item Relations between Patterson--Sullivan distributions and Wigner 
distributions~\cite{Anantharaman_Zelditch}.
\item Numerical investigations of resonances as in~\cite{Borthwick_numerics, 
BPSW}.
\item Distribution and counting results for resonances~\cite{Guillope_Lin_Zworski, JNS, NPS, Pohl_Soares}.
\item Meromorphic continuation of Selberg zeta functions~\cite{Fried_zetafunctionsI, Ruelle_thermo, Ruelle_dynzeta, Ruelle_zeta, Mayer_thermo, Mayer_thermoPSL, Pollicott, Morita_transfer,  Chang_Mayer_extension, Mayer_Muehlenbruch_Stroemberg, Moeller_Pohl, Pohl_hecke_infinite, Pohl_representation, FP_NECM} (some of these results use variants of~\eqref{eq:szf_fredholm_intro} to compensate for non-exact codings)
\end{itemize}

Also Selberg zeta functions twisted by \emph{non-unitary} representations are of interest in mathematics and physics. Recently, it could be shown, see~\cite{FP_NECM}, that for any choice of~$\Gamma$ and any finite-dimensional representation~$\chi\colon\Gamma\to\GL(V)$ with non-expanding cusp monodromy (i.e., for each parabolic element~$p\in\Gamma$, all eigenvalues of the endomorphism~$\chi(p)$ are of modulus~$1$), the infinite product in~\eqref{eq:szf_twist_intro} converges on some right half-plane in~$\C$. For representations~$\chi$ without non-expanding cusp monodromy, the infinite product in~\eqref{eq:szf_twist_intro} diverges. Thus, the class of representations with non-expanding cusp monodromy is the maximal realm of twists for which a theory of Selberg zeta functions, departing at an infinite product as in~\eqref{eq:szf_twist_intro}, can be expected. We emphasize that the class of representations with non-expanding cusp monodromy contains all unitary representations. 

Regarding the meromorphic continuability of these twisted infinite products, the situation is not yet fully settled, in particular not for the non-unitary representations contained in this class. In~\cite{FP_NECM} it is shown that if there exists a so-called \emph{strict transfer operator approach} for the Fuchsian group~$\Gamma$ (or the hyperbolic orbisurface~$\Orbi$), then the infinite product in~\eqref{eq:szf_twist_intro} extends meromorphically to all of~$\C$ for any twist~$\chi$ with non-expanding cusp monodromy. In a nutshell, a strict transfer operator approach refers to the presence of a particularly well-structured transfer operator family~$\{\TO s^\Orbi\}_s$ satisfying the Fredholm determinant identity~\eqref{eq:szf_fredholm_intro} for the untwisted Selberg zeta function. A bit more precisely, it refers to the existence of a particularly well-structured discrete dynamical system on a family of real intervals that yields a bijection between its orbits and a certain large set of geodesics on~$\Orbi$ (including all periodic geodesics), is uniformly expanding and piecewise given by the linear fractional transformation action of a certain generating set of elements of~$\Gamma$. We refer to Section~\ref{SUBSEC:stricttrans} for the proper definition and to Section~\ref{SUBSEC:fasttrans} for the well-structuredness. Strict transfer operator approaches are well-known to exist, e.g., for all Fuchsian Schottky groups (using the Koebe--Morse coding induced by a standard choice of a fundamental domain, see, e.g., \cite{Patterson_Perry, Borthwick_book}) and for all (cofinite as well as non-cofinite) Hecke triangle groups \cite{Moeller_Pohl, Pohl_hecke_infinite, Pohl_representation}. Also the symbolic dynamics in~\cite{Series} and the transfer operators in~\cite{Mayer_thermo, Chang_Mayer_extension} give rise or arise from strict transfer operator approaches. This is much likely also true for several cases in, e.g., \cite{Ruelle_zeta, Fried_zetafunctionsI, Bowen, Pollicott, Bowen_Series, Morita_transfer, Fried_triangle, Dalbo_Peigne}. Despite all these results, for most Fuchsian groups the situation regarding strict transfer operator approaches remains yet unclear. One of our main motivations for this article is to discuss this situation for a large class of Fuchsian groups. Indeed, we will provide strict transfer operator approaches for a huge class of non-cocompact, geometrically finite Fuchsian groups with at least one hyperbolic element or, equivalently, for all non-compact, geometrically finite hyperbolic orbisurfaces with at least one periodic geodesic. 

Considering the identities~\eqref{eq:szf_fredholm_intro} and~\eqref{eq:szf_fredholm_twist_intro}, the natural question arises on the relation between the eigenfunctions with eigenvalue~$1$ of~$\TO {s,\chi}^\Orbi$ and Laplace eigenfunctions, thereby reaching beyond the relation between the existence of transfer operator eigenfunctions and the spectral parameters of Laplace eigenfunctions (or resonances). This leads to another main motivation for this article. For a growing class of hyperbolic orbisurfaces, an interpretation of Laplace eigenfunctions as transfer operator eigenfunctions can indeed be established. More precisely, for geometrically finite, non-compact hyperbolic orbisurfaces~$\Orbi$, well-chosen discretizations of the geodesic flow on~$\Orbi$ give rise to transfer operators that are \emph{finite} sums of certain elementary operators (see Section~\ref{SUBSEC:slowtrans}) \cite{Pohl_Symdyn2d, Pohl_diss, Wabnitz}. These operators are usually called \emph{slow} transfer operators, denoted~$\TO {s,\chi}^{\Orbi, \slow}$, due to step-by-step encodings of windings of geodesics in cusp ends of~$\Orbi$. If $\Orbi$ has cusps and is of finite area, then the (untwisted) Maass cusp forms of~$\Orbi$ with spectral parameter~$s$ (rapidly decaying Laplace eigenfunctions with eigenvalue~$s(1-s)$) are bijective to eigenfunctions with eigenvalue~$1$ of~$\TO s^{\Orbi,\slow}$ \cite{Moeller_Pohl, Pohl_gamma, Pohl_mcf_general}. Recently such a bijection could also be established for a class of hyperbolic orbisurfaces of \emph{infinite} area, namely the Hecke triangle orbisurfaces of infinite area, and more general Laplace eigenfunctions~\cite{Bruggeman_Pohl}. The transfer operators~$\TO {s,\chi}^{\Orbi,\slow}$ are typically non-nuclear and hence do not satisfy the identities~\eqref{eq:szf_fredholm_intro} or~\eqref{eq:szf_fredholm_twist_intro}.

However, for the family of Hecke triangle orbisurfaces of finite as well as of infinite area, a \emph{cuspidal acceleration} procedure was developed in which the discretizations of the geodesic flow that led to~$\TO {s,\chi}^{\Orbi,\slow}$ are modified so that all successive windings in the cusp are encoded in a single step. The cuspidal acceleration gives rise to the family of \emph{fast} transfer operators~$\TO {s,\chi}^{\Orbi,\fast}$, which are indeed nuclear and satisfy~\eqref{eq:szf_fredholm_intro} and~\eqref{eq:szf_fredholm_twist_intro} \cite{Moeller_Pohl, Pohl_hecke_infinite, Pohl_representation, Bruggeman_Pohl}. Further, the eigenfunctions with eigenvalue~$1$ of~$\TO {s,\chi}^{\Orbi,\fast}$ are bijective to those of~$\TO {s,\chi}^{\Orbi,\slow}$ with suitable regularity~\cite{Adam_Pohl_iso_hecke, Bruggeman_Pohl}, which implies further observations for which we refer to~\cite{Adam_Pohl_iso_hecke}. The strongest results in the list above indeed use these fast transfer operators. The fast transfer operators are also useful for characterizing the necessary subspaces of eigenfunctions of slow transfer operators used for the transfer-operator-based interpretation of Laplace eigenfunctions. For completeness, we mention that also the transfer operators and the functional equation in~\cite{Mayer_thermo, Mayer_thermoPSL, Chang_Mayer_extension, Lewis_Zagier} fit into this slow-fast transfer operator scheme~\cite{Pohl_oddeven}. For the setup in~\cite{Bruggeman_Muehlenbruch} it is not yet clarified. 

All these applications of fast transfer operators are reasons for the interest in developing a cuspidal acceleration procedure for more general hyperbolic orbisurfaces and discretizations. While the procedure for Hecke triangle groups makes crucial use of the explicit and rather simple structure of the initial discretizations and the slow transfer operators as well as some special features of Hecke triangle orbisurfaces and could be done hands-on, we here provide a general construction of cuspidal accelerations that is applicable for all (slow) discretizations in~\cite{Pohl_Symdyn2d, Pohl_diss, Wabnitz} and beyond. 

\addtocontents{toc}{\SkipTocEntry}
\subsection*{Main results}

With this article we provide \emph{particularly well-structured strict transfer operator approaches} for a large class of non-uniform, geometrically finite Fuchsian groups. We do not only establish their existence but indeed present algorithmic constructions. 

The starting point for our constructions are cross sections for the geodesic flow on the considered hyperbolic orbisurface that need to satisfy certain structural, but not too restrictive requirements. Each such cross section will henceforth be called a \emph{set of branches}. We refer to Section~\ref{SUBSEC:cuspexp} for a precise, axiomatic definition. As soon as a set of branches can be provided, our construction will produce a strict transfer operator approach; no further restrictions on the hyperbolic orbisurfaces need to be made.  In rough terms, our principal result can be stated as follows. We refer to Theorem~\ref{mainthm} for the detailed statement.

\begin{theorem}\label{thmA}
 Let $\Gamma$ be a geometrically finite Fuchsian group with at least one hyperbolic element that admits a set of branches. Then $\Gamma$ admits a strict transfer operator approach. The constructions in Sections~\ref{SEC:branchred}--\ref{SEC:strictTOAexist} provide explicit examples of strict transfer operator approaches for~$\Gamma$. 
\end{theorem}

Examples of sets of branches exist already for a large class of non-uniform, geometrically finite Fuchsian groups or, equivalently, a large class of non-compact, geometrically finite hyperbolic orbisurfaces (including large classes of hyperbolic orbisurfaces of infinite area). See, e.g.,~\cite{Pohl_Symdyn2d, Pohl_gamma, Pohl_mcf_general, Wabnitz}. Also some Koebe--Markov codings give rise to sets of branches as, e.g., in the case of Schottky groups. See, e.g.,~\cite{Patterson_Perry, Borthwick_book, Bowen_Series, Katok_beyond}. We expect that forthcoming articles will extend the realm of explicit sets of branches to \emph{all} non-compact, geometrically finite hyperbolic orbisurfaces. We further expect that several other examples of sets of branches can be developed. Indeed, one strength of our approach is an axiomatization of properties of cross sections that are sufficient for our constructions. We emphasize that our constructions are uniform for all sets of branches. 

In a nutshell, as already indicated above, strict transfer operator approaches yield one-parameter families of well-structured (fast) transfer operators which are nuclear of order zero on certain Banach spaces (and other suitable spaces), can accommodate any finite-dimensional twist with non-expanding cusp monodromy and---most importantly---who represent the Selberg zeta function via its Fredholm determinants:
\begin{equation}\label{eq:szf_fredholm_fast}
Z_{\Orbi,\chi}(s) = \det\left(1-\TO {s,\chi}^{\Orbi,\fast}\right)\,.
\end{equation}
See also~\eqref{eq:szf_fredholm_twist_intro}. Here, $\Gamma$ is the considered Fuchsian group, $\Orbi=\quod{\Gamma}{\H}$ the associated hyperbolic orbisurface, $\chi\colon \Gamma\to \GL(V)$ the considered finite-dimensional representation of~$\Gamma$ with non-expanding cusp monodromy, and $(\TO {s,\chi}^{\Orbi,\fast})_{s\in\C}$ the provided family of fast transfer operators. 
The identity~\eqref{eq:szf_fredholm_fast} immediately yields a proof of the meromorphic continuation of the Selberg zeta function. See~\cite{FP_NECM}. Thus, an immediate consequence of Theorem~\ref{thmA} is the following statement.

\begin{coroll}
Let $\Gamma$ be a geometrically finite Fuchsian group with at least one hyperbolic element that admits a set of branches, and let $\chi$ be any finite-dimensional twist of~$\Gamma$ with non-expanding cusp monodromy. Then the Selberg zeta function associated to~$(\Gamma,\chi)$ has a meromorphic continuation to all of~$\C$ and is represented by the Fredholm determinant of a family of transfer operators.
\end{coroll}

Cross sections and discretizations for the geodesic flow on~$\Orbi=\quod{\Gamma}{\H}$ for which the associated transfer operator family yields the identity in~\eqref{eq:szf_fredholm_fast} do not need to be very particular. Their most essential properties are uniform expansiveness and a bijection between periodic geodesics of~$\Orbi$ and periodic orbits in the discretization. Therefore, the transfer operator families typically do not exhibit any particular properties. 

However, the discretizations underlying the transfer operators in our constructions derive from sets of branches and hence from much more special cross sections. This also features as special properties of the operators. In particular, each \emph{fast} transfer operator~$\TO {s,\chi}^{\Orbi,\fast}$ is accompanied by a \emph{slow} transfer operator~$\TO {s,\chi}^{\Orbi,\slow}$, which is of a different structure and has different functional analytic properties. We refer to Sections~\ref{SUBSEC:slowtrans}--\ref{SUBSEC:fasttrans} for detailed information. 

For several choices of the Fuchsian group~$\Gamma$ and the representation~$\chi$, it is known that certain classes of eigenfunctions with eigenvalue~$1$ of the slow transfer operator~$\TO {s,\chi}^{\Orbi,\fast}$ are in isomorphism with eigenfunctions with eigenvalue~$s(1-s)$ of the Laplacian~$\Delta$ of~$\Orbi$, by a proof \emph{not} relying on the Selberg zeta function or Selberg trace formula. See~\cite{Lewis_Zagier_survey, Lewis_Zagier, Bruggeman97, Chang_Mayer_eigen, Bruggeman_Muehlenbruch, Deitmar_Hilgert} and \cite{Bruggeman_Lewis_Zagier, Bruggeman_Pohl, Pohl_oddeven, Pohl_gamma, Pohl_mcf_general}. A rather easily accessible survey is provided in~\cite{Pohl_Zagier}. In some of these cases, it could also be shown that eigenfunctions with eigenvalue~$1$ of the slow transfer operator~$\TO {s,\chi}^{\Orbi,\slow}$ are isomorphic to eigenfunctions with eigenvalue~$1$ of the fast transfer operator~$\TO {s,\chi}^{\Orbi,\fast}$. See~\cite{Chang_Mayer_eigen, Chang_Mayer_extension, Moeller_Pohl, Adam_Pohl_iso_hecke, Bruggeman_Pohl}. Since the eigenfunctions with eigenvalue~$1$ of the fast transfer operator determine the zeros of the Selberg zeta function by~\eqref{eq:szf_fredholm_fast}, the full relation provide a spectral interpretation of (some parts of) the zeros of the Selberg zeta function that is alternative to proofs using the Selberg trace formula.

The isomorphism between eigenfunctions of fast and slow transfer operators is currently established for Hecke triangle groups only, due to the previously limited availability of explicit fast and slow transfer operators in the literature. However, this proof takes advantage of a certain structural relation between fast and slow transfer operators for Hecke triangle groups. The pairs of fast and slow transfer operators that we provide here for a larger class of Fuchsian groups enjoy a similar structural relation. Therefore we expect that our article will help to complete the realm of hyperbolic orbisurfaces for this line of research. 

We now provide a brief survey of our construction. A set of branches for the geodesic flow on~$\Orbi = \quod{\Gamma}{\H}$ is a subset of~$S\H$ with essentially the following properties: Its base set in~$\H$ is a finite number of complete geodesics. Clearly, each such geodesic splits~$\H$ into two (open) totally geodesic half-spaces (and the geodesic itself). For each such geodesic, the set of branches contains the unit tangent vectors that are based on the geodesic and point into one of the half-spaces. The geodesics and half-spaces are chosen in such a way that the set of branches descends to a cross section for the geodesic flow on~$\Orbi$. Further, tracking the intersection sequence of a generic geodesic~$\gamma$ on~$\H$ with the set of branches and its $\Gamma$-translates encodes the geodesic by a bi-infinite sequence~$(g_n)_{n\in\Z}$ of elements in~$\Gamma$. In addition, this sequence can be fully determined from the two terminal endpoints of the geodesic~$\gamma$ \emph{combined} with the knowledge of which complete geodesic in the base set of the set of branches is intersected by~$\gamma$ (the ``intersection at time zero''). In particular, the ``positive part'' of the sequence, i.e.~$(g_n)_{n\in\N}$, can be determined from the endpoint of~$\gamma$ in positive time direction, i.e.~$\gamma(+\infty)$, and the intersection at time zero. 

In this way, any set of branches induces a discrete dynamical system which is defined as the action by linear fractional transformations of certain, finitely many elements of~$\Gamma$ on a disjointified union of certain (open connected) subsets of~$P^1(\R)$. It is this discrete dynamical system, which gives rise to well-structured slow transfer operators. (See the discussion above for the significance of slow transfer operators.) E.g., for $\Gamma=\PSL_2(\Z)$, a certain choice of set of branches yields the classical Farey algorithm, also known as the slow continued fraction algorithm. Also for any other generic Fuchsian groups, any set of branches yields a slow continued fraction-type algorithm and hence a discrete dynamical system which typically is \emph{not uniformly expanding}. 

Our task is to modify this discrete dynamical system into one which is \emph{uniformly expanding} without loosing any of its dynamical content. We will  achieve this by modifying the set of branches, and hence the choice of cross section for the geodesic flow on~$\Orbi$. These modifications will be done in a number of steps, each of which serves a different purpose. We emphasize that all steps are algorithmic and that the final object typically is not a set of branches anymore, but still recovers the dynamical essence of the geodesic flow. 

In a first step, we perform a \emph{branch reduction} algorithm. This step is not strictly necessary. However, it reduces the complexity by eliminating branches from the set of branches that do not carry unique information. Moreover, it yields additional slow discretizations and slow discrete dynamical systems associated to the geodesic flow on~$\Orbi$, which are sometimes useful for other purposes. 

For this algorithm, the notion of a successor of a branch (from the set of branches) is essential: for any given branch, say~$\Cs{i}$, any $\Gamma$-translate of a branch, say $g\act\Cs{j}$, is called a direct successor of~$\Cs{i}$ if there exists a unit tangent vector~$\nu\in\Cs{j}$ such that the geodesic~$\gamma_\nu$ determined by~$\nu$ intersects $g\act\Cs{j}$ before intersecting any other $\Gamma$-translate of a branch (in positive time direction). The branch reduction algorithm is separated into two parts. 

In the first part of this algorithm we iteratively eliminate branches that have only a single direct successor, in case that this successor is not a $\Gamma$-translate of the considered branch. In the second part, we iteratively eliminate branches that do not have any $\Gamma$-translate of itself among its direct successors. The resulting object will still be a set of branches. In terms of the initial coding sequence~$(g_n)_n$ of a geodesic, the branch reduction algorithm means that successive elements are combined to a single one if their product variant does not carry any information. 

In the second step, we perform an \emph{identity elimination} algorithm. In contrast to the first step, the second step is strictly necessary. Here, we remove any \emph{inner identities} in iterated successors, i.e., any occurrences  of unit tangent vectors $\nu$ in the set of branches such that the geodesic~$\gamma_\nu$ determined by~$\nu$ intersects again the set of branches, not only a  nontrivial $\Gamma$-translate of a branch. In other words, for the coding sequence~$(g_n)_n$ of a geodesic, initially it might happen that some products of  successive elements combine to the identity element. The identity elimination algorithm eliminates unit tangent vectors from the set of branches in such a way that for the coding sequences arising from the new cross section such phenomenon does not appear. 

The need of identity elimination is explained as follows: the transformations appearing in the coding sequences from a set of generators for~$\Gamma$ (as a group). For a strict transfer operator approach it is required that periodic geodesics have unique coding sequences in terms of the chosen set of generators.  In other words, certain hyperbolic elements of~$\Gamma$ must have unique presentations as product of the elements in this set of generators. We emphasize that the discrete dynamical systems arising from these two steps are typically non-uniformly expanding.

The transition to a uniformly expanding discrete dynamical system is achieved in the third step, the \emph{cuspidal acceleration}. In the situations considered here, non-uniformity in the expansion is always caused by the presence of cusps combined with the property of the \emph{slow} discretizations to encode windings of geodesics around cusps in a \emph{slow} way. That is, each single winding is encoded separately. In a nutshell, the cuspidal acceleration algorithm modifies this discretization and the underlying cross section in such a way that all successive windings are encoded simultaneously. The precise algorithm is rather involved, more precisely, the proof that the arising set of vectors is indeed a cross section with all claimed features is rather involved. Therefore this third step constitutes the main bulk of this article. 

We now give an overview of the structure of this article. In Section~\ref{SEC:prelimi} we survey the elements of hyperbolic geometry that we will need throughout. In Section~\ref{SEC:crosssec} we present the concept of sets of branches, which are the starting objects for our constructions. In addition we discuss their essential properties as well as their relation to cross sections for the geodesic flow, coding sequences of geodesics, associated discrete dynamical systems and slow transfer operators. Furthermore, we recapitulate the concept of strict transfer operator approaches and their relation to fast transfer operators and Selberg zeta functions. In Section~\ref{SEC:branchred} we initiate our constructions with the branch reduction algorithm. In Section~\ref{SEC:stepredux} we present the identity elimination algorithm. The main bulk of the construction, the cuspidal acceleration algorithm, is discussed in Section~\ref{SEC:cuspacc}. In Sections~\ref{SEC:structacc}--\ref{SEC:strictTOAexist} we prove that the arising discrete dynamical systems indeed provide a strict transfer operator approach. Throughout, all constructions are illustrated with small examples. For the convenience of the reader, we provide a detailed long example in the final Section~\ref{SEC:longexample}. In addition, the appendix contains an extensive index of terminology as well as a list of notations.

\addtocontents{toc}{\SkipTocEntry}
\subsection*{Acknowledgement} 
This research was funded by the Deutsche Forschungsgemeinschaft (DFG, German 
Research Foundation) -- project no.~264148330 and PO~1483/2-1. Both authors 
wish to thank the Hausdorff Institute for Mathematics in Bonn for excellent working 
conditions during the HIM trimester program ``Dynamics: Topology and Numbers,'' 
where part of this manuscript was prepared.

\section{Elements of hyperbolic geometry}\label{SEC:prelimi}

In this section we briefly present the background material on hyperbolic orbisurfaces and their geodesic flows that is necessary for our investigations. 
Comprehensive treatises, including proofs for all statements that we leave unproven, can be found in the many excellent textbooks on hyperbolic geometry.
We refer in particular to~\cite{Anderson, Beardon, Ratcliffe, Stahl, Katok_fuchsian, Lehner, Iversen}. 

We will use throughout standard notations such as $\N$, 
\index[symbols]{Na@$\N$}%
$\N_0$ 
\index[symbols]{Naa@$\N_0$}%
and~$\Z$ 
\index[symbols]{Z@$\Z$}%
for the set of positive numbers, non-negative numbers and all integers, respectively.
We use~$\R$
\index[symbols]{R@$\R$}%
and~$\C$
\index[symbols]{C@$\C$}%
for the set of real and complex numbers, respectively, both equipped with the Euclidean topology.
We write~$\SL_2(\R)$ 
\index[symbols]{SL@$\SL_2(\R)$}%
for the group of $2\times2$-matrices of determinant~$1$ with real entries, and equip~$\SL_2(\R)$ with the subspace topology of~$\R^4$.
We denote closed, semi-closed/half-open and open intervals
\index[symbols]{intera@$[\cdot,\cdot]$}%
\index[symbols]{interb@$(\cdot,\cdot]$}%
\index[symbols]{interc@$[\cdot,\cdot)$}%
\index[symbols]{interd@$(\cdot,\cdot)$}%
\index[defs]{interval}%
in~$\R$ by $[a,b]$, $(a,b]$, $[a,b)$ and~$(a,b)$ for any $a,b\in\R$, respectively,
or, if applicable also for $a,b\in\R\cup\{\pm\infty\}$.
In case that we need to stress that intervals are real, we endow them with the subscript~$\R$, e.g., $(a,b)_\R$ (cf.~\ref{SUBSEC:geodplane}).
\index[symbols]{interea@$[\cdot,\cdot]_\R$}%
\index[symbols]{intereb@$(\cdot,\cdot]_\R$}%
\index[symbols]{interec@$[\cdot,\cdot)_\R$}%
\index[symbols]{intered@$(\cdot,\cdot)_\R$}%
For any $z\in\C$ we denote its real part~$\Rea z$ and its imaginary part~$\Ima z$.
\index[symbols]{Re@$\Rea$}%
\index[symbols]{Im@$\Ima$}%

\subsection{The hyperbolic plane}\label{SUBSEC:hypgeo}

As model for the hyperbolic plane we will use the \emph{upper half-plane} 
\[
\H \coloneqq \defset{z\in\C}{\Ima z>0},
\]
\index[symbols]{H@$\H$}\index[defs]{upper half-plane}%
\index[defs]{hyperbolic plane}%
endowed with the hyperbolic metric given by the line element 
\[
\dif s^2_z \coloneqq (\Ima z)^{-2}\dif z\dif\overline{z}
\]
\index[symbols]{ds@$\dif s^2$}%
at any $z\in\H$.
The \emph{geodesic boundary}~$\partial_{\geo}\H$ of~$\H$
\index[symbols]{Hgeo@$\partial_\geo\H$}%
\index[defs]{geodesic boundary}%
can and shall be identified, in the obvious way, with the Alexandroff extension (one-point compactification)
\[
\widehat{\R}\coloneqq\R\cup\left\{\infty\right\}
\]
\index[symbols]{Ra@$\widehat\R$}%
\index[symbols]{10@$\widehat{\cdot}$}%
\index[defs]{one-point compactification}%
\index[defs]{Alexandroff compactification}%
of the real line~$\R$.
Likewise, we understand the \emph{geodesic closure}
\[
 \overline\H^{\geo}\coloneqq\H\cup\partial_{\geo}\H
\]
\index[symbols]{Hclosgeo@$\overline{\H}^\geo$}%
\index[defs]{geodesic closure}%
of~$\H$ as a subset of the Alexandroff compactification 
\[
 \widehat{\C} \coloneqq \C\cup\left\{\infty\right\}
\]
\index[symbols]{Ca@$\widehat\C$}%
\index[defs]{one-point compactification}%
\index[defs]{Alexandroff compactification}%
of~$\C$.
In the topology of~$\widehat\C$, the geodesic boundary~$\partial_\geo\H$ of~$\H$ is indeed the topological boundary of~$\h$, and the geodesic closure~$\overline\H^\geo$ is the topological closure of~$\h$. 
The topology of~$\overline\H^\geo$ can also be characterized intrinsically, 
most conveniently by taking advantage of the Riemannian isometries of~$\h$.
We will recall this characterization in the next subsection, but will not make use of 
it here.

For a subset~$K$ of~$\H$, the closure of~$K$ in the topology of~$\H$ may differ 
from its closure in the topology of~$\overline\H^\geo$.
We will write $\overline{K}$
\index[symbols]{100@$\overline{\;\cdot\;}$}%
\index[defs]{closure}%
for its closure in~$\H$, and $\overline{K}^{\geo}$
\index[symbols]{101@$\overline{\;\cdot\;}^\geo$}%
\index[defs]{geodesic closure}%
for its closure in~$\overline\H^\geo$.
Further, we will write $\partial K$
\index[symbols]{105@$\partial$}%
\index[defs]{boundary}%
for the boundary of~$K$ in~$\H$, and $\partial_{\geo}K$
\index[symbols]{106@$\partial_\geo$}%
\index[defs]{boundary}%
for its boundary in~$\overline{\H}^{\geo}$.
The \emph{geodesic boundary of~$K$}, that is the part of~$\partial_{\geo}K$ which is contained in~$\partial_{\geo}\H$, will be denoted by~$\gbound{K}$.
\index[symbols]{107@$\gbound$}%
\index[defs]{geodesic boundary}%
E.g., for 
\[
 K \coloneqq \defset{ z \in\H }{ 0<\Rea z < 1 }
\]
we have 
\begin{align*}
 \overline K & = \defset{ z\in\H }{ 0\leq \Rea z \leq 1 }\,,
 \\
 \overline K^\geo & = \defset{ z\in\C }{ 0\leq \Rea z\leq 1,\ \Ima z\geq 0 } 
\cup 
\{\infty\}
 \\
 & = \overline K \cup [0,1] \cup\{\infty\}\,,
 \intertext{and}
 \gbound K & = [0,1] \cup\{\infty\}\,.
\end{align*}
The sets of inner points of~$K$ in the two topologies coincide and are 
therefore denoted with the same symbol,~$K^{\circ}$.
\index[symbols]{107@$(\;\cdot\;)^\circ$}%
\index[defs]{inner points}%

We will require the following extension of the notion of intervals in~$\R$ to 
\emph{intervals in~$\widehat\R$}.
\index[defs]{interval}%
For any $a,b\in\R$, $a\not=b$, we let 
\[
 (a,b)_c \coloneqq 
 \begin{cases}
  (a,b) & \text{if $a<b$}
  \\
  (a,+\infty) \cup \{\infty\} \cup (-\infty, b) & \text{if $a>b$}.
 \end{cases}
\]
be the \emph{open interval} in~$\widehat\R$ from~$a$ to~$b$.
For $a=\infty \in\widehat\R$ and $b\in\R$ we set 
\[
 (a,b)_c = (\infty, b)_c = (-\infty, b)\,,
\]
and analogously we define $(b,a)_c$.
\index[symbols]{interaaw@$[\cdot,\cdot]_c$}%
\index[symbols]{interbaw@$(\cdot,\cdot]_c$}%
\index[symbols]{intercaw@$[\cdot,\cdot)_c$}%
\index[symbols]{interdaw@$(\cdot,\cdot)_c$}%
We define semi-open and closed intervals in~$\widehat\R$ in the obvious, analogous way.
The subscript~$c$ refers to the cyclic order of~$\widehat\R$ that is used implicitly in this definition.
We remark that singletons in~$\widehat\R$ and the empty set cannot be defined consistently 
within this notation.
Occasionally we will call the boundary points of intervals \emph{endpoints}.
For subsets~$M$ of~$\wh\R$ we write~$M^\circ$ for the interior and~$\overline{M}$ for the closure of~$M$ in~$\wh\R$.
\index[symbols]{107@$(\;\cdot\;)^\circ$}%
\index[defs]{inner points}%
\index[symbols]{100@$\overline{\;\cdot\;}$}%
\index[defs]{closure}%

In order to distinguish between the point~$\infty$ in~$\wh\R$ and the two infinite endpoints of~$\R$ with its standard order, we will write $\pm\infty$ whenever we refer to the latter ones and use the extended standard order of 
$\R\cup\{\pm\infty\}$ (i.e., $-\infty < r < +\infty$ for all $r\in\R$).
The unsigned symbol~$\infty$ will always refer to the point in~$\wh\R$.
As usual, we consider $\R$ to be embedded into~$\wh\R$.
In particular, we have 
\[
 \wh\R = (-\infty, +\infty) \cup \{\infty\}\,.
\]

\subsection{Riemannian isometries}\label{SUBSEC:riemisom}

The \emph{group of orientation-preserving Riemannian isometries} of~$\H$ is known to be isomorphic to
\begin{equation}\label{eq:def_PSLR}
\PSLR \coloneqq \SL_2(\R) \slash \left\{\pm\id\right\},
\end{equation}
\index[symbols]{PSL@$\PSLR$}%
considered as acting on~$\H$ from the left by linear fractional transformations (M\"obius transformations).
We note that each element~$g\in\PSLR$ has exactly two representatives in~$\SL_2(\R)$.
If $\textmat{a}{b}{c}{d}\in\SL_2(\R)$ is a representative of~$g\in\PSLR$, then the second representative of~$g$ is 
$\textmat{-a}{-b}{-c}{-d}$.
In this case, we shall write $g=\textbmat{a}{b}{c}{d}$ (with square brackets). 
\index[symbols]{801@$\textbmat{a}{b}{c}{d}$}%
\index[symbols]{800@$\textmat{a}{b}{c}{d}$}%
With respect to this notation and identification, the action of~$\PSLR$ on~$\H$ is given by 
\[
 \bmat{a}{b}{c}{d}\act z \coloneqq \frac{az+b}{cz+d}
\]
\index[defs]{action!on~$\H$}%
\index[symbols]{810@$g\act z$}%
for any $g=\textbmat{a}{b}{c}{d}\in\PSLR$ and $z\in\H$.
The action of~$\PSLR$ on~$\H$ is conformal, transitive and faithful.

The action of~$\PSLR$ on~$\H$ extends smoothly to an action of~$\PSLR$ 
on~$\overline\H^\geo$.
For any $g = \textbmat{a}{b}{c}{d}\in\PSLR$ and $z \in \overline\H^{\geo}$ we have
\begin{equation}\label{eq:def_act_PSLR}
g\act z \coloneqq 
\begin{dcases} \frac{a}{c} & \text{if } z = \infty\,, 
\\ \frac{az+b}{cz+d} & \text{otherwise}\,,
\end{dcases}
\end{equation}
\index[defs]{action!on~$\overline\H^{\geo}$}%
\index[symbols]{810@$g\act z$}%
where $g\act z \coloneqq \infty$ whenever the denominator on the right hand side of~\eqref{eq:def_act_PSLR} vanishes.
As for~$\H$, the action of~$\PSLR$ restricted to~$\partial_\geo\H$ is transitive and faithful.
Throughout we will identify the elements in~$\PSLR$ with their action on~$\overline\H^\geo$.

Each element in~$\PSLR$ has at least one fixed point in~$\overline{\H}^{\geo}$.
The identity element
\[
\id = \bmat{1}{0}{0}{1}\in\PSLR
\]
\index[symbols]{id@$\id$}%
is the unique linear fractional transformation that fixes (at least) two points 
in~$\H$.
It is also the unique element with three fixed points in~$\overline{\H}^{\geo}$.
Any other element has at most one fixed point in~$\H$ and at most two fixed points in the geodesic closure~$\overline\H^\geo$. 
Let $g\in\PSLR$, $g\not=\id$.
If $g$ fixes exactly one point in~$\overline\H^\geo$, say~$c$, then $g$ is called \emph{elliptic} 
\index[defs]{elliptic element}%
\index[defs]{element!elliptic}%
if $c\in\H$, and \emph{parabolic}
\index[defs]{parabolic element}%
\index[defs]{element!parabolic}%
if $c\in\partial_{\geo}\H$. If $g$ fixes two points in~$\overline\H^\geo$, then both fixed points are contained in~$\partial_\geo\H$ and $g$ is called \emph{hyperbolic}.
\index[defs]{hyperbolic element}%
\index[defs]{element!hyperbolic}%
The classification of~$g$ being elliptic, parabolic or hyperbolic can be read off the value of the \emph{unsigned trace of~$g$}. To that end we denote with~$\abs{\tr{g}}$ 
\index[symbols]{trace@$\tr{g}$}%
\index[defs]{trace}%
the absolute value of the trace of any matrix in~$\SL_2(\R)$ that represents~$g$. We note that while the trace may 
depend on the choice of this representative, its absolute value does not. 
One easily checks that the element~$g$ is elliptic if and only if~$\abs{\tr{g}}<2$. It is parabolic if and only if~$\abs{\tr{g}}=2$ (recall that $g\not=\id$). And it is hyperbolic if and only if~$\abs{\tr{g}}>2$.

The fixed points of any hyperbolic element~$g$ can be characterized as limit points of the iterated action of~$g$ or~$g^{-1}$ on~$\H$. With a little caution, this characterization extends to~$\overline\H^\geo$. 

\begin{lemdef}\label{LEM:hypfixedconv}
Let $g\in\PSLR$ be hyperbolic. 
\begin{enumerate}[label=$\mathrm{(\roman*)}$, ref=$\mathrm{\roman*}$]
\item For any $z\in\H$, the limit of $(g^n\act z)_{n\in\N}$ in~$\overline\H^\geo$ exists and is independent of the choice of~$z$. We call this limit point the \emph{attracting fixed point} or \emph{attractor of~$g$} and denote it by~$\fixp+g$.
\index[defs]{attracting fixed point}%
\index[defs]{attractor}%
\index[defs]{fixed point!attracting}%
\index[symbols]{fp@$\fixp+{g}$}%
\item For any $z\in\H$, the limit of $(g^{-n}\act z)_{n\in\N}$ in~$\overline\H^\geo$ exists and is independent of the choice of~$z$. We call this limit point the \emph{repelling fixed point} or the \emph{repeller of~$g$} and denote it by~$\fixp-g$.
\index[defs]{repelling fixed point}%
\index[defs]{repeller}%
\index[defs]{fixed point!repelling}%
\index[symbols]{fp@$\fixp-{g}$}%
\item For all $z\in\overline\H^\geo$ we have 
\begin{align*}
\lim_{n\to+\infty}g^n\act z & =
\begin{cases}
\fixp{-}{g} & \text{if $z=\fixp{-}{g}$}\,,
\\
\fixp{+}{g} & \text{otherwise},
\end{cases}
\intertext{and}
\lim_{n\to-\infty}g^n\act z & =
\begin{cases}
\fixp{+}{g} & \text{if $z=\fixp{+}{g}$}\,,
\\
\fixp{-}{g} & \text{otherwise}.
\end{cases}
\end{align*}
\end{enumerate}
\end{lemdef}

This characterization of the fixed points will be of utmost importance for our investigations. 
For the convenience of the reader we indicate a brief proof of Lemma~\ref{LEM:hypfixedconv}, albeit well-known.
To that end let $g\in\PSLR$ be hyperbolic. Then 
$g$ is conjugate within~$\PSLR$ to a unique element of the form 
\begin{align}\label{EQDEF:standh}
\standh{\ell}\coloneqq\begin{bmatrix}e^{\tfrac{\ell}{2}}&0\\0&e^{-\tfrac{\ell}{2
}}\end{bmatrix}
\end{align}
\index[symbols]{hl@$\standh{\ell}$}%
for some $\ell\in\R$, $\ell>0$. The number~$\ell$ is called the \emph{displacement length} of~$g$
\index[defs]{displacement length}%
and will be denoted by~$\ell(g)$.
\index[symbols]{l@$\ell$}%
\index[symbols]{l@$\ell(g)$}%
In Section~\ref{SUBSEC:geodsurfaces} we will recall its relation to periodic geodesics. The statement of Lemma~\ref{LEM:hypfixedconv} is easily verified for transformations of the form~\eqref{EQDEF:standh}, and it remains valid under conjugation in~$\PSLR$.

Taking advantage of the action of~$\PSLR$ on~$\overline\H^\geo$, we can conveniently provide an intrinsic characterization of the topology of~$\overline\H^\geo$. On the subset~$\H$ of~$\overline\H^\geo$ the topology is 
given by the Euclidean topology of~$\C$. A neighborhood basis at~$\infty$
\index[defs]{neighborhood of~$\infty$}%
\index[defs]{neighborhood!$\infty$}%
is given by the family 
\[
 \mathcal{U}_{\infty}\coloneqq\defset{U_\varepsilon}{\varepsilon>0}
\]
\index[symbols]{Ua@$\mathcal{U}_{\infty}$}%
consisting of the open sets
\[
U_\varepsilon\coloneqq\defset{z\in\H}{\Ima(z)>\varepsilon^{-1}}\cup\{\infty\}\,,
\quad \varepsilon>0\,.
\]
\index[symbols]{Ub@$U_\varepsilon$}%
Finally, since $\PSLR$ acts transitively on $\partial_{\geo}\H$, the images~$\PSLR\act\,\mathcal{U}_{\infty}$ of this neighborhood basis at~$\infty$ yield neighborhood bases at any point of~$\partial_{\geo}\H$, each one 
consisting of open sets.
\index[defs]{neighborhood of boundary point}%
\index[defs]{neighborhood!boundary point}%

\subsection{Geodesics on the hyperbolic plane}\label{SUBSEC:geodplane}

Throughout we will consider geodesics on~$\H$ as (certain) curves~$\R\to\H$ and require that they are parametrized by arc length (thus, \emph{unit speed geodesics}). 
\index[defs]{geodesic!on~$\H$}%
\index[defs]{geodesic!unit speed}%
\index[defs]{unit speed geodesic}%
We denote by~$\UTB\H$ the unit tangent bundle of~$\H$.
\index[defs]{unit tangent bundle!of~$\H$}%
\index[symbols]{SH@$\UTB\H$}%
Each unit tangent vector~$\nu\in\UTB\H$
\index[defs]{unit tangent vector}%
uniquely determines a geodesic on~$\H$, namely the (unit speed) geodesic~$\gamma_\nu$
\index[symbols]{geoa@$\gamma_\nu$}%
with derivative $\nu$ at time~$0$, thus
\begin{align}\label{EQDEF:geodeter}
\gamma_{\nu}^{\prime}(0)=\nu\,.
\end{align}
E.g., the geodesic determined by the unit tangent vector~$\partial_y\vert_i$ is
\begin{align}\label{EQDEF:geostandard}
 \gamma_s\coloneqq \gamma_{\partial_y\vert_i}\colon\R\to\H\,,\quad t\mapsto 
ie^t\,,
\end{align}
the \emph{standard geodesic} on~$\H$.
\index[defs]{geodesic!standard}%
\index[symbols]{geob@$\gamma_s$}%
Since the action of~$\PSLR$ on~$\H$ is by Riemannian isometries, it induces an action of~$\PSLR$ on~$\UTB\H$,
\index[defs]{action!on~$\UTB\H$}%
for which we use the same notation as for the action of~$\PSLR$ on~$\H$. The action of~$\PSLR$ on~$\UTB\H$ is simply transitive, and hence $\PSLR$ acts simply transitive on the set of all geodesics on~$\H$. Therefore any geodesic on~$\H$ is a (unique) $\PSLR$-translate of the standard geodesic~$\gamma_s$ from~\eqref{EQDEF:geostandard}. The \emph{(unit speed) geodesic flow} on~$\H$ is the dynamical system
\[
\Phi\colon\R\times\UTB\H \to \UTB\H\,,\quad 
(t,\nu)\mapsto\gamma^{\prime}_{\nu}(t)\,.
\]
\index[defs]{geodesic flow!on~$\H$}%
\index[defs]{unit speed geodesic flow}%
\index[symbols]{Pa@$\Phi$}%

From the characterization above of all geodesics we immediately obtain that the arcs (images) of the geodesics are generalized semicircles that are perpendicular to~$\partial_\geo\H$, and that each such generalized semicircle 
is the arc of a geodesic. We will require the following notion of \emph{hyperbolic intervals} or \emph{geodesic segments}.
\index[defs]{hyperbolic interval}%
\index[defs]{interval!hyperbolic}%
\index[defs]{geodesic segment}%
Let $\gamma\colon\R\to\H$ be a geodesic on~$\H$. We let 
\[
\gamma(\pm\infty)\coloneqq\lim_{r\to\pm\infty}\gamma(r)
\]
\index[symbols]{geof@$\gamma(\pm\infty)$}%
denote the \emph{endpoints} of~$\gamma$ in~$\partial_\geo\H$.
\index[defs]{endpoints of geodesic}%
For any~$r_1,r_2\in\R\cup\{\pm\infty\}$, $r_1<r_2$, we call 
\[
[z_1,z_2]\coloneqq\gamma([r_1,r_2])
\]
\index[symbols]{intera@$[\cdot,\cdot]$}%
the \emph{closed geodesic segment}\footnote{In contrast to geodesics, which are curves, geodesic segments are subsets of~$\overline\H^\geo$.} or \emph{closed hyperbolic interval} from $z_1=\gamma(r_1)$ to $z_2=\gamma(r_2)$.
\index[defs]{geodesic segment!closed}%
\index[defs]{interval!closed hyperbolic}%
Analogously, we define the \emph{open} or \emph{semi-open geodesic segments/hyperbolic intervals}~$(z_1,z_2)$, $[z_1,z_2)$, and~$(z_1,z_2]$.
\index[defs]{geodesic segment!open}%
\index[defs]{geodesic segment!semi-open}%
\index[defs]{interval!open hyperbolic}%
\index[defs]{interval!semi-open hyperbolic}%
\index[symbols]{interb@$(\cdot,\cdot]$}%
\index[symbols]{interc@$[\cdot,\cdot)$}%
\index[symbols]{interd@$(\cdot,\cdot)$}%
We note that for any given~$z_1,z_2\in\overline\H^\geo$, $z_1\not=z_2$, the definition of these hyperbolic intervals does not depend on the choice of the geodesic~$\gamma$ as long as $\gamma(r_1) = z_1$, $\gamma(r_2)=z_2$ for some $r_1<r_2$. We will write~$[z_1,z_2]_{\H}$ (and so on) whenever confusion with intervals in $\widehat{\R}$ might occur.
The latter ones are then denoted~$[z_1,z_2]_\R$ (and so on).
We call a  geodesic segment~$(z_1,z_2)$ \emph{complete} or \emph{geodesic arc} if $z_1,z_2\in\partial_\geo\H$, thus if $(z_1,z_2)=\gamma(\R)$ for some geodesic~$\gamma$ on~$\H$.
\index[defs]{geodesic segment!complete}%
\index[defs]{geodesic arc}%

\subsection{Fuchsian groups and hyperbolic orbisurfaces}\label{SUBSEC:fuchsgroups}

We endow~$\PSLR$ with the quotient topology (see~\eqref{eq:def_PSLR}), and call, as usual, discrete subgroups of~$\PSLR$ \emph{Fuchsian}.
\index[defs]{Fuchsian group}%
For any subgroup~$\Gamma$ of~$\PSLR$,
\index[symbols]{gamma@$\Gamma$}%
being Fuchsian is equivalent to any of the following (equivalent) properties (see, e.g., \cite[Section~2.2]{Katok_fuchsian}): 
\begin{enumerate}[label=$\mathrm{(\alph*)}$, ref=$\mathrm{\alph*}$]
 \item $\Gamma$ acts \emph{properly discontinuously} on~$\H$, that is, any compact subset of~$\H$ contains only finitely many points of each $\Gamma$-orbit.
\index[defs]{properly discontinuous}%
\index[defs]{action!properly discontinuous}%
 \item Each $\Gamma$-orbit is a discrete subset of~$\H$.
\end{enumerate}
From these characterizations it follows immediately that for any Fuchsian group~$\Gamma$ and any point~$z\in\H$, the order of the stabilizer subgroup 
\[
\Stab{\Gamma}{z} \coloneqq \defset{ g\in\Gamma }{ g\act z = z }
\]
of~$\Gamma$ is finite.
\index[defs]{stabilizer group}%
\index[symbols]{stab@$\Stab{\Gamma}{z}$}%

Let $\Gamma$ be a Fuchsian group, let 
\[
\Orbi\coloneqq\quod{\Gamma}{\H}
\]
\index[symbols]{X@$\Orbi$}%
be the orbit space of the action of~$\Gamma$ on~$\H$ and let 
\[
\pi\colon\H\to\Orbi
\]
\index[symbols]{pi@$\pi$}%
\index[defs]{canonical quotient map}%
denote the canonical quotient map. Since $\Gamma$ acts properly discontinuously, the space~$\Orbi$ naturally carries the structure of a good hyperbolic Riemannian orbifold, with a hyperbolic Riemannian metric inherited via the projection map~$\pi$. If~$\Gamma$ has torsion, then $\Orbi$ has conical singularities, and hence it is not a manifold but a genuine orbifold. 
We call any such orbit space~$\Orbi$ a \emph{hyperbolic orbisurface}. 
\index[defs]{hyperbolic orbisurface}%
\index[defs]{hyperbolic orbifold}%
\index[defs]{orbifold}%

\subsection{Geodesics on hyperbolic orbisurfaces}\label{SUBSEC:geodsurfaces}

Let $\Gamma$ be a Fuchsian group and let $\Orbi = \quod{\Gamma}{\H}$ be the associated hyperbolic orbisurface. The (unit speed) \emph{geodesics} on~$\Orbi$
\index[defs]{geodesic!on~$\Orbi$}%
are the images of the geodesics on~$\H$ under the canonical quotient map
\begin{equation}\label{eq:def_pi1}
\pi\colon\H\to\Orbi\,.
\end{equation}
Thus, if $\gamma$ is a geodesic on~$\H$, then 
\begin{equation}\label{eq:def_geodnot}
 \widehat{\gamma}\coloneqq\pi\circ\gamma \colon \R\to\Orbi\,,\quad t\mapsto 
\pi(\gamma(t))\,,
\end{equation}
is the induced geodesic on~$\Orbi$.
\index[symbols]{geob@$\widehat{\gamma}$}%
In this case, we say that~$\gamma$ is a \emph{representative} or a \emph{lift} of~$\widehat\gamma$ to~$\H$.
\index[defs]{representative!geodesic}%
\index[defs]{lift}%
We emphasize that all geodesics on~$\Orbi$ arise in this way.

As for~$\H$, geodesics on~$\Orbi$ are determined by any of its tangent vectors. 
To simplify the further exposition, we recall from the previous section that $\PSLR$, and hence~$\Gamma$, acts on~$\UTB\H$. We identify the unit tangent bundle~$\UTB\Orbi$ of~$\Orbi$ with the $\Gamma$-orbit space of~$\UTB\H$:
\begin{equation}\label{eq:def_USX}
\UTB\Orbi=\quod{\Gamma}{\UTB\H}\,.
\end{equation}
\index[defs]{unit tangent bundle!of~$\Orbi$}%
\index[symbols]{SX@$\UTB\Orbi$}%
We let 
\begin{equation}\label{eq:def_pi2}
 \pi\colon\UTB\H\to\UTB\Orbi
\end{equation}
\index[symbols]{pi@$\pi$}%
\index[defs]{canonical quotient map}%
denote the canonical quotient map, which is indeed the tangent map of the quotient map from~\eqref{eq:def_pi1}. The context will always clarify if~$\pi$ refers to the map in~\eqref{eq:def_pi2} or in~\eqref{eq:def_pi1}. If an object 
on~$\Orbi$ (or related to~$\Orbi$) is defined as the~$\pi$-image of a corresponding object of~$\H$, then we usually denote the object on~$\Orbi$ by the object on~$\H$ decorated with~$\widehat{\ }$. One example is the notation 
in~\eqref{eq:def_geodnot} for geodesics on~$\Orbi$. Analogously, we denote an element in~$\UTB\Orbi$ by $\widehat\nu$ if it is represented by $\nu\in\UTB\H$, thus
\[
 \widehat\nu = \pi(\nu)\,.
\]
\index[symbols]{10@$\widehat{\cdot}$}%
\index[symbols]{v@$\widehat\nu$}%
(However, we note that the~$\widehat{\ }$ on the set~$\R$ denotes its extension to the projective line and not a $\Gamma$-orbit.)

Each unit tangent vector $\widehat\nu\in\UTB\Orbi$ uniquely determines a geodesic on~$\Orbi$, namely the geodesic $\widehat\gamma_{\widehat\nu}$ with derivative~$\widehat\nu$ at time~$t=0$: 
\[
 \widehat\gamma^\prime_{\widehat\nu}(0) = \widehat\nu\,.
\]
\index[symbols]{geoc@$\widehat\gamma_{\widehat\nu}$}%
\index[symbols]{geod@$\widehat\gamma_\nu$}%
We emphasize that for any representative $\nu\in\UTB\H$ of $\widehat\nu$ we have
\[
 \widehat\gamma_{\widehat\nu} = \widehat\gamma_\nu = \pi(\gamma_\nu)\,,
\]
for which reason we will omit from the notation the $\widehat{\ }\ $ in the index. Also the \emph{geodesic flow} on~$\Orbi$, denoted~$\GeoFlow$,
\index[defs]{geodesic flow!on~$\Orbi$}%
\index[symbols]{Pb@$\GeoFlow$}%
is the $\pi$-image of the geodesic flow~$\Phi$ on~$\H$. Thus, 
\begin{equation}\label{eq:GeoFlow}
\GeoFlow\coloneqq\pi\circ\Phi\circ\left(\id\times\pi_0^{-1}
\right)\colon\R\times\UTB\Orbi\to\UTB\Orbi\,,\quad 
(t,\widehat\nu)\mapsto \widehat\gamma_\nu'(t)\,,
\end{equation}
where $\pi_0^{-1}$ is an arbitrary section of $\pi$. (It is straightforward to check that $\GeoFlow$ does not depend on the choice of the section~$\pi_0^{-1}$.)

Whereas the arcs of geodesics on~$\H$ are always generalized semicircles, the arcs of geodesics on~$\Orbi$ enjoy a greater variety of forms. Of particular interest for us are the periodic geodesics, which we discuss now. We say that a geodesic~$\widehat{\gamma}$ on~$\Orbi$ is \emph{periodic}
\index[defs]{geodesic!periodic}%
\index[defs]{periodic geodesic}%
if there exists~$\delta>0$ such that 
\[
\widehat{\gamma}(t+\delta)=\widehat{\gamma}(t)
\]
for all $t\in\R$. Any such~$\delta$ is called a \emph{period} for~$\widehat\gamma$. 
\index[defs]{period}%

Periodic geodesics are closely related to hyperbolic elements in~$\Gamma$ in a way that we recall now. Let $h\in\PSLR$ be a hyperbolic element. Then there exists a geodesic~$\gamma$ on~$\H$ whose endpoints are the fixed points of~$h$. 
More precisely, $\gamma$ satisfies
\begin{equation}\label{eq:geod_h}
 \gamma(+\infty) = \fixp{+}{h}\qquad\text{and}\qquad \gamma(-\infty) = 
\fixp{-}{h}\,.
\end{equation}
We note that the choice of~$\gamma$ is not unique. However, since linear fractional transformations map geodesics to geodesics, it immediately follows that any two geodesics satisfying~\eqref{eq:geod_h} have the same arc but may be  
time-shifted from each other. For that reason, we call any two geodesics~$\gamma_1$ and~$\gamma_2$ on~$\H$ \emph{equivalent} if they differ only by a parameter change,
\index[defs]{geodesic!equivalent}%
i.e., if there exists $t_0\in\R$ such that 
\[
\gamma_1(t)=\gamma_2(t+t_0)
\]
for all $t\in\R$. Thus, the set of geodesics with~\eqref{eq:geod_h} constitutes such an equivalence class, which we call the \emph{axis}~$\alpha(h)$ of~$h$. 
\index[defs]{axis}%
\index[defs]{hyperbolic element!axis}%
\index[defs]{element!axis of hyperbolic}%
\index[symbols]{aa@$\alpha(h)$}%

We denote the equivalence class of a geodesic~$\gamma$ on~$\H$ by~$[\gamma]$. 
\index[defs]{geodesic!equivalence class}%
\index[symbols]{900@$[\gamma]$}%
The action of~$\Gamma$ on the set of geodesics descends to an action on the equivalence classes by
\[
 g\act[\gamma] \coloneqq [g\act\gamma]
\]
\index[defs]{action!on geodesics}%
\index[symbols]{gc@$g\act[\gamma]$}%
for all~$g\in\Gamma$ and all geodesics~$\gamma$ on~$\H$. Analogously, we call any two geodesics on~$\Orbi$ equivalent if they differ only by a parameter change.
\index[defs]{geodesic!equivalent}%
Obviously, the geodesics~$\widehat\gamma_1$ and~$\widehat\gamma_2$ are equivalent if and only if $\Gamma\act[\gamma_1] = \Gamma\act[\gamma_2]$ for any choice of its representatives~$\gamma_1$ and~$\gamma_2$. Periodicity and periods of geodesics on~$\Orbi$ are stable under descend to equivalence classes. Further, 
for any equivalence classes, 
\[
 [\widehat\gamma] =  [\pi(\gamma)] = \pi([\gamma])\,.
\]
We use~$\Geo(\Orbi)$
\index[symbols]{geode@$\Geo(\Orbi)$}%
to denote the set of all equivalence classes of geodesics on~$\Orbi$, and $\Geo_{\Per}(\Orbi)$ to denote the subset of the equivalence classes of periodic geodesics.
\index[symbols]{geodf@$\Geo_{\Per}(\Orbi)$}%

With the following lemma we recall a well-known first observation on the relation between periodic geodesics on~$\Orbi$ and hyperbolic elements in~$\Gamma$ as well as between axes and displacement lengths of different hyperbolic elements. A proof can be deduced, e.g., from~\cite[Observations~3.28, 3.29]{Dupuy_Buser_Semmler}.

\begin{lemma}\label{LEM:periodicaxes}
Let $h\in\Gamma$ be hyperbolic with displacement length~$\ell(h)$ and axis~$\alpha(h)$. Then we have:
\begin{enumerate}[label=$\mathrm{(\roman*)}$, ref=$\mathrm{\roman*}$]
\item\label{periodicaxes:axisisperiod}
The geodesics in the equivalence class~$\pi(\alpha(h))$ are periodic with period~$\ell(h)$.

\item\label{periodicaxes:conjugacy}
For all~$g\in\Gamma$ the element~$ghg^{-1}$ is hyperbolic with displacement length~$\ell(h)$ and axis~$g\act\alpha(h)$.

\item\label{periodicaxes:powers}
For all~$n\in\N$ we have $\alpha(h^n)=\alpha(h)$ and~$\ell(h^n)=n\ell(h)$.
\end{enumerate}
\end{lemma}

We denote by~$[\Gamma]_{\mathrm{h}}$ the set of all $\Gamma$-conjugacy classes of hyperbolic elements in~$\Gamma$.
\index[symbols]{Gammah@$[\Gamma]_{\mathrm{h}}$}%
\index[defs]{hyperbolic element!conjugacy class}%
\index[defs]{element!conjugacy class}%
Lemma~\ref{LEM:periodicaxes} yields that all representatives of~$[g]\in[\Gamma]_{\mathrm{h}}$ give rise to the same 
equivalence class of geodesics in~$\Geo_{\Per}(\Orbi)$. Thus, this relation gives rise to the map
\begin{align}\label{EQDEF:mapaxistoX}
\varrho\colon[\Gamma]_{\mathrm{h}}\to\Geo_{\Per}(\Orbi)\,,\quad 
[h]\mapsto\pi(\alpha(h))\,.
\end{align}
\index[symbols]{rho@$\varrho$}%

Lemma~\ref{LEM:periodicaxes} further shows that for each hyperbolic $h\in\Gamma$, the displacement length is constant in~$[h]$ and is one of the possible periods of the geodesics in~$\varrho([h])$. However, since the 
displacement length scales with powers of~$h$ but the image of~$\varrho$ remains unchanged, $\varrho$ is not a bijection, but an infinite covering. 

For each hyperbolic element~$h\in\Gamma$ there exists a unique pair $(h_0,n)\in\Gamma\times\N$ such that $h_0$ is hyperbolic and $n$ is maximal with the property that $h=h_0^n$. The displacement length of~$h_0$ as well as 
the value of~$n$ are invariants of the equivalence class~$[h]$. We set 
\begin{align*}
 \ct([h]) &\coloneqq n \qquad \text{(``count'')}
\intertext{and}
 \ell_0([h]) &\coloneqq \ell(h_0)\,.
\end{align*}
Further, we call the element~$h_0$ the \emph{primitive} of~$h$.
\index[symbols]{ct@$\ct([h])$}%
\index[symbols]{lb@$\ell_0$}%
\index[symbols]{lc@$\ell_0([h])$}%
\index[symbols]{h@$h_0$}%
\index[defs]{hyperbolic element!primitive}%
\index[defs]{element!primitive hyperbolic}%
\index[defs]{primitive element}%

\begin{prop}[Theorem~3.30 in~\cite{Dupuy_Buser_Semmler}]
\label{PROP:hypconjisperiod}
The map 
\[
\varrho\times\ct \colon [\Gamma]_\hyp \to \Geo_{\Per}(\Orbi)\times\N\,,\quad 
[h] \mapsto \bigl(\pi(\alpha(h)), \ct([h])\bigr)\,,
\]
is a bijection.
\end{prop}

If $\widehat\gamma$ is a periodic geodesic on~$\Orbi$ and $\delta_0$ is its minimal period, then for any $n\in\N$, also $n\delta_0$ is a period of~$\widehat\gamma$. In view of Proposition~\ref{PROP:hypconjisperiod}, we may 
understand $([\widehat\gamma], n) \in\Geo_{\Per}(\Orbi)\times\N$ as the equivalence class~$[\widehat\gamma]$ of periodic geodesics endowed with a fixed choice of period, namely $n\delta_0$. We may further define the \emph{length} of $(\widehat{\gamma},n)\in\Geo_{\Per}(\Orbi)\times\N$
\index[defs]{length}%
as
\[
\ell(\widehat\gamma,n)\coloneqq \ell([h])
\]
\index[symbols]{la@$\ell$}%
\index[symbols]{le@$\ell(\widehat\gamma,n)$}%
for any~$h\in\Gamma$ such that 
\[
 (\varrho\times\ct)([h]) = (\widehat\gamma,n)\,.
\]
If $n=1$ then we call $([\widehat\gamma],n)$ \emph{prime} or \emph{primitive}.
\index[defs]{geodesic!prime}%
\index[defs]{geodesic!primitive}%
Likewise, we call $h$ and~$[h]$ \emph{primitive} if $\ct([h]) = 1$.
\index[defs]{hyperbolic element!primitive}%
\index[defs]{element!primitive hyperbolic}%

Until now we have considered~$\Geo(\Orbi)$ and~$\Geo_{\Per}(\Orbi)$ as sets of equivalence classes of geodesics. 
For convenience, in what follows, we will often consider them as sets of (equivalence classes of) geodesics endowed with a multiplity of the minimal period length. Equivalently, we often consider them as multisets, where every periodic geodesic is repeated countably infinitely many times, once for each multiplicity of its minimal period length. Moreover, we will refer to their elements as geodesics. Thus, by slight abuse of notion, we identify geodesics with their equivalence classes, and we will often indicate the choice of period only implicitly. 
\index[defs]{geodesic!convention}%

\begin{cor}\label{COR:primhypconjisprimperiod}
The conjugacy classes of primitive hyperbolic elements in $\Gamma$ are in bijection with the prime periodic geodesics on $\Orbi$.
\end{cor}

For each prime periodic geodesic~$\widehat\gamma$ on~$\Orbi$, any representative of the associated conjugacy class of primitive hyperbolic elements in~$\Gamma$ serves as the element~$h_{\widehat\gamma}$ in~\eqref{eq:szf_twist_intro}.
\index[symbols]{h@$h_{\widehat\gamma}$}%

\subsection{Geometry at infinity}\label{SUBSEC:geominfinity}

Let $\Gamma$ be a Fuchsian group. We call a subset~$\fund$ of~$\H$ a \emph{fundamental domain} for~$\Gamma$ in~$\H$ if $\fund$ is an open, connected set such that 
\index[defs]{fundamental domain}%
\index[symbols]{F@$\fund$}%
\begin{enumerate}[label=$\mathrm{(\roman*)}$, ref=$\mathrm{\roman*}$]
 \item any two $\Gamma$-translates of~$\fund$ are disjoint: for all 
$g\in\Gamma$, 
 \[
  g\act\fund \cap \fund = \varnothing\,,
 \]
 \item the $\Gamma$-translates of~$\overline\fund$ cover all of~$\H$: 
 \[
  \H = \bigcup_{g\in\Gamma} g\act\overline\fund\,.
 \]
\end{enumerate}
Each Fuchsian group admits at least one fundamental domain. Standard choices are Dirichlet or Ford fundamental domains, which we will use also here and which are examples of fundamental domains in the shape of exact, convex polyhedra (in the sense of~\cite{Ratcliffe}). We call a Fuchsian group \emph{geometrically finite} 
\index[defs]{Fuchsian group!geometrically finite}%
\index[defs]{geometrically finite}%
if it admits such a polyhedral fundamental domain with finitely many sides. By Poincar\'e's theorem on fundamental polyhedra, a Fuchsian group is geometrically finite if and only if it is finitely generated~\cite{Maskit}.

\begin{center}
\framebox{
\begin{minipage}{.53\textwidth}
From now on let $\Gamma$ be a geometrically finite Fuchsian group that contains at least one hyperbolic element.
\end{minipage}
}
\end{center}

Let $\Orbi=\quod{\Gamma}{\H}$ be the associated hyperbolic orbisurface. The geometry of~$\Orbi$ allows us to find a (large) compact subset~$K$ of~$\Orbi$ such that for all compact subsets~$\widetilde K$ with $K\subseteq\widetilde K$, the spaces~$\Orbi\setminus K$ and~$\Orbi\setminus\widetilde K$ have the same number of connected components. For definiteness we may take for~$K$ the \emph{compact core} of~$\Orbi$. As we will not use any further features of the compact core other than this separation property, we refer to~\cite{Borthwick_book} for a definition. The connected components of~$\Orbi\setminus K$ are the \emph{ends} of~$\Orbi$.
\index[defs]{ends}%
\index[defs]{hyperbolic orbisurface!end}%
(We ignore here the slight inexactness in that this notion of ends depends on the choice of~$K$ if we do not pick the compact core, as we will need only the general concept.) The geometric finiteness of~$\Gamma$ yields that $\Orbi$ has only a finite number of ends.   

The hyperbolic orbisurface~$\Orbi$ has at least one periodic geodesic as $\Gamma$ contains a hyperbolic element. Therefore $\Orbi$ admits only two types of ends: 
\begin{enumerate}[label=$\mathrm{(\alph*)}$, ref=$\mathrm{\alph*}$]
\item Ends of finite area are called \emph{cusps}.
\index[defs]{cusp}%
Via the canonical quotient map~$\pi$ from~\eqref{eq:def_pi1}, each cusp of~$\Orbi$ can be identified with the $\Gamma$-orbit of the fixed point~$c$ of some parabolic transformation in~$\Gamma$. The stabilizer 
group~$\Stab{\Gamma}{c}$ of~$c$ is a cyclic subgroup of~$\Gamma$. In particular, there exists $g\in\PSL_2(\R)$ and a unique~$\lambda>0$ such that~$\Stab{\Gamma}{c}$ is generated by
\[
 g\bmat{1}{\lambda}{0}{1}g^{-1}\,.
\]
We call~$c$ a \emph{cusp representative} or a \emph{cuspidal point}
\index[defs]{cusp representative}%
\index[defs]{representative!cusp}%
\index[defs]{cuspidal point}%
and denote the corresponding cusp by~$\wh c$ (see also Section~\ref{SUBSEC:geodsurfaces}). Further, we call~$\lambda$ the \emph{cusp width} of~$\wh c$.
\index[defs]{cusp width}%
\item Ends of infinite area are called \emph{funnels}.
\index[defs]{funnel}%
Funnels can be identified with certain subsets of the geodesic boundary of a fundamental domain for~$\Gamma$. Further below, after the introduction of the limit set of~$\Gamma$, we will give a second characterization of funnels.
\end{enumerate}

The hyperbolic orbisurface~$\Orbi$ is compact if and only if it neither has cusps nor funnels. In this case, $\Gamma$ is called \emph{cocompact} or \emph{uniform}.
\index[defs]{Fuchsian group!cocompact}%
\index[defs]{Fuchsian group!uniform}%
\index[defs]{uniform}%
\index[defs]{cocompact}%
If $\Orbi$ is not compact, then $\Gamma$ is called \emph{non-cocompact} or \emph{non-uniform}.
\index[defs]{Fuchsian group!non-cocompact}%
\index[defs]{Fuchsian group!non-uniform}%
\index[defs]{non-cocompact}%
\index[defs]{non-uniform}%
If $\Orbi$ has no cusps, but probably funnels, and is a proper surface (i.e., $\Gamma$ has no elliptic elements), then $\Gamma$ is called \emph{convex cocompact}.
\index[defs]{Fuchsian group!convex cocompact}%
\index[defs]{convex cocompact}%
(The naming originates from the fact that in this case the convex core is compact. We refer to \cite{Borthwick_book} for the definition of the convex core.)

The periodic geodesics on~$\Orbi$ are dense in the set of all geodesics on $\Orbi$ in a certain sense which we describe in what follows and which will be crucial for our investigations. Since $\Gamma$ is discrete,  $\Gamma$-orbits do not accumulate in~$\H$ but they may do so in~$\overline\H^\geo$. We denote by~$\Lambda(\Gamma)$ the set of all limit points (accumulation points) of $\Gamma$-orbits in~$\overline\H^\geo$\!, the \emph{limit set} of~$\Gamma$.
\index[defs]{limit set}%
\index[symbols]{L@$\Lambda$}%
\index[symbols]{L@$\Lambda(\Gamma)$}%
The set~$\Lambda(\Gamma)$ equals the set of all limit points of the single orbit~$\Gamma\act z$ for any $z\in\H$ with trivial stabilizer subgroup in~$\Gamma$. We note that $\Lambda(\Gamma)$ is a $\Gamma$-invariant subset 
of~$\partial_\geo\H$. The Fuchsian group~$\Gamma$ is called \emph{elementary}
\index[defs]{Fuchsian group!elementary}%
\index[defs]{elementary}%
if $\Lambda(\Gamma)$ is finite. If $\Gamma$ is non-elementary, then the limit set~$\Lambda(\Gamma)$ is either all of~$\partial_\geo\H$ or a perfect, nowhere dense subset of~$\partial_\geo\H$. We let
\begin{align}\label{EQDEF:EOrbi}
E(\Orbi)\coloneqq
\defset{ (\gamma(+\infty),\gamma(-\infty)) }{ 
\text{$\gamma$ geodesic on~$\H$, $\widehat\gamma\in\Geo_{\Per}(\Orbi)$} }
\end{align}
\index[symbols]{EX@$E(\Orbi)$}%
denote the set of the endpoints of all lifts of periodic geodesics on~$\Orbi$. By Proposition~\ref{PROP:hypconjisperiod}, this set can also be characterized as  
\[
E(\Orbi)=\defset{ (\fixp{+}{h},\fixp{-}{h}) }{ 
[h]\in[\Gamma]_{\mathrm{h}} }\,.
\]
In particular,~$E(\Orbi)$ is a subset of~$\Lambda(\Gamma)\times\Lambda(\Gamma)$.

\begin{prop}[\cite{Eberlein_curvedMflds}]\label{PROP:EXliesdense}
For any geometrically finite Fuchsian group~$\Gamma$ with hyperbolic elements, the set $E(\Orbi)$ is dense in~$\Lambda(\Gamma)\times\Lambda(\Gamma)$.
\end{prop}

The limit set~$\Lambda(\Gamma)$ allows us to give another interpretation of funnels as follows. The set
\[
 \widehat\R\setminus\Lambda(\Gamma) 
\]
decomposes into countably many connected (open) components. Each such component we call a \emph{funnel interval}.
\index[defs]{funnel interval}%
\index[defs]{interval!funnel}%
Further, we call an interval in~$\widehat\R$  a \emph{funnel representative}
\index[defs]{funnel representative}%
\index[defs]{representative!funnel}%
if it is contained in a funnel interval consisting of points that are pairwise 
non-equivalent under the $\Gamma$-action and that is maximal with obeying this property.
One easily sees that each funnel interval is the union of several funnel representatives. The $\Gamma$-orbits of 
funnel intervals coincide with the $\Gamma$-orbits of funnel representatives (as sets or as equivalence classes), and each such $\Gamma$-orbit corresponds to a unique funnel of~$\Orbi$. We may identify each funnel of~$\Orbi$ with such a $\Gamma$-orbit (understood as equivalence class) or with a funnel representative.

Finally, we introduce a few more definitions that are not classical but will be used extensively throughout. For any parabolic element $p\in\Gamma$ we denote its fixed point by~$\fixp{}{p}$. 
\index[defs]{fixed point!parabolic element}%
\index[symbols]{fp@$\fixp{}{p}$}%
We set
\begin{equation}\label{eq:def_st}
\widehat{\R}_{\st}\coloneqq\Lambda(\Gamma)\setminus\bigcup\defset{\fixp{}{p}}{p\in\Gamma~\text{parabolic}}\,.
\end{equation}
\index[symbols]{Rst@$\widehat{\R}_{\st}$}%
\index[symbols]{st@${\st}$}%
\index[symbols]{910@${}_{\st}$}%
\index[defs]{st}%
The set~$\wh\R_{\st}$ is $\Gamma$-invariant and contains all hyperbolic fixed points of~$\Gamma$ (and typically further points).
For every subset~$M\subseteq\widehat{\R}$ we set
\[
M_{\st}\coloneqq M\cap\widehat{\R}_{\st}.
\]
\index[symbols]{Mst@$M_{\st}$}%
\index[symbols]{Iaa@$I_{\st}$}%
For elements of some family $\{M_j\}_{j}$ of subsets of $\widehat{\R}$ we set $M_{j,\st}\coloneqq\left(M_j\right)_{\st}$.
Finally, we call a subset~$I'\subseteq\widehat{\R}_{\st}$ an \emph{interval} in~$\widehat{\R}_{\st}$,
\index[defs]{interval!in~$\widehat{\R}_{\st}$}%
if there exists an interval~$I$ in~$\widehat{\R}$ such that~$I'=I_{\st}$. The following characterization of 
cocompact Fuchsian groups is immediate.

\begin{lemma}\label{LEM:noncocomp}
Let $\Gamma$ be a geometrically finite Fuchsian group. Then $\Gamma$ is cocompact if and only if $\widehat{\R}=\widehat{\R}_{\st}$.
\end{lemma}

\section{Cross sections, sets of branches and transfer operators}\label{SEC:crosssec}

Throughout this section let $\Gamma$ be a geometrically finite, non-cocompact Fuchsian group with at least one hyperbolic element and let $\Orbi$ denote the associated hyperbolic orbisurface. In this section we present the starting point of our constructions, the so-called \emph{sets of branches}, and prove various crucial properties of these seminal objects.

We start by presenting a notion of cross sections for the geodesic flow that we borrowed from~\cite{Pohl_Symdyn2d} and will use throughout. As for the classical notion, a cross section in our sense is a subset of the unit tangent bundle of~$\Orbi$ such that all intersections between any geodesic and the cross section are discrete with respect to the time parameter of the considered geodesic. However, it differs from the classical one in regard to the set of geodesics that is requested to intersect the cross section. We refer to the exposition in Section~\ref{SUBSEC:cross} for details.

We will define any cross section by choosing a set of representatives for it, i.e., a subset of the unit tangent bundle of~$\H$ that is bijective to the cross section. More precisely, we may and will consider the set of representatives as the primary object and the cross section as a consequential object that inherits all its properties from the set of representatives. The starting point of our constructions are well-structured sets of representatives, called \emph{sets of branches}, whose definition we introduce in Section~\ref{SUBSEC:cuspexp} and 
whose first essential properties we discuss in Sections~\ref{SUBSEC:propbranches}--\ref{SUBSEC:slowtrans}.

In a nutshell, the notion of sets of branches constitutes an equilibrium between our wish to keep the framework as general as possible and the requirements of a descent algorithm of cuspidal acceleration and a nicely structured passage from slow to fast transfer operators.  
Indeed, all sets of representatives for cross sections in~\cite{Pohl_Symdyn2d, Bruggeman_Pohl, Pohl_diss, Wabnitz} are sets of branches.
Thus, a large class of geometrically finite, non-compact hyperbolic orbisurfaces admits several choices of sets of branches.
Moreover, also many more choices of sets of branches with much different properties are possible. Example~\ref{EX:branchramification} below provides a first indication. 
On the other hand, the notion of a set of branches is sufficiently rigid to give rise to well-structured families of slow transfer operators, as we show in Section~\ref{SUBSEC:slowtrans}. Further, it allows for an acceleration/induction algorithm that turns any set of branches into a strict transfer operator approach. The notions of strict transfer operator approaches and fast transfer operators as well as their importance for Selberg zeta functions are provided in Section~\ref{SUBSEC:stricttrans}--\ref{SUBSEC:fasttrans}.

\subsection{Cross sections for the geodesic flow}\label{SUBSEC:cross}

For the notion of cross sections for the geodesic flow~$\GeoFlow$ on the hyperbolic orbisurface~$\Orbi$ we use a measure~$\mu$ on the set of all geodesics on~$\Orbi$ to single out those geodesics whose behavior is essential for the applications in mind. For the general statement of the definition of a cross section, we neither require the measure~$\mu$ to be finite or even a probability measure nor ask for any specific properties of the implicitly fixed $\sigma$-algebra on the set of geodesics. We refer to the discussion below the definition and to 
Section~\ref{SUBSEC:crossbranches} for the class of measures relevant for our applications. For any subset~$M$ of~$\UTB\Orbi$, any geodesic~$\widehat\gamma$ on~$\Orbi$ and any~$t\in\R$ we say that $\widehat\gamma$ \emph{intersects~$M$ at time~$t$} if $\widehat{\gamma}^{\prime}(t)\in M$.
\index[defs]{intersect}%
\index[defs]{geodesic!intersection}%

We call a subset~$\CrSc$ of~$\UTB\Orbi$ a \emph{cross section for~$\GeoFlow$ with respect to~$\mu$} if
\index[defs]{cross section}%
\index[symbols]{C@$\CrSc$}%
\begin{enumerate}[label=$\mathrm{(CS\arabic*)}$, ref=$\mathrm{CS}$\arabic*]
\item\label{CS:infinitelyoften}
$\mu$-almost every geodesic~$\widehat{\gamma}$ on~$\Orbi$ intersects~$\CrSc$ infinitely often in past and future, i.e., there exists a two-sided sequence~$(t_n)_{n\in\Z}$ with 
\[
\lim_{n\to\pm\infty} t_n = \pm\infty
\]
such that for each~$n\in\Z$, the geodesic~$\wh\gamma$ intersects~$\CrSc$ at time~$t_n$, and 

\item\label{CS:discreteintime}
each intersection of any geodesic~$\widehat\gamma$ on~$\Orbi$ and $\CrSc$ is discrete (in time). That is, for any $t\in\R$ with $\widehat\gamma^\prime(t)\in\CrSc$, there exists $\varepsilon>0$ such that 
\[
\widehat{\gamma}^{\prime}((t-\varepsilon,t+\varepsilon))\cap\CrSc=\{\widehat{
\gamma}^{\prime}(t)\}\,.
\]
\end{enumerate}
\index[defs]{0CS01@(CS1)}%
\index[symbols]{CS1@(CS1)}%
\index[defs]{0CS02@(CS2)}%
\index[symbols]{CS2@(CS2)}%
We call a cross section~$\CrSc$ for~$\GeoFlow$ \emph{strong} if it satisfies that
\index[defs]{cross section!strong}%
\begin{enumerate}[resume, label=$\mathrm{(CS\arabic*)}$, 
ref=$\mathrm{CS}$\arabic*]
\item\label{CS:strong}
every geodesic on~$\Orbi$ that intersects~$\CrSc$ at all, intersects~$\CrSc$ infinitely often both in past and future. 
\end{enumerate}
\index[defs]{0CS02@(CS3)}%
\index[symbols]{CS3@(CS3)}%

We remark that this notion of cross section deviates from the classical notion in that it does \emph{not} require that \emph{every} geodesic intersects the cross section. For the applications that motivate our article we may restrict to certain measures whose support contains all periodic geodesics, and we may relax~\eqref{CS:infinitelyoften} to 
\begin{enumerate}[label=$\mathrm{(CS\arabic*')}$, ref=$\mathrm{CS}$\arabic*']
\item\label{CS:infoftenalt} Every periodic geodesic~$\widehat{\gamma}$ on~$\Orbi$ intersects~$\CrSc$. 
\end{enumerate}
\index[defs]{0CS01@(CS1')}%
\index[symbols]{CS1@(CS1')}%
In~\cite{Pohl_Symdyn2d, Moeller_Pohl, Pohl_gamma, Pohl_mcf_general, Bruggeman_Pohl} it was seen that cross sections of this kind capture just the right amount of geometry and simultaneously allow for sufficient freedom to construct discretizations of the geodesic flow for which the associated transfer operators mediate between the geodesic flow and the Laplace eigenfunctions of~$\Orbi$.

Substituting~\eqref{CS:infinitelyoften} by~\eqref{CS:infoftenalt} would allow us to omit the choice of a measure from the definition of cross section. However, to achieve greater flexibility in view of potential further applications, we will work with a larger class of measures. Starting with Section~\ref{SUBSEC:crossbranches} we  will consider all those measures that do not assign positive mass to the geodesics that ``vanish'' into an end 
of~$\Orbi$. We refer to Proposition~\ref{CofSoB:crosssection} for a precise statement. In what follows we will often suppress the choice of the measure~$\mu$ from the notation.

Suppose that~$\CrSc$ is a strong cross section for~$\GeoFlow$. An immediate consequence of~\eqref{CS:strong} is that for any~$\widehat\nu\in\CrSc$, the \emph{first return time} 
\index[defs]{first return time}%
\index[symbols]{t@$\retime{\CrSc}$}%
\[
\retime{\CrSc}(\nu)\coloneqq\min\defset{t>0}{\widehat{\gamma}_{\nu}^{\prime}
(t)\in\CrSc}
\]
of~$\widehat\nu$ with respect to~$\CrSc$ exists. Further, the \emph{first return map}
\index[defs]{first return map}%
\index[symbols]{R@$\wh\Return$}%
\[
\wh\Return\colon\CrSc\to \CrSc\,,\quad \widehat{\nu}\mapsto\widehat{\gamma}_{\nu}^{\prime}(\retime{\CrSc}(\widehat{
\nu}))\,,
\]
is well-defined. The dynamical system
\[
 \Z\times\CrSc \to \CrSc\,,\quad (n,\wh\nu) \mapsto \wh\Return^n(\wh\nu)\,,
\]
for short $(\CrSc,\wh\Return)$, constitutes a discretization of the geodesic flow~$\GeoFlow$ on~$\Orbi$. We will apply the notions of first return time and first return map also to cross sections~$\CrSc$ that are not necessarily strong. In this case, $\retime{\CrSc}$ and~$\wh\Return$ might be defined only on a subset of~$\CrSc$, resulting in partial maps.

Let $\CrSc$ now be a cross section that may be non-strong. Recall the quotient map~$\pi\colon \UTB\H\to\UTB\Orbi$ from~\eqref{eq:def_pi2}. We call a subset $\BrU$ of~$\UTB\H$ a \emph{set of representatives} for~$\CrSc$ 
\index[symbols]{C@$\BrU$}%
\index[defs]{set of representatives}%
\index[defs]{representative!set of}%
if $\BrU$ and $\CrSc$ are bijective via~$\pi$, i.e., $\pi(\BrU) = \CrSc$ and the restricted map
\[
 \pi\vert_{\BrU}\colon \BrU \to \CrSc
\]
is bijective. For any set of representatives~$\BrU$, the first return map~$\wh\Return$ induces a \emph{first return map~$\Return$ on~$\BrU$}
\index[defs]{first return map}%
\index[symbols]{R@$\Return$}%
by 
\begin{equation}\label{eq:firstreturnH}
 \Return \coloneqq \pi\vert_{\BrU}^{-1}\circ\wh\Return\circ\pi\,.
\end{equation}
In other words, the diagram
\[
\xymatrix{
\BrU \ar[r]^{\Return} \ar[d]_{\pi}  & \BrU \ar[d]^{\pi}
\\
\CrSc \ar[r]^{\wh\Return} & \CrSc
}
\]
commutes. If $\CrSc$ is not strong and hence $\wh\Return$ is only partially defined, then $\Return$ is also only partially defined. Sometimes it is possible to find a partition of~$\BrU$ into (finitely or infinitely) many subsets, say
\begin{equation}\label{eq:decomp_branches}
 \BrU = \bigcup_{a\in A} \Cs{a}\,, 
\end{equation}
such that for each~$a\in A$, the map 
\[
 \Cs{a} \to \wh\R\,,\quad \nu\mapsto \gamma_\nu(+\infty)
\]
is injective. In this case, we set 
\[
 D_a \coloneqq \defset{ (\gamma_\nu(+\infty), a) }{ \nu\in \Cs{a} }
\]
and 
\begin{equation}\label{eq:surveyDunion}
 D \coloneqq \bigcup_{a\in A} D_a\,.
\end{equation}
We emphasize that the union in~\eqref{eq:surveyDunion} is disjoint.
Then $\Return$ induces a (well-defined, unique) map $F\colon D\to D$ that makes the diagram
\begin{equation}\label{eq:induced_system}
\xymatrix{
\BrU \ar[r]^{\Return} 
\ar[d]_{\overset{\nu}{\underset{\gamma_\nu(+\infty)}{\downmapsto}}} & \BrU 
\ar[d]^{\overset{\nu}{\underset{\gamma_\nu(+\infty)}{\downmapsto}}}
\\
D \ar[r]^{F} & D
}
\end{equation}
commutative. In the first component, the map~$F$ is piecewise given by the action of certain elements of~$\Gamma$ by linear fractional transformations. In the second component, $F$ is a certain symbol transformation.
We call $(D,F)$ the \emph{discrete dynamical system} induced by~$\BrU$. 
\index[defs]{discrete dynamical system}%

\subsection{Sets of branches}\label{SUBSEC:cuspexp}

If $\BrU$ is a set of representatives for a cross section~$\CrSc$ for the geodesic flow on~$\Orbi$, then $\BrU$ completely determines~$\CrSc$. For that reason, we may turn around the order of definitions. That means, for defining a cross section, we may start by picking a subset~$\BrU$ of~$\UTB\H$ such that the quotient map~$\pi\colon \BrU \to \pi(\BrU)$ is bijective and the image set~$\pi(\BrU)$ is a cross section. Then all properties of~$\CrSc=\pi(\BrU)$ are controlled by the properties of~$\BrU$, and specific requirements on a set of representatives can sometimes be guaranteed by a suitable choice of~$\BrU$. 

The concept of sets of branches, which we will introduce in this section, implements this idea. A set of branches is a family of subsets of~$\UTB\H$ that serves as a set of representatives with a decomposition as in~\eqref{eq:decomp_branches}, namely the elements of this family, and which induces a nicely structured 
discrete dynamical system as in~\eqref{eq:induced_system}. See Sections~\ref{SUBSEC:propbranches}--\ref{SUBSEC:slowtrans}. This concept takes advantage of points in~$\widehat{\R}\setminus\widehat{\R}_{\st}$, whose 
existence is equivalent to the non-cocompactness of~$\Gamma$ by Lemma~\ref{LEM:noncocomp}. This also explains why we restrict our considerations to non-cocompact Fuchsian groups.

We let
\begin{align}\label{EQDEF:basepoint}
\mathrm{bp}\colon\UTB\H\to\H\,,\quad (z,\nu)\mapsto z\,,
\end{align}
\index[symbols]{bzp@$\mathrm{bp}$}%
\index[defs]{projection on base points}%
denote the projection onto base points. For any convex subset~$M$ of~$\H$ and any unit tangent vector~$\nu\in\UTB\H$ satisfying $\base{\nu}\in\partial M$ we say that $\nu$ \emph{points into}~$M$
\index[defs]{points into}%
if there exists $t^*>0$ such that for all~$t\in(0,t^*)$ we have $\gamma_{\nu}(t)\in M^\circ$, where $M^\circ$ denotes the interior of~$M$.
\index[symbols]{Ma@$M^\circ$}%
\index[defs]{interior of a set}%
We further set
\begin{align}\label{EQDEF:Gammastar}
\Gamma^*\coloneqq\Gamma\setminus\{\id\}\,.
\end{align}
\index[symbols]{Gammas@$\Gamma^*$}%
We recall the definition of intersection between subsets of~$\UTB\H$ (or~$\UTB\Orbi$) and geodesics on~$\H$ (or~$\Orbi$) from Section~\ref{SUBSEC:cross}, and we recall the quotient map~$\pi$ from~\eqref{eq:def_pi2}.

\begin{defi}\label{DEF:setofbranches}
Let~$N\in\N$ and let~$\Cs1,\ldots, \Cs N$ be subsets of~$\UTB\H$. 
Set~$A\coloneqq\{1,\dots,N\}$, 
\[
 \BrU \coloneqq \bigcup_{j\in A} \Cs j\qquad\text{and}\qquad \BrS \coloneqq 
\defset{\Cs j }{ j\in A }\,.
\]
We call $\BrS$ a \emph{set of branches for the geodesic flow on}~$\Orbi$ if it satisfies the following properties:
\index[defs]{set of branches}%
\index[defs]{branch!set of}%
\index[symbols]{C@$\BrS$}%
\index[symbols]{C@$\BrU$}%
\index[symbols]{A@$A$}%

\begin{enumerate}[label=$\mathrm{(B\arabic*)}$, ref=$\mathrm{B\arabic*}$, 
itemsep=2ex]
\item\label{BP:closedgeodesicsHtoX} 
\index[defs]{0B01@(B1)}%
\index[symbols]{B01@(B1)}%
For each~$j\in A$ there exists~$\nu\in\Cs{j}$ such that $\pi(\gamma_\nu)$ is a periodic geodesic on~$\Orbi$.
\item\label{BP:completegeodesics}
\index[defs]{0B02@(B2)}%
\index[symbols]{B02@(B2)}%
For each~$j\in A$, the set~$\overline{\base{\Cs{j}}}$ is a complete geodesic segment in~$\H$ and its endpoints are 
in~$\widehat{\R}\setminus\widehat{\R}_{\st}$. In particular, for each $j\in A$, the set~$\H\setminus\overline{\base{\Cs{j}}}$ decomposes uniquely into two (geodesically) convex open half-spaces.
\item\label{BP:pointintohalfspaces}
\index[defs]{0B03@(B3)}%
\index[symbols]{B03@(B3)}%
For each~$j\in A$, all elements of~$\Cs{j}$ point into the same open half-space relative to~$\overline{\base{\Cs{j}}}$. We denote this half-space by~$\Plussp{j}$ and set
\[
\Minussp{j}\coloneqq\H\setminus(\overline{\base{\Cs{j}}}\cup\Plussp{j})\,.
\]
\index[symbols]{H@$\Plussp{j}$}%
\index[symbols]{H@$\Minussp{j}$}%
\index[defs]{half-space}%
\index[symbols]{I@$I_j$}%
Further, we denote by~$I_j$ the largest open subset of~$\wh\R$ that is contained in~$\geo\Plussp{j}$, and by~$J_j$ the largest open subset of~$\wh\R$ contained in~$\geo\Minussp{j}$.
\index[symbols]{J@$J_j$}%
\item\label{BP:coverlimitset}
\index[defs]{0B04@(B4)}%
\index[symbols]{B04@(B4)}%
The~$\Gamma$-translates of~$\defset{I_j}{j\in A}$ cover the set~$\wh \R_{\st}$, i.e.,
\[
\wh \R_{\st}\subseteq\bigcup_{j\in A}\bigcup_{g\in\Gamma}g\act I_j\,.
\]
\item\label{BP:allvectors}
\index[defs]{0B05@(B5)}%
\index[symbols]{B05@(B5)}%
For each $j\in A$ and each pair~$(x,y)\in \Iset{j}\times\Jset{j}$ there exists a (unique) vector~$\nu\in\Cs{j}$ such that
\[
(x,y)=(\gamma_\nu(+\infty),\gamma_\nu(-\infty))\,.
\]
\item\label{BP:disjointunion}
\index[defs]{0B06@(B6)}%
\index[symbols]{B06@(B6)}%
If $\overline{\base{\Cs{j}}}\cap g\act\overline{\base{\Cs{k}}}\ne\varnothing$ for some~$j,k\in A$ and~$g\in\Gamma$, then either $j=k$ and $g=\id$, or $\PMsp{j} = g\act\MPsp{k}$.
\item\label{BP:intervaldecomp}
\index[defs]{0B07@(B7)}%
\index[symbols]{B07@(B7)}%
For each pair~$(a,b)\in A\times A$ there exists a (possibly empty) subset~$\Trans{}{a}{b}$ of~$\Gamma$ such that
\begin{enumerate}[label=$\mathrm{(\alph*)}$, ref=$\mathrm{\alph*}$]
\item\label{BP:intervaldecompGdecomp}
for all~$j\in A$ we have 
\begin{align*}
\bigcup_{k\in A}\bigcup_{g\in\Trans{}{j}{k}}g\act I_{k}&\subseteq 
I_j
\intertext{and}
\bigcup_{k\in 
A}\bigcup_{g\in\Trans{}{j}{k}}g\act\Iset{k}&=\Iset{j}\,,
\end{align*}
and these unions are disjoint,
\item\label{BP:intervaldecompGgeod}
for each pair~$(j,k)\in A\times A$, each $g\in\Trans{}{j}{k}$ and each pair of points~$(z,w)\in\base{\Cs{j}}\times g\act\base{\Cs{k}}$, the geodesic segment~$(z,w)_{\H}$ is nonempty, is contained in~$\Plussp{j}$ and does not intersect~$\Gamma\act\BrU$,
\item\label{BP:intervaldecompback}
for all~$j\in A$ we have 
\[
 \Jset{j}\subseteq\bigcup_{k\in A}\bigcup_{h\in\Trans{}{k}{j}}h^{-1}\act\Jset{k}\,.
\]
\end{enumerate}
\end{enumerate}

We call the sets~$\Cs{j}$, $j\in A$, the \emph{branches} of~$\BrS$,
\index[defs]{branch}%
\index[symbols]{C@$\Cs{j}$}%
and $\BrU$ the \emph{branch union}.
\index[defs]{branch union}%
\index[symbols]{C@$\BrU$}%
Further, we call the sets~$\Trans{}{j}{k}$, $j,k\in A$, the \emph{(forward) transition sets} of~$\BrS$,
\index[defs]{transition set}%
\index[defs]{forward transition set}%
\index[defs]{transition set!forward}%
\index[symbols]{G@$\Trans{}{j}{k}$}%
with $\Trans{}{j}{k}$ being the \emph{(forward) transition set from~$\Cs{j}$ to~$\Cs{k}$}.

A set of branches is called \emph{admissible} if it satisfies the following property:
\index[defs]{admissible}%
\index[defs]{set of branches!admissible}%

\begin{enumerate}[resume*]
\item\label{BP:leavespaceforflip}
\index[defs]{0B08@(B8)}%
\index[symbols]{B08@(B8)}%
There exist a point~$q\in\widehat{\R}$ and an open neighborhood~$\mathcal{U}$ of~$q$ in~$\widehat{\R}$ such that 
\[
\mathcal{U}\cap\bigcup_{j\in A}\Iset{j}=\varnothing\qquad\text{and}\qquad q\notin I_j\,,
\]
for every~$j\in A$.
\end{enumerate}

A set of branches~$\BrS$ is called \emph{non-collapsing} if it satisfies 
\index[defs]{non-collapsing}%
\index[defs]{set of branches!non-collapsing}%

\begin{enumerate}[resume*]
\item\label{BP:noidentity}
\index[defs]{0B09@(B9)}%
\index[symbols]{B09@(B9)}%
For every~$n\in\N$, every choice of~$j_1,\dots,j_{n+1}\in A$ and every choice of elements~$g_i\in\Trans{}{j_i}{j_{i+1}}$ for~$i\in\{1,\ldots,n\}$, we have
\[
g_1\cdots g_n\ne\id\,.
\]
\end{enumerate}
If~$\BrS$ does not satisfy~\eqref{BP:noidentity}, then it is called \emph{collapsing}.
\index[defs]{collapsing}%
\index[defs]{set of branches!collapsing}%
\end{defi}

The following remark highlights some immediate consequences of the properties of a set of branches. 

\begin{remark}\label{REM:SoBrem}
We comment on some properties of a set of branches that will be used throughout and that are immediate by its definition. We resume the notation from Definition~\ref{DEF:setofbranches}.

\begin{enumerate}[label=$\mathrm{(\alph*)}$, ref=$\mathrm{\alph*}$, itemsep=1ex]

\item\label{SoBrem:periodgeod}
A close relationship between the set of branches~$\BrS$ and periodic geodesics 
on~$\Orbi$ is guaranteed by~\eqref{BP:closedgeodesicsHtoX} and, in 
fact,~\eqref{BP:coverlimitset}.
The property \eqref{BP:closedgeodesicsHtoX} assures that every branch 
contributes in a meaningful way to the complete collection of branches by detecting at least one periodic geodesic on~$\Orbi$ or, more precisely, a lift to~$\H$ of a periodic geodesic on~$\Orbi$. 
In particular, it implies that each branch is a nonempty set.
Therefore, for orbifolds without 
periodic geodesics (e.g., a parabolic cylinder, 
see~\cite{Borthwick_book}) a set of branches in the sense of 
Definition~\ref{DEF:setofbranches} does not exist.
On the contrary, property~\eqref{BP:coverlimitset} has the 
consequence that every periodic geodesic on~$\Orbi$ is detected by~$\BrS$. See 
Proposition~\ref{PROP:oldB1} below.

\item The emphasis on periodic geodesics is due to our applications. For a 
strict transfer operator approach and hence a representation of the Selberg zeta 
function as a Fredholm determinant of a transfer operator family, we need to 
provide a certain symbolic presentation of each periodic geodesic on~$\Orbi$ by 
means of iterated intersections with a cross section. For more details we refer 
to the brief discussion in Section~\ref{SUBSEC:cross} as well as 
to~\cite{Moeller_Pohl, Pohl_representation, FP_NECM}.

\item\label{SoBrem:structure}
Properties~\eqref{BP:completegeodesics}--\eqref{BP:allvectors} determine the 
structure of branches. Each branch partitions the hyperbolic plane into two 
geodesically convex half-spaces and a complete geodesic segment. The requirement 
that, for each $j\in A$, the endpoints of~$\overline{\base{\Cs{j}}}$ are 
in~$\wh\R\setminus\wh\R_{\st}$ implies that geodesics~$\gamma$ on~$\H$ with 
$\gamma(\R)=\overline{\base{\Cs{j}}}$ do not represent any periodic 
geodesic on~$\Orbi$. This condition further implies that 
$\Iset{j}\cap\Jset{j}=\varnothing$.

\item\label{SoBrem:descrHalfspaces}
For each~$j\in A$, the requirements 
of~\eqref{BP:pointintohalfspaces}--\eqref{BP:allvectors} yield that the union 
\[
I_j\cup J_j\cup\geo\base{\Cs{j}}
\]
is disjoint and equals~$\widehat{\R}$. The set~$\overline{\base{\Cs{j}}}$ is the 
complete geodesic segment in~$\H$ that connects the two endpoints of~$I_j$ or, 
equivalently, of~$J_j$. The boundary of the half-spaces~$\Plussp{j}$ 
and~$\Minussp{j}$ in~$\H\cup\partial_{\geo}\H$ is 
\begin{eqnarray*}
\partial_{\geo}\Plussp{j}=\overline{\base{\Cs{j}}}\cup\overline{I_j}^{\geo}
&\text{and}&\partial_{\geo}\Minussp{j}=\overline{\base{\Cs{j}}}\cup\overline{
J_j}^{\geo}\,,
\end{eqnarray*}
respectively. It will be useful to fix the following notation for the endpoints 
of~$\overline{\base{\Cs{j}}}$: let $\point{j}{\eX}, \point{j}{\eY}$ be the 
elements in~$\wh\R$ such that 
\[
\{\point{j}{\eX},\point{j}{\eY}\}=\gbound{\overline{\base{\Cs{j}}}}
\]
\index[symbols]{X@$\eX_j$}%
\index[symbols]{Y@$\eY_j$}%
and that, when traveling along the geodesic segment~$\overline{\base{\Cs{j}}}$ 
from~$\point{j}{\eX}$ to~$\point{j}{\eY}$, the half-space~$\Plussp{j}$ lies to 
the right of the path of travel. See Figure~\ref{FIG:halfspaces}.

\item\label{SoBrem:vecfrompoint}
Property~\eqref{BP:allvectors} further has the following consequence for all 
$j\in A$: Let $(x,y)\in\Iset{j}\times\Jset{j}$ and let $\gamma$ be a geodesic 
on~$\H$ from~$x$ to~$y$. The unique vector~$\nu\in\Cs{j}$ with 
$(\gamma_\nu(+\infty),\gamma_\nu(-\infty))=(x,y)$ is then 
\[
\nu=\gamma^{\prime}(t)\,,
\]
where $t\in\R$ is the unique time with $\gamma(t)\in\base{\Cs{j}}$. We emphasize 
that~\eqref{BP:allvectors} does not prevent the branches from containing 
vectors~$\nu$ such that 
$(\gamma_\nu(+\infty),\gamma_\nu(-\infty))\notin\R_{\st}\times\R_{\st}$.

\item\label{SoBrem:pairwisedisjoint}
Properties~\eqref{BP:disjointunion} and~\eqref{BP:intervaldecomp} describe 
the mutual interplay of the branches. Property~\eqref{BP:disjointunion} implies 
that a set of branches~$\{\Cs{1},\dots,\Cs{n}\}$ is pairwise disjoint, 
which will be crucial for the well-definedness of the intersection sequences in 
Section~\ref{SUBSEC:reducedbranches} below.  A stronger 
statement is shown in Proposition~\ref{PROP:CofSoBnonst}\eqref{CofSoB:disjointunion} below.
Property~\eqref{BP:intervaldecomp} 
uses the close relation between each branch~$\Cs{j}$, $j\in A$,  and its 
associated sets~$I_j$ and~$J_j$ in~$\wh\R$ in order to provide the tools 
necessary to track the behavior of geodesics which intersect~$\Cs{j}$ in future 
and past time directions. As we will see in Sections~\ref{SUBSEC:reducedbranches} 
and~\ref{SUBSEC:slowtrans}, the rather precise tracking makes it possible to deduce an explicit discrete model of the geodesic flow, or in other words, of the arising symbolic dynamics or intersection sequences.

\item\label{SoBrem:coincide} In the situation of~\eqref{BP:disjointunion} we 
always have $\overline{\base{\Cs{j}}}=g\act\overline{\base{\Cs{k}}}$. However, 
it does not necessarily follow that $\base{\Cs{j}}=g\act\base{\Cs{k}}$. This 
subtle difference makes it necessary to be rather careful with choices and 
argumentation at some places.

\item\label{SoBrem:infinitesets}
We emphasize that the uniqueness or non-uniqueness of the forward transition 
sets is not part of the requirements in~\eqref{BP:intervaldecomp}. For the 
moment and in particular in isolated consideration of~\eqref{BP:intervaldecomp} 
it may well be that different choices for the families of forward transition 
sets~$(\Trans{}{j}{k})_{j,k\in A}$ can be made. However, in 
Proposition~\ref{PROP:CofSoBst} we will see that the interplay of all properties 
of a set of branches enforces uniqueness of these sets. 

We further emphasize that the transition sets~$\Trans{}{j}{k}$, $j,k\in A$, are 
allowed to be infinite. 
In Example~\ref{EX:branchramification} we show that there are sets of branches with finite as well as with infinite transition sets.
In Section~\ref{SUBSEC:branchram} we provide a characterization of sets of branches 
with infinite transition sets.

\item\label{SoBrem:leavespaceforflip}
Property~\eqref{BP:leavespaceforflip} allows us to suppose that for 
all~$j\in A$, the intervals~$I_j$ are contained in~$\R$. 
For this we possibly need to conjugate $\Gamma$ by some element~$g\in\PSL_2(\R)$, translate~$\BrS$ by~$g$ and consider a set of branches for $g\Gamma g^{-1}$.
In other words, \eqref{BP:leavespaceforflip} allows us to suppose without loss of generality that the discrete dynamical system induced by~$\BrS$ is completely defined within~$\R$ and any handling of a second manifold chart to investigate neighborhoods of~$\infty$ can be avoided.
This often simplifies the discussion, in particular in Section~\ref{SEC:strictTOAexist}.

It is immediately clear that a sufficient (but not necessary) condition for~\eqref{BP:leavespaceforflip} is that~$\wh\R\setminus\bigcup_{k=1}^NI_k$ contains an open interval.
Let~$j\in A$.
Property~\eqref{BP:closedgeodesicsHtoX} is equivalent to the existence of an equivalence class of geodesics~$[\gamma]$ on~$\H$ such that~$\pi(\gamma)\in\Geo_{\Per}(\Orbi)$ and 
\[
(\gamma(+\infty),\gamma(-\infty))\in \Iset{j}\times\Jset{j}\subseteq\ I_j\times J_j\,,
\]
for every representative~$\gamma$ of~$[\gamma]$.
The class~$[\gamma]$ is then the axis of some hyperbolic transformation~$h\in\Gamma$, which, because of Lemma~\ref{LEM:periodicaxes}, contracts the interval~$I_j$ towards~$\fixp{+}{h}=\gamma(+\infty)$.
For every~$j\in A$ a hyperbolic transformation~$h_j\in\Gamma$ can be found in this way, and the contracting behavior assures that we find~$i_1,\dots,i_N\in\N$ such that
\[
\wh\R\setminus\bigcup_{k=1}^Nh_k^{i_k}\act I_k
\]
contains an open interval.
Lemma~\ref{LEM:nouniqueSoB} below implies that
\[
\{h_k^{i_1}\act\Cs{1},\dots,h_N^{i_N}\act\Cs{N}\}
\]
is again a set of branches for the geodesic flow on $\Orbi$.
For orbisurfaces with cusps it is often possible and appropriate to work with parabolic instead of hyperbolic transformations.

We refer to Proposition~\ref{PROP:noncollaps} below for a rigorous treatment of this issue.

\item\label{SoBrem:noidentity}
Indispensable for our approach is a \emph{unique} coding of periodic geodesics in 
terms of the chosen generators of the Fuchsian group.
This requires in particular that the identity transformation will not be encountered during the tracking of geodesics.
This property is formulated as~\eqref{BP:noidentity}. 
Even though \eqref{BP:noidentity} will eventually be fundamental, this property need not be guaranteed immediately during the construction of a set of branches. Indeed, as we will see, every set that fulfills~\eqref{BP:closedgeodesicsHtoX}--\eqref{BP:intervaldecomp} can be transformed into one that fulfills~\eqref{BP:noidentity} (at the tolerable cost of weakening others, most profoundly~\eqref{BP:allvectors}).
This is done by means of a reduction procedure, which we call~\emph{identity elimination}. It is discussed in Section~\ref{SEC:stepredux} below.
\end{enumerate}
\end{remark}

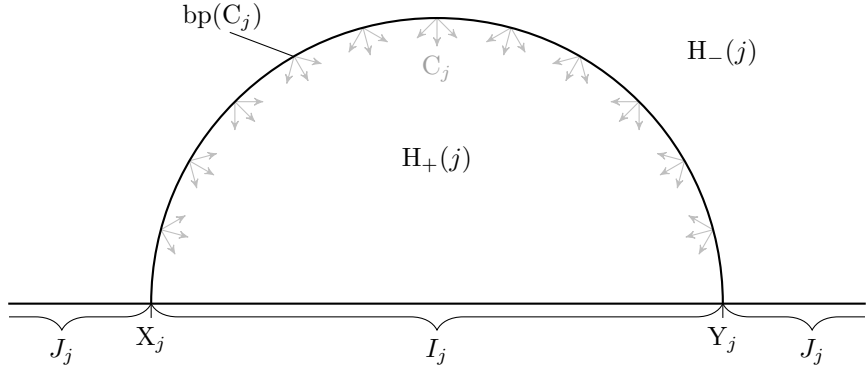
\begin{figure}[h]
\begin{tikzpicture}[scale=9.5]
\foreach \x in {15,30,45,60,75,90,105,120,135,150,165}{
	\foreach \y in {0,45,-45}{
		\tikzmath{\p=.5 + .4*cos(\x); \q=.4*sin(\x);
						 \v=(.9*\p+.05-\p)*cos(\y)-(.9*\q-\q)*sin(\y)+\p; 
						 \w=(.9*\p+.05-\p)*sin(\y)+(.9*\q-\q)*cos(\y)+\q;
						}
		\draw[->,color=lightgray] (\p,\q) -- (\v,\w);
		}
	}
\draw[style=thick] (0:.9) arc (0:180:.4);
 \foreach \x/\y in {.1/$\point{j}{\eX}$,.9/$\point{j}{\eY}$}
    \draw (\x,0.00) -- (\x,-0.02) node [below] {\y};
\coordinate [label=below:$\color{gray!75}\Cs{j}$] (C) at (.5,.36);
\coordinate [label={$\overline{\base{\Cs{j}}}$}] (BC) at (.2,.37);
\draw (.21,.38) -- (.3,.346);
\draw [decorate,decoration={brace,amplitude=10pt,mirror}]
(0.1,0) -- (0.9,0) node [below,midway,yshift=-10pt]
{$I_j$};
\draw [decorate,decoration={brace,amplitude=10pt,mirror}]
(0.9,0) -- (1.15,0) node [below,midway,yshift=-10pt]
{$J_j$};
\draw [decorate,decoration={brace,amplitude=10pt,mirror}]
(-.15,0) -- (0.1,0) node [below,midway,yshift=-10pt]
{$J_j$};
\fill[color=white] (1.1,-.03)--(1.155,-.03)--(1.155,.0)--(1.1,.0)--cycle;
\fill[color=white] (-.1,-.03)--(-.155,-.03)--(-.155,.0)--(-.1,.0)--cycle;
\draw[style=thick] (-.1,0) -- (1.1,0);
\coordinate [label=$\Plussp{j}$] (H+) at (.5,.17);
\coordinate [label=$\Minussp{j}$] (H-) at (.9,.32);
\end{tikzpicture}
\caption[halfspaces]{The relation between the sets $\overline{\base{\Cs{j}}}$, 
$\Plussp{j}$, $\Minussp{j}$, $I_j$, $J_j$ and the points~$\eX_j$ and~$\eY_j$ for 
the branch~$\Cs{j}$.}\label{FIG:halfspaces}
\end{figure}

Examples for sets of branches can be found in~\cite{Pohl_Symdyn2d, 
Pohl_representation, Pohl_diss, Wabnitz}. Indeed, all cross sections 
constructed there arise from sets of branches. In what follows we provide sets of branches for Schottky groups as well as for one further class of Fuchsian groups. Throughout, we will come back to these examples to illustrate our techniques.

\begin{example}\label{EX:schottkySoB}
Let $\Gamma_{\mathrm{S}}$ be a Schottky group,
\index[defs]{Schottky group}%
\index[defs]{Fuchsian group!Schottky}%
that is, a geometrically finite, 
non-cofinite Fuchsian group consisting solely of hyperbolic elements and the 
identity. By~\cite{Button_Schottky}, we may associate to~$\Gamma_{\mathrm{S}}$ a 
choice of \emph{Schottky data},
\index[defs]{Schottky data}%
that is, a tuple
\[
\left(r,\left\{\mathcal{D}_{j},\mathcal{D}_{-j}\right\}_{j=1}^r,\left\{s_{j},s_{
-j}\right\}_{j=1}^r\right),
\]
where $r\in\N$, $\{s_1,\dots,s_r\}\subseteq\PSLR$ is a set of 
generators of~$\Gamma_{\mathrm{S}}$, $s_{-j}\coloneqq s_j^{-1}$ for 
$j\in\{1,\ldots, r\}$, and 
$\mathcal{D}_1,\dots,\mathcal{D}_r,\mathcal{D}_{-1},\dots,\mathcal{D}_{-r}$ are 
mutually disjoint open Euclidean disks in~$\C$ centered on~$\R$ such that for 
each~$j=1,\dots,r$, the element~$s_j$ maps the exterior of~$\mathcal{D}_j$ to 
the interior of~$\mathcal{D}_{-j}$, and such that 
\[
 \H \setminus \bigcup_{j=1}^r \bigl(\overline{\mc D_j} \cup \overline{\mc D_{-j}}\bigr)
\]
is a fundamental domain for~$\Gamma_{\mathrm{S}}$. For $j\in\{\pm1,\dots,\pm 
r\}$ we let $\Cs{j}$ be the set of unit tangent vectors~$\nu\in\UTB\H$ that are 
based on the boundary~$\partial\mathcal{D}_j$ of~$\mc D_j$ and that point 
into~$\mathcal{D}_j$ (thus, 
$\gamma_\nu(+\infty)\in\Rea{\left(\mathcal{D}_j\right)}$). Then 
\[
\{\Cs{1},\dots,\Cs{r},\Cs{-1},\dots,\Cs{-r}\}
\]
is a set of branches for the geodesic flow on the hyperbolic surface~$\quod{\Gamma_{\mathrm{S}}}{\H}$.
\end{example}

\begin{example}\label{EX:G3Def}
Let $\lambda>1$. Let $\Gamma_{\lambda}$ be the subgroup of $\PSLR$ generated by the elements
\begin{eqnarray*}
t_{\lambda}\coloneqq\begin{bmatrix}1&\lambda\\0&1\end{bmatrix}&\text{and}
&h\coloneqq\begin{bmatrix}2&-1\\3&-1\end{bmatrix}\,.
\end{eqnarray*}
Then $\Gamma_{\lambda}$ is discrete in $\PSLR$ and (obviously) geometrically 
finite. It has one funnel (represented by the interval $[1,\lambda]$), one cusp 
(represented by $\infty$), and one elliptic point of order $3$ (represented by 
$\tfrac{1}{2}+\i\tfrac{1}{2\sqrt{3}}$).
A fundamental domain for $\Gamma_{\lambda}$ is indicated in 
Figure~\ref{FIG:G3fund}.
A first choice of a set of branches~$\BrS$ is indicated in Figure~\ref{FIG:G3SoB}. It arises from the cusp expansion algorithm in~\cite{Pohl_diss}.

\begin{figure}[h]
\begin{tikzpicture}[scale=10]
\tikzmath{\q=1/(2*sqrt(3));
				  \r=.8*1/3;
				  \v=.1+\r;
				  \w=.1+2*\r;
				  }
\fill[color=lightgray!50] (1,0) -- (.9,0) arc (0:120:\r) -- (.5,.8*\q) arc 
(60:180:\r) -- (0,0)  -- (0,.6) -- (1,.6) -- cycle;
\draw (.9,0) arc (0:120:\r);
\draw (.5,.8*\q) arc (60:180:\r);
\draw (0,0) -- (0,.6);
\draw (1,0) -- (1,.6);
 \foreach \x/\y in 
{.1/$0$,.9/$1$,0/$\frac{1-\lambda}{2}$,1/$\frac{1+\lambda}{2}$,\v/$\frac{1}{3}$,
\w/$\frac{2}{3}$}
    \draw (\x,0.00) -- (\x,-0.02) node [below] {\y};
\draw[dashed] (\w,0) arc (0:60:\r);
\draw[dashed] (.5,.8*\q) arc (120:180:\r);
\draw[style=thick] (-.1,0) -- (1.1,0);
\coordinate [label=below:$\color{gray}\fund$] (F) at (.5,.5);
\end{tikzpicture}
\caption[halfspaces]{A Ford fundamental domain domain for $\Gamma_{\lambda}$. 
The side-pairing transforms are $t_{\lambda}^{\pm1}$ for the vertical sides and 
$h^{\pm1}$ for the non-vertical sides, respectively.}\label{FIG:G3fund}
\end{figure}
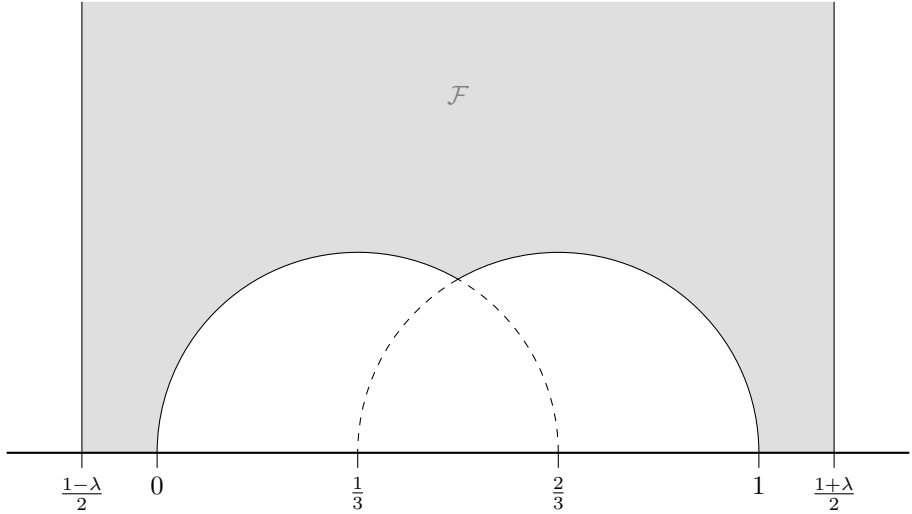

\begin{figure}[h]
\begin{tikzpicture}[scale=9.5]
\tikzmath{\q=1/(2*sqrt(3));
				  \r=.8/6;
				  \v=.1+.8*1/3;
				  \w=.1+2*.8*1/3;
				  }
\foreach \x in {0,.1,\v,\w,.9}
	\fill[color=gray!50] (\x,0) -- (\x+.02,0) -- (\x+.02,.6) -- (\x,.6) -- 
cycle;
\foreach \x in {.1,.9,1}
	\fill[color=gray!50] (\x,0) -- (\x-.02,0) -- (\x-.02,.6) -- (\x,.6) -- 
cycle;
\foreach \x in {0,\v,\w}
	\fill[color=gray!15] (\x,0) -- (\x-.02,0) -- (\x-.02,.6) -- (\x,.6) -- 
cycle;	
\fill[color=gray!15] (1,0) -- (1+.02,0) -- (1+.02,.6) -- (1,.6) -- cycle;	
\foreach \x/\y in 
{\v/$h^{-1}\act\!\Cs{5}$,\w/$h\act\!\Cs{4}$,.9/$h\act\!\Cs{8}$}{
	\fill[color=gray!15] (\x,0) arc (0:180:\r);
	\fill[color=white] (\x-.02,0) arc (0:180:\r-.02);
	\draw (\x,0) arc (0:180:\r);
	\coordinate [label=below:\y] (C\x) at (\x-\r,\r-.01);
	}	
 \foreach \x/\y/\z/\s in 
{0/$\frac{1-\lambda}{2}$/$\Cs{1}$/$t_{\lambda}^{-1}\act\!\Cs{6}$,
 									   .1/$0$/$\Cs{2}$/$\Cs{8}$,
\v/$\frac{1}{3}$/$\Cs{3}$/$h^{-1}\act\!\Cs{4}$,
\w/$\frac{2}{3}$/$\Cs{4}$/$h\act\!\Cs{3}$,
 									   .9/$1$/$\Cs{5}$/$\Cs{7}$,
1/$\frac{1+\lambda}{2}$/$t_{\lambda}\act\!\Cs{1}$/$\Cs{6}$}{
    \draw (\x,0.00) -- (\x,-0.02) node [below] {\y};
    \draw (\x,0) -- (\x,.45) node [right] {\z};
    \draw (\x,.6) -- (\x,.4) node [left] {\s};
    }
\fill[color=white] (1.1,-.03)--(1.155,-.03)--(1.155,.0)--(1.1,.0)--cycle;
\fill[color=white] (-.1,-.03)--(-.155,-.03)--(-.155,.0)--(-.1,.0)--cycle;
\draw[style=thick] (-.1,0) -- (1.1,0);
\end{tikzpicture}
\caption[halfspaces]{A set of branches~$\BrS=\defset{\Cs{j}}{j=1,\dots,8}$ for 
$\Gamma_{\lambda}$ (dark gray) together with their translates relevant for the 
determination of the sets $\Trans{}{.}{.}$ (light gray).}\label{FIG:G3SoB}
\end{figure}
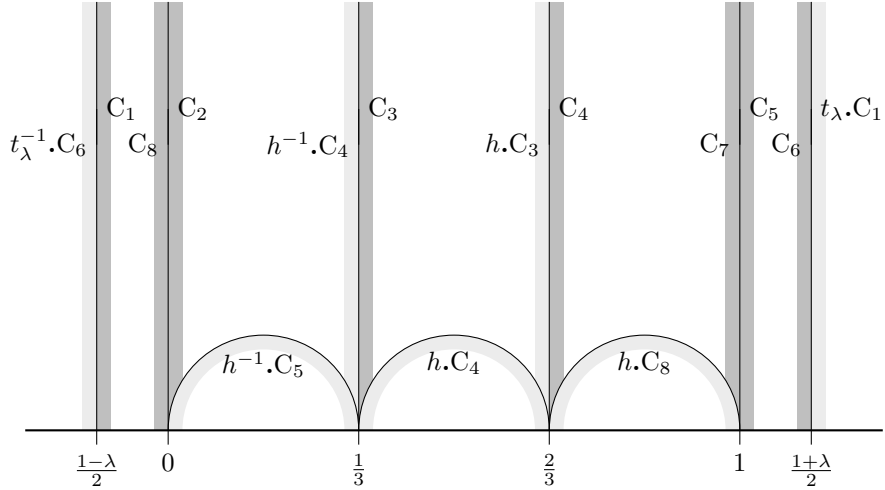
\end{example}

\subsection{Elementary properties of sets of 
branches}\label{SUBSEC:propbranches}

Throughout this section let $\BrS = \{ \Cs{1},\ldots, \Cs{N}\}$ be a set of 
branches for the geodesic flow on~$\Orbi$, set $A\coloneqq \{1,\ldots, N\}$ and 
let $\BrU = \bigcup_{j\in A} \Cs{j}$ denote the branch union of~$\BrS$. In the 
course of this section and the following two sections we will show that 
\[
\CrSc\coloneqq\pi(\BrU)
\]
is a cross section with respect to any measure in a certain class and also in 
the sense of~\eqref{CS:infoftenalt} and~\eqref{CS:discreteintime}. See 
Proposition~\ref{CofSoB:crosssection}. We will further find a subset of~$\CrSc$ 
that is a strong cross section. See Corollary~\ref{CofSoB:strongcrosssection}.

For any $j\in A$, we resume the notation for the sets~$I_j$, $J_j$, $\Plussp{j}$ 
and~$\Minussp{j}$ from~\eqref{BP:pointintohalfspaces}. We fix a family of 
forward transition sets~$(\Trans{}{j}{k})_{j,k\in A}$. See~\eqref{BP:intervaldecomp}. Further, for $j,k\in A$ 
we set
\begin{equation}\label{eq:backsets}
\Past{}{k}{j}\coloneqq 
\Trans{}{k}{j}^{-1}=\defset{g^{-1}}{g\in\Trans{}{k}{j}}\,,
\end{equation}
which we call the \emph{backwards transition set from~$\Cs{k}$ to~$\Cs{j}$}.
\index[defs]{transition set!backwards}%
\index[defs]{backwards transition set}%
\index[symbols]{V@$\Past{}{k}{j}$}%
In this section 
we establish some first properties of the set of branches, the 
transition sets and their interactions, which will be used throughout this 
article.

\begin{prop}\label{PROP:BPintervaldecompV}
The backwards transition sets satisfy the properties dual 
to~\eqref{BP:intervaldecomp}. That is, \eqref{BP:intervaldecomp} holds also for  
$(\Past{}{k}{j})_{j,k\in A}$ in place of~$(\Trans{}{j}{k})_{j,k\in A}$ and the 
roles of~$(I_j)_{j\in A}$ and $(J_j)_{j\in A}$ as well as $(\Plussp{j})_{j\in 
A}$ and $(\Minussp{j})_{j\in A}$ interchanged. More precisely:
\begin{enumerate}[label=$\mathrm{(\roman*)}$, ref=$\mathrm{\roman*}$]
\item\label{BP:intervaldecompVdecomp}
For all~$j\in A$ we have 
\begin{align*}
\bigcup_{k\in A}\bigcup_{g\in\Past{}{k}{j}}g\act J_{k}&\subseteq 
J_j
\intertext{and}
\bigcup_{k\in A}\bigcup_{g\in\Past{}{k}{j}}g\act\Jset{k}&=\Jset{j}\,,
\end{align*}
and these unions are disjoint,

\item\label{BP:intervaldecompVgeod}
for each pair~$(j,k)\in A\times A$, each $g\in\Past{}{k}{j}$ and each pair of 
points~$(v,z)\in g\act\base{\Cs{k}}\times \base{\Cs{j}}$, the geodesic segment~
$(v,z)_{\H}$ is nonempty, contained in~$\Minussp{j}$ and does not intersect
$\Gamma\act\BrU$.
\end{enumerate}
\end{prop}

\begin{proof}
We first establish~\eqref{BP:intervaldecompVdecomp}. Let~$j\in A$. For any $k\in 
A$ and $g\in \Past{}{k}{j}=\Trans{}{k}{j}^{-1}$ we have 
\[
 g^{-1}\act I_j \subseteq I_k
\]
by~(\ref{BP:intervaldecomp}\ref{BP:intervaldecompGdecomp}). Therefore 
$g^{-1}\act J_j\supseteq J_k$ and hence 
\[
 g\act J_k \subseteq J_j\,.
\]
It follows that 
\begin{equation}\label{eq:Vunion}
 \bigcup_{k\in A}\bigcup_{g\in\Past{}{k}{j}}g\act J_k \subseteq J_j
\end{equation}
and further 
\[
 \bigcup_{k\in A}\bigcup_{g\in\Past{}{k}{j}}g\act \Jset{k} \subseteq \Jset{j}\,.
\]
Combining the latter with~(\ref{BP:intervaldecomp}\ref{BP:intervaldecompback}) 
shows the claimed equality of sets. It remains to show that the unions 
in~\eqref{eq:Vunion} are disjoint. To that end let $k_1,k_2\in A$ and 
$g_1\in\Past{}{k_1}{j}$, $g_2\in\Past{}{k_2}{j}$ such that 
\[
 g_1\act J_{k_1} \cap g_2\act J_{k_2}\not=\varnothing. 
\]
If we assume that 
\[
 g_1\act J_{k_1} \not= g_2\act J_{k_2}\,,
\]
then 
\begin{align*}
 g_2\act I_{k_2} &\cap g_1\act J_{k_1} \not=\varnothing\qquad\text{and}\qquad 
 g_2\act I_{k_2} \cap g_1\act I_{k_1} \not=\varnothing\,,
\end{align*}
and hence 
\begin{align*}
 g_2\act \Plussp{k_2} & \cap g_1\act \Minussp{k_1}\not=\varnothing\,,
 \\
 g_2\act \Plussp{k_2} & \cap g_1\act \Plussp{k_1}\not=\varnothing\,, 
 \intertext{and}
 g_2\act \overline{\base{\Cs{k_2}}} &\cap g_1\act \overline{\base{\Cs{k_1}}} 
\not=\varnothing\,.
\end{align*}
Property~\eqref{BP:disjointunion} implies that this constellation is impossible. In 
turn, 
\[
 g_1\act J_{k_1} = g_2\act J_{k_2}\,.
\]
It follows that $g_1\act I_{k_1} = g_2\act I_{k_2}$ and further 
\begin{align*}
 g_1\act \overline{\base{\Cs{k_1}}} & = g_2\act\overline{\base{\Cs{k_2}}}
 \intertext{as well as}
 g_1\act \Plussp{k_1} & = g_2\act \Plussp{k_2}\,.
\end{align*}
Thus, $k_1=k_2$ and $g_1=g_2$ by~\eqref{BP:disjointunion}, which shows that the 
unions in~\eqref{eq:Vunion} are disjoint.

We now show~\eqref{BP:intervaldecompVgeod}. To that end let $j,k\in A$, $g\in 
\Past{}{k}{j}$ and $(v,z)\in g\act\base{\Cs{k}}\times \base{\Cs{j}}$. Then, 
since 
\[
 g^{-1}\act (v,z) = (g^{-1}\act v, g^{-1}\act z) \in \base{\Cs{k}} \times 
g^{-1}\act \base{\Cs{j}}
\]
and $g^{-1}\in\Trans{}{k}{j}$, 
(\ref{BP:intervaldecomp}\ref{BP:intervaldecompGgeod}) shows that $g^{-1}\act 
(v,z)_{\H}$ (and hence $(v,z)_\H$) is nonempty and does not 
intersect~$\Gamma\act\BrU$. By 
(\ref{BP:intervaldecomp}\ref{BP:intervaldecompGdecomp}), $g^{-1}\act I_j 
\subseteq I_k$ and therefore $\Plussp{j} \subseteq g\act \Plussp{k}$. 
Combination with (\ref{BP:intervaldecomp}\ref{BP:intervaldecompGgeod}) yields 
\[
 (v,z)_\H \subseteq g\act \Plussp{k} \setminus \Plussp{j} \subseteq 
\Minussp{j}\,.
\]
This completes the proof.
\end{proof}

We recall that a family~$\mf B$ of subsets of~$\H$ is called \emph{locally 
finite in~$\H$}
\index[defs]{locally finite}%
if for each~$z\in\H$ there exists an open neighborhood~$U$ 
of~$z$ in~$\H$ such that at most finitely many members of~$\mf B$ intersect~$U$.

\begin{prop}\label{PROP:branches_locfinite}
The family 
\[
 \mf B\coloneqq\defset{ g\act \overline{\base{\Cs{j}}} }{ g\in\Gamma,\ j\in A }
\]
of $\Gamma$-translates of the closures of the base sets of the branches in~$\BrS$ is locally finite in~$\H$. 
\end{prop}

\begin{proof}
Let $\sigma \subseteq\H$ be a complete geodesic segment with both its endpoints 
in~$\wh\R\setminus\wh\R_{\st}$. In what follows we show that the family of 
$\Gamma$-translates of~$\sigma$ is locally finite in~$\H$. Since the family of the closures of the base sets of the branches~$\BrS$ consists of finitely many of such geodesic segments, the statement of the proposition follows then immediately. 

We find and fix a Siegel-set-type fundamental domain~$\mc F$ for~$\Gamma$ 
in~$\H$ (see, e.g., \cite{Garland_Raghunathan} for cofinite Fuchsian groups 
and~\cite[Section~10]{Beardon} for general, geometrically finite Fuchsian 
groups). That is, each cusp~$\wh c$ of~$\Orbi$ is represented by a point, 
say~$c$, in~$\overline{\mc F}^{\geo}$, and the part of~$\mc F$ near~$c$ is of the 
form (\emph{Siegel set})
\index[defs]{Siegel set}%
\[
 g\act\defset{ z\in\H }{ |\Rea z| < w,\ \Ima z > h }
\]
for some~$g=g(c)\in\PSL_2(\R)$, $w=w(c)>0$ and $h=h(c)>0$. The neighboring 
translates of~$\mc F$ at~$c$ are given by~$p\act\mc F$ and~$p^{-1}\act\mc F$ 
where $p\in\Gamma$ is a generator of~$\Stab{\Gamma}{c}$. Each funnel 
of~$\Orbi$ is represented by a funnel representative, say~$I$, in~$\overline{\mc 
F}^{\geo}$, and the part of~$\mc F$ near~$I$ is bounded by two geodesic segments, 
each one having one of the two boundary points of~$I$ as an endpoint. The 
neighboring translates of~$\mc F$ at this part are of the form~$b\act\mc F$ 
and~$b^{-1}\act\mc F$, where $b$ is a primitive element of~$\Gamma$ whose axis 
represents a funnel bounding geodesic. A funnel bounding geodesic is a   periodic geodesic on~$\Orbi$ that is furthest into the funnel. Funnel bounding geodesics are unique up to orientation.
The whole fundamental domain~$\mc F$ is a finite union of geodesically convex subsets of~$\H$. 

Since each endpoint of~$\sigma$ is either cuspidal or contained in a funnel 
representative, the shape of~$\mc F$ implies that we find~$g_1,g_2\in\Gamma$ 
such that~$g_1\act\mc F\cup g_2\act\mc F$ covers ``most'' of~$\sigma$. In other 
words, 
\[
 \beta \coloneqq \sigma\setminus (g_1\act\mc F \cup g_2\act\mc F)
\]
is a geodesic segment of finite hyperbolic length. Again, the shape of~$\mc F$ 
implies that~$\beta$ can be covered by finitely many $\Gamma$-translates 
of~$\overline{\mc F}$. In total, the complete geodesic segment~$\sigma$ 
intersects only finitely many $\Gamma$-translates of~$\overline{\mc F}$. 
Equivalently, $\overline{\mc F}$ contains only finitely many $\Gamma$-translates 
of~$\sigma$. Since $\sigma$, and each of its $\Gamma$-translates, is a geodesic 
segment, it immediately follows that the family~$\Gamma\act\sigma$ is locally 
finite in~$\overline{\mc F}$, and hence in all of~$\H$. This completes the 
proof.
\end{proof}

We now aim to prove that the set of branches~$\BrS$ for~$\Gamma$ accounts for all periodic geodesics on~$\Orbi$, in the sense that every~$\widehat{\gamma}\in\Geo_{\Per}(\Orbi)$ has a representative~$\gamma\in\Geo(\H)$ that intersects~$\BrU$.
For this we take advantage of the following equivalent formulation 
of~\eqref{BP:coverlimitset}.

\begin{lemma}\label{LEM:coverlimitset}
Set~$\Geo_{\Per,\Gamma}(\H)\coloneqq\defset{\gamma\in\Geo(\H)}{
\pi(\gamma)\in\Geo_{\Per}(\Orbi)}$.
Then~\eqref{BP:coverlimitset} is equivalent to
\begin{equation}\tag{$\mathrm{B4'}$}\label{EQN:coverplusinfty}
\forall\gamma\in\Geo_{\Per,\Gamma}(\H)\colon\gamma(+\infty)\in\Gamma\act\bigcup_
{j\in A}I_j\,.
\end{equation}
\index[defs]{0B04@(B4')}%
\index[symbols]{B04@(B4')}%
\end{lemma}

\begin{proof}
Recall from~\eqref{EQDEF:EOrbi} the set
\[
E(\Orbi)=\defset{(\gamma(+\infty),\gamma(-\infty))}{\gamma\in\Geo_{\Per,\Gamma}
(\H)}\,.
\]
The inclusion relation~$E(\Orbi)\subseteq\wh\R_{\st}\times\wh\R_{\st}$ immediately shows that~\eqref{BP:coverlimitset} implies~\eqref{EQN:coverplusinfty}.
Since~$E(\Orbi)$ is dense in~$\Lambda(\Gamma)\times\Lambda(\Gamma)$ by 
Proposition~\ref{PROP:EXliesdense}, so 
does~$\wh\R_{\st}\times\wh\R_{\st}$.
Thus, the openness of the sets~$I_j$, $j\in A$, yields that \eqref{EQN:coverplusinfty} implies~\eqref{BP:coverlimitset}.
\end{proof}

\begin{prop}\label{PROP:oldB1}
Under the assumption of~\eqref{BP:allvectors} and~\eqref{BP:intervaldecomp}, 
property~\eqref{BP:coverlimitset} is equivalent to the following statement:
\begin{enumerate}[label=$\mathrm{(B_{\Per})}$, ref=$\mathrm{B_{\Per}}$]
\item\label{BP:closedgeodesicsXtoH}
For all~$\widehat{\gamma}\in\Geo_{\Per}(\Orbi)$ there exists~$\gamma\in\Geo(\H)$ 
such that~$\pi(\gamma)=\widehat{\gamma}$ and~$\gamma$ intersects~$\BrU$.
\end{enumerate}
\index[defs]{0B12@(B$_{\Per}$)}%
\index[symbols]{B12@(B$_{\Per}$)}%
\end{prop}

\begin{proof}
Let~$\gamma\in\Geo_{\Per,\Gamma}(\H)$ and 
set~$\widehat{\gamma}\coloneqq\pi(\gamma)\in\Geo_{\Per}(\Orbi)$.
By combining Proposition~\ref{PROP:hypconjisperiod} and 
Lemma~\ref{LEM:periodicaxes}\eqref{periodicaxes:conjugacy} one sees that the set 
of all representatives of~$\widehat{\gamma}$ is given by~$\Gamma\act\gamma$.
Hence, if one representative of~$\widehat{\gamma}$ intersects~$\BrU$, then each 
of its representatives intersect~$\Gamma\act\BrU$.
Hence, given~\eqref{BP:closedgeodesicsXtoH}, there exists~$(j,g)\in 
A\times\Gamma$ such that~$\gamma$ intersects~$g\act\Cs{j}$, i.e.,
\[
(\gamma(+\infty),\gamma(-\infty))\in g\act\Iset{j}\times g\act\Jset{j}\subseteq 
g\act I_j\times g\act J_j\,.
\]
Thus,~\eqref{BP:closedgeodesicsXtoH} implies~\eqref{EQN:coverplusinfty}.
For the converse implication, we suppose that~\eqref{EQN:coverplusinfty} is satisfies and let~$\widehat{\gamma}\in\Geo_{\Per}(\Orbi)$.
As discussed above, it suffices to find a representative of~$\widehat{\gamma}$ that intersects~$\Gamma\act\BrU$.
Let~$\gamma\in\Geo(\H)$ be any representative of~$\widehat{\gamma}$.
Then~$\gamma\in\Geo_{\Per,\Gamma}(\H)$. Further, \eqref{EQN:coverplusinfty} yields the existence of a pair~$(k_1,g_1)\in A\times\Gamma$ such that~$\gamma(+\infty)\in g_1\act I_{k_1}$.
By~\eqref{BP:allvectors}, for a geodesic~$\eta\in\Geo(\H)$ to 
intersect~$g\act\Cs{k}$, $(k,g)\in A\times\Gamma$, it suffices to have
\[
(\eta(+\infty),\eta(-\infty))\in g\act\Iset{k}\times g\act\Jset{k}\,.
\]
Since~$E(\Orbi)\subseteq\wh\R_{\st}\times\wh\R_{\st}$ we immediately 
have~$(\gamma(+\infty),\gamma(-\infty))\in\wh\R_{\st}\times\wh\R_{\st}$.
Therefore, if~$\gamma(-\infty)\in g_1\act J_{k_1}$, the statement 
of~\eqref{BP:closedgeodesicsXtoH} follows.
In order to seek a contradiction, we assume that this is not the case.
Since~$\eX_j,\eY_j\in\wh\R\setminus\wh\R_{\st}$, it follows 
that~$\gamma(-\infty)\in g_1\act I_{k_1}$.
By~(\ref{BP:intervaldecomp}\ref{BP:intervaldecompGdecomp}) we 
find~$(k_2,g_2)\in\left(A\times\Gamma\right)\setminus\{(k_1,g_1)\}$ such that
\[
g_1^{-1}g_2\in\Trans{}{k_1}{k_2}\qquad\text{and}\qquad\gamma(+\infty)\in g_2\act 
I_{k_2}\,.
\]
Now the same argumentation as before applies: If~$\gamma(-\infty)\in g_2\act 
J_{k_2}$, then~\eqref{BP:closedgeodesicsXtoH} follows.
If this is not the case, then necessarily~$\gamma(-\infty)\in g_2\act I_{k_2}$ 
and we apply~(\ref{BP:intervaldecomp}\ref{BP:intervaldecompGdecomp}) to 
find~$(k_3,g_3)\in\left(A\times\Gamma\right)\setminus\{(k_1,g_1),(k_2,g_2)\}$ 
such that
\[
g_1^{-1}g_2^{-1}g_3\in\Trans{}{k_2}{k_3}\qquad\text{and}\qquad\gamma(+\infty)\in 
g_3\act I_{k_3}\,.
\]
We now show that iterations of this procedure must terminate after finitely many 
steps by finding~$\gamma(-\infty)\in g_i\act J_{k_i}$ for some~$i\in\N$.
Assume for contradiction that this is not the case. Thus, the above procedure 
yields a sequence~$((k_n,g_n))_{n\in\N}$ in~$A\times\Gamma$ such that
\[
\{\gamma(+\infty),\gamma(-\infty)\}\subseteq\bigcap_{n\in\N}g_n\act I_{k_n}\,.
\]
Then Proposition~\ref{PROP:hypconjisperiod} provides a hyperbolic 
element~$h\in\Gamma$ such that
\[
\gamma(+\infty)=\fixp{+}{h}\ne\fixp{-}{h}=\gamma(-\infty)\,,
\]
meaning the geodesic arc~$\gamma(\R)$ is non-degenerate.
Further, from the construction it is clear that
\[
g_{i+1}\act I_{k_{i+1}}\subseteq g_{i}\act I_{k_{i}}
\]
for all~$i\in\N$.
Hence, the sequence~$((g_n\act\eX_{k_n},g_n\act\eY_{k_n}))_{n\in\N}$ converges 
to some pair $(x,y)\in\widehat{\R}\times\widehat{\R}$ with~$x,y\in \left(g_1\act 
I_{k_1}\right)\setminus\Rea{\left(\gamma(\R)\right)}$.
But this entails the convergence of the 
sequence~$(g_n\act\overline{\base{\Cs{k_n}}})_{n\in\N}$ to the geodesic 
arc~$(x,y)_{\H}$.
Any neighborhood of any point~$z\in(x,y)_{\H}$ therefore intersects infinitely 
many members of the family~$\defset{g_n\act\overline{\base{\Cs{k_n}}}}{n\in\N}$, 
which contradicts to the local finiteness of its 
superset~$\mathfrak{B}=\defset{g\act\overline{\base{\Cs{j}}}}{(j,g)\in 
A\times\Gamma}$ ensured by Proposition~\ref{PROP:branches_locfinite}.
In turn,  the above procedure must terminate after finitely many steps, thereby 
showing that~$\gamma$ intersects~$\Gamma\act\BrU$.
Combining this with Lemma~\ref{LEM:coverlimitset} finishes the proof.
\end{proof}

The following proposition is the first immediate step towards proving that 
$\CrSc$ is a cross section with $\BrU$ as set of representatives. We show that 
the intersections of geodesics on~$\Orbi$ with~$\CrSc$ are bijective to the 
intersections of geodesics on~$\H$ with~$\BrU$. This observation will be crucial 
for establishing discreteness of intersections. We show further that the map 
\begin{equation}\label{eq:mapsetrep}
 \pi\vert_{\BrU}\colon \BrU\to\CrSc\,,\quad \nu\mapsto \pi(\nu)\,,
\end{equation}
is a bijection. Hence, as soon as $\CrSc$ is known to be a cross section, then 
$\BrU$ is a set of representative for~$\CrSc$. To simplify the exposition, we 
will already call $\BrU$ a set of representatives, thereby refering 
to~\eqref{eq:mapsetrep}.

\begin{prop}\label{PROP:CofSoBnonst}
The set~$\BrU$, the family of the~$\Gamma$-translates of its elements, and its 
image~$\wh\BrU$ under~$\pi$ satisfy the following properties:
\begin{enumerate}[label=$\mathrm{(\roman*)}$, ref=$\mathrm{\roman*}$]
\item\label{CofSoB:disjointunion}
The members of the family $\defset{g\act\nu}{g\in\Gamma,\,j\in 
A,\,\nu\in\Cs{j}}$ are pairwise distinct. In particular, $\BrU$ is a set of 
representatives for~$\CrSc$.

\item\label{CofSoB:uniqueintersect}
Let $\widehat{\gamma}$ be a geodesic on~$\Orbi$ that
intersects~$\CrSc$ at time~$t$. Then there exists a unique geodesic~$\gamma$ 
on~$\H$ such that $\pi(\gamma)=\widehat{\gamma}$ and $\gamma$ 
intersects~$\BrU$ at time~$t$.
\end{enumerate}
\end{prop}

\begin{proof}
In order to prove~\eqref{CofSoB:disjointunion}, let $j,k\in A$, $\nu\in\Cs{j}$,  
$\eta\in\Cs{k}$ and $g\in\Gamma$ such that $\nu=g\act\eta$. Thus 
$\base{\Cs{j}}\cap g\act\base{\Cs{k}}\ne\varnothing$. 
Then~\eqref{BP:disjointunion} implies $\base{\Cs{j}}=g\act\base{\Cs{k}}$. Since 
$\nu=g\act\eta$, we have 
$\Plussp{j}\cap g\act\Plussp{k}\ne\varnothing$ 
by~\eqref{BP:pointintohalfspaces}. Using again~\eqref{BP:disjointunion}, we 
obtain $j=k$ and $g=\id$. This shows~\eqref{CofSoB:disjointunion}.

In order to prove~\eqref{CofSoB:uniqueintersect} let~$\widehat{\gamma}$ be a 
geodesic on~$\Orbi$ that intersects~$\CrSc$ at~$t$. Without loss of generality, 
we may suppose that $t=0$ (otherwise we apply a reparametrization 
of~$\widehat\gamma$). Let $\widehat{\nu}\coloneqq \widehat{\gamma}^{\prime}(0)$. 
Since $\BrU$ is a set of representatives for~$\CrSc$ 
by~\eqref{CofSoB:disjointunion}, there exists a unique element~$\nu\in\BrU$ such 
that~$\pi(\nu)=\wh\nu$. Thus, $\gamma_\nu$ is the unique lift of~$\wh\gamma$ 
to~$\H$ that intersects~$\CrSc$ at $t=0$. This completes the proof.
\end{proof}

The following result, which is a straightforward consequence of the conformity 
of linear fractional transformations and the definitions of the sets involved, shows 
that there is no unique choice of a set of branches for the set~$\CrSc = 
\pi(\BrU)$.  

\begin{lemma}\label{LEM:nouniqueSoB}
Let $g_1,\dots,g_N\in\Gamma$. Then $\{g_1\act\Cs{1},\dots,g_N\act\Cs{N}\}$ is a 
set of branches for~$\GeoFlow$, and $\pi(\bigcup_{j=1}^N g_j\act\Cs{j})=\CrSc$.
\end{lemma}

\subsection{Strong branches and iterated 
intersections}\label{SUBSEC:reducedbranches}

Throughout this section we continue to set $A \coloneqq \{1,\ldots, N\}$, let 
$\BrS = \defset{ \Cs{j} }{ j\in A }$ be a set of branches and let $\BrU = 
\bigcup_{j\in A} \Cs{j}$ denote the branch union of~$\BrS$. We pick again a 
family of forward transition sets~$(\Trans{}{j}{k})_{j,k\in A}$ and denote the 
family of backwards transition sets by $(\Past{}{k}{j})_{j,k\in A}$ (cf.\@ 
\eqref{eq:backsets}). For any $j\in A$, we resume the notation for the 
sets~$I_j$, $J_j$, $\Plussp{j}$ and~$\Minussp{j}$ 
from~\eqref{BP:pointintohalfspaces}. 

In this section we show that the transition sets are indeed unique, and we 
provide an alternative characterization of them. See 
Proposition~\ref{PROP:CofSoBst}. Moreover, we prepare the ground for showing 
that $\CrSc \coloneqq \pi(\BrU)$ is intersected by almost all geodesics 
infinitely often in future and past, for finding a strong cross section as a 
subset of~$\CrSc$ and for determining the induced discrete dynamical system on 
subsets of~$\wh\R$.

For each~$j\in A$ we call 
\begin{equation}\label{eq:def_redbranch}
\Cs{j,\st}\coloneqq\defset{\nu\in\Cs{j}}{\left(\gamma_\nu(+\infty),
\gamma_\nu(-\infty)\right)\in\widehat{\R}_{\st}\times\widehat{\R}_{\st}},
\end{equation}
the \emph{strong branch},
\index[defs]{strong branch}%
\index[defs]{branch!strong}%
\index[symbols]{Cjs@$\Cs{j,\st}$}%
and we let 
\begin{equation}\label{eq:def_redset}
\BrU_{\st}\coloneqq\bigcup_{j\in A}\Cs{j,\st}
\end{equation}
denote the \emph{strong set of representatives}. 
\index[defs]{strong set of representatives}%
\index[defs]{set of representatives!strong}%
\index[defs]{representative!strong set}%
\index[symbols]{Cst@$\BrU_{\st}$}%

Our first goal is to show that for each vector~$\nu\in\BrU_{\st}$, the 
geodesic~$\gamma_\nu$ on~$\H$ has a (well-defined) minimal intersection 
time~$t^+>0$ with~$\Gamma\act\BrU_{\st}$, the \emph{next intersection time},
\index[defs]{next intersection time}%
\index[defs]{intersection time!next}%
\index[symbols]{ta@$t^+$}%
as 
well as a maximal intersection time $t^-<0$ with~$\Gamma\act\BrU_{\st}$, the 
\emph{previous intersection time}.
\index[defs]{previous intersection time}%
\index[defs]{intersection time!previous}%
\index[symbols]{taa@$t^-$}%
We start with a remark and some preparatory 
lemmas.

\begin{remark}\label{REM:inst} After the restriction to strong sets and branches, we observe the following relations:
\begin{enumerate}[label=$\mathrm{(\roman*)}$, ref=$\mathrm{\roman*}$]
\item\label{insti} For each~$j\in A$, property~\eqref{BP:closedgeodesicsHtoX} 
implies that $\Iset{j}\not=\varnothing$ and $\Jset{j}\not=\varnothing$.
\item\label{instii} For each~$j\in A$ and each pair~$(x,y)\in I_{j,\st}\times 
J_{j,\st}$, 
there exists a unique vector~$\nu\in\BrU_j$ such that 
\[
 (\gamma_\nu(+\infty),\gamma_\nu(-\infty)) = (x,y)
\]
by~\eqref{BP:allvectors}. Clearly, $\nu\in\BrU_{j,\st}$.
\item\label{instiii} Let $j\in A$ and $\nu\in\Cs j$. If the 
geodesic~$\gamma_\nu$ represents a periodic geodesic on~$\Orbi$, then $\nu\in\Cs 
{j,\st}$ (see Proposition~\ref{PROP:EXliesdense} and~\eqref{eq:def_st}).
\end{enumerate}
\end{remark}

From the definition of the transition sets~$\Trans{}{j}{k}$ it is obvious 
that $g\act\Plussp{k}\subseteq\Plussp{j}$ for all~$g\in\Trans{}{j}{k}$. 
The following lemma shows that this inclusion is indeed proper, and that also 
the dual property holds with the backwards transition sets.

\begin{lemma}\label{LEM:halfspaceincludeGV}
Let $j,k\in A$.
\begin{enumerate}[label=$\mathrm{(\roman*)}$, ref=$\mathrm{\roman*}$]
\item\label{halfspaceincludeGV+} For all~$g\in\Trans{}{j}{k}$ we have 
$g\act\Plussp{k}\varsubsetneq\Plussp{j}$.
\item\label{halfspaceincludeGV-} For all~$g\in\Past{}{k}{j}$ we have 
$g\act\Minussp{k}\varsubsetneq\Minussp{j}$.
\end{enumerate}
\end{lemma}

\begin{proof}
It suffices to establish~\eqref{halfspaceincludeGV+} as the proof 
of~\eqref{halfspaceincludeGV-} is analogous. Let $g\in\Trans{}{j}{k}$. 
Then~(\ref{BP:intervaldecomp}\ref{BP:intervaldecompGdecomp}) shows that
$g\act I_k\subseteq I_j$. Remark~\ref{REM:SoBrem}\eqref{SoBrem:descrHalfspaces} 
and the convexity of the half-spaces $\Plussp{.}$ imply that 
$g\act\Plussp{k}\subseteq\Plussp{j}$. We now assume that 
$g\act\Plussp{k}=\Plussp{j}$, in order to seek a contradiction. Then
\[
g\act\overline{\base{\Cs{k}}}\cup 
g\act\overline{I_k}^{\geo}=g\act\partial_{\geo}\Plussp{k}=\partial_{\geo}\Plussp
{j}=\overline{\base{\Cs{j}}}\cup\overline{I_j}^{\geo}
\]
by Remark~\ref{REM:SoBrem}\eqref{SoBrem:descrHalfspaces} and the continuity of 
the action of~$g$. Since $\H$ and~$\partial_{\geo}\H$ are both stable under the 
action of~$\PSLR$, it follows that 
$g\act\overline{\base{\Cs{k}}}=\overline{\base{\Cs{j}}}$. Thus, for 
any~$z\in\overline{\base{\Cs{j}}}$ we have $z\in g\act\overline{\base{\Cs{k}}}$ 
and $(z,z)_{\H} = \varnothing$, in contradiction 
to~(\ref{BP:intervaldecomp}\ref{BP:intervaldecompGgeod}). In turn,  
$g\act\Plussp{k}\varsubsetneq\Plussp{j}$.
\end{proof}

We emphasize that in the following two lemmas, $g$ is not required to be 
in~$\Trans{}{j}{k}$ or~$\Past{}{k}{j}$, respectively. These lemmas allow us to 
establish next and previous intersections and to narrow down their locations by coarse-locating the endpoints of geodesics.

\begin{lemma}\label{LEM:halfspaceFuture}
Let $j,k\in A$ and let $g\in\Gamma$ be such that $g\act I_k\subseteq I_j$. Then 
we have
\begin{enumerate}[label=$\mathrm{(\roman*)}$, ref=$\mathrm{\roman*}$]
\item\label{halfspaceFutureinclude}
$g\act\Plussp{k}\subseteq\Plussp{j}$.
\item\label{halfspaceFutureintersect}
For all $\nu\in\Cs{j,\st}$ with $\gamma_\nu(+\infty)\in g\act I_k$ there exists 
$t\geq0$ such that $\gamma_\nu^{\prime}(t)\in g\act\Cs{k}$.
\item\label{halfscapeFuturestrict}
If $g\act I_k\varsubsetneq I_j$, then $g\act\Plussp{k}\varsubsetneq \Plussp{j}$ and, 
in~\eqref{halfspaceFutureintersect}, $t>0$.
\end{enumerate}
\end{lemma}

\begin{proof}
For the proof of~\eqref{halfspaceFutureinclude} and the first part 
of~\eqref{halfscapeFuturestrict} we use the characterization of the 
half-space~$\Plussp{j}$ from 
Remark~\ref{REM:SoBrem}\eqref{SoBrem:descrHalfspaces}. The continuity of~$g$ 
implies $g\act\overline{I_k}^{\geo}=\overline{g\act 
I_k}^{\geo}\subseteq\overline{I_j}^{\geo}$. We denote the two endpoints of~$I_k$ 
by~$x$ and~$y$. From  $(x,y)_{\H}=\overline{\base{\Cs{k}}}$ and $g\act x,\,g\act 
y\in\overline{I_j}^{\geo}$ we obtain 
$g\act\overline{\base{\Cs{k}}}\subseteq\overline{\Plussp{j}}^{\geo}$. Thus, 
$g\act\Plussp{k}\subseteq\Plussp{j}$. If $g\act I_k\not= I_j$, then $g\act 
(x,y)_\H$ passes through (the interior) of~$\Plussp{j}$. Further, at least one of the points~$g\act x$ and~$g\act y$ is in $I_j$. Thus, in this case, 
$g\act\Plussp{k}\varsubsetneq\Plussp{j}$.

For the proof of~\eqref{halfspaceFutureintersect} and the second part 
of~\eqref{halfscapeFuturestrict} let $\nu\in\Cs{j,\st}$ be such that 
$\gamma_\nu(+\infty)\in g\act I_k$. The hypothesis $g\act I_k\subseteq I_j$ 
implies $g\act\Iset{k}\subseteq\Iset{j}$ and $g\act J_k\supseteq J_j$ and 
further $g\act\Jset{k}\supseteq\Jset{j}$. Since 
$(\gamma_\nu(+\infty),\gamma_\nu(-\infty))\in\Iset{j}\times\Jset{j}$ as 
$\nu\in\Cs{j,st}$, it follows that $\gamma_\nu(+\infty)\in\widehat{\R}_{\st}\cap 
g\act 
I_k=g\act\Iset{k}$. Therefore, $(\gamma_\nu(+\infty),\gamma_\nu(-\infty))\in 
g\act\Iset{k}\times g\act\Jset{k}$ or, equivalently, 
\[
 g^{-1}\act(\gamma_\nu(+\infty),\gamma_\nu(-\infty))\in \Iset{k}\times 
\Jset{k}\,.
\]
By~\eqref{BP:allvectors} there exists a (unique) vector~$\eta\in\Cs{k}$ such 
that 
\[
(\gamma_\eta(+\infty),\gamma_\eta(-\infty))=g^{-1}\act(\gamma_{\nu}(+\infty),
\gamma_{\nu}(-\infty))\,.
\]
The uniqueness of geodesics connecting two points in~$\overline{\H}^{\geo}$ 
implies that there exists~$t\in\R$ such that $\gamma_\nu^{\prime}(t)=g\act\eta$. 
The combination of~\eqref{BP:pointintohalfspaces}, 
\eqref{halfspaceFutureinclude} and the first part 
of~\eqref{halfscapeFuturestrict} yield that $t\geq 0$ and, in case of $g\act 
I_k\varsubsetneq I_j$, $t>0$.
\end{proof}

The proof of the following result is analogous to that of 
Lemma~\ref{LEM:halfspaceFuture}, for which reason we omit it.

\begin{lemma}\label{LEM:halfspacePast}
Let $j,k\in A$ and let $g\in\Gamma$ be such that $g\act J_k\subseteq J_j$. Then 
we have
\begin{enumerate}[label=$\mathrm{(\roman*)}$, ref=$\mathrm{\roman*}$]
\item\label{halfspacePastinclude}
$g\act\Minussp{k}\subseteq\Minussp{j}$,
\item\label{halfspacePastintersect}
for all $\nu\in\Cs{j,\st}$ with $\gamma_\nu(-\infty)\in g\act J_k$ there exists 
$t\leq0$ such that $\gamma_\nu^{\prime}(t)\in g\act\Cs{k}$,
\item\label{halfspacePaststrict}
if $g\act J_k\varsubsetneq J_j$, then $g\act\Minussp{k}\varsubsetneq\Minussp{j}$ and, 
in~\eqref{halfspacePastintersect}, $t<0$.
\end{enumerate}
\end{lemma}

For each~$j\in A$ and $\nu\in\Cs{j}$ we set 
\begin{align}
\retime{\BrU}(\nu) & \coloneqq\min\defset{t>0}{\gamma_{\nu}^{\prime}
(t)\in\Gamma\act\BrU}
\intertext{and}
\pretime{\BrU}(\nu) & \coloneqq\max\defset{t<0}{\gamma_\nu^{\prime}
(t)\in\Gamma\act\BrU}
\end{align}
\index[symbols]{tb@$\retime{\BrU}(\nu)$}%
\index[symbols]{tbb@$\pretime{\BrU}(\nu)$}%
if the respective elements exist. In this case, we call~$\retime{\BrU}(\nu)$ 
the \emph{next intersection time of~$\nu$ in~$\UTB\H$ with respect to~$\BrU$},
\index[defs]{next intersection time!with respect to}%
\index[defs]{intersection time!next, with respect to}%
and $\pretime{\BrU}(\nu)$ the \emph{previous intersection time of~$\nu$ 
in~$\UTB\H$ with respect to~$\BrU$}.
\index[defs]{previous intersection time!with respect to}%
\index[defs]{intersection time!previous, with respect to}%

With these preparations we can now establish the existence of next and previous 
intersection times. That also allows us to present a characterization of the 
transition sets, which implies their uniqueness.

\begin{prop}\label{PROP:CofSoBst}
Let $j\in A$.

\begin{enumerate}[label=$\mathrm{(\roman*)}$, ref=$\mathrm{\roman*}$]
\item\label{CofSoB:dontseeVan}
We have
\begin{align*}
\Iset{j} &=\defset{\gamma_\nu(+\infty)}{\nu\in\Cs{j,\st}}
\intertext{and}
\Jset{j} & =\defset{\gamma_\nu(-\infty)}{\nu\in\Cs{j,\st}}\,.
\end{align*}

\item\label{CofSoB:firstreturn}
For each $\nu\in\Cs{j,\st}$, the next intersection time~$\retime{\BrU}(\nu)$, as 
well as the previous intersection time~$\pretime{\BrU}(\nu)$, exists.

\item\label{CofSoB:representsets} For each~$k\in A$ we have
\begin{align*}
\Trans{}{j}{k}&=\defset{g\in\Gamma}{\exists\,\nu\in\Cs{j,\st}\colon\gamma_\nu^{
\prime}(\retime{\BrU}(\nu))\in g\act\Cs{k}}
\intertext{and}
\Past{}{k}{j}&=\defset{g\in\Gamma}{\exists\,\nu\in\Cs{j,\st}\colon 
\gamma_\nu^{\prime}(\pretime{\BrU}(\nu))\in g\act\Cs{k}}\,.
\end{align*}
\end{enumerate}
\end{prop}

\begin{proof}
For the proof of~\eqref{CofSoB:dontseeVan}, we recall from Remark~\ref{REM:inst} 
that $\Jset{j}\not=\varnothing$ and fix any  $y_0\in\Jset{j}$. For each 
$x\in\Iset{j}$, the combination of~\eqref{BP:allvectors} and 
Remark~\ref{REM:inst} implies the existence of~$\nu\in\Cs{j,\st}$ such that 
\[
 (x,y_0) = (\gamma_\nu(+\infty), \gamma_\nu(-\infty))\,.
\]
Thus, 
\[
 x\in \defset{ \gamma_\nu(+\infty) }{ \nu \in \Cs{j,\st} }
\]
and hence 
\[
 \Iset{j} \subseteq \defset{ \gamma_\nu(+\infty) }{ \nu \in \Cs{j,\st} }\,.
\]
Conversely, since $\defset{ \gamma_\nu(+\infty) }{ \nu\in\Cs{j} }\subseteq I_j$ 
by~\eqref{BP:pointintohalfspaces}, 
\[
 \defset{\gamma_\nu(+\infty) }{ \nu\in\Cs{j,\st} } \subseteq I_j \cap 
\wh{\R}_{\st} = \Iset{j}\,.
\]
This shows the first part of~\eqref{CofSoB:dontseeVan}. The second part follows 
analogously. 

For the proof of~\eqref{CofSoB:firstreturn} we fix $\nu\in\Cs{j,\st}$. Then 
$\gamma_\nu(+\infty)\in\Iset{j}$ by~\eqref{CofSoB:dontseeVan}, 
and~(\ref{BP:intervaldecomp}\ref{BP:intervaldecompGdecomp}) shows the existence 
of unique elements~$k\in A$ and $g\in\Trans{}{j}{k}$ such that 
$\gamma_\nu(+\infty)\in g\act I_{k}$. 
From Lemma~\ref{LEM:halfspaceincludeGV}\eqref{halfspaceincludeGV+} we 
obtain $g\act\Plussp{k}\varsubsetneq\Plussp{j}$ and hence $g\act I_k\varsubsetneq 
I_j$. Therefore we find $t>0$ such that $\gamma_\nu^{\prime}(t)\in 
g\act\Cs{k}$, as proven in Lemma~\ref{LEM:halfspaceFuture}. Thus, 
\begin{equation}\label{eq:intersectset}
 t\in \defset{ s>0 }{ \gamma_\nu^\prime(s) \in \Gamma\act\BrU }\,.
\end{equation}
In order to show that the minimum of this set exists and is assumed by~$t$, we 
set $z\coloneqq\gamma_\nu(0)$ and $w\coloneqq\gamma_\nu(t)$ and observe that
\[
 (z,w)_{\H}\cap\Gamma\act\BrU=\varnothing
\]
by using (\ref{BP:intervaldecomp}\ref{BP:intervaldecompGgeod}) and the fact that
$g\in\Trans{}{j}{k}$. Thus, there is no ``earlier'' intersection between 
$\gamma_\nu$ and $\Gamma\act\BrU$, and hence $\retime{\BrU}(\nu)$ exists and 
equals $t$.  The existence of $\pretime{\BrU}(\nu)$ follows 
analogously by taking advantage of Proposition~\ref{PROP:BPintervaldecompV},
Lemma~\ref{LEM:halfspaceincludeGV}\eqref{halfspaceincludeGV-} and 
Lemma~\ref{LEM:halfspacePast}.

In order to establish~\eqref{CofSoB:representsets} we fix $k\in A$ and set
\[
\mathrm{G}_{j,k}\coloneqq\defset{g\in\Gamma}{\exists\,
\nu\in\Cs{j,\st}\colon\gamma_{\nu}^{\prime}(\retime{\BrU}(\nu))\in 
g\act\Cs{k}}\,.
\]
We first aim at showing that $\mathrm{G}_{j,k} = \Trans{}{j}{k}$. Let 
$g\in\mathrm{G}_{j,k}$ and $\nu\in\Cs{j,\st}$ such that 
$\gamma_{\nu}^{\prime}(\retime{\BrU}(\nu))\in g\act\Cs{k}$. As in the proof 
of~\eqref{CofSoB:firstreturn} we obtain the existence and uniqueness of~$\ell\in 
A$ and~$h\in\Trans{}{j}{\ell}$ such that 
$\gamma_\nu^\prime(\retime{\BrU}(\nu))\in h\act\Cs{\ell}$. Thus, $g\act 
\Cs{k}\cap h\act\Cs{\ell}\not=\varnothing$, which yields $g=h$ and $k=\ell$ 
by~\eqref{BP:disjointunion}. In turn, $g\in \Trans{}{j}{k}$ and hence
\[
 \mathrm{G}_{j,k} \subseteq \Trans{}{j}{k}\,.
\]
For the converse inclusion relation, we pick $g\in\Trans{}{j}{k}$. 
From~\eqref{BP:intervaldecomp}, in combination with~\eqref{BP:disjointunion}, we 
get 
\[
 g\act \Iset{k} \varsubsetneq \Iset{j}\,.
\]
We pick any $(x,y)\in g\act \Iset{k}\times  \Jset{j}$. Then $(x,y)\in 
\Iset{j}\times\Jset{j}$, and hence there exists, by~\eqref{BP:allvectors}, a unique vector~$\nu\in\Cs{j,\st}$ such that 
\[
 (\gamma_\nu(+\infty), \gamma_\nu(-\infty)) = (x,y)\,.
\]
By Lemma~\ref{LEM:halfspaceFuture}, there exists~$t>0$ such that 
$\gamma_\nu^\prime(t)\in g\act\Cs{k}$. As in the proof 
of~\eqref{CofSoB:firstreturn} we obtain $t=\retime{\BrU}(\nu)$. Therefore $g\in 
\mathrm{G}_{j,k}$ and hence 
\[
 \Trans{}{j}{k}\subseteq \mathrm{G}_{j,k}\,.
\]
This proves the first part of~\eqref{CofSoB:representsets}. For the second part 
we observe that 
\begin{eqnarray*}
\exists\,\nu\in\Cs{j}\colon\gamma_\nu^{\prime}(\pretime{\BrU}(\nu))\in 
g\act\Cs{k}&\Longleftrightarrow&\exists\,\eta\in\Cs{k}\colon\gamma_\eta^{\prime}
(\retime{\BrU}(\eta))\in g^{-1}\act\Cs{j}\,.
\end{eqnarray*}
The equality $\Past{}{k}{j}=\Trans{}{k}{j}^{-1}$ now completes the proof. 
\end{proof}

\begin{remark}
Let $j,k\in A$. We note that for the characterization of the transition 
sets~$\Trans{}{j}{k}$ and~$\Past{}{k}{j}$ in 
Proposition~\ref{PROP:CofSoBst}\eqref{CofSoB:representsets} we used the strong 
branch~$\Cs{j,\st}$ instead of the orginal branch~$\Cs{j}$. This is necessary 
due to the possibility that branches are \emph{not full}, i.e., there might 
exist $k\in A$ and a geodesic~$\gamma$ on~$\H$ with $\gamma(+\infty)\in I_k$ and 
$\gamma(-\infty)\in J_k$ that does not intersect~$\Cs{k}$. In other words, the 
geodesic~$\gamma$ is passing through the geodesic 
segment~$\overline{\base{\Cs{k}}}$ from the half-space~$\Minussp{k}$ 
into~$\Plussp{k}$ and hence has the potential to intersect~$\Cs{k}$, but the 
necessary vector at the intersection point with~$\overline{\base{\Cs{k}}}$ is 
not contained in~$\Cs{k}$.  If we now have a vector 
~$\nu\in\Cs{j}\setminus\Cs{j,\st}$ such that the next intersection 
time~$\retime{\BrU}(\nu)$ exists but the next intersection between $\gamma_\nu$ 
and $\Gamma\act\overline{\base{\BrU}}$ is at an earlier time (due to a 
``missing'' vector in~$\BrU$ as above), then 
\[
 \gamma_\nu'(\retime{\BrU}(\nu)) \in h\act\Cs{k}
\]
for some $k\in A$ and $h\in\Gamma$ with $h$ typically not in~$\Trans{}{j}{k}$.
However, if all branches are full, then
\begin{align*}
\Trans{}{j}{k}&=\defset{g\in\Gamma}{\exists\,\nu\in\Cs{j}\colon\text{$\retime{
\BrU}(\nu)$ exists and $\gamma_\nu^{\prime}(\retime{\BrU}(\nu))\in 
g\act\Cs{k}$}}
\intertext{and}
\Past{}{k}{j}&=\defset{g\in\Gamma}{\exists\,\nu\in\Cs{j}\colon 
\text{$\pretime{\BrU}(\nu)$ exists and 
$\gamma_\nu^{\prime}(\pretime{\BrU}(\nu))\in g\act\Cs{k}$}}\,.
\end{align*}
\end{remark}

The next observation follows immediately from Proposition~\ref{PROP:CofSoBst}. 

\begin{cor}\label{COR:uniquereturn}
Let $j,k\in A$, $\nu\in\Cs{j,\st}$ and $g\in\Gamma$. Then 
\begin{enumerate}[label=$\mathrm{(\roman*)}$, ref=$\mathrm{\roman*}$]
\item\label{uniquereturn+}
$\gamma_\nu^{\prime}(\retime{\BrU}(\nu))\in g\act\Cs{k}$ if and only if 
$g\in\Trans{}{j}{k}$ and $\gamma_\nu(+\infty)\in g\act I_k$,

\item\label{uniquereturn-}
$\gamma_\nu^{\prime}(\pretime{\BrU}(\nu))\in g\act\Cs{k}$ if and only if 
$g\in\Past{}{k}{j}$ and $\gamma_\nu(-\infty)\in g\act J_k$.
\end{enumerate}
\end{cor}

We now associate to each element~$\nu\in\BrU_{\st}$ the three sequences defined 
in what follows. The combination of Propositions~\ref{PROP:CofSoBnonst}, 
\ref{PROP:CofSoBst}, Corollary~\ref{COR:uniquereturn} and Remark~\ref{REM:inst} 
shows their existence, well-definedness and the claimed properties.

We define the \emph{sequence~$(\ittime{\BrU,n}(\nu))_{n\in\Z}$ of iterated 
intersection times of~$\nu$ with respect to~$\BrU$} by 
\index[defs]{iterated intersection times}%
\index[defs]{intersection time!iterated}%
\index[defs]{iterated sequence!times}%
\index[symbols]{tc@$\ittime{\BrU,n}(\nu)$}%
\begin{align}
\label{EQDEF:ittime0}
\ittime{\BrU,0}(\nu) &\coloneqq 0
\intertext{and}
\label{EQDEF:ittimen}
\ittime{\BrU,n}(\nu) &\coloneqq
\begin{cases}
\min\defset{t>\ittime{\BrU,n-1}(\nu)}{\gamma_\nu^{\prime}(t)\in\Gamma\act\BrU}
&\text{for $n\geq 1$}\,,
\\[1ex]
\max\defset{t<\ittime{\BrU,n+1}(\nu)}{\gamma_\nu^{\prime}
(t)\in\Gamma\act\BrU}&\text{for $n\leq-1$}\,.
\end{cases}
\end{align}
This sequence is strictly increasing. For each $n\in\Z$ we have 
\[
\sgn(\ittime{\BrU,n}(\nu))=\sgn(n)
\]
and
\begin{equation}\label{eq:iternexttime}
\ittime{\BrU,n}(\nu)=\retime{\BrU}\bigl(\gamma_\nu^{\prime}(\ittime{\BrU,n-1}
(\nu))\bigr)=\pretime{\BrU}\bigl(\gamma_\nu^{\prime}(\ittime{\BrU,n+1}
(\nu))\bigr)\,.
\end{equation}
Given the sequence~$(\ittime{\BrU,n}(\nu))_{n\in\Z}$, for each $n\in\Z$ the 
branch translate $g\act\Cs{k}$ containing the vector 
$\gamma_\nu^{\prime}(\ittime{\BrU,n}(\nu))$ is uniquely determined.  This allows 
us to define the \emph{sequence~$(\itindex{\BrU,n}(\nu))_{n\in\Z}$ of iterated 
intersection branches of~$\nu$ with respect to~$\BrU$}
\index[defs]{iterated intersection branches}%
\index[defs]{branch!iterated intersection}%
\index[defs]{iterated sequence!branches}%
\index[symbols]{k@$\itindex{\BrU,n}(\nu)$}%
as the sequence in~$A$ 
given by
\begin{align}\label{EQDEF:itindex}
\itindex{\BrU,n}(\nu)=k \qquad \logeq  \qquad \exists\,
g\in\Gamma\colon\gamma_\nu^{\prime}(\ittime{\BrU,n}(\nu))\in g\act\Cs{k}
\end{align}
for all~$n\in\Z$. The \emph{sequence~$(\ittrans{\BrU,n}(\nu))_{n\in\Z}$ of 
iterated intersection transformations of~$\nu$ with respect to~$\BrU$}
\index[defs]{iterated intersection transformations}%
\index[defs]{iterated sequence!transformations}%
\index[symbols]{g@$\ittrans{\BrU,n}(\nu)$}%
is a 
sequence in~$\Gamma$ that is given by 
\begin{align}\label{EQDEF:ittrans1}
\ittrans{\BrU,0}(\nu) & \coloneqq\id
\end{align}
and for $n\in\Z$, $n\not=0$, by 
\begin{align}\label{EQDEF:ittrans2}
& \ittrans{\BrU,n}(\nu)  = g  
\\
& \qquad \logeq  \quad 
\begin{cases}
\gamma_ 
\nu^{\prime}(\ittime{\BrU,n}(\nu))\in\ittrans{\BrU,1}(\nu)\cdots\ittrans{\BrU,
n-1}(\nu)g\act\Cs{\itindex{\BrU,n}(\nu)}&\text{for $n\geq1$}\,,
\\[2mm]
\gamma_\nu^{\prime}(\ittime{\BrU,n}(\nu))\in\ittrans{\BrU,-1}(\nu)\cdots\ittrans
{\BrU,
n+1}(\nu)g\act\Cs{\itindex{\BrU,n}(\nu)}&\text{for $n\leq-1$}\,.
\end{cases}
\nonumber
\end{align}
For each $n\in\N$ we then have
\[
\ittrans{\BrU,n}(\nu)\in\Trans{}{\itindex{\BrU,n-1}(\nu)}{\itindex{\BrU,n}(\nu)}
\]
and for each $n\in -\N$ we have
\[
\ittrans{\BrU,n}(\nu)^{-1}\in\Trans{}{\itindex{\BrU,n}(\nu)}{\itindex{\BrU,n+1}
(\nu)}\,.
\]
We call the ordered set 
\begin{equation}\label{eq:itseq}
[(\ittime{\BrU,n}(\nu))_n,(\itindex{\BrU,n}(\nu))_n,(\ittrans{\BrU,n}(\nu))_n]
\end{equation}
the \emph{system of iterated sequences of~$\nu$ with respect to~$\BrU$}. 
\index[defs]{system of iterated sequences}%
\index[defs]{iterated sequence!system}%

\begin{lemma}\label{LEM:accpointsofit}
For each $\nu\in\BrU_{\st}$ we have
\[
\lim_{n\to\pm\infty}\ittime{\BrU,n}(\nu)=\pm\infty\,.
\]
\end{lemma}

\begin{proof}
We establish the claimed properties via a proof by contradiction. Let 
$\nu\in\BrU_{\st}$ and consider the system 
\[
 [(\ittime{\BrU,n}(\nu))_n,(\itindex{\BrU,n}(\nu))_n,(\ittrans{\BrU,n}(\nu))_n]
\]
of iterated sequences of~$\nu$ from~\eqref{eq:itseq}. For $n\in\N$ set 
\[
 t_n \coloneqq \ittime{\BrU,n}(\nu)
\]
and recall that $(t_n)_{n\in\N}$ is strictly increasing. We assume that 
$(t_n)_n$ converges in~$\R$, say
\[
 \lim_{n\to+\infty} t_n = \tau \quad\in\R\,.
\]
Since 
\[
 \R\to\H\,,\quad t\mapsto \gamma_\nu(t)\,,
\]
is an isometric embedding, it follows that the sequence~$(\gamma_\nu(t_n))_n$ 
converges in~$\H$, namely
\[
 \lim_{n\to+\infty}\gamma_\nu(t_n) = \gamma_\nu(\tau)\,,
\]
and the elements of the sequence $(\gamma_\nu(t_n))_n$ are pairwise distinct. 
For $n\in\N$ let 
\[
 k_n\coloneqq \itindex{\BrU,n}(\nu)
\]
and 
\[
 h_n \coloneqq \ittrans{\BrU,1}(\nu)\cdots \ittrans{\BrU,n}(\nu)\,.
\]
Then 
\[
 \gamma_\nu(t_n) \in h_n\act \overline{\base{\Cs{k_n}}}
\]
for each~$n\in\N$ (see~\eqref{EQDEF:ittrans1}--\eqref{EQDEF:ittrans2}). Further, 
the shape of geodesics in~$\H$ implies that the tuples $(k_n,h_n)$, $n\in\N$, 
are pairwise distinct. Hence, each neighborhood of~$\gamma_\nu(\tau)$ in~$\H$ 
intersects infinitely many members of the family 
\[
 \defset{h_n\act\overline{\base{\Cs{k_n}}}}{n\in\N}\,.
\]
This contradicts~Proposition~\ref{PROP:branches_locfinite}. In turn, 
\[
 \lim_{n\to+\infty} t_n = +\infty\,.
\]
The statement for $n\to-\infty$ can be shown analogously.
\end{proof}

The following proposition shows that each intersection between 
$\Gamma\act\BrU_{\st}$ and a geodesic determined by an element of~$\BrU_{\st}$ 
is indeed detected by the iterated sequences. This observation will be crucial 
for establishing that $\CrSc = \pi(\BrU)$ is a cross section.

\begin{prop}\label{PROP:allinit}
Let $\nu\in\BrU_{\st}$, $k\in A$, $t\in\R$ and $g\in\Gamma$ be such that
\[
\gamma_\nu^{\prime}(t)\in g\act\Cs{k}\,.
\]
Then there exists a unique element~$n\in\Z$ such that~$\sgn(n)=\sgn(t)$ and
\begin{eqnarray*}
k=\itindex{\BrU,n}(\nu),&t=\ittime{\BrU,n}(\nu),\text{ 
and}&g=\ittrans{\BrU,\sgn(t)}(\nu)\ittrans{\BrU,2\sgn(t)}(\nu)\cdots\ittrans{
\BrU,n}(\nu)\,.
\end{eqnarray*}
\end{prop}

\begin{proof}
It suffices to show the uniqueness of~$n\in\Z$ with~$t=\ittime{\BrU,n}(\nu)$. The remaining statements are then immediate from the definitions.
By the strict monotony of the sequence $(\ittime{\BrU,n}(\nu))_{n\in\Z}$ and 
Lemma~\ref{LEM:accpointsofit} we find exactly one $n\in\Z$ such that
\[
\ittime{\BrU,n-1}(\nu)< t\leq\ittime{\BrU,n}(\nu)\,.
\]
If $t<\ittime{\BrU,n}(\nu)$, then the hypothesis
$\gamma_\nu^{\prime}(t)\in\Gamma\act\BrU$ implies 
$\ittime{\BrU,n}(\nu)\ne\retime{\BrU}(\gamma_\nu^{\prime}(\ittime{\BrU,n-1}
(\nu)))$. This contradicts~\eqref{eq:iternexttime}.  Hence, 
$t=\ittime{\BrU,n}(\nu)$.
\end{proof}

The remainder of this section is devoted to the proof that every set of branches 
can be rearranged into an admissible one (see~\eqref{BP:leavespaceforflip}).
We further introduce an additional property of sets of branches, which will be 
needed later on (see Proposition~\ref{PROP:stepreduxworx} below) as a 
prerequisite in order to assure a non-collapsing behavior in the sense 
of~\eqref{BP:noidentity}.
Again, every set of branches can be rearranged to one with that property, and so 
that, simultaneously, admissibility is assured.

\begin{defi}\label{DEF:noncollaps}
A set of branches~$\BrS=\defset{\Cs{j}}{j\in A}$ is called \emph{weakly 
non-collaps\-ing}
\index[defs]{weakly non-collapsing}%
\index[defs]{non-collapsing!weakly}%
\index[defs]{set of branches!weakly non-collapsing}%
\index[defs]{branch!set of weakly non-collapsing}%
if
\begin{enumerate}[label=$\mathrm{(B_{col})}$, ref=$\mathrm{B_{col}}$]
\item\label{BP:noncollaps}
For every pair~$(j,k)\in A\times A$, for every~$\nu\in\Cs{j}$ such 
that~$\gamma_\nu$ intersects~$\Cs{k}$ at some time~$t^*>0$, the geodesic 
segment~$\gamma_\nu((0,t^*))$ does not intersect~$g\act\BrU$ for 
any~$g\in\Gamma^*$.
\end{enumerate}
\index[defs]{0B12@(B$_{\mathrm{col}}$)}%
\index[symbols]{B12@(B$_{\mathrm{col}}$)}%
\end{defi}

\begin{remark}\label{B9impliesBcol}
Via contraposition it is easy to see that~\eqref{BP:noidentity} 
implies~\eqref{BP:noncollaps}:
Assume that a given set of branches~$\BrS=\defset{\Cs{j}}{j\in A}$ is not weakly 
non-collapsing. Then we find~$j,k,\ell\in A$,~$g\in\Gamma^*$,~$0<t_1<t_2$, 
and~$\nu\in\UTB\H$ such that
\[
\gamma_\nu'(0)\in\Cs{j}\,,\quad\gamma_\nu'(t_1)\in 
g\act\Cs{\ell}\,,\quad\text{and}\quad\gamma_\nu'(t_2)\in\Cs{k}\,.
\]
Consider the system of iterated sequences 
$[(\ittime{\BrU,n}(\nu))_n,(\itindex{\BrU,n}(\nu))_n,(\ittrans{\BrU,n}(\nu))_n]$ 
associated to~$\nu$ by~\eqref{eq:itseq}.
By Proposition~\ref{PROP:allinit} there exist~$n_1,n_2\in\N$ such that
\[
\ittime{\BrU,n_1}(\nu)=t_1\qquad\text{and}\qquad\ittime{\BrU,n_2}(\nu)=t_2\,.
\]
The above then implies
\[
\ittrans{\BrU,1}(\nu)\ittrans{\BrU,2}(\nu)\cdots\ittrans{\BrU,n_2}(\nu)=\id\,.
\]
Hence,~$\BrS$ does not fulfill~\eqref{BP:noidentity}.
However,~\eqref{BP:noidentity} and~\eqref{BP:noncollaps} are not equivalent, 
since~\eqref{BP:noncollaps} allows for elements~$\nu\in\BrU$ whose induced 
geodesics have future intersections with~$\BrU$ that are not immediate, while~\eqref{BP:noidentity} does not.
\end{remark}

Property~\eqref{BP:noncollaps} demands that the set of branches is structured in a specific way, relative to the non-trivial $\Gamma$-translates of itself.
Lemma~\ref{LEM:nouniqueSoB} allows us to exchange branches with $\Gamma$-translates.
In what follows, we describe a sorting algorithm that transforms every set of branches into a weakly non-collapsing one.

For reasons of effectivity, we introduce so-called \emph{branch trees}:
\index[defs]{branch tree}%
Let~$j\in A$.
The root node (level~$0$) of the \emph{branch tree relative to~$j$} is~$(j,\id)\in A\times\Gamma$.
The nodes at level~$1$ are all tuples of the form~$(\itindex{\BrU,1}(\nu),\ittrans{\BrU,1}(\nu))\in A\times\Gamma$, for~$\nu\in\Cs{j,\st}$.
By virtue of~\eqref{BP:intervaldecomp}, the finiteness of~$A$ and the discreteness of~$\Gamma$, there are at most countably many such tuples in level~$1$.
The nodes at level~$r$, for~$r\in\N$, are all tuples of the form~$(\itindex{\BrU,r}(\nu),\ittrans{\BrU,1}(\nu)\cdots\ittrans{\BrU,r}(\nu))\in A\times\Gamma$, for~$\nu\in\Cs{j,\st}$.
Two nodes are linked by an edge if and only if they are of the form
\[
(\itindex{\BrU,r}(\nu),\ittrans{\BrU,1}(\nu)\cdots\ittrans{\BrU,r}(\nu))\quad\text{and}\quad(\itindex{\BrU,r+1}(\nu),\ittrans{\BrU,1}(\nu)\cdots\ittrans{\BrU,r}(\nu)\ittrans{\BrU,r+1}(\nu))
\]
for \textit{the same}~$\nu\in\Cs{j,\st}$.
Hence, every path in the branch tree is of the form
\begin{equation}\label{EQ:pathform}
(j,\id)\edge{\phantom{-.}}(\itindex{\BrU,1}(\nu),\ittrans{\BrU,1}(\nu))\edge{\phantom{-.}}(\itindex{\BrU,2}(\nu),\ittrans{\BrU,1}(\nu)\ittrans{\BrU,2}(\nu))\edge{\phantom{-.}}\dots
\end{equation}
for \textit{the same}~$\nu\in\Cs{j,\st}$ in each tuple, where~$(k_1,h_1)\edge{\phantom{-.}}(k_2,h_2)$ denotes an edge in the tree, for~$k_1,k_2\in A$,~$h_1,h_2\in\Gamma$.
Hence, a path in a branch tree corresponds to the existence of a geodesic on~$\H$ intersecting~$h\act\Cs{k}$ for every node~$(k,h)$ on that path.
We denote the branch tree with the root~$(j,\id)$ by~$B_j$.
\index[symbols]{Bzj@$B_j$}%
The collection of one or more branch trees is called a \emph{branch forest}.
\index[defs]{branch forest}%
In what follows we adopt a straightforward terminology of sub- and super-trees as sub- and super-graphs of trees.
A sub- or super-tree is called \emph{complete} if it contains all child nodes of its root node.
The \emph{level} of a node in a tree is the number of edges in the path connecting the root node to it.
Furthermore, we consider trees as directed graphs, with direction towards increasing level.

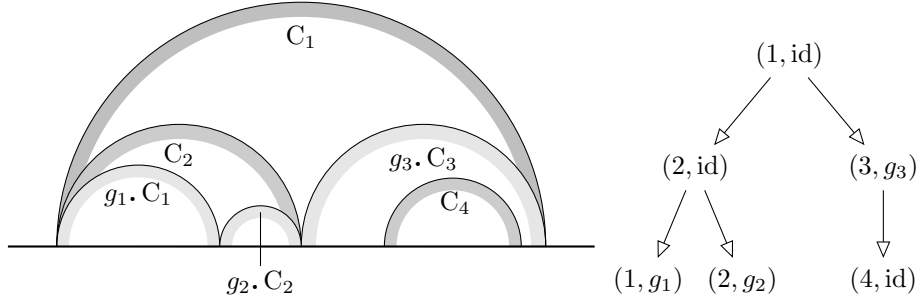
\begin{figure}
\begin{tikzpicture}[scale=6.5]
\fill[color=lightgray] (.5,0) -- +(0:.5) arc (0:180:.5);
\fill[color=white] (.5,0) -- +(0:.47) arc (0:180:.47);
\fill[color=lightgray!80] (.25,0) -- +(0:.25) arc (0:180:.25);
\fill[color=white] (.25,0) -- +(0:.22) arc (0:180:.22);
\draw (.25,0) -- +(0:.25) arc (0:180:.25);
\foreach \x in {.75}{
	\fill[color=lightgray!40] (\x,0) -- +(0:.25) arc (0:180:.25);
	\fill[color=white] (\x,0) -- +(0:.22) arc (0:180:.22);
	\draw (\x,0) -- +(0:.25) arc (0:180:.25);
	}
\foreach \x in {.1666666666, .4166666666}{
	\fill[color=lightgray!40] (\x,0) -- +(0:{abs(\x - 1/3)}) arc (0:180:{abs(\x - 1/3)});
	\fill[color=white] (\x,0) -- +(0:{abs(\x - 1/3)-.025}) arc (0:180:{abs(\x - 
1/3)-.025});
	\draw (\x,0) -- +(0:{abs(\x - 1/3)}) arc (0:180:{abs(\x - 1/3)});
}
\fill[color=lightgray!80] (.7,0) -- +(0:.25) arc (0:180:.14);
\fill[color=white] (.7,0) -- +(0:.225) arc (0:180:.115);
\draw (.7,0) -- +(0:.25) arc (0:180:.14);
\draw (.5,0) -- +(0:.5) arc (0:180:.5);
\draw (.416666666,-.04) -- (.416666666,.07);
\draw[style=thick] (-.1,0) -- (1.1,0);
\coordinate [label=below:$\Cs{1}$] (C1) at (.5,.47);
\coordinate [label=below:$\Cs{2}$] (C2) at (.25,.23);
\coordinate [label=below:$g_3\act\Cs{3}$] (ST1SC3) at (.75,.22);
\coordinate [label=below:$g_1\act\Cs{1}$] (STC1) at (.16666666,.15);
\coordinate [label=below:$\Cs{4}$] (C4) at (.815,.13);
\coordinate [label=below:$g_2\act\Cs{2}$] (STSC2) at (.416666666,-.03);
\end{tikzpicture}
\begin{tikzpicture}[level/.style={sibling distance=25mm/#1},edge from parent/.style={draw,-open triangle 45}]
\node (z){$(1,\id)$}
  child {node (a) {$(2,\id)$}
    child {node (b) {$(1,g_1)$}
    }
    child {node (c) {$(2,g_2)$}}
   }
  child {node (d) {$(3,g_3)$}
  child {node (e) {$(4,\id)$}}
  };
\end{tikzpicture}
\caption[branchtree1]{A schematic example of a local relationship in a set of branches and one of the branch trees emerging from it, up to level~$2$.}\label{FIG:branchtree1}
\end{figure}

We define a left multiplication of elements~$g\in\Gamma$ on the set~$\defset{B_j}{j\in A}$ by defining~$gB_j$ to be the tree that arises from~$B_j$ by exchanging every node~$(k,h)$ in~$B_j$ for~$(k,gh)$.
We emphasize that we construct further trees in this way.
Then the complete sub-tree with the root~$(\itindex{\BrU,r}(\nu),\ittrans{\BrU,1}(\nu)\cdots\ittrans{\BrU,r}(\nu)\bigr)$ is given by~$\ittrans{\BrU,1}(\nu)\cdots\ittrans{\BrU,r}(\nu)B_{\itindex{\BrU,r}(\nu)}$, at every level~$r\in\N$.
Conversely, every complete sub-tree of~$B_j$,~$j\in A$, is of the form~$gB_k$, for some~$k\in A$ and~$g\in\Gamma$.
Every node~$(k,h)$, with~$k\in A$ and~$h\in\Gamma$, is unique within a branch tree~$B_j$ (but not necessarily within a branch forest).
A complete sub-tree of some tree~$B_j$,~$j\in A$, is called~\emph{$\Gamma$-trivial}
\index[defs]{Gtriv@$\Gamma$-trivial}%
\index[defs]{trivial, $\Gamma$-}%
if it is of the form~$\id\!B_k$ for some~$k\in A$.
This is obviously the case if and only if~$\ittrans{\BrU,1}(\nu)\cdots\ittrans{\BrU,r}(\nu)=\id$, with~$r\in\N$ and~$\nu\in\Cs{j,\st}$ such that~$k=\itindex{\BrU,r}(\nu)$.
This fact implies the following characterization.

\begin{lemma}\label{LEM:wnctree}
The set of branches~$\BrS=\defset{\Cs{j}}{j\in A}$ is weakly non-collapsing if and only if, for every~$j\in A$, every complete super-tree of any~$\Gamma$-trivial sub-tree of~$B_j$ is itself~$\Gamma$-trivial.
\end{lemma}

Using this characterization we can implement an algorithm to rearrange a set of branches in order to obtain a weakly non-collapsing structure.
To that end, we first fix a convenient choice of root nodes given by initial branches:
A branch~$\Cs{j}$,~$j\in A$, is called~\emph{initial}
\index[defs]{initial branch}%
\index[defs]{branch!initial}%
if
\begin{equation}\label{EQN:initiality}
\forall\, k\in A\setminus\{j\}\colon \Minussp{k}\not\subseteq \Minussp{j}\,.
\end{equation}
In particular this implies that
\[
\ittrans{\BrU,-1}(\nu)^{-1}\cdots\ittrans{\BrU,-n}(\nu)^{-1}\ne\id
\]
for every~$\nu\in\Cs{j}$ and every~$n\in\N$.
Hence, a branch~$\Cs{j}$ is initial if and only if $B_j$ does not appear as a~$\Gamma$-trivial sub-tree in any branch tree other than~$B_j$.
We set
\begin{equation}\label{eq:D_initial}
D_\ini \coloneqq \defset{j\in A}{\text{$\Cs{j}$ is initial}}\,.
\end{equation}
\index[symbols]{Dini@$D_\ini$}%
Since~$A$ is finite, the set~$D_\ini$ is nonempty.
For every~$k\in A$ there exists~$j\in D_\ini$ such that either~$k=j$, or~$\Plussp{k}\subseteq\Plussp{j}$.
In the latter case there exists~$\nu\in\Cs{j}$ and~$n\in\N$ such that
\[
(\itindex{\BrU,n}(\nu),\ittrans{\BrU,1}(\nu)\cdots\ittrans{\BrU,n-1}(\nu))=(k,\id)\,.
\]
Hence, every branch tree~$B_k$,~$k\in A$, is contained as a~$\Gamma$-trivial sub-tree in some member of the branch forest~$F_\ini\coloneqq\defset{B_j}{j\in D_\ini}$.
\index[symbols]{Fini@$F_\ini$}%

The following algorithm defines transformations~$q_j\in\Gamma$ for every~$j\in A$.
In it we apply a notion of \emph{cutting off} nodes~$(k,h)$ from the branch forest~$F_\ini$.
By that we mean that, for the remainder of the algorithm, one is restricted from considering the complete sub-tree~$hB_k$ in \emph{any} branch tree~$B_j$,~$j\in D_\ini$.
The \emph{remaining} nodes are all nodes not contained in a cut off sub-tree.
For~$r\in\N$ and~$k\in A$ we define~$L_r(k)$ to be the set of all nodes of the form~$(k,h)$,~$h\in\Gamma$, at level~$r$, anywhere in the branch forest~$F_\ini$.
\index[symbols]{Lr@$L_r(k)$}%

\begin{algo}\label{ALGO:decollaps}
The index $r$ below starts at $1$.
\vspace*{-\baselineskip}
\vspace*{4pt}
\begin{enumerate}
\renewcommand{\labelenumi}{\emph{Step\,0.}}
\item
Set~$q_j\coloneqq\id$ for every~$j\in D_\ini$.\newline
Cut off all nodes from~$F_\ini$ for which the complete sub-tree with this node as root node does not contain a~$\Gamma$-trivial sub-tree.\newline
Carry out \emph{Step 1}.
\renewcommand{\labelenumi}{\emph{Step\,r.}}
\item
If no nodes remain at level~$r$, the algorithm terminates.\newline
Otherwise, cut off all nodes~$(j,g)$ from~$F_\ini$ in level~$r$ with $j\in A$ for which~$q_j$ has already been defined.\newline
For all~$k\in A\setminus D_\ini$ for which~$L_r(k)$ contains remaining nodes, choose such a node~$(k,h)$, set~$q_k\coloneqq h$, and cut off all remaining nodes in~$L_r(k)\setminus\{(k,h)\}$ from~$F_\ini$.\newline
Carry out \emph{Step r+1}.\algofin
\end{enumerate}
\end{algo}

\begin{lemma}\label{LEM:noncollaps}
Algorithm~\ref{ALGO:decollaps} defines for every~$j\in A$ a transformation~$q_j\in\Gamma$, before terminating after at most~$N+2$ steps.
(Recall that $\#A = N$.)
The arising set~$\BrS'\coloneqq\defset{q_j\act\Cs{j}}{j\in A}$ is a weakly non-collapsing set of branches for the geodesic flow on~$\Orbi$.
\end{lemma}

\begin{proof}
Since~$\#A=N<+\infty$, the level at which nodes of the form~$(k,\id)$,~$k\in A$, may appear is bounded from above.
Hence, from the cutting off of sub-trees that do not contain~$\Gamma$-trivial sub-trees in \emph{Step 0} onward, the (non-complete) sub-trees of remaining nodes are all finite.
Thus, Algorithm~\ref{ALGO:decollaps} will eventually fail to encounter remaining nodes and will thus terminate after finitely many steps.
Since in every step the algorithm either terminates or defines at least one transformation~$q_j$ which has not yet been defined, the number of steps is bounded by~$N+2$.

Let~$r\in\N$ and~$k\in A$, and assume that~$q_k$ has not yet been defined at the start of \emph{Step r}.
Since the branch tree~$B_k$ is contained in some member of~$F_\ini$, we find~$j\in A$ and $g\in\Gamma$ such that~$q_j$ has been defined at \emph{Step $r-$1} and~$gB_j$ contains the node~$(k,\id)$.
But then the sub-tree~$q_jB_j$ contains the node~$(k,q_jg^{-1})$, which is remaining at the start of \emph{Step r}.
Since this argument applies for every~$r$ for which Algorithm~\ref{ALGO:decollaps} does not terminate in \emph{Step r}, and since no sub-tree containing~$B_k$ is cut off in \emph{Step 0}, the algorithm must eventually encounter a node of the form~$(k,h)$ and thus define the transformation~$q_k$.

Finally, the set~$\BrS'$ is a set of branches for the geodesic flow on~$\Orbi$ by virtue of Lemma~\ref{LEM:nouniqueSoB}.
After termination of Algorithm~\ref{ALGO:decollaps}, every path in the sub-tree of remaining nodes in any member of~$F_\ini$ is of the form
\[
(k_1,q_{k_1})\edge{\phantom{-.}}(k_2,q_{k_2})\edge{\phantom{-.}}\dots\edge{\phantom{-.}}(k_n,q_{k_n})\,,
\]
for some~$n\in\N$ and~$k_1,\dots,k_n\in A$.
Since the branch forest of all branch trees with respect to the set of branches~$\BrS'$ is given by~$\defset{q_j^{-1}B_j}{j\in A}$, this path then reads as
\[
(k_1,\id)\edge{\phantom{-.}}(k_2,\id)\edge{\phantom{-.}}\dots\edge{\phantom{-.}}(k_n,\id)\,.
\]
By Lemma~\ref{LEM:wnctree}, this implies that~$\BrS'$ is weakly non-collapsing.
\end{proof}

\begin{example}\label{EX:G3wnc}
Recall the group~$\Gamma_\lambda$ from Example~\ref{EX:G3Def} and its set of branches~$\BrS=\{\Cs{1},\dots,\Cs{8}\}$ from Figure~\ref{FIG:G3SoB}.
The set of transformations defined by Algorithm~\ref{ALGO:decollaps} for~$\BrS$ is given by
\[
q_1=q_2=\ldots=q_7=\id\qquad\text{and}\qquad q_8=h\,,
\]
meaning the only branch that gets swapped with one of its translates in order to obtain a weakly non-collapsing set of branches~$\BrS'$ is~$\Cs{8}$.
\end{example}

We now tend to show that each set of branches can be turned into one that is simultaneously weakly non-collapsing and admissible.
To that end a sorting of branches more restricted than what is provided by the branch trees is needed.
To be more precise, we suppose that~$\BrS=\defset{\Cs{j}}{j\in A}$ is weakly non-collapsing and let~$k\in A\setminus D_\ini$.
Then there might be more than one~$i\in D_\ini$ for which~$(k,\id)$ is a node in~$B_i$.
To overcome this issue, we define inductively for every node an associated initial node.
For~$i\in D_\ini$ we set~$\mathscr{j}(i)\coloneqq i$.
For~$(k,\id)$ a node at level~$1$ we pick one~$i\in D_\ini$ with~$(k,\id)\in B_i$ and set~$\mathscr{j}(k)\coloneqq i$.
Now let~$r\in\N$ and assume that for all nodes~$(k,\id)$ up to level~$r$ the index~$\mathscr{j}(k)$ has already been defined.
Then, for~$(k,\id)$ a node at level~$r+1$, we pick one~$i\in A$ with~$(i,\id)$ a node at level~$r$ and~$(k,\id)$ a node in~$B_i$, and set~$\mathscr{j}(k)\coloneqq\mathscr{j}(i)$.
\index[symbols]{jk@$\mathscr{j}(k)$}%
That way we obtain a map~$\mathscr{j}\colon A\to D_\ini$.
For~$i\in D_\ini$ we define
\begin{equation}\label{EQNDEF:Bj}
D_i\coloneqq\defset{k\in A}{\mathscr{j}(k)=i}\,.
\end{equation}
\index[symbols]{Dinj@$D_j$}%
Then
\begin{equation}\label{EQ:Dinidecomp}
A=\bigcup_{i\in D_\ini}D_i
\end{equation}
and the union on the right hand side is disjoint.

\begin{prop}\label{PROP:noncollaps}
For every set of branches~$\BrS=\{\Cs{1},\dots,\Cs{N}\}$ there exist 
transformations~$g_1,\dots,g_N\in\Gamma$ such 
that~$\wt{\BrS}\coloneqq\{g_1\act\Cs{1},\dots,g_N\act\Cs{N}\}$ is admissible and 
weakly non-collapsing.
\end{prop}

\begin{proof}
Because of Lemma~\ref{LEM:noncollaps} we may assume that~$\BrS$ is weakly non-collapsing.
Let~$i\in D_\ini$ and let~$k\in D_i$ be of maximal level in~$B_i$, that is,~$\id\notin\bigcup_{\ell\in A}\Trans{}{k}{\ell}$.
Then, whenever~$k\ne i$,
\[
I_k\varsubsetneq I_i\qquad\text{and}\qquad J_i\varsubsetneq J_k\,.
\]
Since~$I_k$ and~$J_i$ are both open and contain elements of~$\Lambda(\Gamma)$ by virtue of~\eqref{BP:closedgeodesicsHtoX}, Proposition~\ref{PROP:EXliesdense} yields a hyperbolic transformation~$h_i\in\Gamma$ such that
\[
(\fixp{+}{h_i},\fixp{-}{h_i})\in I_k\times J_i\,.
\]
Then, because of~\eqref{BP:allvectors}, the axis of~$h_i$ intersects each branch~$\Cs{\ell}$ with~$\ell\in D_i$.
Thus, for every~$\ell\in D_i$ we have
\[
(\fixp{+}{h_{\mathscr{j}(\ell)}},\fixp{-}{h_{\mathscr{j}(\ell)}})\in I_\ell\times J_\ell\subseteq \wh\R\setminus\{\eX_\ell,\eY_\ell\}\times\wh\R\setminus\{\eX_\ell,\eY_\ell\}\,.
\]
By Lemma~\ref{LEM:hypfixedconv} we therefore find
\[
\lim_{n\to+\infty}h_{\mathscr{j}(\ell)}^n\act \eX_\ell=\lim_{n\to+\infty}h_{\mathscr{j}(\ell)}^n\act \eY_\ell=\fixp{+}{h_{\mathscr{j}(\ell)}}\,,
\]
meaning~$h_{\mathscr{j}(\ell)}$ contracts the interval~$I_\ell$ towards~$\fixp{+}{h_{\mathscr{j}(\ell)}}$.
Hence, for~$\wt n\in\N$ sufficiently large,
\[
\wh\R\setminus\bigcup_{\ell\in A}h_{\mathscr{j}(\ell)}^{\wt n}\act I_\ell
\]
contains an open interval.
In turn,~$\wt\BrS\coloneqq\defset{h_{\mathscr{j}(\ell)}^{\wt n}\act\Cs{\ell}}{\ell\in A}$ is an admissible set of branches (see also Lemma~\ref{LEM:nouniqueSoB}).

Now let~$k\in A$ and denote the branch tree of~$k$ w.r.t.~$\wt\BrS$ by~$\wt B_k$.
Let~$\mathscr{j}(k)=i$.
Then~$\wt B_i$ contains the path
\begin{equation}\label{EQ:pathintree}
(i,\id)\edge{\phantom{-.}}(k_1,h_1)\edge{\phantom{-.}}\dots\edge{\phantom{-.}}(k_n,h_n)\edge{\phantom{-.}}(k,\id)\,,
\end{equation}
with~$k_1,\dots,k_n\in A$ and~$h_1,\dots,h_n\in\Gamma$,~$n\in\N$.
Since~$\mathscr{j}(k)=i$, by construction we have~$\mathscr{j}(k_\iota)=i$ for all~$\iota=1,\dots,n$. Therefore, the branch tree~$B_i$ in the branch forest of~$\BrS$ contains the path~\eqref{EQ:pathintree} as well.
Since~$\BrS$ is weakly non-collapsing, we have~$h_1=\dots=h_n=\id$.
This shows that~$\wt\BrS$ is also weakly non-collapsing. This finishes the proof.
\end{proof}

\subsection{Cross sections induced by sets of 
branches}\label{SUBSEC:crossbranches}

Throughout this section, let $\BrS = \{\Cs{1},\ldots,\Cs{N}\}$ be a set of 
branches for the geodesic flow~$\GeoFlow$ on~$\Orbi$, let $\BrU \coloneqq 
\bigcup_{j=1}^N \Cs{j}$ denote the branch union and set 
\[
 \CrSc \coloneqq \pi\left( \BrU \right)\,. 
\]
In what follows we show that $\CrSc$ is indeed a cross section for~$\GeoFlow$ 
with respect to certain measures and that the strong branch union~$\BrU_{\st}$ 
induces a strong cross section.  
To that end let $\Vanish(\Orbi)$
\index[symbols]{van@$\Vanish(\Orbi)$}%
denote the subset of~$\Geo(\Orbi)$ of all 
geodesics for which there exists a lift on~$\H$ having at least one endpoint 
in~$\widehat{\R}\setminus\widehat{\R}_{\st}$ (and thus all of its lifts on~$\H$ 
have this property). We note that 
$\Geo_{\Per}(\Orbi)\subseteq\Geo(\Orbi)\setminus\Vanish(\Orbi)$. We denote 
by~$\mathcal{M}_{\Vanish(\Orbi)}$
\index[symbols]{Mvan@$\mathcal{M}_{\Vanish(\Orbi)}$}%
the set of measures~$\mu$ on (a $\sigma$-algebra on)~$\Geo(\Orbi)$ 
with the property that $\mu(\Vanish(\Orbi))=0$. In particular, the counting measure of periodic geodesics belongs to~$\mathcal{M}_{\Vanish(\Orbi)}$. Throughout we use the standard notation from the previous sections.
In particular, we set $A = \{1,\ldots, N\}$ and define $I_j$, $J_j$, etc. as 
in~\eqref{BP:pointintohalfspaces}.

\begin{prop}\label{CofSoB:crosssection}
For each~$\mu\in\mathcal{M}_{\Vanish(\Orbi)}$, the set~$\CrSc$ is a cross 
section for the geodesic flow~$\GeoFlow$ on~$\Orbi$ with respect to~$\mu$. In 
particular, each geodesic in~$\Geo(\Orbi)\setminus\Vanish(\Orbi)$ 
intersects~$\CrSc$ infinitely often in past and future.
\end{prop}

\begin{proof}
We start by establishing~\eqref{CS:infinitelyoften}. To that end let 
$\widehat{\gamma}\in\Geo(\Orbi)\setminus\Vanish(\Orbi)$ and $\gamma$ be any lift 
of~$\wh\gamma$ to~$\H$. We first need to show that $\wh\gamma$ 
intersects~$\CrSc$ or, equivalently, that $\gamma$ intersects~$\Gamma\act\BrU$. 
Without loss of generality we may suppose that $\gamma(\pm\infty)\not=\infty$ 
(otherwise, we pick another representing geodesic for~$\wh\gamma$). Thus, the 
two endpoints~$\gamma(\pm\infty)$ of~$\gamma$ are in~$\R_{\st}$. In what follows 
we show (via proof by contradiction) that there exist~$g\in\Gamma$ and $j\in A$ 
such that $\gamma(+\infty)\in g\act I_j$ and $\gamma(-\infty)\in g\act J_j$. 
Then $\gamma$ intersects~$g\act\Cs{j}$ by~\eqref{BP:allvectors}, and 
hence~$\wh\gamma$ intersects~$\CrSc$. 

In order to seek a contradiction to the existence of such elements in~$\Gamma$ 
and~$A$, we assume that for all~$g\in\Gamma$ and all~$j\in A$ the two 
endpoints~$\gamma(\pm\infty)$ of~$\gamma$ are either both in~$g\act I_j$ or both 
in~$g\act J_j$. Since $\gamma(\pm\infty)\in\R_{\st}$ and $\R_{\st}$ is contained 
in the limit set~$\Lambda(\Gamma)$ of~$\Gamma$, Proposition~\ref{PROP:EXliesdense} implies 
that for each~$\varepsilon>0$ we find a geodesic~$\eta_\eps$ on~$\H$ such that 
$\pi(\eta_\eps)\in\Geo_{\Per}(\Orbi)$ and the endpoints~$(x(\eps),y(\eps)) 
\coloneqq (\eta_\eps(+\infty), \eta_\eps(-\infty))$ of~$\eta_\eps$ are 
$\eps$-near to~$\gamma(\pm\infty)$, respectively, i.e.,
\begin{eqnarray}\label{EQN:epsapart}
\abs{x(\eps)-\gamma(+\infty)}<\eps&\text{and}&\abs{y(\eps)-\gamma(-\infty)}
<\eps\,.
\end{eqnarray}
Again using Proposition~\ref{PROP:EXliesdense}, we may and shall suppose that 
$x(\eps)$ and~$y(\eps)$ are exterior to the interval in~$\R$ that is spanned 
by~$\gamma(+\infty)$ and~$\gamma(-\infty)$. By~\eqref{BP:closedgeodesicsXtoH} we 
find~$h_\eps\in\Gamma$ and~$j_\eps\in A$ such that $h_\eps\act\eta_\eps$ 
intersects~$\Cs{j_\eps}$. Thus, 
\[
(x(\eps),y(\eps)) \in h_\eps^{-1}\act I_{j_\eps}\times h_\eps^{-1}\act 
J_{j_\eps}\,.
\]
By the assumption, either
\begin{eqnarray}\label{EQN:IxorJ}
\gamma(\pm\infty)\in h_\eps^{-1}\act I_{j_\eps}&\text{or}&\gamma(\pm\infty)\in 
h_\eps^{-1}\act J_{j_\eps}\,.
\end{eqnarray}
We consider $h_\eps$ and $j_\eps$ to be fixed once and for all for each~$\eps>0$ 
separately.

We now construct inductively a sequence of ``nested'' translates of a complete 
geodesic segment as follows. Without loss of generality we suppose that 
$\gamma(+\infty)<\gamma(-\infty)$. We pick a (small) $\eps_1>0$ and fix a 
geodesic~$\eta_{\eps_1}$ on~$\H$ with the properties as above with $\eps_1$ in 
place of~$\eps$. We let $a_1,b_1\in\wh\R$ be the endpoints 
of~$h_{\eps_1}^{-1}\act\overline{\base{\Cs{j_{\eps_1}}}}$, ordered such that 
\[
 a_1 < \gamma(+\infty) < \gamma(-\infty) < b_1\,.
\]
If $a_1=\infty \in\wh\R$, then the left hand part of this inequality is 
understood as $-\infty < \gamma(+\infty)$ in~$\R$; analogously for the right 
hand part if $b_1=\infty$.
This configuration is the only one feasible under the condition 
that~\eqref{EQN:epsapart} and~\eqref{EQN:IxorJ} remain both valid for a 
sufficiently small~$\eps_1$.

We set 
\[
 \eps_2 \coloneqq \min\left\{ \frac{|a_1-\gamma(+\infty)|}{2}  , \frac{|b_1 - 
\gamma(-\infty)|}{2} \right\}\,.
\]
Then $\eps_1>\eps_2$, and we repeat with $\eps_2$ in place of~$\eps_1$. We 
emphasize that the chosen geodesic~$\eta_{\eps_2}$ does not intersect 
$h_{\eps_1}^{-1}\act \Cs{j_{\eps_1}}$ as it is contained in either 
$h_{\eps_1}^{-1}\act\Plussp{j_{\eps_1}}$ or 
$h_{\eps_1}^{-1}\act\Minussp{j_{\eps_1}}$. Further, the endpoints $(a_2,b_2)$ do 
not coincide with $(a_1,b_1)$, and since $h_{\eps_2}^{-1}\act 
\overline{\base{\Cs{j_{\eps_2}}}}$ may not intersect $h_{\eps_1}^{-1}\act 
\overline{\base{\Cs{j_{\eps_1}}}}$ by~\eqref{BP:disjointunion}, we have $a_1\leq 
a_2$ and $b_2\leq b_1$ with at least one of the inequalities being strict.

We repeat inductively and obtain a sequence
\[
 h_{\eps_1}^{-1}\act \Cs{j_{\eps_1}}\,,\ h_{\eps_2}^{-1}\act \Cs{j_{\eps_2}}\,,\ 
h_{\eps_3}^{-1}\act \Cs{j_{\eps_3}}\,,\ \ldots
\]
of certain $\Gamma$-translates of elements of the set of branches, and 
sequences 
\[
 (a_n)_{n\in\N}\qquad\text{and}\qquad (b_n)_{n\in\N}
\]
of the endpoints of the elements of the sequence~$(h_{\eps_n}^{-1}\act 
\overline{\base{\Cs{j_{\eps_n}}}})_n$. Then the sequence~$(a_n)_n$ is monotone 
increasing and bounded above by~$\gamma(+\infty)$, and $(b_n)_n$ is monotone 
decreasing and bounded below by~$\gamma(-\infty)$. Let 
\[
 a \coloneqq \lim_{n\to\infty} a_n\qquad\text{and}\qquad b\coloneqq 
\lim_{n\to\infty} b_n\,.
\]
We fix a point on the geodesic segment~$(a,b)_\H$, say~$z$. Then each 
neighborhood of~$z$ intersects infinitely many members of the family
\[
 \defset{ h_{\eps_n}^{-1}\act\overline{\base{\Cs{j_{\eps_n}}}} }{ n\in\N }\,,
\]
which are pairwise disjoint. This 
contradicts~Proposition~\ref{PROP:branches_locfinite}. In turn, 
$\wh\gamma$ intersects~$\CrSc$. 

Without loss of generality we may suppose that $\widehat{\gamma}$ 
intersects~$\CrSc$ at time~$t=0$. By 
Proposition~\ref{PROP:CofSoBnonst}\eqref{CofSoB:uniqueintersect} there exists a 
unique lift~$\eta$ of~$\widehat{\gamma}$ that intersects~$\BrU$ at $t=0$.  Let 
$\nu\coloneqq\eta^{\prime}(0)$. Since $\widehat{\gamma}\notin\Vanish(\Orbi)$, we 
have $\{\eta(\pm\infty)\}\subseteq\widehat{\R}_{\st}$, and hence 
$\nu\in\BrU_{\st}$. Using Lemma~\ref{LEM:accpointsofit} we now obtain 
\eqref{CS:infinitelyoften} with 
$(t_n)_{n\in\Z}=(\ittime{\BrU,n}(\nu))_{n\in\Z}$.

In order to establish~\eqref{CS:discreteintime} we let $\wh\gamma$ be any 
geodesic on~$\Orbi$ and pick any representing geodesic~$\gamma$ on~$\H$. We note 
that the intersection times of~$\wh\gamma$ with~$\CrSc$ and of~$\gamma$ 
with~$\Gamma\act\BrU$ coincide (we picked geodesics with coinciding time 
parametrizations). The local finiteness of~$\Gamma\act\BrU$, as guaranteed 
by~Proposition~\ref{PROP:branches_locfinite}, immediately implies that the 
intersection times form a discrete subset of~$\R$.
\end{proof}

\begin{cor}\label{CofSoB:strongcrosssection}
For each~$\mu\in\mathcal{M}_{\Vanish(\Orbi)}$ the 
set~$\CrSc_{\st}\coloneqq\pi(\BrU_{\st})$ is a strong cross 
section for the geodesic flow~$\GeoFlow$ on~$\Orbi$ with respect to~$\mu$. Each 
geodesic in~$\Geo(\Orbi)\setminus\Vanish(\Orbi)$ intersects~$\CrSc_{\st}$ 
infinitely often in past and future.
\end{cor}

\begin{proof}
The proof of Proposition~\ref{CofSoB:crosssection} already 
establishes~\eqref{CS:infinitelyoften} for~$\CrSc_{\st}$ since it shows that any 
$\wh\gamma\in\Geo(\Orbi)\setminus\Vanish(\Orbi)$ intersects~$\CrSc_{\st}$, not 
only~$\CrSc$, at least once and then infinitely often in past and future. Also 
the proof of~\eqref{CS:discreteintime} for~$\CrSc_{\st}$ can be taken directly 
from the proof of Proposition~\ref{CofSoB:crosssection}. For~\eqref{CS:strong} 
we let~$\wh\gamma$ be any geodesic on~$\Orbi$ that intersects~$\CrSc_{\st}$ at 
least once. Then $\wh\gamma\in\Geo(\Orbi)\setminus\Vanish(\Orbi)$ and it can be 
seen as in the proof of Proposition~\ref{CofSoB:crosssection} that $\wh\gamma$ 
intersects~$\CrSc_{\st}$ infinitely often in past and future.
\end{proof}

\subsection{Slow transfer operators}\label{SUBSEC:slowtrans}

Let $\BrS \coloneqq \defset{ \Cs{j} }{ j\in A }$ with $A\coloneqq \{1,\ldots, 
N\}$ be a set of branches for the geodesic flow on~$\Orbi$. In this section we 
present the discrete dynamical system induced by~$\BrS$ and the associated 
family of transfer operators. These transfer operators are the so-called 
\emph{slow} transfer operators. The notion of \emph{fast} transfer operators 
will be discussed in Section~\ref{SUBSEC:fasttrans}. 

As before we let $\BrU \coloneqq \bigcup_{j\in A} \Cs{j}$ denote the branch 
union and resume the notation from~\eqref{BP:pointintohalfspaces}, 
\eqref{BP:intervaldecomp}, and \eqref{eq:def_redbranch}--\eqref{eq:def_redset}. 
For $\nu\in\BrU_{\st}$ we recall from 
\eqref{EQDEF:ittime0}--\eqref{EQDEF:ittrans2} its system of iterated sequences
\[
[(\ittime{\BrU,n}(\nu))_n,(\itindex{\BrU,n}(\nu))_n,(\ittrans{\BrU,n}(\nu))_n]\,
.
\]
The first return map $\Return\colon \BrU_{\st} \to \BrU_{\st}$ (cf.\@ 
\eqref{eq:firstreturnH}) is given by 
\[
\Return\vert_{\Cs{j,\st}}\colon\Cs{j,\st} \to \Cs{
\itindex{\BrU,1}(\nu),\st}\,,\quad \nu\mapsto 
\ittrans{\BrU,1}(\nu)^{-1}\act\gamma_{\nu}^{\prime}(\ittime{\BrU,1}(\nu))\,,
\]
for any $j\in A$. In order to present the discrete dynamical system induced 
by~$\BrS$, as defined at the end of Section~\ref{SUBSEC:cross}, we set for 
any~$j,k\in A$ and $g\in\Trans{}{j}{k}$, 
\[
D_{j,k,g}\coloneqq g\act\Iset{k}\times\{j\}\,.
\]
By~\eqref{BP:intervaldecomp},
\[
 \bigcup_{j,k\in A}\bigcup_{g\in\Trans{}{j}{k}}D_{j,k,g}=\bigcup_{j\in 
A}\Iset{j}\times\{j\}\,.
\]
Then
\[
D\coloneqq\bigcup_{j\in A}\Iset{j}\times\{j\}
\]
is the domain of the induced discrete dynamical system~$F\colon D\to D$, and the 
map~$F$ decomposes into the \emph{submaps} (local bijections)
\[
F\vert_{D_{j,k,g}}\colon g\act\Iset{k}\times\{j\} \to \Iset{k}\times\{k\}\,,\quad  (x,j)\mapsto(g^{-1}\act x,k)
\]
for $j,k\in A$, $g\in\Trans{}{j}{k}$. One easily checks that the diagram
\begin{center}
\begin{tikzcd}
\BrU_{\st}\arrow[r, "\Return"]\arrow[d, "\iota"]& \BrU_{\st}\arrow[d, "\iota"] 
\\
D\arrow[r, "F"]& D
\end{tikzcd}
\end{center}
indeed commutes, where the surjection $\iota\colon\BrU_{\st}\mapsto D$ is  
piecewise defined by 
\[
\iota\vert_{\Cs{j,\st}}\colon\nu\mapsto(\gamma_{\nu}(+\infty),j)
\]
with $j\in A$. 

Let $V$ be a finite-dimensional complex vector space and let 
\[
 \Fct(D;V) \coloneqq \defset{ f\colon D\to V }{ \text{$f$ function} }
\]
denote the space of $V$-valued functions on~$D$. Further let 
$\chi\colon\Gamma\to\GL(V)$ be a representation of~$\Gamma$ on~$V$. We define an 
associated \emph{weight function}~$\omega$ on~$D$ by
\[
\omega\vert_{D_{j,k,g}}\colon(x,j)\mapsto\chi(g^{-1})\,,
\]
for any $j,k\in A$ and $g\in\Trans{}{j}{k}$.
\index[defs]{transfer operator!slow}%
\index[symbols]{Ls@$\TO{s}$}%
The \emph{(slow) transfer 
operator}~$\TO s$ with parameter $s\in\C$ and weight~$\omega$ associated to the 
map~$F$ is (initially only formally) defined as an operator on~$\Fct(D;V)$ by 
\[
\TO{s}f((x,k))\coloneqq\sum_{(y,j)\in 
F^{-1}((x,k))}\omega((y,j))\abs{F^{\prime}(y,j)}^{-s} f((y,j))
\]
for any $(x,k)\in D$ and $f\in \Fct(D;V)$. The space of functions which is used 
as domain of the transfer operator and on which it then defines an actual 
operator depends a lot on the intended application. It typically is a subset 
of~$\Fct(D;V)$ or a closely related space of functions with complex domains. For 
the transfer-operator based interpretations of Laplace eigenfunctions, which 
motivate this article and that we briefly surveyed in the introduction, the 
function spaces of choice are subspaces of~$\Fct(V;D)$ consisting of highly 
regular functions. We omit any further discussion and refer 
to~\cite{Moeller_Pohl, Pohl_gamma, Pohl_mcf_general, Pohl_oddeven, 
Bruggeman_Pohl, Pohl_Zagier} for details. However, we note that if all 
transition sets~$\Trans{}{j}{k}$ ($j,k\in A$) are finite, then $\TO s$ is 
already well-defined as an operator on~$\Fct(D;V)$. For infinite transition 
sets, questions of convergence need to be considered. 

We end this section with the following example, which illustrates the structure 
of the slow transfer operators in the case of Schottky surfaces. For other 
Fuchsian groups, the structure is similar. For Schottky surfaces, transfer 
operators are classically defined using a Koebe--Morse coding for the geodesic 
flow (see, e.g., \cite{Borthwick_book}). This example also shows that the 
approach via sets of branches reproduces the classical transfer operators and 
generalizes the classical construction.

\begin{example}\label{EX:schottkyTO}
\sloppy Let $\Gamma_{\mathrm{S}}$ be a Schottky group with Schottky data 
$(r,\left\{\mathcal{D}_{\pm j}\right\}_{j=1}^r,\left\{s_{\pm 
j}\right\}_{j=1}^r)$, and recall the set of branches $\{\Cs{\pm1},\dots,\Cs{\pm 
r}\}$ from Example~\ref{EX:schottkySoB}.
Let 
\[
\BrU_{\mathrm{S}}\coloneqq\bigcup_{j=1}^r\Cs{j}\cup\bigcup_{k=1}^r\Cs{-k}\,.
\]
For $j\in\{\pm1,\dots,\pm r\}$ consider the subspace $A^2(\mathcal{D}_j)$ of 
$L^2(\mathcal{D}_j)$ of holomorphic functions (Bergman space with $p=2$).
Then the direct sum
\[
\mathcal{H}\coloneqq\bigoplus_{j=1}^r A^2(\mathcal{D}_j)\oplus\bigoplus_{k=1}^r 
A^2(\mathcal{D}_{-k})
\]
is a Hilbert space.
If we identify functions $f\in\mathcal{H}$ with the function vectors
\begin{eqnarray*}
\bigoplus_{j=1}^rf_j\oplus\bigoplus_{k=1}^rf_{-k}\,,&\text{with}& f_l\in 
A^2(\mathcal{D}_l)\,,\quad l\in\{\pm 1,\dots,\pm r\}\,,
\end{eqnarray*}
then the slow transfer operator for $\BrU_{\mathrm{S}}$ with parameter~$s\in\C$ 
and constant weight~$\omega\equiv 1$ (trivial representation) takes the form
\[
\TO{s}=\begin{pmatrix}
\tau_s(s_1)&\tau_s(s_2)&\dots&\tau_s(s_r)&0&\tau_s(s_{-2})&\dots&\tau_s(s_{-r}
)\\
\tau_s(s_1)&\tau_s(s_2)&\dots&\tau_s(s_r)&\tau_s(s_{-1})&0&\dots&\tau_s(s_{-r}
)\\
\vdots&\vdots&\ddots&\vdots&\vdots&\vdots&\ddots&\vdots\\
\tau_s(s_1)&\tau_s(s_2)&\dots&\tau_s(s_r)&\tau_s(s_{-1})&\tau_s(s_{-2}
)&\dots&0\\
0&\tau_s(s_2)&\dots&\tau_s(s_r)&\tau_s(s_{-1})&\tau_s(s_{-2})&\dots&\tau_s(s_{-r
})\\
\tau_s(s_1)&0&\dots&\tau_s(s_r)&\tau_s(s_{-1})&\tau_s(s_{-2})&\dots&\tau_s(s_{-r
})\\
\vdots&\vdots&\ddots&\vdots&\vdots&\vdots&\ddots&\vdots\\
\tau_s(s_1)&\tau_s(s_2)&\dots&0&\tau_s(s_{-1})&\tau_s(s_{-2})&\dots&\tau_s(s_{-r
})
\end{pmatrix},
\]
where
\begin{eqnarray*}
\tau_s(g^{-1})f(x)\coloneqq\left(g^{\prime}(x)\right)^sf(g\act x)
\end{eqnarray*}
for $f\colon U\to\C$, $x\in U$, $g\in\Gamma_{\mathrm{S}}$.
\end{example}

\subsection{Strict transfer operator approaches}\label{SUBSEC:stricttrans}

This section serves to recall the concept of strict transfer operator approaches 
from~\cite{FP_NECM}. The relation between these approaches and the Selberg zeta 
function will be discussed in Section~\ref{SUBSEC:fasttrans}. As already stated 
in the introduction, the aim of this article is to construct strict transfer 
operator approaches for a large class of Fuchsian groups, starting at sets of 
branches. 

We say that $\Gamma$ admits a \emph{strict 
transfer operator approach}
\index[defs]{strict transfer operator approach}%
if there exists a \emph{structure tuple}
\[
\mathcal{S} \coloneqq \bigl( \Index, (\wh{I}_{a})_{a \in \Index}, (P_{a,b})_{a,b \in \Index}, (C_{a,b})_{a,b \in \Index},((g_p)_{p\in P_{a,b}})_{a,b\in \Index} \bigr)
\]
\index[defs]{structure tuple}%
\index[symbols]{S@$\mathcal S$}%
consisting of
\index[symbols]{Ab@$\Index$}%
\index[symbols]{Ia@$\wh{I}_a$}%
\index[symbols]{Pab@$P_{a,b}$}%
\index[symbols]{Cab@$C_{a,b}$}%
\index[symbols]{gp@$g_p$}%
\begin{itemize}
\item a finite set $\Index$,
\item a family $(\wh{I}_{a})_{a \in \Index}$ of (not necessarily disjoint) intervals in 
$\widehat{\R}$,
\item a family $(P_{a,b})_{a,b \in \Index}$ of finite (possibly empty) sets of 
parabolic elements in $\Gamma$,
\item a family $(C_{a,b})_{a,b \in \Index}$ of finite (possibly empty) subsets of 
$\Gamma$, and
\item a family $((g_p)_{p\in P_{a,b}})_{a,b\in \Index}$ of elements of $\Gamma$ 
(which may be the identity)
\end{itemize}
which satisfies the following five properties.
\begin{propsperty}\label{staPROP1}
\index[defs]{property!1}%
For all $a,b \in \Index$
\begin{enumerate}[label=$\mathrm{(\Roman*)}$,ref=$\mathrm{\Roman*}$]
\item\label{staPROP1:Pinclude} we have $p^{-n}g_p^{-1}\act \wh{I}_{a,\st} 
\subseteq \wh{I}_{b,\st}$ for all $p \in P_{a,b}$ and $n \in \N$, and $p^{n} 
\notin P_{a,b}$ for $n\geq2$,
\item\label{staPROP1:Cinclude} we have $g^{-1}\act \wh{I}_{a,\st} \subseteq 
\wh{I}_{b,\st}$ for all $g \in C_{a,b}$,
\item\label{staPROP1:decompofIsets} the sets in the family
\[
\defset{g^{-1}\act \wh{I}_{j,\st}}{j \in \Index, g \in C_{j,b}} \cup 
\defset{p^{-n}g_p^{-1}\act \wh{I}_{j,\st}}{j \in \Index,\, p \in P_{j,b},\, n \in 
\N}
\]
are pairwise disjoint and
\[
\wh{I}_{b,\st} = \bigcup_{j \in \Index} \Biggl( \bigcup_{g \in C_{j,b}}g^{-1}\act 
\wh{I}_{j,\st} \cup \bigcup_{p \in P_{j,b}} \bigcup_{n = 
1}^{\infty}p^{-n}g_p^{-1}\act \wh{I}_{j,\st}\Biggr)\,.
\]
\end{enumerate}
\end{propsperty}
Property 1 induces a discrete dynamical system $(D,F)$,
\index[defs]{discrete dynamical system}%
\index[symbols]{DF@$(D,F)$}%
where
\[
D \coloneqq \bigcup_{a\in \Index}\wh{I}_{a,\st}\times \{a\}\,,
\]
\index[symbols]{D@$D$}%
and $F$ splits into the submaps (bijections, that are local parts of the map~$F$)
\index[defs]{submap}%
\begin{align*}
g^{-1}\act \wh{I}_{a,\st}\times \{b\} \ni (x,b) &\mapsto (g\act x,a) \in 
\wh{I}_{a,\st} \times \{a\}\,, 
\intertext{and}
p^{-n}g_p^{-1}\act \wh{I}_{a,\st}\times \{b\} \ni (x,b) &\mapsto 
(g_pp^{n}\act x,a) \in \wh{I}_{a,\st} \times \{a\}\,,
\end{align*}
for all $a,b \in \Index$, $g \in C_{a,b}$, $p \in P_{a,b}$ and $n \in \N$, which 
completely determine~$F$.

\begin{propsperty}\label{staPROP2}
\index[defs]{property!2}%
For $n \in \N$ denote by $\Per_{n}$
\index[symbols]{Pern@$\Per_n$}%
the subset of~$\Gamma$ of all~$g$ for which 
there exists $a \in \Index$ such that
\[
g^{-1}\act \wh{I}_{a,\st}\times \{a\} \ni (x,a) \mapsto (g\act x,a) \in 
\wh{I}_{a,\st}\times \{a\}
\]
is a submap of $F^{n}$. Then the union
\[
\Per \coloneqq \bigcup_{n = 1}^{\infty}\Per_{n}
\]
is disjoint.
\index[symbols]{Per@$\Per$}%
\end{propsperty}
As before, we denote by~$[\Gamma]_{\mathrm{h}}$ the set of all 
$\Gamma$-conjugacy classes of hyperbolic elements in~$\Gamma$.
\begin{propsperty}\label{staPROP3}
\index[defs]{property!3}%
Let $\Per$ be as in Property 2. Then
\begin{enumerate}[label=$\mathrm{(\Roman*)}$, ref=$\mathrm{\Roman*}$]
\item\label{staPROP3:hyperbolic}
all elements of $\Per$ are hyperbolic,

\item\label{staPROP3:primitive}
for each $h \in \Per$ also its primitive $h_{0}$ is contained in $\Per$,

\item\label{staPROP3:representatives}
for each $[g] \in [\Gamma]_{\mathrm{h}}$ there exists a unique element $n \in 
\N$ such that $\Per_{n}$ contains an element that represents $[g]$.
\end{enumerate}
\end{propsperty}

Suppose that $[g]\in [\Gamma]_{\mathrm{h}}$ is represented by $h\in\Per_{n}$, 
$n\in\N$.
Because of Property~\ref{staPROP2} we shall define the \emph{word length of~$h$} 
as
\index[defs]{word length}%
\index[symbols]{om@$\omega(h)$}%
\begin{align}\label{EQDEF:wordlengthofh}
\omega(h)\coloneqq n\,.
\end{align}
We denote by $m=m(h)\in\N$
\index[symbols]{m@$m(h)$}%
the unique number such that $h = h_{0}^{m}$ for a 
primitive hyperbolic element $h_{0}\in\Gamma$, and we set
\begin{align}\label{EQDEF:ratioofh}
p(h)\coloneqq\tfrac{\omega(h)}{m(h)}\,.
\end{align}
\index[symbols]{p@$p(h)$}%
Further we set $\omega(g)\coloneqq\omega(h)$ as well as $m(g)\coloneqq m(h)$ and 
$p(g)\coloneqq p(h)$. By Property~\ref{staPROP3} these values are well-defined.

\begin{propsperty}\label{staPROP4}
\index[defs]{property!4}%
For each element $[g] \in [\Gamma]_{\mathrm{h}}$ there are exactly $p(g)$ 
distinct elements $h \in \Per_{\omega(g)}$ such that $h \in [g]$.
\end{propsperty}

\begin{propsperty}\label{staPROP5}
\index[defs]{property!5}%
There exists a family $(\mathcal{E}_ {a})_{a \in \Index}$ of open, bounded, 
connected and simply connected sets in $\widehat{\C}$ such that
\index[symbols]{Ea@$\mathcal{E}_a$}%
\begin{enumerate}[label=$\mathrm{(\Roman*)}$, ref=$\mathrm{\Roman*}$]
\item\label{staPROP5:cover}
for all $a \in \Index$ we have
\[
\overline{\wh{I}_{a,\st}} \subseteq \mathcal{E}_{a}\,,
\]
\item\label{staPROP5:fliptheworld}
there exists $\xi \in \PSLR$ such that for all $a \in \Index$ we have 
$\xi\act\overline{\mathcal{E}_{a}} \subseteq \C$, and for all $b \in \Index$ and all 
$g \in C_{a,b}$ we have
\[
g\xi^{-1}\act\infty \notin \overline{\mathcal{E}_{a}}\,,
\]
\item\label{staPROP5:Cinclude}
for all $a,b \in \Index$ and all $g \in C_{a,b}$ we have
\[
g^{-1}\act\overline{\mathcal{E}_{a}} \subseteq \mathcal{E}_{b}\,,
\]
\item\label{staPROP5:Pinclude}
for all $a,b \in \Index$ and all $p \in P_{a,b}$ there exists a compact subset 
$K_{a,b,p}$ of $\widehat{\C}$ such that for all $n \in \N$ we have
\[
p^{-n}g_p^{-1}\act\overline{\mathcal{E}_{a}} \subseteq K_{a,b,p} \subseteq 
\mathcal{E}_{b}\,,
\]
\item\label{staPROP5:nofixed}
for all $a,b \in \Index$ and all $p \in P_{a,b}$ the set 
$g_p^{-1}\act\overline{\mathcal{E}_{a}}$ does not contain the fixed point of 
$p$.
\end{enumerate}
\end{propsperty}

\subsection{Fast transfer operators and Selberg zeta 
functions}\label{SUBSEC:fasttrans}

In this section we briefly explain the use of strict transfer operator 
approaches for Selberg zeta functions and introduce for this the notion of 
fast transfer operators. 
To that end we suppose that $\Gamma$ admits a strict transfer operator approach 
with structure tuple
\[
\mathcal{S} \coloneqq \bigl( \Index, (I_{a})_{a \in \Index}, (P_{a,b})_{a,b \in \Index}, (C_{a,b})_{a,b \in \Index},((g_p)_{p\in P_{a,b}})_{a,b\in \Index} \bigr)\,.
\]
We let $V$ be a finite-dimensional vector space and let
$\chi\colon\Gamma\to\GL(V)$ be a representation of~$\Gamma$ on~$V$ with \emph{non-expanding cusp monodromy},
\index[defs]{non-expanding cusp monodromy}%
i.e., for each parabolic 
element~$p\in\Gamma$ the endomorphism~$\chi(p)$ has only eigenvalues with 
absolute value~$1$. (We refer to ~\cite{EKMZ_Lyapunovexp} and~\cite{FP_NECM} for 
an extended discussion of this property.) 

For $s\in\C$, $U\subseteq\widehat{\C}$, any $f\in\Fct(U;V)$, $g\in\Gamma$ and 
$z\in U$ we set
\[
\alpha_s(g^{-1})f(z)\coloneqq(g^{\prime}(z))^s\chi(g)f(g\act z)\,,
\]
\index[symbols]{as@$\alpha_s$}%
whenever it is well-defined. (We note that $\alpha_s$ is typically not a 
representation of~$\Gamma$ on~$\Fct(U;V)$,
\index[symbols]{Fct@$\Fct(U;V)$}%
but it satisfies some restricted 
homomorphism properties, which motivated the notation. We refer to the 
discussion in~\cite[Section~6.3]{Bruggeman_Pohl} for details.) For any open set 
$U\subseteq\C$ we set
\[
\mathcal{B}(U;V)\coloneqq\defset{f\in C(\overline{U};V)}{\text{$f\vert_U$ 
holomorphic}}\,.
\]
\index[symbols]{BUV@$\mathcal{B}(U;V)$}%
Then $\mathcal{B}(U;V)$, endowed with the supremum norm, is a Banach space. 
We write
\[
\mathcal{B}(\mathcal{E}_{\Index};V)\coloneqq\bigoplus_{a\in 
\Index}\mathcal{B}(\mathcal{E}_a;V)
\]
for the product space, where $\mathcal{E}_{\Index}=(\mathcal{E}_a)_{a\in \Index}$ is a family of open sets as provided by Property~\ref{staPROP5}.

For $s\in\C$ with $\Rea s$ sufficiently large, the transfer operator~$\TO{s,\chi}$ 
with parameter~$s$ associated to~$\mc S$ and~$\chi$ is an operator on the Banach 
space~$\mc  B(\mc E_{\Index};V)$. If we identify the elements 
$f\in\mathcal{B}(\mathcal{E}_{\Index};V)$ with the function vectors $f=(f_a)_{a\in \Index}$, 
where 
\[
f_a\colon I_{a,\st}\to V
\]
for $a\in \Index$, then $\TO{s,\chi}$ reads 
\[
\TO{s,\chi}=\biggl(\sum_{g\in C_{a,b}}\alpha_s(g)+\sum_{p\in 
P_{a,b}}\sum_{n\in\N}\alpha_s(g_pp^n)\biggr)_{a,b\in \Index}\,.
\]
We call $\{\TO{s,\chi}\}_{s}$ the \emph{fast transfer operator family} for~$\Gamma$ 
associated to~$\mathcal{S}$. 
\index[defs]{transfer operator!fast}%
\index[symbols]{Ls@$\TO{s}$}%

One of the main motivations for this article is the following result 
of~\cite{FP_NECM}, which shows that (fast) transfer operator families arising 
from strict transfer operator approaches provide Fredholm determinant 
representations of Selberg zeta functions.

\begin{thm}[\cite{FP_NECM}]\label{mainthm:F_P_NECM}
Let $\Gamma$ be a geometrically finite Fuchsian group which admits a strict 
transfer operator approach, and let $\chi\colon\Gamma\to\GL(V)$ be a 
finite-dimensional representation of~$\Gamma$ on the finite-dimensional vector 
space~$V$ having non-expanding cusp monodromy. Let $\mathcal{S}$ be a structure 
tuple for~$\Gamma$ with associated fast transfer operator 
family~$\{\TO{s,\chi}\}_{s}$.
Let~$\Orbi=\quod{\Gamma}{\H}$ and resume the notation of Section~\ref{SUBSEC:stricttrans}.
Then we have:
\begin{enumerate}[label=$\mathrm{(\roman*)}$, ref=$\mathrm{\roman*}$]
\item There exists $\delta>0$ only depending on $\Gamma$ and $(V,\chi)$ such 
that for $s\in\C$ with $\Rea s>\delta$ the operator $\TO{s,\chi}$ on 
$\mathcal{B}(\mathcal{E}_{\Index};V)$ is bounded and nuclear of order 0, independently of the choice of the family $\mathcal{E}_{\Index}$.
\item The map $s\mapsto\TO{s,\chi}$ extends meromorphically to all of $\C$ with 
values in nuclear operators of order 0 on $\mathcal{B}(\mathcal{E}_{\Index};V)$. All poles are simple. There exists $d\in\N$ such that each pole is contained in 
$\frac{1}{2}(d-\N_0)$. 
\item\label{eq:szf_id} For $\Rea s\gg1$, we have
\[
Z_{\Orbi,\chi}(s)=\det(1-\TO{s,\chi})\,.
\]
\item The Selberg zeta function $Z_{\Orbi,\chi}$ extends to a 
meromorphic function on $\C$ with poles contained in $\frac{1}{2}(d-\N_0)$ and 
the identity in~\eqref{eq:szf_id} extends to all of~$\C$.
\end{enumerate}
\end{thm}

We reduced the statement of Theorem~\ref{mainthm:F_P_NECM} to match our needs here. 
The result in~\cite{FP_NECM} (see Theorem~4.2 ibid.) contains further 
information about the rank of the operators~$\TO{s,\chi}$ as well as the order of the poles of $Z_{\Orbi,\chi}$.
In particular, explicit values for~$\delta$ and $d$ are given. 
However, all of these additional informations are not needed for our purposes here. 

Combining Theorem~\ref{mainthm:F_P_NECM} with the main result of this paper, 
i.e., with Theorem~\ref{mainthm},  yields a representation of the Selberg zeta 
function in terms of Fredholm determinants of transfer operators for each 
Fuchsian group that admits the construction of a set of branches.

\subsection{Ramification of branches}\label{SUBSEC:branchram}

Let $\BrS\coloneqq\defset{\Cs{j}}{j\in A}$ be a set of branches for the geodesic 
flow on~$\Orbi$ with $A=\{1,\dots,N\}$ and $N\in\N$ and let $\Trans{}{.}{.}$ 
denote the forward transition sets. As Example~\ref{EX:branchramification} 
shows, it is possible that transition sets are infinite. However, for our 
constructions of strict transfer operator approaches further below we will 
suppose that all transition sets are finite. 

The purpose of this section is twofold. We first present a simple-to-check  
criterium that allows us to distinguish sets of branches with infinite 
transition sets from those for which all transition sets are finite. Then we provide an 
algorithm that turns each set of branches with infinite transition sets into one 
with only finite transition sets by adding a limited number of specific 
branches.

\begin{defi}\label{DEF:branchramification}
For $j\in A$ we define its \emph{ramification number}
\index[defs]{ramification!number}%
\index[symbols]{ram@$\ram{j}$}%
by
\[
\ram{j}\coloneqq\sum_{k\in A}\#\Trans{}{j}{k}
\]
and the \emph{ramification in} $\BrS=\defset{\Cs{j}}{j\in A}$
\index[defs]{ramification}%
\index[symbols]{ram@$\Ram{\BrS}$}%
by
\[
\Ram{\BrS}\coloneqq\sup_{j\in A}\ram{j}\,.
\]
A set of branches $\BrS$ is called \emph{infinitely 
ramified} if $\Ram{\BrS}=+\infty$, and \emph{finitely ramified} 
otherwise.
\index[defs]{ramified!finitely}%
\index[defs]{ramified!infinitely}%
If we need to emphasize the choice of the Fuchsian group~$\Gamma$ for 
the ramification (as, e.g., in Example~\ref{EX:branchramification}), then we 
write~$\Ram{\BrS,\Gamma}$ for~$\Ram{\BrS}$.
\end{defi}

Starting on the branch $\Cs j$ the number $\ram{j}$ encodes the number of 
distinct directions in which one can travel with regard to the next intersection 
branches.

With the following example we illustrate that the ramification heavily depends 
on the combination of Fuchsian group \emph{and} set of branches. We provide two 
Fuchsian groups that admit the same set of branches, but the ramification is 
finite only for one of them. We also show that a different choice of set of 
branches may yield a finite ramification.

\begin{example}\label{EX:branchramification}
We consider the modular group 
\[
 \Gamma_1 \coloneqq \PSL_2(\Z)
\]
and the projective Hecke congruence group of level~$2$
\begin{align*}
 \Gamma_2 &\coloneqq \PGamma_0(2) = \defset{ \bmat{a}{b}{c}{d} \in \PSL_2(\Z) }{ 
c \equiv 0 \mod 2}\,.
\end{align*}
These groups are well-known to be discrete and geometrically finite. Fundamental 
domains for them are indicated in Figure~\ref{FIG:branchram:fund}. We fix the 
elements
\begin{eqnarray*}
s_1\coloneqq\begin{bmatrix}0&1\\-1&0\end{bmatrix},&s_2\coloneqq\begin{bmatrix}
1&-1\\2&-1\end{bmatrix}\quad\text{ 
and}\quad &t\coloneqq\begin{bmatrix}1&1\\0&1\end{bmatrix}
\end{eqnarray*}
of~$\PSLR$. Then 
\begin{align*}
\Gamma_1 = \langle s_1,t\rangle \qquad\text{and}\qquad 
\Gamma_2 = \langle s_2,t\rangle\,.
\end{align*}

\begin{figure}
\begin{tikzpicture}[scale=2.5]
\draw[color=white] (-.5,0) -- (-.5,2.5);
\fill[color=lightgray!40] (0,0) -- (1,0) -- (1,2.5) -- (0,2.5) -- cycle;
\draw[dashed] (0,1) -- (0,2.5);
\draw[dashed] (1,1) -- (1,2.5);
\fill[color=white]  (0,0) -- +(90:1) arc (90:0:1) -- cycle;
\fill[color=white]  (1,0) -- +(180:1) arc (180:90:1) -- cycle;
\draw[style=thick,color=white] (0,0) -- (0,1);
\draw[style=thick,color=white] (1,0) -- (1,1);
\draw[dashed] (0,0) -- +(0:1) arc (0:110:1);
\draw[dashed] (1,0) -- +(180:1) arc (180:70:1);
 \foreach \x/\y in {0/$0$,.5/$\tfrac{1}{2}$,1/$1$}
    \draw (\x,0.04) -- (\x,-0.04) node [below] {\y};
\draw[style=thick] (-.5,0) -- (1.5,0);
\coordinate [label=below:$\mathcal{F}_{\Gamma_1}$] (F) at (.5,2);
\end{tikzpicture}
\qquad
\begin{tikzpicture}[scale=2.5]
\draw[color=white] (-.5,0) -- (-.5,2.5);
\fill[color=lightgray!40] (0,0) -- (1,0) -- (1,2.5) -- (0,2.5) -- cycle;
\draw[dashed] (0,0) -- (0,2.5);
\draw[dashed] (1,0) -- (1,2.5);
\fill[color=white]  (.5,0) -- +(0:.5) arc (0:180:.5) -- cycle;
\draw[dashed] (.5,0) -- +(0:.5) arc (0:180:.5);
 \foreach \x/\y in {0/$0$,.5/$\tfrac{1}{2}$,1/$1$}
    \draw (\x,0.04) -- (\x,-0.04) node [below] {\y};
\draw[style=thick] (-.5,0) -- (1.5,0);
\coordinate [label=below:$\mathcal{F}_{\Gamma_2}$] (F) at (.5,2);
\end{tikzpicture}
\caption[branchramification]{Examples of fundamental domains for $\Gamma_1$ and 
$\Gamma_2$, respectively.}\label{FIG:branchram:fund}
\end{figure}
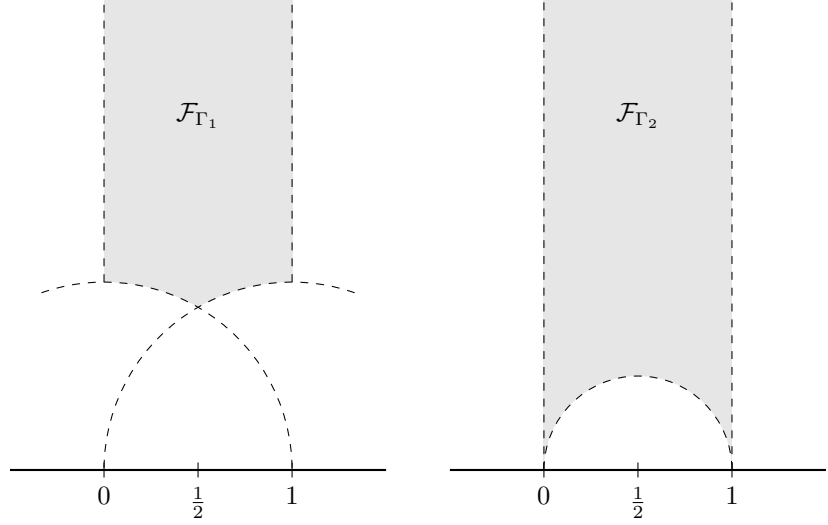

\begin{enumerate}[label=$\mathrm{(\roman*)}$, ref=$\mathrm{\roman*}$]
\item We let $\gamma\coloneqq(0,1)_{\H}$ be the geodesic segment from~$0$ to~$1$ 
and define $\Cs{1}\subseteq\UTB\H$ to be the set of all vectors based 
on~$\gamma$ and pointing into the half-space to the right of it. Thus, 
\[
I_1=\defset{\gamma_\nu(+\infty)}{\nu\in\Cs{1}}=(0,1)
\]
and 
\[
J_1=\defset{\gamma_\nu(-\infty)}{\nu\in\Cs{1}}=(-\infty,0)\cup(1,+\infty)
\]
(see also Figures~\ref{FIG:branchram:Gamma1} and~\ref{FIG:branchram:Gamma2}).
\begin{figure}
\begin{tikzpicture}[scale=8]
\fill[color=lightgray] (.5,0) -- +(0:.5) arc (0:180:.5);
\fill[color=white] (.5,0) -- +(0:.47) arc (0:180:.47);
\draw (.5,0) -- +(0:.5) arc (0:180:.5);
\foreach \x in {.25,.75}{
	\fill[color=lightgray!40] (\x,0) -- +(0:.25) arc (0:180:.25);
	\fill[color=white] (\x,0) -- +(0:.22) arc (0:180:.22);
	\draw (\x,0) -- +(0:.25) arc (0:180:.25);
	}
 \foreach \x/\y in {0/$0$,.5/$\tfrac{1}{2}$,1/$1$}
    \draw (\x,0.02) -- (\x,-0.02) node [below] {\y};
\draw[style=thick] (-.1,0) -- (1.1,0);
\coordinate [label=below:$\Cs{1}$] (C) at (.5,.47);
\coordinate [label=below:$s_1t^{-1}s_1\act\Cs{1}$] (ST1SC) at (.25,.22);
\coordinate [label=below:$s_1t^{-2}\act\Cs{1}$] (ST2C) at (.75,.22);
\end{tikzpicture}
\caption[branchramGamma1]{The set of branches $\{\Cs{1}\}$ for $\Gamma_1$ and 
its 
successors.}\label{FIG:branchram:Gamma1}
\end{figure}
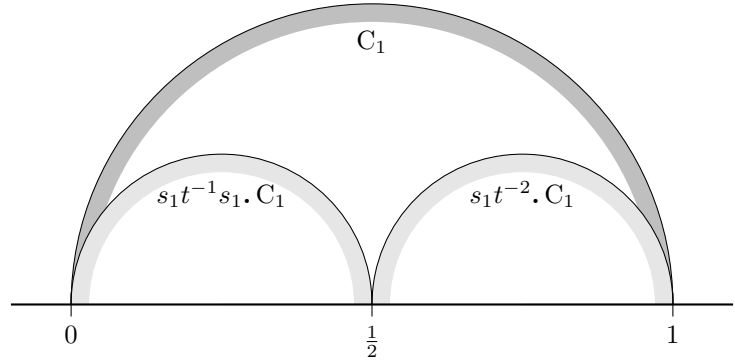
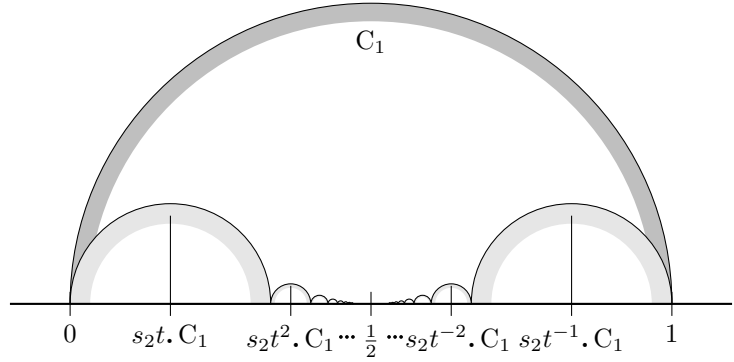
\begin{figure}
\begin{tikzpicture}[scale=8]
\fill[color=lightgray] (.5,0) -- +(0:.5) arc (0:180:.5);
\fill[color=white] (.5,0) -- +(0:.47) arc (0:180:.47);
\draw (.5,0) -- +(0:.5) arc (0:180:.5);
\foreach \x in {1,2,3,4,5,6,7,8}{
	\pgfmathsetmacro{\Rad}{1/(8*\x*\x - 2)}%
	\pgfmathsetmacro{\Mit}{(4*\x*\x - 2*\x - 1)/(8*\x*\x - 2)}%
	\pgfmathsetmacro{\Mint}{(4*\x*\x + 2*\x - 1)/(8*\x*\x - 2)}%
	\fill[color=lightgray!40] (\Mit,0) -- +(0:\Rad) arc (0:180:\Rad);
	\fill[color=white] (\Mit,0) -- +(0:\Rad*.8) arc (0:180:\Rad*.8);
	\draw (\Mit,0) -- +(0:\Rad) arc (0:180:\Rad);
	\fill[color=lightgray!40] (\Mint,0) -- +(0:\Rad) arc (0:180:\Rad);
	\fill[color=white] (\Mint,0) -- +(0:\Rad*.8) arc (0:180:\Rad*.8);
	\draw (\Mint,0) -- +(0:\Rad) arc (0:180:\Rad);
	}
 \foreach \x/\y in {0/$0$,.5/$\tfrac{1}{2}$,1/$1$}
    \draw (\x,0.02) -- (\x,-0.02) node [below] {\y};
\draw[style=thick] (-.1,0) -- (1.1,0);
\coordinate [label=below:$\Cs{1}$] (C) at (.5,.47);
\draw (.16666666,-0.02) -- (.16666666,.14666666);
\coordinate [label=below:$s_2t\act\Cs{1}$] (STC) at (.16666666,-0.02);
\draw (.36666666,-0.02) -- (.36666666,.03);
\coordinate [label=below:$s_2t^2\act\Cs{1}$] (ST2C) at (.36666666,-0.02);
\node at (.45,-0.055)[circle,fill,inner sep=.5pt]{};
\node at (.46,-0.055)[circle,fill,inner sep=.5pt]{};
\node at (.47,-0.055)[circle,fill,inner sep=.5pt]{};
\draw (1-.16666666,-0.02) -- (1-.16666666,.14666666);
\coordinate [label=below:$s_2t^{-1}\act\Cs{1}$] (STinvC) at (1-.16666666,-0.02);
\draw (1-.36666666,-0.02) -- (1-.36666666,.03);
\coordinate [label=below:$s_2t^{-2}\act\Cs{1}$] (STinv2C) at (1-.355,-0.02);
\node at (1-.45,-0.055)[circle,fill,inner sep=.5pt]{};
\node at (1-.46,-0.055)[circle,fill,inner sep=.5pt]{};
\node at (1-.47,-0.055)[circle,fill,inner sep=.5pt]{};
\end{tikzpicture}
\caption[branchramGamma2]{The set of branches $\{\Cs{1}\}$ for $\Gamma_2$ and 
its 
successors.}\label{FIG:branchram:Gamma2}
\end{figure}
Then $\{\Cs{1}\}$ is a set of branches for both groups, $\Gamma_1$ and 
$\Gamma_2$. However, 
\begin{eqnarray*}
\Ram{\{\Cs{1}\},\Gamma_1}=2,&\text{while}&\Ram{\{\Cs{1}\},\Gamma_2}=+\infty\,,
\end{eqnarray*}
as indicated in Figures~\ref{FIG:branchram:Gamma1} 
and~\ref{FIG:branchram:Gamma2}.
\item We now let $\eta \coloneqq (0,1/2)_\H$ be the geodesic segment from~$0$ 
to~$1/2$, and let $\Cs{2}$ and~$\Cs{3}$ be the set of unit tangent vectors based 
on~$\eta$ that point into the half-space to the right or left of~$\eta$, 
respectively. Then $\BrS = \{ \Cs{1}, \Cs{2},\Cs{3}\}$ is a set of branches 
for~$\Gamma_2$, and $\Ram{\BrS,\Gamma_2} = 2$ (see 
Figure~\ref{FIG:branchram2}). 

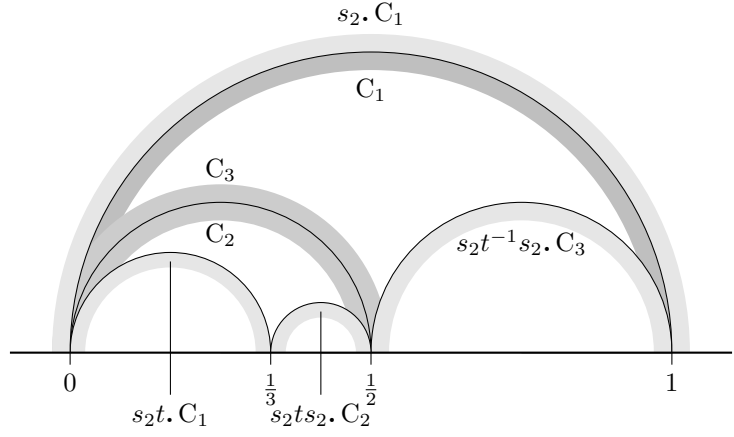
\begin{figure}
\begin{tikzpicture}[scale=8]
\fill[color=lightgray] (.5,0) -- +(0:.5) arc (0:180:.5);
\fill[color=white] (.5,0) -- +(0:.47) arc (0:180:.47);
\fill[color=lightgray!80] (.25,0) -- +(0:.28) arc (0:180:.28);
\fill[color=lightgray!80] (.25,0) -- +(0:.25) arc (0:180:.25);
\fill[color=white] (.25,0) -- +(0:.22) arc (0:180:.22);
\draw (.25,0) -- +(0:.25) arc (0:180:.25);
\foreach \x in {.75}{
	\fill[color=lightgray!40] (\x,0) -- +(0:.25) arc (0:180:.25);
	\fill[color=white] (\x,0) -- +(0:.22) arc (0:180:.22);
	\draw (\x,0) -- +(0:.25) arc (0:180:.25);
	}
\foreach \x in {.1666666666, .4166666666}{
	\fill[color=lightgray!40] (\x,0) -- +(0:{abs(\x - 1/3)}) arc (0:180:{abs(\x 
- 1/3)});
	\fill[color=white] (\x,0) -- +(0:{abs(\x - 1/3)-.025}) arc (0:180:{abs(\x - 
1/3)-.025});
	\draw (\x,0) -- +(0:{abs(\x - 1/3)}) arc (0:180:{abs(\x - 1/3)});
	\draw (\x,-.07) -- (\x,{abs(\x - 1/3)-.015});
}
\fill[color=lightgray!40] (1,0) -- (0:1.03) arc (0:180:.53) -- (0,0) arc 
(180:0:.5);
\draw (.5,0) -- +(0:.5) arc (0:180:.5);
 \foreach \x/\y in {0/$0$,.3333333333/$\tfrac{1}{3}$,.5/$\tfrac{1}{2}$,1/$1$}
    \draw (\x,0) -- (\x,-0.02) node [below] {\y};
\draw[style=thick] (-.1,0) -- (1.1,0);
\coordinate [label=below:$\Cs{1}$] (C1) at (.5,.47);
\coordinate [label=above:$s_2\act\Cs{1}$] (SC1) at (.5,.53);
\coordinate [label=below:$\Cs{2}$] (C2) at (.25,.23);
\coordinate [label=above:$\Cs{3}$] (C3) at (.25,.27);
\coordinate [label=below:$s_2t^{-1}s_2\act\Cs{3}$] (ST1SC3) at (.75,.22);
\coordinate [label=below:$s_2t\act\Cs{1}$] (STC1) at (.16666666,-.07);
\coordinate [label=below:$s_2ts_2\act\Cs{2}$] (STSC2) at (.416666666,-.07);
\end{tikzpicture}
\caption[branchram2]{The set of branches $\BrS=\{\Cs{1}, \Cs{2}, \Cs{3}\}$ for 
$\Gamma_2$ and its 
successors.}\label{FIG:branchram2}
\end{figure}

\end{enumerate}
\end{example}

The following result shows that ramification is invariant under the action 
of~$\Gamma$.

\begin{lemma}\label{LEM:ramtrans}
Let $\defset{p_j}{j\in A}\subseteq\Gamma$. Then the ramification of the set of 
branches $\BrS'\coloneqq\defset{p_j\act\Cs{j}}{j\in A}$ equals the one 
of~$\BrS$, thus
\begin{align*}
\Ram{\BrS'}=\Ram{\BrS}\,.
\end{align*}
\end{lemma}

\begin{proof}
We recall from Lemma~\ref{LEM:nouniqueSoB} that $\BrS'$ is indeed a set of 
branches and denote all objects related to~$\BrS'$ by the same letter as 
for~$\BrS$ but with an additional prime (${ }'$). Thus, $\Iset{j}' = p_j\act 
\Iset{j}$ for all $j\in A$, etc. 
Then~(\ref{BP:intervaldecomp}\ref{BP:intervaldecompGdecomp}) yields
\[
p_j\act\Iset{j}=p_j\act\Bigl(\bigcup_{k\in 
A}\bigcup_{g\in\Trans{}{j}{k}}g\act\Iset{k}\Bigr)=\bigcup_{k\in 
A}\bigcup_{g\in\Trans{}{j}{k}}p_jgp_k^{-1}\act\Iset{k}'
\]
for all~$j,k\in A$, and hence $\Trans{}{j}{k}'=p_j\Trans{}{j}{k}p_k^{-1}$. 
Therefore the ramification number of~$j$ with respect to~$\BrS$ is the same as 
with respect to~$\BrS'$.
\end{proof}

In what follows, we determine the geometric structure of finitely and 
infinitely ramified sets of branches and find that the cusps of~$\Orbi$ play a 
central role. It is therefore convenient to first study the case that~$\Orbi$ 
has no cusps.

\begin{lemma}\label{LEM:nocuspnoram}
Let $\Gamma$ be a geometrically finite Fuchsian group without parabolic 
elements. Then every set of branches for the geodesic flow on $\Orbi$ is 
finitely ramified. 
\end{lemma}

\begin{proof}
Let $\BrS = \{\Cs{1},\dots,\Cs{N}\}$ be a set of branches for the geodesic flow 
on~$\Orbi$ and adopt the standard notation from the beginning of this section. 
In order to seek a contradiction, we assume that $\Ram{\BrS}=+\infty$. Then we 
find and fix $j,k\in A$ such that $\#\Trans{}{j}{k} = +\infty$. 

In what follows, we will take advantage of the euclidean structure and the 
standard ordering of~$\R$ to simplify the argumentation. To that end, we may 
suppose without loss of generality that the interval~$I_j$ is contained in~$\R$ 
and bounded (if necessary, we conjugate~$\Gamma$ and the set of branches by a 
suitable element~$g\in\PSL(2,\R)$). 

From~\eqref{BP:intervaldecomp} we obtain that the open, nonempty intervals 
\begin{equation}\label{eq:intervalsdisjoint}
 g\act I_k\quad (g\in\Trans{}{j}{k})
\end{equation}
are pairwise disjoint and all contained in~$I_j$. In particular, for all 
$g\in\Trans{}{j}{k}$ the two boundary points $g\act\eX_k$, $g\act\eY_k$ of~$g\act I_k$ 
are contained in~$\overline{I_j}$. These properties, together with the 
boundedness of~$I_j$, allow us to find a strictly increasing, convergent 
sequence in $\defset{ g\act\eX_k }{ g\in\Trans{}{j}{k} }$, say 
$(g_n\act\eX_k)_{n\in\N}$ with 
\[
 \lim_{n\to+\infty} g_n\act\eX_k \eqqcolon a\,.
\]
We note that $a\in\overline{I_j}$. Further, the disjointness and convexity of 
the intervals in~\eqref{eq:intervalsdisjoint} and the strict monotony of the 
sequence~$(g_n\act\eX_k)_{n\in\N}$ implies that, for each index~$n$, the 
point~$g_n\act\eY_k$ is contained in the interval~$(g_n\act\eX_k, 
g_{n+1}\act\eX_k)$. Therefore 
\[
 \lim_{n\to+\infty} g_n\act\eY_k = a\,.
\]
We fix a point on the geodesic segment~$\overline{\base{\Cs{k}}}$, say 
$z_0\in\H$. Then $(g_n\act z_0)_{n\in\N}$ is a sequence in~$\H$ with 
\[
 \lim_{n\to+\infty} g_n\act z_0 = a\,.
\]
Thus, $a\in\Lambda(\Gamma)$. Since $\Gamma$ contains no parabolic elements, we 
have $\Lambda(\Gamma) = \wh\R_{\st}$ (see \eqref{eq:def_st}) and hence 
\[
 a\in \wh\R_{\st} \cap \overline{I_j} = \Iset{j}\,.
\]
By~\eqref{BP:intervaldecomp} we find a (unique) pair~$(\ell,h)\in 
A\times\Trans{}{j}{\ell}$ such that 
\[
 a\in h\act \Iset{\ell} \subseteq h\act I_\ell\,.
\]
Since $h\act I_\ell$ is open and $a$ is the limit of the sequences from above, 
we find $n\in\N$ such that $g_n\act\eX_k, g_n\act\eY_k\in h\act I_\ell$. Thus,
\[
 g_n\act I_k \cap h\act I_\ell\not=\varnothing\,.
\]
This contradicts the disjointness of the unions in~\eqref{BP:intervaldecomp}. In 
turn, $\BrS$ is finitely ramified.
\end{proof}

We now consider the case where~$\Orbi$ is allowed to have cusps.
In Proposition~\ref{PROP:Attandram} we will see that the ramification with 
respect to a given set of branches~$\BrS$ depends on how thoroughly~$\BrS$ 
accounts for the cusps of~$\Orbi$.
To make this statement rigorous we require the following notion regarding the 
local structure of sets of branches in the vicinity of cusps.

\begin{defi}\label{DEF:attachedcusp}
Let $\wh c$ be a cusp of $\Orbi$, let $c\in\wh\R$ be a representative of~$\wh c$, and let
\[
\Att{\BrS}{c}\coloneqq\defset{(j,h)\in 
A\times\Gamma}{c \in h\act\geo\overline{\base{\Cs{j}}}}\,.
\]
\index[symbols]{Att@$\Att{\BrS}{c}$}%
\index[defs]{set of branches!attached to}%
\index[defs]{attached to}%
We say that $\wh c$ is \emph{attached to the set of branches 
}~$\BrS$ if the interval
\[
\IAtt{\BrS}{c}\coloneqq\overline{\bigcup_{(j,h)\in\Att{\BrS}{c}} 
h\act I_j}
\]
is a neighborhood of~$c$ in~$\wh\R$.
\end{defi}

Obviously, the definition of~$\nbAtt{\BrS}$ is independent of the choice of the 
representative~$c$ of~$\wh c$ and the notion of attachedness is well-defined. If 
$\wh c$ is attached to~$\BrS$, then $\#\Att{\BrS}{c}\geq 2$ because for 
each~$j\in A$, the set~$\geo\overline{\base{\Cs{j}}}$ consists of exactly the 
two boundary points of the interval~$I_j$.

The following result is a technical observation that comes in handy for the 
remaining proofs of this section.

\begin{lemma}\label{LEM:maxheight}
Suppose that the hyperbolic orbisurface~$\Orbi$ has cusps and that one of them, 
say~$\widehat{c}$, is represented by~$\infty$ and let $\lambda$ be the cusp 
width of~$\widehat{c}$. Suppose further that the set of 
branches~$\BrS=\defset{\Cs{j}}{j\in A}$ 
satisfies~$\Att{\BrS}{\infty}=\varnothing$. Then, for each~$j\in A$, the maximum
\[
\mathscr{h}_{\widehat{c}}(j)\coloneqq\max\defset{\Ima{z}}{
z\in\Gamma\act\overline{\base{\Cs{j}}}}
\]
exists and is bounded from above by $\lambda/2$.
\index[symbols]{hc@$\mathscr{h}_{\widehat{c}}(j)$}%
\end{lemma}

\begin{proof}
Let $j\in A$. For each~$g\in\Gamma$, the 
hypothesis~$\Att{\BrS}{\infty}=\varnothing$ implies that the 
set~$g\act\overline{\base{\Cs{j}}}$ is a euclidean semicircle. Hence, the 
maximum of the set
\[
\defset{\Ima{z}}{z\in g\act\overline{\base{\Cs{j}}}}
\]
exists and equals the radius of~$g\act\overline{\base{\Cs{j}}}$. Let 
$x_{g},y_{g}\in\R$, $x_{g}<y_{g}$, denote the two endpoints 
of~$g\act\overline{\base{\Cs{j}}}$, i.e., 
\[
(x_{g},y_{g})_{\H}=g\act\overline{\base{\Cs{j}}}\,.
\]
Let~$t_{\lambda}=\begin{bsmallmatrix}1&\lambda\\0&1\end{bsmallmatrix}$ be the 
generator of the cyclic group~$\Stab{\Gamma}{\infty}$ satisfying $\lambda>0$. 
Then the set~$\defset{\Ima z}{z\in\Gamma\act\overline{\base{\Cs{j}}}}$ is 
bounded from above by~$\lambda/2$ because otherwise we would find 
$h\in\Gamma$ such that $y_{h}-x_{h}>\lambda$ and then
\[
x_{h}<x_{h}+\lambda=t_{\lambda}\act x_{h}=x_{t_{\lambda}h}<y_{h}<t_{\lambda}\act 
y_{h}=y_{t_{\lambda}h}\,.
\]
Hence, the geodesic segments~$g\act\overline{\base{\Cs{j}}}$ 
and~$t_{\lambda}g\act\overline{\base{\Cs{j}}}$ would intersects without 
coinciding, contradicting~\eqref{BP:disjointunion}.

We consider now the strip~$\mc S\coloneqq(-\lambda,\lambda)+\i\R_{>0}$.
Let $\Lambda\subseteq\Gamma$ be the set of elements~$g\in\Gamma$ such that 
\[
 g\act \overline{\base{\Cs{j}}}\subseteq \mc S\,.
\]
For each~$g\in\Lambda$ set 
\[
\mathscr{h}_{g}\coloneqq\max\defset{\Ima{z}}{z\in 
g\act\overline{\base{\Cs{j}}}}\,.
\]
In order to seek a contradiction, we suppose that there exists a 
sequence~$(g_m)_{m\in\N}$ in~$\Lambda$ such that the sequence of 
maxima~$(\mathscr{h}_{g_m})_{m\in\N}$ is strictly increasing. By the previous 
considerations, $(\mathscr{h}_{g_m})_{m\in\N}$ is bounded from above 
by~$\lambda/2$ and hence convergent. Further, $\mathscr{h}_{g_m}>0$ for 
all~$m\in\N$. Since the 
union~$\bigcup_{g\in\Lambda}g\act\overline{\base{\Cs{j}}}$ is disjoint 
by~\eqref{BP:disjointunion} and contained in the strip~$\mc S$ of finite width, 
we find a subsequence~$(g_{m_k})_{k\in\N}$ of~$(g_m)_{m\in\N}$ such that for 
all~$k\in\N$,
\[
g_{m_k}\act\overline{\base{\Cs{j}}}\subseteq\conv_{\mathrm{E}}\left(g_{m_{k+1}}
\act\overline{\base{\Cs{j}}}\right)\,,
\]
where~$\conv_{\mathrm{E}}(M)$ denotes the convex hull of the set~$M$ in~$\C$ 
with respect to the euclidean metric. Since~$(\mathscr{h}_{g_{m_k}})_{k\in\N}$ 
converges, the family~$(g_{m_k}\act\overline{\base{\Cs{j}}})_{k\in\N}$ is not locally 
finite in~$\H$, which contradicts Proposition~\ref{PROP:branches_locfinite}. 
Thus, such a sequence~$(g_m)_{m\in\N}$ cannot exist. In turn, there exists an 
element~$g^*\in\Lambda$ such that
\[
\mathscr{h}_{g^*}=\max_{g\in\Lambda}\mathscr{h}_{g}\,.
\]
The same argument applies to all strips of the form 
\[
 \bigl( (n-1)\lambda, (n+1)\lambda\bigr) + i\R_{>0} = t_\lambda^n\act\mc S
\]
with~$n\in\Z$ and $\Lambda$ replaced by~$t_\lambda^n\act\Lambda$. With the 
$t_\lambda$-invariance of~$\mathscr{h}$, i.e., $\mathscr{h}_g = 
\mathscr{h}_{t_\lambda g}$ for all~$g\in\Gamma$, it follows that 
\[
\mathscr{h}_{g^*}=\max_{g\in t_\lambda^n\act\Lambda}\mathscr{h}_{g}
\]
for all~$n\in\Z$. Since $\Gamma = \bigcup_{n\in\Z} t_\lambda^n\act\Lambda$, the 
existence of~$\mathscr{h}_{\widehat{c}}(j)$ is shown.
\end{proof}

\begin{prop}\label{PROP:Attandram}
The set of branches~$\BrS$ is finitely ramified if and only if all cusps 
of~$\Orbi$ are attached to~$\BrS$.
\end{prop}

\begin{proof}
We suppose first that $\BrS$ is finitely ramified. In order to seek a 
contradiction we assume that there exists a cusp of~$\Orbi$ that is not attached 
to~$\BrS$. Without loss of generality we may suppose that this cusp is 
represented by~$\infty$. Then $\infty\in\Lambda(\Gamma)$. Moreover, $\infty$ is 
approximated from both sides by suitable sequences in~$\wh\R_{\st}$, as can be 
seen by taking any element~$w\in\wh\R_{\st}$ (which is necessarily 
not~$\infty$), any element~$t$ in~$\Stab{\Gamma}{\infty}$ and considering the 
two sequences~$(t^n\act w)_{n\in\N}$ and~$(t^{-n}\act w)_{n\in\N}$, which both 
converge to~$\infty$ but from different sides.

We now claim that there exists a pair~$(j,g)\in A\times\Gamma$ such that 
\begin{equation}\label{eq:inftyforward}
\infty\in g\act I_j\,. 
\end{equation}
In order to see this, we pick a periodic geodesic~$\wh\gamma$ on~$\Orbi$. 
By~\eqref{BP:closedgeodesicsXtoH} we find a geodesic~$\gamma$ on~$\H$ 
representing~$\wh\gamma$ and a pair~$(j,p)\in A\times\Gamma$ such that $\gamma$ 
intersects~$p\act\Cs{j,\st}$. Then $\gamma(+\infty)\in p\act\Iset{j}$ and 
$\gamma(-\infty)\in p\act\Jset{j}$. Further we find a hyperbolic 
element~$h\in\Gamma$ such that $\fixp{\pm}{h} = \gamma(\pm\infty)$. In other 
words, the geodesic~$\gamma$ represents the axis~$\alpha(h)$ of~$h$. Under the 
iterated action of~$h^{-1}$, the two endpoints~$p\act\eX_j$ and~$p\act\eY_j$ 
of~$p\act J_j$ tend to~$\fixp{-}{h} = \gamma(-\infty)$. More precisely, 
\[
 h^{-(n+1)}p\act J_j \varsubsetneq h^{-n}p\act J_j\qquad\text{for all~$n\in\N$}\,,
\]
and 
\[
 \bigcap_{n\in\N} h^{-n}p\act J_j = \{ \gamma(-\infty) \}
\]
(see Lemma~\ref{LEM:hypfixedconv}). Since $\wh\gamma$ is periodic, 
$\gamma(-\infty)$ is a hyperbolic fixed point and hence cannot coincide with the 
cuspidal point~$\infty$. Therefore, for some sufficiently large~$N\in\N$, the 
interval~$h^{-N}p\act J_j = (h^{-N}p\act\eY_j, h^{-N}p\act\eX_j)_c$ is the real 
interval $(h^{-N}p\act\eY_j, h^{-N}p\act\eX_j)$. Thus, $h^{-N}p\act I_j = 
(h^{-N}p\act\eX_j, h^{-N}p\act\eY_j)_c$ contains~$\infty$. This establishes the 
existence of a pair~$(j,g)\in A\times\Gamma$ satisfying~\eqref{eq:inftyforward}. 

We now fix such a pair~$(j,g)$. Without loss of generality (using 
Lemma~\ref{LEM:ramtrans}), we may suppose that $g=\id$. From the density 
of~$\wh\R_{\st}$ at~$\infty$ we obtain that $\infty\in \overline{\Iset{j}}$. 
Since $\BrS$ is finitely ramified, 
(\ref{BP:intervaldecomp}\ref{BP:intervaldecompGdecomp}) implies that 
\[
 \overline{\Iset{j}} = \bigcup_{k\in A} \bigcup_{h\in \Trans{}{j}{k}} h\act 
\overline{\Iset{k}}\,.
\]
Thus, we find $k\in A$ and $h\in\Trans{}{j}{k}$ such that $\infty\in 
h\act\overline{I_k}$. 

We suppose first that~$\infty$ is an endpoint, hence a boundary point, of the 
interval~$h\act I_k$. Without loss of generality, we may suppose that 
$\infty=h\act\eY_k$. Then 
\[
 h\act \overline{I_k} \cap I_j \subseteq (\eX_j, \infty]_c\,.
\]
Since $\infty$ is contained in the open interval~$I_j$ and is approximated 
within~$\wh\R_{\st}$ from both sides, we see that 
\[
 (\infty, \eY_j)_c \cap \Iset{j} \not=\varnothing
\]
(the part of~$\Iset{j}$ at the other side of~$\infty$) and we find $\ell\in A$ 
and $p\in\Trans{}{j}{\ell}$ with $(\ell,p)\not=(k,h)$ such that $\infty\in 
p\act\overline{I_\ell}$. Since $p\act I_\ell$ and $h\act I_k$ are disjoint 
by~(\ref{BP:intervaldecomp}\ref{BP:intervaldecompGdecomp}), we obtain that 
$\infty$ is a boundary point of~$p\act I_\ell$ and further that 
$(k,h),(\ell,p)\in \Att{\BrS}{\infty}$ and $p\act\overline{I_\ell}\cup 
h\act\overline{I_k}$ is a neighborhood of~$\infty$ in~$\wh\R$. This contradicts 
our hypothesis that $\wh\infty$ is not attached to~$\BrS$. 

In turn, since $h\act I_k$ is open, $\infty\in h\act I_k$ 
(and hence $(k,h)$ is uniquely determined). The combination 
of~(\ref{BP:intervaldecomp}\ref{BP:intervaldecompGgeod}) 
and~\eqref{BP:disjointunion} implies that 
\[
 h\act I_k \varsubsetneq I_j\,.
\]
Inductively we obtain a sequence $((k_n,g_n))_{n\in\N}$ in $A\times\Gamma$ such 
that for each $n\in\N$, 
\[
 \infty \in g_n\act I_{k_n}
\]
and 
\[
 g_{n+1}\act I_{k_{n+1}} \varsubsetneq g_n\act I_{k_n}\,.
\]
Since the family of the geodesic segments $g_n\act\overline{\base{\Cs{k_n}}}$, 
$n\in\N$, is locally finite by Proposition~\ref{PROP:branches_locfinite}, the 
intervals $g_n\act I_{k_n}$, $n\in\N$, zero in on~$\infty$. But this implies 
that the family of maxima
\[
\mathscr{h}_{k_n, g_n} \coloneqq \max\defset{\Ima z}{ z \in 
g_n\act\overline{\base{\Cs{k_n}}}}\qquad (n\in\N)
\]
is unbounded, which contradicts Lemma~\ref{LEM:maxheight}. Thus, the assumption 
that $\wh\infty$ is not attached to~$\BrS$ fails. This completes the proof that 
$\BrS$ being finitely ramified implies that all cusps of~$\Orbi$ are attached 
to~$\BrS$. 

In the case that~$\Orbi$ does not have cusps, the converse implication (i.e., if 
all cusps are attached to~$\BrS$, then $\BrS$ is finitely ramified) is already 
established in Lemma~\ref{LEM:nocuspnoram}. For its proof in the general case we 
suppose that $\Orbi$ has cusps and that every cusp of~$\Orbi$ is attached to 
$\BrS=\{\Cs{1},\dots,\Cs{n}\}$. We aim to show that~$\BrS$ is finitely ramified. 
However, in order to seek a contradiction we assume that $\BrS$ is infinitely 
ramified. As in the proof of Lemma~\ref{LEM:nocuspnoram} we find and fix $j,k\in 
A$ such that $\#\Trans{}{j}{k} = +\infty$, we may suppose that the 
interval~$I_j$ is contained in~$\R$ and bounded, and we find a 
sequence~$(g_n)_{n\in\N}$ in~$\Trans{}{j}{k}$ such that the endpoint sequences 
$(g_n\act\eX_k)_{n\in\N}$ and $(g_n\act\eY_k)_{n\in\N}$ are contained in~$I_j$ and converge to 
an element 
\[
 a \in \Lambda(\Gamma) \cap \overline{I_j}\,.
\]
From the proof of Lemma~\ref{LEM:nocuspnoram} we obtain further that 
$a\notin\wh\R_{\st}$. Thus, it remains to consider the case that $a$ is a 
parabolic fixed point (cf.~\eqref{eq:def_st}). By hypothesis, the cusp~$\wh a$ 
of~$\Orbi$ is attached to~$\BrS$.
Preparatory for the following constructions, we now pick 
$(i,h)\in\Att{\BrS}{a}$ such that $I_j\subseteq h\act I_i$ (if such a pair 
exists, otherwise we omit this step) and show that we also find 
$(\ell,g)\in\Att{\BrS}{a}$ such that $g\act I_{\ell}\subseteq I_j$. To that end 
we note that since $a\in \overline{I_j}$ and $a$ is an endpoint of the 
interval~$h\act I_i$, the point~$a$ is also an endpoint of~$I_j$. Let 
$p\in\Gamma$ be a parabolic element that fixes~$a$. Then the pairs~$(j,p)$ 
and~$(j,p^{-1})$ belong to~$\Att{\BrS}{a}$ and either 
\[
 p\act I_j \subseteq I_j \qquad\text{or}\qquad p^{-1}\act I_j\subseteq I_j\,,
\]
which shows the existence of such a pair~$(\ell,g)$.

From the attachment property of the cusp~$\wh a$ it follows that we find two 
distinct pairs~$(\ell_1,h_1), (\ell_2,h_2)\in A\times\Gamma$ such that $a$ is a 
joint endpoint of the intervals~$h_1\act I_{\ell_1}$ and~$h_2\act I_{\ell_2}$, 
and 
\[
 h_1\act \overline{I_{\ell_1}} \cup h_2\act\overline{I_{\ell_2}}
\]
is a neighborhood of~$a$ in~$\R$. By the previous argument we may further 
suppose that at least one of these intervals intersects~$I_j$ but does not 
cover~$I_j$. Without loss of generality, we suppose it to be $h_1\act 
I_{\ell_1}$. 

We suppose first that $h_1\act I_{\ell_1}\subseteq I_j$ and fix $n\in\N$ such 
that $g_n\act I_k\varsubsetneq h_1\act I_{\ell_1}$ (the existence of~$n$ follows 
directly from the properties of the two sequences~$(g_n\act\eX_k)_n$ 
and~$(g_n\act\eY_k)_n$). By the density of~$E(\Orbi)$ (see 
Proposition~\ref{PROP:EXliesdense}) we find 
\[
 (x,y) \in E(\Orbi) \cap \bigl( g_n\act I_k \times J_j\bigr)\,.
\]
Then~\eqref{BP:allvectors} implies that the geodesic segment~$(x,y)_\H$ 
intersects $\base{\Cs{j}}$ in some point, say~$z$, and intersects 
$g_n\act\base{\Cs{k}}$ in some point, say~$w$, and 
intersects~$h_1\act\base{\Cs{\ell_1}}$ in some point, say~$u$, with $u\in 
(z,w)_\H$. This contradicts~\eqref{BP:intervaldecomp}. 

In turn, $h_1\act I_{\ell_1}\varsubsetneq I_j$. Then one endpoint of~$h_1\act 
I_{\ell_1}$, namely~$a$, is contained in~$I_j$, while the other endpoint is not. 
Convexity implies that $\overline{\base{\Cs{j}}} \cap 
h_1\act\overline{\base{\Cs{\ell_1}}} \not=\varnothing$ but 
$\overline{\base{\Cs{j}}} \ne h_1\act \overline{\base{\Cs{\ell_1}}}$. But this 
contradicts~\eqref{BP:disjointunion}. It follows that~$\BrS$ is finitely 
ramified.
\end{proof}

Proposition~\ref{PROP:Attandram} already indicated how we could turn an 
infinitely ramified set of branches into a finitely ramified one: if we find a 
way to augment the initial (infinitely ramified) set of branches with further 
branches such that all cusps of~$\Orbi$ are attached to the enlarged family of 
branches, then the ramification becomes finite. By comparing 
Figure~\ref{FIG:branchram:Gamma2} to Figure~\ref{FIG:branchram2} in 
Example~\ref{EX:branchramification} above, one sees that this approach has been 
carried out successfully for the group~$\Gamma_2$. 
Proposition~\ref{PROP:allwithfiniteram} below states that this can always be 
done. We emphasize that its proof is constructive and provides an algorithm for 
the enlargement procedure.

For the proof of Proposition~\ref{PROP:allwithfiniteram} we will take advantage 
of Ford fundamental domains and some of their specific properties, which we 
briefly survey now. For more details we refer to, e.g., \cite{Pohl_diss, 
Pohl_isofunddom}. We start with the concept of isometric spheres. 

Let~$g=\begin{bsmallmatrix}a&b\\c&d\end{bsmallmatrix}\in\Gamma\setminus\Stab{
\Gamma}{\infty}$. The \emph{isometric sphere} of~$g$ is the set
\index[defs]{isometric sphere}%
\index[symbols]{ig@$\iso{g}$}%
\[
\iso{g}\coloneqq\defset{z\in\H}{\abs{g'(z)}=1}=\defset{z\in\H}{\abs{cz+d}=1}\,.
\]
Thus, the isometric sphere~$\iso{g}$ is the complete geodesic segment connecting 
the two real points~$(-d+1)/c$ and $(-d-1)/c$. The \emph{radius} of~$\iso{g}$ is 
$1/|c|$.
\index[defs]{radius}%
\index[defs]{isometric sphere!radius}%
The \emph{interior} of~$\iso{g}$ is the set
\index[defs]{interior}%
\index[defs]{isometric sphere!interior}%
\index[symbols]{int@$\intiso{g}$}%
\begin{align*}
\intiso{g}&\coloneqq\defset{z\in\H}{\abs{g'(z)}>1}=\defset{z\in\H}{\abs{cz+d}<1}
\,,
\intertext{and its \emph{exterior} is}
\extiso{g}&\coloneqq\defset{z\in\H}{\abs{g'(z)}<1}=\defset{z\in\H}{\abs{cz+d}>1}
\,.
\end{align*}
\index[defs]{exterior}%
\index[defs]{isometric sphere!exterior}%
\index[symbols]{ext@$\extiso{g}$}%
We set
\begin{align}\label{EQNDEF:setK}
\mathcal{K}\coloneqq\bigcap_{g\in\Gamma\setminus\Stab{\Gamma}{\infty}}\extiso{g}
\,.
\end{align}
\index[symbols]{K@$\mathcal K$}%
This is a geodesically convex set whose boundary decomposes into geodesic 
segments contained in isometric spheres. Suppose that $\wh\infty$ is a cusp 
of~$\Orbi$ and let $\lambda$ be the cusp width of~$\wh\infty$. Then 
$\Stab{\Gamma}{\infty}$ is generated 
by~$t_{\lambda}=\begin{bsmallmatrix}1&\lambda\\0&1\end{bsmallmatrix}$. The 
set~$\mc K$ is invariant under~$t_{\lambda}$, or in other words, all geometric 
structures of~$\mc K$ are $\lambda$-periodic. Further, for sufficiently 
large~$y_0>0$, the set 
\begin{equation}\label{eq:horoballinfty}
 \defset{ z\in\H }{ \Ima z > y_0 }
\end{equation}
is contained in~$\mc K$. (In other words, all horoballs centered at~$\infty$ 
with sufficiently small radii are contained in~$\mc K$.) These two properties 
are the most crucial features for deducing fundamental domains from~$\mc K$ in 
the following way: For every~$x\in\R$ the strip
\begin{equation}\label{EQNDEF:fundxinfty}
\fund_{x,\infty}\coloneqq(x,x+\lambda)+\i\R_{>0}
\end{equation}
\index[symbols]{Fx@$\fund_{x,\infty}$}%
is a fundamental domain for the action of~$\Stab{\Gamma}{\infty}$ on~$\H$.
Then
\begin{align}\label{EQNDEF:fundx}
\fund_x\coloneqq\fund_{x,\infty}\cap\mathcal{K}
\end{align}
\index[symbols]{Fx@$\fund_x$}%
is a geometrically finite convex fundamental polyhedron for~$\Gamma$, a 
\emph{Ford fundamental domain}.
\index[defs]{Ford fundamental domain}%
\index[defs]{fundamental domain!Ford}%
We emphasize that for each vertical geodesic 
segment~$(b,\infty)_{\H}$ (including complete segments), there exists~$x\in\R$ 
such that~$(b,\infty)_{\H}$ is contained in~$\fund_{x,\infty}$. Further, the 
part of~$(b,\infty)_\H$ in $\fund_x$ is again a vertical geodesic segment.

\begin{lemma}\label{LEM:maxintersect}
Suppose that $\wh\infty$ is a cusp of~$\Orbi$ and that the set of 
branches~$\BrS$ satisfies~$\Att{\BrS}{\infty}=\varnothing$. Let $(j,g)\in 
A\times\Gamma$ be such that 
\[
 \max\defset{ \Ima z }{ z \in \Gamma\act\overline{\base{\BrU}} } = \max\defset{ 
\Ima z }{ z \in g\act\overline{\base{\Cs{j}}} }
\]
(the existence of the pair~$(j,g)$ is guaranteed by Lemma~\ref{LEM:maxheight}). 
Pick $x\in\R$ such that $(g\act\eX_j,\infty) \subseteq \fund_{x,\infty}$. Then 
it follows that 
\begin{enumerate}[label=$\mathrm{(\roman*)}$, ref=$\mathrm{\roman*}$]
\item\label{MI:z0contained} $g\act\overline{\base{\Cs{j}}} \cap \overline{\mc K} 
\not=\varnothing$. More precisely, the point of maximal height 
of~$g\act\overline{\base{\Cs{j}}}$ is contained in~$\overline{\mc K}$. 
\item\label{MI:minigap} there exists~$\eps>0$ such that 
\[
  (g\act\eX_j -\eps, g\act\eY_j -\eps) \subseteq \Rea(\fund_{x,\infty})
\]
or
\[
  (g\act\eY_j +\eps, g\act\eX_j +\eps) \subseteq \Rea(\fund_{x,\infty}) 
\]
\item\label{MI:nicefunddom} there exists~$x\in\R$ such that 
$(g\act\eX_j,\infty)\cap\fund_x\not=\varnothing$ and the point of maximal height 
of~$g\act\overline{\base{\Cs{j}}}$ is contained in~$\overline{\fund_x}$.
\end{enumerate}
\end{lemma}

\begin{proof}
Let $z_0$ denote the (unique) point of maximal height 
of~$g\act\overline{\base{\Cs{j}}}$. By the choice of~$(j,g)$, we have
\begin{equation}\label{eq:z0globalmax}
 \Ima z_0 = \max\defset{ \Ima z }{ z\in \Gamma\act\overline{\base{\BrU}} }\,.
\end{equation}
In order to show~\eqref{MI:z0contained}, we assume, with the goal to find a 
contradiction, that $g\act\overline{\base{\Cs{j}}}\cap\overline{\mc K} = 
\varnothing$. Then 
\[
 g\act\overline{\base{\Cs{j}}} \subseteq 
\bigcup_{h\in\Gamma\setminus\Stab{\Gamma}{\infty}} \intiso{h}\,.
\]
We fix $p = \textbmat{a}{b}{c}{d} \in \Gamma\setminus\Stab{\Gamma}{\infty}$ such 
that $z_0\in\intiso{p}$. Thus, $|cz_0+d|<1$, and it follows that 
\[
 \Ima (p\act z_0) = \frac{\Ima z_0}{|cz_0+d|^2} > \Ima z_0\,,
\]
which contradicts the choice of~$(j,g)$ and~$z_0$ 
(see~\eqref{eq:z0globalmax}). In turn, 
\[
 g\act\overline{\base{\Cs{j}}}\cap\overline{\mc K} \not=\varnothing\,.
\]
For~\eqref{MI:minigap}, let $\lambda>0$ be the cusp width of~$\wh\infty$. 
Lemma~\ref{LEM:maxheight} shows that the height 
of~$g\act\overline{\base{\Cs{j}}}$ is bounded above by~$\lambda/2$. Thus, 
\[
 |g\act\eX_j - g\act\eY_j|\leq \lambda\,.
\]
Since~$\fund_{x,\infty} = (x,x+\lambda) + i\R_{>0}$ and 
$(g\act\eX_j,\infty)\in\fund_{x,\infty}$, the statement of~\eqref{MI:minigap} 
follows immediately.
Statement~\eqref{MI:nicefunddom} is an immediate consequence 
of~\eqref{MI:z0contained} and~\eqref{MI:minigap}.
\end{proof}

With these preparations we can now provide and prove the enlargement procedure, 
in the proof of the following proposition.

\begin{prop}\label{PROP:allwithfiniteram}
Suppose that $\BrS=\{\Cs 1, \ldots, \Cs N\}$ is infinitely ramified and let~$m$ 
be the number of cusps of~$\Orbi$ not attached to~$\BrS$. Then there exists a 
finitely ramified set of branches for~$\GeoFlow$ of the form
\[
\BrS'\coloneqq \{\Cs{1},\dots,\Cs{N},\Cs{N+1},\dots,\Cs{N+k}\}
\]
for some $k\in\N$, $k\leq 2m$.
\end{prop}

\begin{proof}
By Proposition~\ref{PROP:Attandram}, the hyperbolic orbisurface~$\Orbi$ has at least 
one cusp that is not attached to~$\BrS$, say~$\wh c$. We will enlarge~$\BrS$ to 
a set of branches~$\BrS'$ to which $\wh c$ is attached and which contains at 
most two branches more than~$\BrS$. Since $\Orbi$ has only finitely many cusps 
as a geometrically finite orbifold, a finite induction then yields the 
statement, including the counting bound.

Without loss of generality, we may suppose that $\wh c$ is represented 
by~$\infty$. (If not, then we pick any representative of~$\wh c$, say~$c$, and 
any $q\in\PSL_2(\R)$ such that $q\act c = \infty$, consider $\defset{ q\act 
\Cs{j}}{j\in A}$ instead of~$\BrS$, $q\Gamma q^{-1}$ instead of~$\Gamma$, 
perform the enlargement as described in what follows and finally undo the 
transformation by applying~$q^{-1}$.) We distinguish the following two cases: 
\begin{align}
 \Att{\BrS}{\infty}&\not=\varnothing  \label{ATT:caseA} \tag{A}
 \intertext{and}
 \Att{\BrS}{\infty}& = \varnothing\,. \label{ATT:caseB} \tag{B}
\end{align}

In Case~\eqref{ATT:caseA} we pick $(j,g)\in\Att{\BrS}{\infty}$ and let 
\[
 \Cs{n+1} \coloneqq\defset{ v\in \UTB\H }{ \base{v} \in 
g\act\overline{\base{\Cs{j}}},\ \gamma_v(+\infty) \in g\act J_j}
\]
be the set of unit tangent vectors that are based at 
$g\act\overline{\base{\Cs{j}}}$ but point into the opposite direction as 
$g\act\Cs{j}$. We emphasize that we allow the whole 
set~$g\act\overline{\base{\Cs{j}}}$ as base points and do not restrict to those 
that lie on a geodesic connecting points in~$\wh\R_{\st}$. We set 
\[
 \BrS'\coloneqq \BrS \cup \{ \Cs{n+1} \}\qquad\text{and}\qquad A'\coloneqq A 
\cup \{n+1\}\,.
\]
In Case~\eqref{ATT:caseB}, Lemma~\ref{LEM:maxheight} shows the existence 
of~$(j,g)\in A\times\Gamma$ such that 
\[
 \max\defset{ \Ima z }{ z\in g\act\overline{\base{\Cs{j}}} } = \max_{k\in A} 
\mathscr{h}_{\wh c}(k)\,.
\]
We let 
\begin{align*}
 \Cs{n+1} & \coloneqq \defset{ v\in \UTB\H }{ \base{v}\in (g\act\eX_j, 
\infty)_\H,\ \gamma_v(+\infty)\in (g\act\eX_j,+\infty) }
 \intertext{and}
 \Cs{n+2} & \coloneqq \defset{ v\in \UTB\H }{ \base{v}\in (g\act\eX_j, 
\infty)_\H,\ \gamma_v(+\infty)\in (-\infty, g\act\eX_j) }
\end{align*}
be the sets of unit tangent vectors based on the geodesic 
segment~$(g\act\eX_j,\infty)_\H$ and pointing into the one or the other of the 
associated half-spaces. We set
\begin{align*}
 \BrS' & \coloneqq \BrS \cup \{ \Cs{n+1}, \Cs{n+2} \}
 \intertext{and}
 A' & \coloneqq A \cup \{n+1, n+2\}\,.
\end{align*}
In both cases we set
\[
 \BrU'\coloneqq \bigcup_{j\in A'} \Cs{j}\,.
\]
To show \eqref{BP:closedgeodesicsXtoH} for~$\BrS'$, we note that $\BrS'$ is a 
superset of~$\BrS$. Since $\BrS$ satisfies~\eqref{BP:closedgeodesicsXtoH}, 
$\BrS'$ does so as well.
This yields~$\eqref{BP:coverlimitset}$ by virtue of Proposition~\ref{PROP:oldB1}.

The validity of~\eqref{BP:completegeodesics},~\eqref{BP:pointintohalfspaces}, and~\eqref{BP:allvectors} for~$\BrS'$ 
is obvious from the construction of the additional branches. For each $j\in A'$ 
we define the sets $I_{j}$, $J_{j}$, $\Plussp{j}$ and~$\Minussp{j}$ as 
in~\eqref{BP:pointintohalfspaces}. For $j\in A$ these sets obviously coincide 
with those related to~$\BrS$. In Case~\eqref{ATT:caseA} we have 
\[
 I_{n+1} = g\act J_j\,,\quad J_{n+1} = g\act I_j\quad\text{and}\quad \PMsp{n+1} 
= g\act \MPsp{j}\,.
\]
In Case~\eqref{ATT:caseB} we have 
\[
 I_{n+1} = J_{n+2} = (g\act\eX_j,+\infty)\,,\quad J_{n+1} = I_{n+2} = 
(-\infty,g\act\eX_j)
\]
and 
\begin{align*}
 \Plussp{n+1} & = \Minussp{n+2} = \defset{ z\in\H }{ \Rea z > g\act\eX_j }
 \\
 \Minussp{n+1} & = \Plussp{n+2} = \defset{ z\in\H }{ \Rea z < g\act\eX_j }\,.
\end{align*}

We now prove \eqref{BP:closedgeodesicsHtoX} for~$\BrS'$. For $j\in A'\setminus 
A$ we pick $(x,y) \in \Iset{j}\times\Jset{j}$ and fix $\eps>0$ such that 
$B_\eps(x)\subseteq I_j$ and $B_\eps(y)\subseteq J_j$. (Note that 
$\infty\notin\wh\R_{\st}$.) Then $(x,y)\in 
\Lambda(\Gamma)\times\Lambda(\Gamma)$. By Proposition~\ref{PROP:EXliesdense} we 
find a geodesic~$\gamma$ on~$\H$ that represents a periodic geodesic on~$\Orbi$ 
and satisfies
\[
 \gamma(+\infty) \in B_\eps(x) \qquad\text{and}\qquad \gamma(-\infty) \in 
B_\eps(y)\,.
\]
The geodesic intersects~$\Cs{j}$ as can be seen directly from the definition of 
this set. This shows~\eqref{BP:closedgeodesicsHtoX}.

Property~\eqref{BP:disjointunion} for~$\BrS'$ in Case~\eqref{ATT:caseA} follows 
immediately from~$\overline{\base{\Cs{n+1}}} = g\act\overline{\base{\Cs{j}}}$ 
and $\PMsp{n+1} = g\act\MPsp{j}$. To establish \eqref{BP:disjointunion} 
for~$\BrS'$ in Case~\eqref{ATT:caseB} we let $a,b\in A'$, $h\in\Gamma$ be such 
that 
\[
 \overline{\base{\Cs{a}}} \cap h\act\overline{\base{\Cs{b}}} \not= 
\varnothing\,.
\]
We consider first the case that $a=n+1$, $b=n+2$ and $\overline{\base{\Cs{a}}} = 
h\act\overline{\base{\Cs{b}}}$. Recall that 
\[
 \overline{\base{\Cs{n+1}}} = \overline{\base{\Cs{n+2}}} = 
(g\act\eX_j,\infty)_\H\,.
\]
From $\Att{\BrS}{\infty} = \varnothing$ it follows that $\wh{g\act\eX_j} \not= 
\wh\infty$ (otherwise $(j,g)\in\Att{\BrS}{\infty}$). Thus, $h$ fixes both 
endpoints of~$\overline{\base{\Cs{n+1}}}$. Since $\infty$ is cuspidal, $h=\id$. 
Further, by construction, $\PMsp{n+1} = \MPsp{n+2}$. Hence, 
\eqref{BP:disjointunion} is satisfied in this case.

The previous case in combination with the fact that \eqref{BP:disjointunion} is 
satisfied for~$\BrS$ allows us to restrict all further considerations to the 
case that $a=n+1$ and $b\in A\cup\{n+1\}$. We show first that necessarily 
$b=n+1$. To that end, in order to seek a contradiction, we assume that $b\in A$. 
Then $\overline{\base{\Cs{a}}} \not=h\act\overline{\base{\Cs{b}}}$ as 
$\Att{\BrS}{\infty} = \varnothing$, and hence the geodesic 
segments~$\overline{\base{\Cs{a}}}$ and $h\act\overline{\base{\Cs{b}}}$ 
intersect transversally. We recall the tuple $(j,g)\in A\times\Gamma$ from the 
construction of~$\BrS'$. Since the geodesic 
segment~$g\act\overline{\base{\Cs{j}}}$ has maximal radius among all 
semi-circles in $\Gamma\act\overline{\base{\Cs{k}}}$ with $k\in A$, and 
$g\act\eX_j$ is a joint endpoint of~$g\act\overline{\base{\Cs{j}}}$ 
and~$\overline{\base{\Cs{n+1}}}$, the geodesic 
segment~$h\act\overline{\base{\Cs{b}}}$ 
intersects~$g\act\overline{\base{\Cs{j}}}$. Since $\BrS$ 
satisfies~\eqref{BP:disjointunion}, it follows that 
$h\act\overline{\base{\Cs{b}}} = g\act\overline{\base{\Cs{j}}}$. But then 
$h\act\overline{\base{\Cs{b}}}$ does not intersect~$\overline{\base{\Cs{n+1}}}$. 
In turn, this case is impossible. 

It remains to consider the case that $a=b=n+1$ and 
\[
 \overline{\base{\Cs{n+1}}} \cap h\act\overline{\base{\Cs{n+1}}} 
\not=\varnothing
\]
but 
\[
 \overline{\base{\Cs{n+1}}} \not= h\act\overline{\base{\Cs{n+1}}}\,.
\]
We let $B\coloneqq \overline{\base{\Cs{n+1}}}$ and suppose without loss of 
generality that the endpoints of~$g\act\overline{\base{\Cs{j}}}$ satisfy 
\[
 g\act\eX_j < g\act\eY_j\,.
\]
(If $g\act\eX_j>g\act\eY_j$, the argumentation in what follows applies with some 
changes of orderings.) Since $h\act B$ is a non-vertical complete geodesic 
segment such that the real interval enclosed between its two endpoints contains 
$g\act\eX_j$, which is the common endpoint of~$B$ 
and~$g\act\overline{\base{\Cs{j}}}$, the argumentation in the previous paragraph 
yields that 
\begin{equation}\label{eq:inconvhull}
 g\act\overline{\base{\Cs{j}}} \subseteq \conv_{\mathrm{E}}(h\act B) \setminus 
\partial\conv_{\mathrm{E}}(h\act B) \,, 
\end{equation}
where, as in the proof of Lemma~\ref{LEM:maxheight}, $\conv_{\mathrm{E}}(M)$ 
denotes the convex hull
\index[defs]{convex hull}%
\index[symbols]{conv@$\conv_{\mathrm{E}}$}%
of the set~$M$ in~$\C$ with respect to the euclidean 
metric. 

We recall the definition of the sets~$\mc K, \fund_{x,\infty}, \fund_x$ 
from~\eqref{EQNDEF:setK}, \eqref{EQNDEF:fundxinfty} and~\eqref{EQNDEF:fundx} and 
fix $x\in\R$ such that $B\subseteq \fund_{x,\infty}$ and such that the point~$z_0$ of 
maximal height of~$g\act\overline{\base{\Cs{j}}}$  is contained 
in~$\overline{\fund_x}$ and hence in~$\overline{\mc K}$. The choice of~$x$ is 
possible by Lemma~\ref{LEM:maxintersect}. With~\eqref{eq:inconvhull}, we obtain 
that 
\begin{equation}\label{eq:z0inboth}
z_0\in \conv_{\mathrm{E}}(h\act B) \cap \fund_x\,.
\end{equation}
We consider the strip-shaped set
\[
 S \coloneqq \defset{ w + \i t }{ w \in g\act\overline{\base{\Cs{j}}} \cap 
\overline{\mc K}\,,\ \Rea w \in \Rea(\fund_{x,\infty})\,,\ t>0 }\,.
\]
The set~$S$ is convex due to the convexity of~$\mc K$ and~$\fund_{x,\infty}$ and 
the boundary structure of~$\mc K$. Further, 
\begin{equation}\label{eq:SinF}
S\subseteq\fund_x 
\end{equation}
and $z_0+\i \R_{>0}\subseteq S$. The latter implies that $h\act B\cap S 
\not=\varnothing$. We now define the two domains
\begin{align*}
L&\coloneqq\left(\H\setminus\conv_{\mathrm 
E}(g\act\overline{\base{\Cs{j}}})\right)\cap\defset{z\in\H}{\Rea{z}<\inf_{w\in 
S}\Rea{w}}
\intertext{and}
R&\coloneqq\left(\H\setminus\conv_{\mathrm 
E}(g\act\overline{\base{\Cs{j}}})\right)\cap\defset{z\in\H}{\Rea{z}>\sup_{w\in 
S}\Rea{w}}\,.
\end{align*}
Because the sets~$L$,~$R$ and~$\overline{L}\cup\overline{R}\cup\overline{S}$ are 
convex, and $h\act B$ intersects~$S$, for the pair~$(h\act\infty,hg\act\eX_j)$ 
of the two endpoints of~$h\act B$ we obtain
\begin{align}\label{EQN:endpointsinLR}
\left(h\act\infty,hg\act\eX_j\right)\in\left(\geo L\times\geo 
R\right)\cup\left(\geo R\times \geo L\right)\,.
\end{align}
Since $B\cap\fund_x$ is of the form~$(b,\infty)_{\H}$ for some point~$b\in\H$, 
and $\fund_x$ contains all subsets of the form $\defset{ z\in\H }{ \Rea z\in 
(b-\eps, b+\eps),\ \Ima z > y_0}$ for sufficiently small~$\eps>0$ and 
sufficiently large~$y_0>0$, \eqref{EQN:endpointsinLR} and the convexity of the 
sets~$L,R$ imply that we have
\begin{align}
h\act\fund_x\cap L\ne\varnothing \quad & \text{if $h\act\infty\in \geo L$}
\label{eq:nonemptL}
\intertext{and}
h\act\fund_x\cap R\ne\varnothing \quad & \text{if $h\act\infty\in \geo R$}\,.
\label{eq:nonemptR}
\end{align}
We now aim to show that $h\act\fund_x$ indeed intersects $L$ and~$R$. To that 
end we recall that we suppose that $g\act\eX_j<g\act\eY_j$. 
If~$h\act\infty\in\geo L$, then~$hg\act\eX_j\in\geo R$ and
\[
\Rea{\left(hg\act\overline{\base{\Cs{j}}}\right)}\subseteq\left[hg\act\eX_j,
+\infty\right)\,.
\]
It follows that $h\act z_0\in h\act\fund_x \cap R$. If~$h\act\infty\in\geo R$, 
then~$hg\act\eX_j\in\geo L$ and, taking advantage of~\eqref{BP:disjointunion} 
for~$\BrS$, we find
\[
\Rea{\left(hg\act\overline{\base{\Cs{j}}}\right)}\subseteq\left[hg\act\eX_j,
g\act\eX_j\right]
\]
and hence $h\act z_0\in h\act\fund_x \cap L$. Combining with~\eqref{eq:nonemptL} 
and~\eqref{eq:nonemptR}, respectively, we find 
\[
 h\act\fund_x\cap L\not=\varnothing\qquad\text{and}\qquad h\act\fund_x\cap 
R\not=\varnothing\,.
\]
From the convexity of~$h\act\fund_x$ and the definitions of the sets~$L$,~$R$, 
and~$S$ it now follows that
\[
h\act\fund_x\cap S\ne\varnothing\,.
\]
In combination with~\eqref{eq:SinF} this yields a contradiction. In turn, 
\eqref{BP:disjointunion} is valid for~$\BrS'$.

In order to establish~\eqref{BP:intervaldecomp} for~$\BrS'$, we first show that 
the next and previous intersection times exist for all elements 
in~$\BrU'_{\st}$. To that end let $j\in A'$ and $\nu\in\Cs{j,\st}$. Using that 
$\Gamma\act\,\BrS'$ is locally finite by 
Proposition~\ref{PROP:branches_locfinite}, we see that if there exists any 
intersection between the geodesic~$\gamma_\nu$ and~$\Gamma\act\BrU'$ at some 
time~$t>0$, then there exists a time-minimal one and hence~$\retime{\BrU'}(\nu)$ 
exists. 

We suppose first that $j\in A$. Then $\retime{\BrU}(\nu)$ exists by 
Proposition~\ref{PROP:CofSoBst}\eqref{CofSoB:firstreturn}. Thus, 
$\retime{\BrU'}(\nu)$ exists as well. We suppose now that $j\in A'\setminus A$. 
Then 
\[
 \bigl( \gamma_\nu(+\infty), \gamma_\nu(-\infty) \bigr) \in 
\wh\R_{\st}\times\wh\R_{\st}\,.
\]
As $\wh\R_{\st} \subseteq \Lambda(\Gamma)$, Proposition~\eqref{PROP:EXliesdense} 
shows that for each~$\eps>0$ we find a geodesic~$\eta_\eps$ on~$\H$ that 
represents a periodic geodesic on~$\Orbi$ and whose endpoints satisfy
\[
 \eta_\eps(\pm\infty) \in B_\eps\bigl( \gamma_\nu(\pm\infty) \bigr)\,.
\]
Since \eqref{BP:closedgeodesicsXtoH} is valid for~$\BrS$, we find $(k_\eps, 
g_\eps)\in A\times\Gamma$ such that $\eta_\eps$ 
intersects~$g_\eps\act\Cs{k_\eps}$. Thus,
\[
 \bigl( \eta_\eps(+\infty), \eta_\eps(-\infty) \bigr) \in g_\eps\act 
\Iset{k_\eps} \times g_\eps\act \Jset{k_\eps}\,.
\]
Since $g_\eps\act I_{k_\eps}$ and $g_\eps\act J_{k_\eps}$ are open and 
$\eX_{k_\eps}, \eY_{k_\eps}\notin \wh\R_{\st}$, we can choose $\eps$ so small 
that 
\[
 \bigl( \gamma_\nu(+\infty), \gamma_\nu(-\infty) \bigr) \in g_\eps\act 
I_{k_\eps} \times g_\eps\act J_{k_\eps}\,.
\]
We fix such an~$\eps$ and set $g\coloneqq g_\eps$, $k\coloneqq k_\eps$. Now 
\[
 \bigl( \gamma_\nu(+\infty), \gamma_\nu(-\infty) \bigr) \in g\act\Iset{k} \times 
g\act\Jset{k}
\]
and \eqref{BP:allvectors} for~$\BrS$ show that $g^{-1}\act\gamma_\nu$ 
intersects~$\Cs{k,\st}$, say in $\eta\in\Cs{k,\st}$ at time~$t_0$. We consider 
the system of iterated sequences of~$\eta$ with respect to~$\BrS$, as defined 
in~\eqref{EQDEF:ittime0}--\eqref{eq:itseq}. Lemma~\ref{LEM:accpointsofit} shows 
the existence of~$n\in\Z$ such that 
\[
 \ittime{\BrU,n}(\nu) > t_0\,,
\]
which means that $g^{-1}\act\gamma_\nu$ intersects~$\ittrans{\BrU,n}\act 
\Cs{\itindex{\BrU,n}(\nu)}$ at a time larger than~$t_0$. Thus, there exists an 
intersection between~$\gamma_\nu$ and~$\BrU'$ at a positive time, and hence 
$\retime{\BrU'}(\nu)$ exists. Analogously, we can show the existence 
of~$\pretime{\BrU'}(\nu)$ in both cases. 

For $j,k\in A'$ we set, motivated by 
Proposition~\ref{PROP:CofSoBst}\eqref{CofSoB:representsets},
\[
 \Transprime{}{j}{k} \coloneqq \defset{ g\in\Gamma }{ 
\exists\,\nu\in\Cs{j,\st}\colon \gamma_\nu'(\retime{\BrU'}(\nu)) \in g\act\Cs{k} 
}\,.
\]
\index[symbols]{Ga@$\Transprime{}{j}{k}$}%
We now show that $\BrS'$ 
satisfies~(\ref{BP:intervaldecomp}\ref{BP:intervaldecompGdecomp}) with 
$(\Transprime{}{j}{k})_{j,k\in A'}$ in place of~$(\Trans{}{a}{b})_{a,b\in A'}$. 
Let $j,k\in A'$ and $g\in\Transprime{}{j}{k}$. Thus, we find~$\nu\in\Cs{j,\st}$ 
such that $\gamma_\nu(\retime{\BrU'}(\nu))\in g\act\Cs{k}$ and hence 
\[
 \gamma_\nu(+\infty) \in \Iset{j}\cap g\act I_k\,.
\]
\eqref{BP:disjointunion} implies that $g\act I_k\subseteq I_j$. It follows that 
\begin{align}
 \bigcup_{k\in A'} \bigcup_{g\in\Transprime{}{j}{k}} g\act I_k \subseteq I_j
 \intertext{and}
 \bigcup_{k\in A'} \bigcup_{g\in\Transprime{}{j}{k}} g\act \Iset{k} \subseteq 
\Iset{j}\,. 
 \label{eq:prime_st_union}
\end{align}
The disjointness of these unions follows immediately 
from~\eqref{BP:disjointunion}. Further, for any~$x\in\Iset{j}$ we choose $y\in 
\Jset{j}$. By~\eqref{BP:allvectors}, the geodesic~$\gamma$ from~$y$ to~$x$ 
intersects~$\Cs{j,\st}$, say in~$\nu$. Now~$\retime{\BrU'}(\nu)$ exists as we 
have seen above. Thus, 
\[
\gamma_\nu( \retime{\BrU'}(\nu) ) \in 
\ittrans{\BrU',1}(\nu)\act\Cs{\itindex{\BrU',1}(\nu),\st}
\]
and hence $\ittrans{\BrU',1}(\nu) \in \Transprime{}{j}{\itindex{\BrU',1}(\nu)}$. 
Therefore
\[
 x = \gamma_\nu(+\infty) \in \ittrans{\BrU',1}(\nu)\act 
\Iset{\itindex{\BrU',1}(\nu)}\,.
\]
It follows that the inclusion in~\eqref{eq:prime_st_union} is indeed an 
equality. This completes the proof 
of~(\ref{BP:intervaldecomp}\ref{BP:intervaldecompGdecomp}). The proof 
of~(\ref{BP:intervaldecomp}\ref{BP:intervaldecompback}) is analogous, using the 
existence of~$\pretime{\BrU'}(\nu)$ for all~$\nu\in\BrU'_{\st}$. 
(\ref{BP:intervaldecomp}\ref{BP:intervaldecompGgeod}) follows immediately 
from~\eqref{BP:disjointunion} and the definition of the 
sets~$\Transprime{}{j}{k}$ for $j,k\in A'$.
\end{proof}

\begin{example}\label{EX:G3infram}
Recall the group~$\Gamma_{\lambda}$ from Example~\ref{EX:G3Def}.
A possible alternative choice of a set of branches is given by~$\BrS'=\{\Cs{1}\}$ 
with
\[
\Cs{1}\coloneqq\defset{\nu\in\UTB\H}{\base{\nu}\in(-\varepsilon,1+\varepsilon)_{
\H},\,\gamma_{\nu}(+\infty)\in(-\varepsilon,1+\varepsilon)_{\R}}\,,
\]
for some~$\varepsilon\in[0,\tfrac{\lambda - 1}{2}]$.
Figure~\ref{FIG:G3infram} visualizes the location of~$\Cs{1}$ in relation to the 
fundamental domain~$\fund$ from Figure~\ref{FIG:G3fund} as well as the 
translates of~$\Cs{1}$ determining the set~$\Trans{}{1}{1}$.
We see that~$\BrS'$ is infinitely ramified (we 
have~$\Att{\BrS}{\infty}=\varnothing$).
The (slow) transfer operator with parameter~$s\in\C$,~$\Rea{s}\gg1$, induced 
by~$\BrS'$ is given by
\[
\TO{s}=\sum_{\substack{n=-\infty\\n\ne0}}^{\infty}\left(\alpha_s(t_{\lambda}
^nh)+\alpha_s(t_{\lambda}^nh^{-1})\right)\,.
\]
Note that, since~$\lambda>1$, the 
transformations~$t_{\lambda}^nh,\,t_{\lambda}^nh^{-1}$ are hyperbolic for all 
exponents~$n\in\Z\setminus\{0\}$.

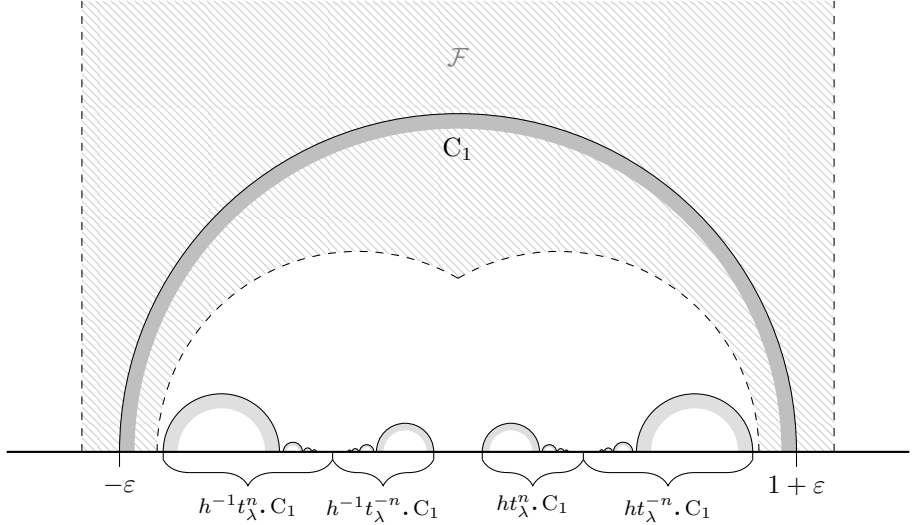
\begin{figure}[h]
\begin{tikzpicture}[scale=10]
\tikzmath{\q=1/(2*sqrt(3));
				  \r=.8*1/3;
				  \v=.1+\r;
				  \w=.1+2*\r;
				  }
\fill[pattern=north west lines, pattern color=lightgray!70] (1,0) -- (.9,0) arc 
(0:120:\r) -- (.5,.8*\q) arc (60:180:\r) -- (0,0)  -- (0,.6) -- (1,.6) -- cycle;
\draw[dashed] (.9,0) arc (0:120:\r);
\draw[dashed] (.5,.8*\q) arc (60:180:\r);
\draw[dashed] (0,0) -- (0,.6);
\draw[dashed] (1,0) -- (1,.6);
\fill[color=gray!50] (.95,0) arc (0:180:3*\r/2+.05) -- (.07,0) arc 
(180:0:3*\r/2+.03) -- cycle;
\draw (.95,0) arc (0:180:3*\r/2+.05);
\foreach \x/\y in {.05/$-\varepsilon$,.95/$1+\varepsilon$}
    \draw (\x,0.00) -- (\x,-0.02) node [below] {\y};
\foreach \s in {1,-1}{
	\foreach \ss in {1,-1}{
		\foreach \n in {1,2,3,4}{
			\tikzmath{
				\htl = ((3/2 + \s/2)*(-.04) + (3/2 + \s/2)*\ss*\n*1.2 - 
1)/(3*(-.04) + 3*\ss*\n*1.2 - (3/2 - \s/2));
				\htr = ((3/2 + \s/2)*(1.04) + (3/2 + \s/2)*\ss*\n*1.2 - 
1)/(3*(1.04) + 3*\ss*\n*1.2 - (3/2 - \s/2));
				\rad = .5*(\htr - \htl);
				\kleinrad = \rad/(5 - \n);
			}
			\fill[color=gray!25] (\htl,0) arc (180:0:\rad) -- (\htr - 
\kleinrad,0) arc (0:180:\rad - \kleinrad) -- cycle;
			\draw (\htl,0) arc (180:0:\rad);
		}
	}
	\tikzmath{
		\htl = ((3/2 + \s/2)*(-.04) + (3/2 + \s/2)*1.2 - 1)/(3*(-.04) + 3*1.2 - 
(3/2 - \s/2));
		\htr = ((3/2 + \s/2)*(1.04) + (3/2 + \s/2)*(-1.2) - 1)/(3*(1.04) + 
3*(-1.2) - (3/2 - \s/2));
		\cen = (3/2 + \s/2)/3;
	}
	\draw (\htl,0) -- (\htl,-.01);
	\draw (\htr,0) -- (\htr,-.01);
	\draw (\cen,0) -- (\cen,-.01);
	\draw [decorate,decoration={brace,amplitude=10pt,mirror}]
(\htl,-.01) -- (\cen,-.01) node [below,midway,yshift=-10pt]
{\ifthenelse{\s=1}{$\footnotesize ht_{\lambda}^n\act\Cs{1}$}{$\footnotesize 
h^{\s}t_{\lambda}^n\act\Cs{1}$}};	
	\draw [decorate,decoration={brace,amplitude=10pt,mirror}]
(\cen,-.01) -- (\htr,-.01) node [below,midway,yshift=-10pt]
{\ifthenelse{\s=1}{$\footnotesize ht_{\lambda}^{-n}\act\Cs{1}$}{$\footnotesize 
h^{\s}t_{\lambda}^{-n}\act\Cs{1}$}};	
}
\draw[style=thick] (-.1,0) -- (1.1,0);
\coordinate [label=below:$\color{gray}\fund$] (F) at (.5,.55);
\coordinate [label=below:$\Cs{1}$] (C1) at (.5,3*\r/2+.03);
\end{tikzpicture}
\caption[halfspaces]{The set of branches~$\BrS'=\{\Cs{1}\}$ 
for~$\Gamma_{\lambda}$ and its successors. For~$n\to\infty$ the successors 
converge towards the boundary points~$\tfrac{1}{3}$ 
resp.~$\tfrac{2}{3}$.}\label{FIG:G3infram}
\end{figure}

Since~$\base{\Cs{1}}\subseteq\fund$ and the radii of all isometric spheres 
of~$\Gamma_{\lambda}$ are bounded by~$1/3$, the branch~$\Cs{1}$ is 
clearly of maximal height in the sense that
\[
\max\defset{\Ima{z}}{z\in\base{\Cs{1}}}=\mathscr{h}_{\wh \infty}(1)\,.
\]
Hence, by following the instructions given in the proof of 
Proposition~\ref{PROP:allwithfiniteram}, we obtain a finitely ramified set of 
branches from~$\BrS'$ by adding the two branches
\begin{align*}
\Cs{2}&\coloneqq\defset{\nu\in\UTB\H}{\base{\nu}\in(-\varepsilon,\infty)_{\H},\,
\gamma_{\nu}(+\infty)\in(-\infty,\varepsilon)_{\R}}
\intertext{and}
\Cs{3}&\coloneqq\defset{\nu\in\UTB\H}{\base{\nu}\in(-\varepsilon,\infty)_{\H},\,
\gamma_{\nu}(+\infty)\in(-\varepsilon,+\infty)_{\R}}\,.
\end{align*}
In Figure~\ref{FIG:G3finram} the emerging set of 
branches~$\BrS''=\{\Cs{1},\Cs{2},\Cs{3}\}$ together with its relevant successors 
in the strip~$\left(\tfrac{1-\lambda}{2},\tfrac{1+\lambda}{2}\right)+\i\R_{>0}$ 
is portrayed.
Note that the additional successors~$t_{\lambda}\act\Cs{3}$ 
resp.~$t_{\lambda}^{-1}\act\Cs{1}$ and~$t_{\lambda}^{-1}\act\Cs{2}$ appear to 
the right resp. the left of it.
Now the slow transfer operator induced by~$\BrS''$, in matrix form, reads as
\[
\TO{s}=\begin{pmatrix}
0&\alpha_s(t_{\lambda})&1\\
\alpha_s(t_{\lambda}h)+\alpha_s(t_{\lambda}h^{-1})&\alpha_s(t_{\lambda})&0\\
\alpha_s(t_{\lambda}^{-1}h)+\alpha_s(t_{\lambda}^{-1}h^{-1})&0&\alpha_s(t_{
\lambda}^{-1})
\end{pmatrix}\,.
\]

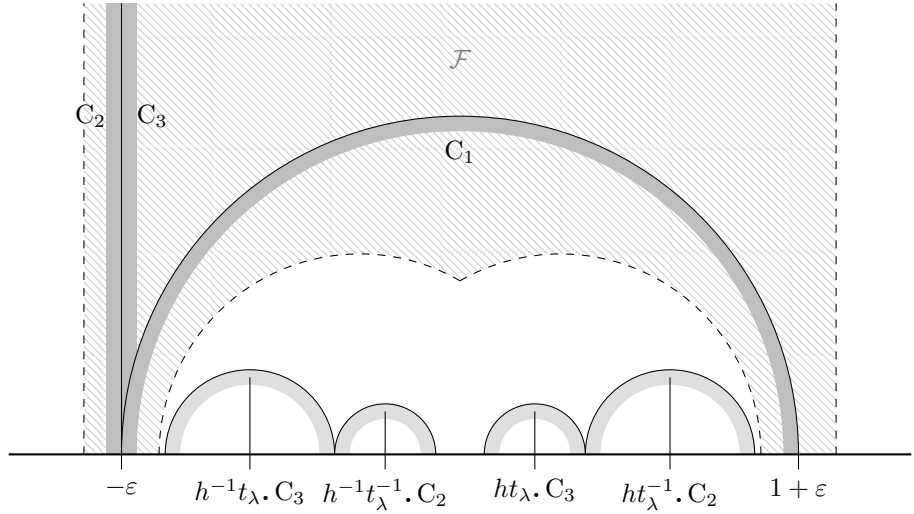
\begin{figure}[h]
\begin{tikzpicture}[scale=10]
\tikzmath{\q=1/(2*sqrt(3));
				  \r=.8*1/3;
				  \v=.1+\r;
				  \w=.1+2*\r;
				  }
\fill[pattern=north west lines, pattern color=lightgray!70] (1,0) -- (.9,0) arc 
(0:120:\r) -- (.5,.8*\q) arc (60:180:\r) -- (0,0)  -- (0,.6) -- (1,.6) -- cycle;
\draw[dashed] (.9,0) arc (0:120:\r);
\draw[dashed] (.5,.8*\q) arc (60:180:\r);
\draw[dashed] (0,0) -- (0,.44);
\draw[dashed] (0,.47) -- (0,.6);
\draw[dashed] (1,0) -- (1,.6);
\fill[color=gray!50] (.07,0) -- (.07,.6) -- (.03,.6) -- (.03,0) -- cycle;
\draw (.05,0) -- (.05,.6);
\fill[color=gray!50] (.95,0) arc (0:180:3*\r/2+.05) -- (.07,0) arc 
(180:0:3*\r/2+.03) -- cycle;
\draw (.95,0) arc (0:180:3*\r/2+.05);
\foreach \x/\y in {.05/$-\varepsilon$,.95/$1+\varepsilon$}
    \draw (\x,0.00) -- (\x,-0.02) node [below] {\y};
\def \lam {1.2}
\def \epp {.04}
\tikzmath{
	\poia = (2*\lam - 1 - 2*\epp)/(3*\lam - 1 - 3*\epp);
	\poib = (2*\lam - 1 - 2*\epp)/(3*\lam - 2 - 3*\epp);
	\poic = (\lam - 1 - \epp)/(3*\lam - 2 - 3*\epp);
	\poid = (\lam - \epp)/(3*\lam - 1 - 3*\epp);
}
\foreach \x/\y/\z in 
{\poia/2/$ht_{\lambda}\act\Cs{3}$,\poib/2/$ht_{\lambda}^{-1}\act\Cs{2}$,
\poic/1/$h^{-1}t_{\lambda}\act\Cs{3}$,\poid/1/$h^{-1}t_{\lambda}^{-1}\act\Cs{2}$
}{
	\tikzmath{
		\poi = min(\x,\y/3);
		\rad = .5*abs(\x - \y/3);
		\kleinrad = .02;
	}
	\fill[color=gray!25] (\poi,0) arc (180:0:\rad) -- (\poi + 2*\rad - 
\kleinrad,0) arc (0:180:\rad - \kleinrad) -- cycle;
	\draw (\poi,0) arc (180:0:\rad);
	\draw (\poi + \rad,\rad - .01) -- (\poi + \rad,-.02) node [below] {\z};
}
\draw[style=thick] (-.1,0) -- (1.1,0);
\coordinate [label=below:$\color{gray}\fund$] (F) at (.5,.55);
\coordinate [label=below:$\Cs{1}$] (C1) at (.5,3*\r/2+.03);
\coordinate [label=left:$\Cs{2}$] (C2) at (.042,.45);
\coordinate [label=right:$\Cs{3}$] (C3) at (.058,.45);
\end{tikzpicture}
\caption[halfspaces]{The set of branches~$\BrS''=\{\Cs{1},\Cs{2},\Cs{3}\}$ 
for~$\Gamma_{\lambda}$ and its successors.}\label{FIG:G3finram}
\end{figure}
\end{example}

\section{Branch reduction}\label{SEC:branchred}

Throughout this section let $\Gamma$ be a geometrically finite Fuchsian group 
with at least one hyperbolic element and let $\Orbi\coloneqq\quod{\Gamma}{\H}$ be the 
associated hyperbolic orbisurface. Let $N\in\N$, $A\coloneqq\{1,\ldots,N\}$ and let 
\[
 \BrS \coloneqq \defset{ \Cs j }{ j\in A }
\]
be a set of branches for the geodesic flow on~$\Orbi$. Further let 
\[
 \BrU \coloneqq \bigcup_{j\in A} \Cs j
\]
denote the branch union, and set 
\[
 \CrSc \coloneqq \pi(\BrU)\,,
\]
where $\pi\colon\UTB\H\to\UTB\Orbi$ is the canonical quotient map 
from~\eqref{eq:def_pi2}. In general, the cross section~$\CrSc$ and the set of branches~$\BrS$ do not yet give rise to a strict transfer operator approach. More precisely, using the notation from Section~\ref{SUBSEC:stricttrans}, if we attempt to use the family of intervals~$(I_a)_{a\in A}$ as part of a structure tuple and form, for $a,b\in A$, the sets~$P_{a,b}, C_{a,b}$ and $(g_p)_{p\in P_{a,b}}$ of elements of~$\Gamma$ such that the associated discrete dynamical system~$(D,F)$ (see explanation right after Property~\ref{staPROP1}) coincides with the discrete dynamical system associated to~$\BrS$ (see Section~\ref{SUBSEC:slowtrans}), then Properties~\ref{staPROP1}--\ref{staPROP5} from Section~\ref{SUBSEC:stricttrans} are typically not satisfied. For the associated transfer operators it means that we typically cannot find a Banach space on which they act as nuclear operators and have a well-defined Fredholm determinant (even ignoring the requirement that it should be related to the Selberg zeta function). This issue, if present, originates from $(D,F)$ not being uniformly expanding. The non-uniform expansiveness of~$(D,F)$ can have the following two reasons: 
\begin{enumerate}[label=$\mathrm{(\roman*)}$, ref=$\mathrm{\roman*}$]
 \item\label{issue1} The identity element of~$\Gamma$ is among the action elements of~$F^n$ for some $n\in\N$. That is, some iterate of the map~$F$ has a submap of the form 
 \[
  \wt{\Iset{a}}\times \{b\} \to \wt{\Iset{a}}\times \{a\}\,,\quad (x,b)\mapsto (x,a)
 \]
 for some interval~$\wt{\Iset{a}}$ (being a naturally formed subinterval of~$\Iset{a}$).
 \item\label{issue2} Some iterate of the map~$F$ has a submap consisting of the action of a parabolic element of~$\Gamma$, and the fixed point of this parabolic element is an inexhaustible source for iterations. That is, there exists~$n\in\N$ such that $F^n$ has a submap conjugate to 
 \[
  (1,\infty)_{\st} \times \{a\} \to (0,\infty)_{\st} \times \{a\}\,,\quad (x,a)\mapsto (x-1,a)\,.
 \]
 Then any iterate of~$F^n$ has a submap of this form, and hence a ``big part'' in which no expansion takes place. 
\end{enumerate}

For the set of branches~$\BrS$, issue~\eqref{issue1} means that $\BrS$ contains a branch which contains an element, say~$\nu$, such that the associated geodesic~$\gamma_\nu$ intersects another branch in~$\BrS$. Issue~\eqref{issue2} is present if $\Orbi$ has cusps. The cross section~$\CrSc$ detects every winding of a geodesic around a cusp as a separate event, and hence the set of branches~$\BrS$ and the associated discrete dynamical system~$(D,F)$ encode each single of them separately. 

To overcome these issues we require an appropriate acceleration of the dynamics, which 
translates to a reduction procedure of the branches. This will be done in three 
separate steps, which we call \emph{branch reduction}, \emph{identity elimination} and \emph{cuspidal acceleration}. 

The branch reduction, which we present in this section (Algorithms~\ref{nodereductionI} and~\ref{nodereductionII}), aims at simplifying the constructions by reducing the number of branches to a ``minimum.'' Albeit not being absolutely necessary, in many cases the branch reduction considerably reduces the complexity of the situation. In addition, it provides several intermediate ``reduced slow transfer operator families,'' which are typically useful for other applications as well. The identity elimination is discussed in Section~\ref{SEC:stepredux}, the cuspidal acceleration in Section~\ref{SEC:cuspacc}.

\subsection{Return graphs}\label{SUBSEC:returngraphs}

We associate to the set of branches~$\BrS$ a directed graph~$\ReturnGraph{}$ 
with weighted edges, called \emph{return graph}, which encodes the next intersection properties among the branches in~$\BrS$ in just the right way for an efficient presentation and discussion of the branch reduction algorithm. The return graph~$\ReturnGraph{}$ associated to~$\BrS$ is defined as follows:
\index[defs]{return graph}%
\index[symbols]{RG@$\ReturnGraph{}$}%
\index[defs]{vertex}%
\index[defs]{edge}%
\begin{itemize}
 \item The set of nodes or vertices of~$\ReturnGraph{}$ is~$A$.
 \item The set of edges of~$\ReturnGraph{}$ from the node~$j$ to the node~$k$ is 
bijective to the forward transition set~$\Trans{}{j}{k}$. For each $g\in 
\Trans{}{j}{k}$, the graph~$\ReturnGraph{}$ contains an edge from~$j$ to~$k$ 
with weight~$g$.
\end{itemize}

Let $j,k$ be nodes of~$\ReturnGraph{}$. We call $k$ a \emph{successor}
\index[defs]{successor}%
of~$j$ if 
$\ReturnGraph{}$ has an edge from~$j$ to~$k$, and a \emph{predecessor}
\index[defs]{predecessor}%
of~$j$  
if $\ReturnGraph{}$ has an edge from~$k$ to~$j$. For an edge from~$j$ to~$k$ 
weighted by~$g$ we will often write $j\edge{g}k$. A \emph{path}
\index[defs]{path}%
\index[symbols]{j@$j\edge{g}k$}%
in~$\ReturnGraph{}$ of length~$m\in\N$ is a sequence of consecutive edges of the 
form
\begin{equation}\label{EQ:path}
j_0\edge{g_1}j_1\edge{g_2}\dots
\edge{g_m}j_{m}\,.
\end{equation}
If $j_m = j_0$, then we call this path also a \emph{cycle}.
\index[defs]{cycle}%
See 
Example~\ref{EX:G3Graph} and Figure~\ref{FIG:G3returngraph} below for an example 
of a return graph.

We recall from Proposition~\ref{PROP:CofSoBst}\eqref{CofSoB:representsets} that 
for any $j,k\in A$, the forward transition set is 
\[
\Trans{}{j}{k}=\defset{g\in\Gamma}{\exists 
\nu\in\Cs{j}\colon\gamma_{\nu}^{\prime}(\retime{\BrU}(\nu))\in g\act\Cs{k}}\,.
\]
Its elements determine exactly the translates of the branch~$\Cs{k}$ on which 
the next intersections of geodesics starting on~$\Cs{j}$ are located (see 
Section~\ref{SUBSEC:reducedbranches}). We consider weights of paths as 
multiplicative. In other words, the path 
\begin{equation}\label{eq:minipath}
j\edge{g}k\edge{h}\ell
\end{equation}
gives rise to the total weight~$gh$. Any path in~$\ReturnGraph{}$ of the form as 
in~\eqref{eq:minipath} indicates that there is (at least) one geodesic starting 
on the branch~$\Cs{j}$, traversing~$g\act\Cs{k}$ and then 
intersecting~$gh\act\Cs{\ell}$, and not intersecting any other translates of 
branches inbetween. The following lemma shows that this interpretation of paths 
is indeed correct, and that all paths  in~$\ReturnGraph{}$ arise in this way. We emphasize that the vector~$\nu$ in the second part of the following lemma is not 
necessarily unique.

\begin{lemma}\label{LEM:paths_gen}
The paths in~$\ReturnGraph{}$ are fully characterized by the systems of iterated 
sequences from Section~\ref{SUBSEC:reducedbranches} of the elements in~$\BrU$. 
More precisely: 
\begin{enumerate}[label=$\mathrm{(\roman*)}$, ref=$\mathrm{\roman*}$]
\item\label{path:allinittopath}
For all $\nu\in\BrU_{\st}$ and all $n\in\N$, the return 
graph~$\ReturnGraph{}$ contains the path
\[
\itindex{\BrU,0}(\nu)
\edge{\ittrans{\BrU,1}(\nu)}\itindex{\BrU,1}(\nu)\edge{\ittrans{\BrU,2}
(\nu)}\dots
\edge{\ittrans{\BrU,n}(\nu)}\itindex{\BrU,n}(\nu)\,.
\] 

\item\label{path:allbyCS} Let $m\in\N$ and suppose that 
\[
 k_0\edge{g_1}k_1\edge{g_2}
\dots\edge{g_{m-1}}k_{m-1}
\edge{g_m}
k_m
\]
is a path in the return graph~$\ReturnGraph{}$. Then there exists $\nu\in 
\BrU_{k_0,\st}$ such that 
\begin{align*}
k_j & = \itindex{\BrU,j}(\nu) \qquad \text{for $j\in\{0,\ldots, m\}$}
\intertext{and}
g_j & = \ittrans{\BrU,j}(\nu) \qquad \text{for $j\in\{1,\ldots, m\}$\,.}
\end{align*}
\end{enumerate}
\end{lemma}

\begin{proof}
To prove~\eqref{path:allinittopath} let $\nu\in\Cs{\st}$ and set 
$\nu_{m}\coloneqq\gamma_{\nu}^{\prime}(\ittime{\BrU,m}(\nu))$ for 
all~$m\in\N_0$. (We recall from Section~\ref{SUBSEC:reducedbranches} that 
$\ittime{\BrU,m}(\nu)$ is well-defined for all~$m\in\N_0$.) For each~$m\in\N_0$ 
we have 
$\ittrans{\BrU,1}(\nu_m)\in\Trans{}{\itindex{\BrU,0}(\nu_m)}{
\itindex{\BrU,1}(\nu_m)}$. Hence, the return graph~$\ReturnGraph{}$ contains 
the edge
\begin{equation}\label{eq:m_edge}
\itindex{\BrU,0}(\nu_m)
\edge{\ittrans{\BrU,1}(\nu_m)}\itindex{\BrU,1}(\nu_m)\,.
\end{equation}
Since 
\[
 \ittrans{\BrU,1}(\nu_m)=\ittrans{\BrU,m+1}(\nu)\,,\quad 
\itindex{\BrU,0}(\nu_m)=\itindex{\BrU,m}(\nu)\,,\quad\text{and}\quad 
\itindex{\BrU,1}(\nu_m)=\itindex{\BrU,m+1}(\nu)\,,
\]
the edge in~\eqref{eq:m_edge} equals
\[
 \itindex{\BrU,m}(\nu)
\edge{\ittrans{\BrU,m+1}(\nu)}\itindex{\BrU,m+1}(\nu)\,.
\]
Letting $m$ run through~$\{0,\ldots, n-1\}$,  we now obtain that the path 
\[
\itindex{\BrU,0}(\nu)
\edge{\ittrans{\BrU,1}(\nu)}\itindex{\BrU,1}(\nu)\edge{\ittrans{\BrU,2}
(\nu)}\dots
\edge{\ittrans{\BrU,n}(\nu)}\itindex{\BrU,n}(\nu)
\]
is contained in~$\ReturnGraph{}$. This proves~\eqref{path:allinittopath}.

We now show~\eqref{path:allbyCS}. Let $j\in \{1,\ldots, m\}$. Since 
\[
k_{j-1}\edge{g_j} k_j 
\]
is an edge in~$\ReturnGraph{}$ by hypothesis, we have 
$g_j\in\Trans{}{k_{j-1}}{k_j}$. Lemma~\ref{LEM:halfspaceincludeGV} shows 
\[
 g_j\act \Plussp{k_j} \varsubsetneq \Plussp{k_{j-1}}
\]
and hence 
\begin{equation}\label{EQ:nestedinclusion}
 g_j\act I_{k_j,\st} \subseteq I_{k_{j-1},\st}\,.
\end{equation}
Set $g\coloneqq g_1\cdots g_m$. Applying repeatedly these inclusions 
considerations we obtain
\[
 g\act I_{k_m,\st} \subseteq I_{k_0,\st}\,.
\]
We pick 
\[
(x,y)\in g\act I_{k_m,\st}\times J_{k_0,\st}\,.
\]
By~\eqref{BP:allvectors} and Remark~\ref{REM:inst} we find (a unique) 
$\nu\in\BrU_{k_0,\st}$ such that 
\[
 (x,y) = (\gamma_\nu(+\infty),\gamma_\nu(-\infty))\,.
\]
In order to show that $\nu$ satisfies the claimed properties we proceed 
inductively. Since $g\act I_{k_m,\st}\subseteq g_1\act I_{k_1,\st}$ 
(which can be seen via iteration of~\eqref{EQ:nestedinclusion}), $\gamma_\nu(+\infty)\in 
g_1\act 
I_{k_1,\st}$. From this property and $g_1\in\Trans{}{k_0}{k_1}$ it follows that
\[
 \gamma'_\nu(\retime{\BrU,1}(\nu)) \in g_1\act \BrU_{k_1,\st}\,,
\]
by Corollary~\ref{COR:uniquereturn} and Remark~\ref{REM:inst}. Then, by definition,
\[
 k_0 = \itindex{\BrU,0}(\nu)\,,\quad k_1 = 
\itindex{\BrU,1}(\nu)\quad\text{and}\quad g_1 = \ittrans{\BrU,1}(\nu)\,.
\]
Suppose now that for some $j_0\in\{2,\ldots, m-1\}$ we have 
already established that 
\[
 \gamma'_\nu(\retime{\BrU,j}(\nu)) \in g_j\cdots g_1\act \BrU_{k_j,\st} 
\qquad\text{for $j\in\{1,\ldots, j_0-1\}$}
\]
as well as 
\begin{align*}
 k_j & = \itindex{\BrU,j}(\nu)\qquad \text{for $j\in\{0,\ldots, j_0-1\}$}
 \intertext{and}
 g_j & = \ittrans{\BrU,j}(\nu)\qquad \text{for $j\in\{1,\ldots, j_0-1\}$.}
\end{align*}
Then 
\[
 \wt \nu \coloneqq g_1^{-1}\cdots g_{j_0-1}^{-1}\act 
\gamma'_\nu(\retime{\BrU,j_0-1}(\nu)) \in \BrU_{k_{j_0-1},\st}
\]
and the associated geodesic~$\gamma_{\wt \nu}$ is given by
\[
 \gamma_{\wt\nu}(t) = g_1^{-1}\cdots g_{j_0-1}^{-1}\act \gamma_\nu(t + 
\retime{\BrU,j_0-1}(\nu))\qquad\text{for all $t\in\R$.}
\]
Since
\[
 \gamma_\nu(+\infty) \in g_1\cdots g_{j_0}\act I_{k_{j_0},\st}\,,
\]
we have 
\[
 \gamma_{\wt\nu}(+\infty)= g_1^{-1}\cdots 
g_{j_0-1}^{-1}\act\gamma_\nu(+\infty) \in g_{j_0}\act I_{k_{j_0},\st}\,.
\]
Together with $g_{j_0}\in\Trans{}{k_{j_0-1}}{k_{j_0}}$ this yields that 
\[
 \gamma'_{\wt\nu}(\retime{\BrU,1}(\wt\nu)) \in 
g_{j_0}^{-1}\act\BrU_{k_{j_0},\st}\,,
\]
using Corollary~\ref{COR:uniquereturn} and Remark~\ref{REM:inst}. Therefore,
\begin{align*}
 k_{j_0-1} & = \itindex{\BrU,0}(\wt\nu) = \itindex{\BrU,j_0-1}(\nu)\,,
 \\
 k_{j_0} & = \itindex{\BrU,1}(\wt\nu) = \itindex{\BrU,j_0}(\nu)\,,
 \\
 g_{j_0} & = \ittrans{\BrU,1}(\wt\nu) = \ittrans{\BrU,j_0}(\nu)
\end{align*}
as well as
\[
 \gamma'_{\wt\nu}(\retime{\BrU,1}(\wt\nu)) = 
\gamma'_\nu(\retime{\BrU,j_0}(\nu))\,.
\]
This completes the proof of~\eqref{path:allbyCS}.
\end{proof}

The return graph~$\ReturnGraph{}$ is highly connected and weighted paths are essentially unique as the following proposition proves. These properties are crucial for the proof of the correctness of the branch reduction algorithm presented below (Algorithms~\ref{nodereductionI} and~\ref{nodereductionII}). See Proposition~\ref{PROP:algorithmsbehave}.

\begin{prop}\label{PROP:returngraph}
The paths in the return graph~$\ReturnGraph{}$ obey the following structures:
\begin{enumerate}[label=$\mathrm{(\roman*)}$, ref=$\mathrm{\roman*}$]
\item\label{returngraph:cyclelvlr}
Every node in~$\ReturnGraph{}$ is contained in a cycle.
 
\item\label{returngraph:disjointpaths}
Let $j,k\in A$ and suppose that 
\begin{align*}
j&\edge{g_1}p_1\edge{g_2}
\dots\edge{g_{m-1}}p_{m-1}
\edge{g_m}
k
\intertext{and}
j&\edge{h_1}q_1\edge{h_2}\dots\edge{h_{n-1}}q_{n-1}
\edge{h_n}k
\end{align*}
are paths in~$\ReturnGraph{}$ such that $g_1g_2\cdots g_m=h_1h_2\cdots h_n$. 
Then $m=n$ and, for all $i\in\{1,\ldots, m-1\}$ and $\ell\in\{1,\ldots, m\}$, we 
have $p_i=q_i$ and $g_\ell=h_\ell$.

\item\label{returngraph:awayback}
Let $j,k\in A$. If $\ReturnGraph{}$ contains a path from~$j$ to~$k$, then
it also contains a path from~$k$ to~$j$.
\end{enumerate}
\end{prop}

\begin{proof}
In order to prove~\eqref{returngraph:cyclelvlr} we fix~$j\in A$. 
By~\eqref{BP:closedgeodesicsHtoX} and Remark~\ref{REM:inst} we 
find~$\nu\in\Cs{j,\st}$ such that $\widehat{\gamma}_\nu\coloneqq\pi(\gamma_\nu)$ 
is a 
periodic geodesic on~$\Orbi$. Thus, we find $t_0\in (0,\infty)$ such that 
$\widehat{\gamma}'_\nu(t_0)=\widehat{\gamma}'_\nu(0)$. Consequently, 
$\eta\coloneqq\gamma_{\nu}'(t_0)\in g\act\Cs{j}$ for some~$g\in\Gamma$, 
by Proposition~\ref{PROP:CofSoBnonst}. Proposition~\ref{PROP:allinit} 
now shows that there exists a unique~$n\in\N$ (we note that $t_0>0$) such that 
$t_0=\ittime{\BrU,n}(\nu)$, $j=\itindex{\BrU,n}(\nu)$ and 
$g=\ittrans{\BrU,1}(\nu)\cdots\ittrans{\BrU,n}(\nu)$. By 
Lemma~\ref{LEM:paths_gen}\eqref{path:allinittopath}, $\ReturnGraph{}$ contains 
the path 
\[
j=\itindex{\BrU,0}(\nu)
\edge{\ittrans{\BrU,1}(\nu)}\itindex{\BrU,1}(\nu)\edge{\ittrans{\BrU,2}
(\nu)}\dots
\edge{\ittrans{\BrU,n}(\nu)}\itindex{\BrU,n}(\nu)=j\,.
\] 
Hence, $j$ is contained in a cycle of~$\ReturnGraph{}$, which 
establishes~\eqref{returngraph:cyclelvlr}.

For the proof of~\eqref{returngraph:disjointpaths} we pick, as seen to be 
possible by Lemma~\ref{LEM:paths_gen}\eqref{path:allbyCS}, elements 
$\nu_1,\nu_2\in\Cs{j,\st}$ such that the path generated by~$\nu_1$ is 
\[
 j\edge{g_1}p_1\edge{g_2} \dots\edge{g_{m-1}}p_{m-1} \edge{g_m} k
\]
and the path generated by~$\nu_2$ is 
\[
 j\edge{h_1}q_1\edge{h_2}\dots\edge{h_{n-1}}q_{n-1}\edge{h_n}k\,.
\]
This means that 
\[
 j = \itindex{\BrU,0}(\nu_1) = \itindex{\BrU,0}(\nu_2)
\]
and 
\begin{align*}
 p_\ell = \itindex{\BrU,\ell}(\nu_1) & \quad\text{for $\ell\in\{1,\ldots, 
m-1\}$}
 \\
 g_\ell = \ittrans{\BrU,\ell}(\nu_1) & \quad\text{for $\ell\in\{1,\ldots, m\}$}
 \\
 q_\ell = \itindex{\BrU,\ell}(\nu_2) & \quad\text{for $\ell\in\{1,\ldots, 
n-1\}$}
 \\
 h_\ell = \ittrans{\BrU,\ell}(\nu_2) & \quad\text{for $\ell\in\{1,\ldots, 
n\}$}\,.
\end{align*}
Let 
\[
 g\coloneqq g_1\cdots g_m = h_1\cdots h_n\,.
\]
By combining Lemma~\ref{LEM:paths_gen}, Corollary~\ref{COR:uniquereturn} (recall 
\eqref{EQDEF:itindex}) and \eqref{BP:intervaldecomp} we find, 
considering~$\nu_1$, 
\[
g\act I_k = g_1\cdots g_m\act I_k \subseteq g_1\ldots g_{m-1}\act I_{p_{m-1}} 
\subseteq \ldots \subseteq g_1g_2\act I_{p_2} \subseteq g_1\act I_{p_1} 
\subseteq I_j\,. 
\]
Considering~$\nu_2$, we obtain 
\[
 g\act I_k = h_1\cdots h_n\act I_k \subseteq h_1\cdots h_{n-1}\act I_{q_{n-1}} 
\subseteq \ldots \subseteq h_1h_2\act I_{q_2} \subseteq h_1\act I_{q_1} 
\subseteq I_j\,.
\]
The disjointness of the unions in~\eqref{BP:intervaldecomp} yields that $g_1\act 
I_{p_1}\cap h_1\act I_{q_1}=\varnothing$ whenever $(g_1,p_1)\not=(h_1,q_1)$. 
However, since $g_1\act I_{p_1}$ and $h_1\act I_{q_1}$ both contain $g\act I_k$ 
(which is nonempty), we obtain $(g_1,p_1) = (h_1,q_1)$. Applying this argument 
iteratively, we find $n=m$, and $g_\ell=h_\ell$, $p_i=q_i$ for all 
$\ell\in\{1,\ldots, n\}$ and all $i\in\{1,\ldots, n-1\}$.

In order to prove~\eqref{returngraph:awayback} let 
\begin{equation}\label{eq:gen_path}
j\edge{g_1}p_1\edge{g_2}\dots\edge{g_{m-1}}p_{m-1}\edge{g_m}k
\end{equation}
be a path in the return graph~$\ReturnGraph{}$ of length~$m\in\N$. 
Lemma~\ref{LEM:paths_gen}\eqref{path:allbyCS} shows that there 
exists~$\nu\in\Cs{j,\st}$ such that the first $m$ elements of the system of 
iterated sequences of~$\nu$ produce the path~\eqref{eq:gen_path}. The choice 
of~$\nu$ is not unique. In what follows, we show that $\nu$ can be chosen such 
that $\wh\gamma_\nu$ is a periodic geodesic on~$\Orbi$. 

As in the proof of~\eqref{returngraph:disjointpaths} we see that
\begin{equation}\label{eq:nestedforward}
 g_1\cdots g_m\act I_k \subseteq g_1\cdots g_{m-1}\act I_{p_{m-1}} \subseteq 
\ldots \subseteq g_1g_2\act I_{p_2} \subseteq g_1\act I_{p_1} \subseteq I_j\,.
\end{equation}
Using Remark~\ref{REM:SoBrem}\eqref{SoBrem:descrHalfspaces}, we then obtain 
\begin{equation}\label{eq:nestedbackwards}
 g_1\cdots g_m\act J_k \supseteq g_1\cdots g_{m-1}\act J_{p_{m-1}} \supseteq 
\ldots \supseteq g_1g_2\act J_{p_2} \supseteq g_1\act J_{p_1} \supseteq J_j\,.
\end{equation}
Let $g\coloneqq g_1\cdots g_m$ and recall from~\eqref{EQDEF:EOrbi} the 
set~$E(\Orbi)$ of endpoint pairs of representatives of periodic geodesics 
on~$\Orbi$. Since $E(\Orbi)$ is dense in~$\Lambda(\Gamma)\times\Lambda(\Gamma)$ 
(see Proposition~\eqref{PROP:EXliesdense}) and 
$\wh\R_{\st}\subseteq\Lambda(\Gamma)$, we find 
\[
 (x,y) \in g\act \Iset{k}\times \Jset{j}\,.
\]
Let $\gamma$ be a geodesic on~$\H$ with $(\gamma(+\infty),\gamma(-\infty)) = 
(x,y)$. By~\eqref{BP:allvectors} and Remark~\ref{REM:inst}, $\gamma$ 
intersects~$\Cs{j,\st}$, say in~$\nu$. Iterated application of 
Corollary~\ref{COR:uniquereturn} shows that $\nu$ produces the 
path~\eqref{eq:gen_path}, i.e., 
\begin{align*}
 j=\itindex{\BrU,0}(\nu) &\edge{g_1=\ittrans{\BrU,1}(\nu)} p_1 = 
\ittrans{\BrU,1}(\nu) \edge{g_2=\ittrans{\BrU,2}(\nu)} \ldots 
 \\
 & \edge{g_{m-1} = \ittrans{\BrU,m-1}(\nu)} p_{m-1} = \itindex{\BrU,m-1}(\nu) 
\edge{g_m=\ittrans{\BrU,m}(\nu)} k = \itindex{\BrU,m}(\nu)\,.
\end{align*}
Since $\gamma_\nu$ represents a periodic geodesic on~$\Orbi$ as being a 
reparametrization of~$\gamma$, the system of iterated sequences of~$\nu$ is 
periodic and hence the (infinite) path in~$\ReturnGraph{}$ determined by~$\nu$ 
(see Lemma~\ref{LEM:paths_gen}) contains a subpath from~$k$ to~$j$. This 
completes the proof.
\end{proof}

\begin{example}\label{EX:G3Graph}
Recall the group~$\Gamma_{\lambda}$ from Example~\ref{EX:G3Def} and its set of 
branches~$\BrS$ as indicated in Figure~\ref{FIG:G3SoB}.
The return graph $\ReturnGraph{0}$ of $\Gamma_{\lambda}$ can easily be read off from 
Figure~\ref{FIG:G3SoB} and is given in Figure~\ref{FIG:G3returngraph}.

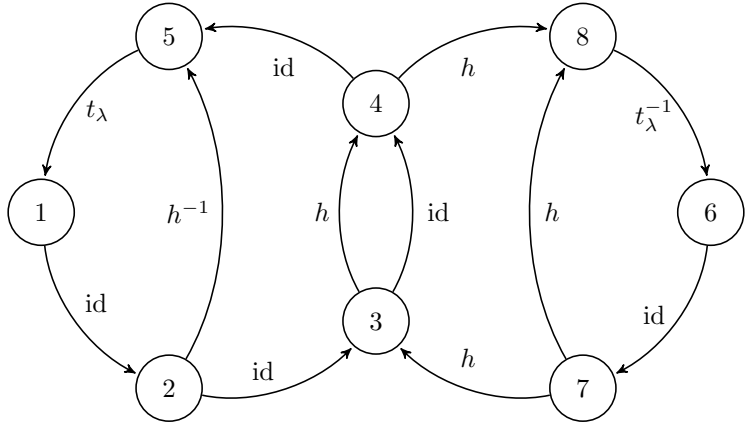
\begin{figure}[h]
\begin{tikzpicture}[->,shorten >=1pt,auto, semithick]
\def \fak {7 em}
\begin{scope}[on grid]
\node[state] (C6) {$6$};
\node[state] (C8) [above left= .951*\fak and (1-.309)*\fak of C6] {$8$};
\node[state] (C4) [above left= .588*\fak and 1.809*\fak of C6] {$4$};
\node[state] (C3) [below left= .588*\fak and 1.809*\fak of C6] {$3$};
\node[state] (C7) [below left= .951*\fak and (1-.309)*\fak of C6] {$7$};
\node[state] (C2) [below left= .951*\fak and (1.809+.809+.309)*\fak of C6] 
{$2$};
\node[state] (C5) [above left= .951*\fak and (1.809+.809+.309)*\fak of C6] 
{$5$};
\node[state] (C1) [left=(1.809+.809+1)*\fak of C6] {$1$};
\def \los {.85}
\path (C1) edge [bend right,looseness=\los] node {$\id$} (C2)
		  (C2) edge [bend right,looseness=\los] node {$\id$} (C3)
		  		 edge [bend right,looseness=\los] node[left=1pt] {$h^{-1}$} (C5)
		  (C3) edge [bend left,looseness=\los] node {$h$} (C4)
		  		 edge [bend right,looseness=\los] node[right=2pt] {$\id$} (C4)
		  (C4) edge [bend right,looseness=\los] node[below=5pt] {$\id$} (C5)
		  		 edge [bend left,looseness=\los] node [below=5pt] {$h$} (C8)
		  (C5) edge [bend right,looseness=\los] node [pos=.3,below=5pt] 
{$t_{\lambda}$} (C1)
		  (C6) edge [bend left,looseness=\los] node[pos=.7,above=5pt] {$\id$} 
(C7)
		  (C7) edge [bend left,looseness=\los] node[right=2pt] {$h$} (C8)
		  		 edge [bend left,looseness=\los] node[pos=.5, above=5pt] {$h$} 
(C3)
		  (C8) edge [bend left,looseness=\los] node[pos=.3, below=5pt] 
{$t_{\lambda}^{-1}$} (C6);
\end{scope}
\end{tikzpicture}
\caption[halfspaces]{The return graph $\ReturnGraph{0}$ for the set of 
branches~$\BrS$ for~$\Gamma_{\lambda}$.}\label{FIG:G3returngraph}
\end{figure}

The associated (slow) transfer operator in matrix form reads as
\[
\TO{s}=\begin{pmatrix}
0&				0&				0&			0&\alpha_s(t_{\lambda}^{-1})&0&			
0&				0\\
1&				0&				0&				0&				0&				
0&				0&				0\\
0&				1&				0&				0&				0&				
0&	\alpha_s(h^{-1})&	0\\
0&				0&  1+\alpha_s(h^{-1})&  0&				0&				0&			
	0&				0\\
0&		\alpha_s(h)&		0&				1&				0&				0&		
		0&				0\\
0&				0&				0&				0&				0&		 0&	   
0&\alpha_s(t_{\lambda})\\
0&				0&				0&				0&				0&				
1&				0&				0\\
0&				0&				0&	\alpha_s(h^{-1})&	0&				0&	
\alpha_s(h^{-1})&	0
\end{pmatrix}.
\]
Similarly, for the weakly non-collapsing set of branches for~$\Gamma_\lambda$ given in Example~\ref{EX:G3wnc} we obtain the return graph~$\ReturnGraph{0}$ as depicted in Figure~\ref{FIG:G3returngraphwnc}.

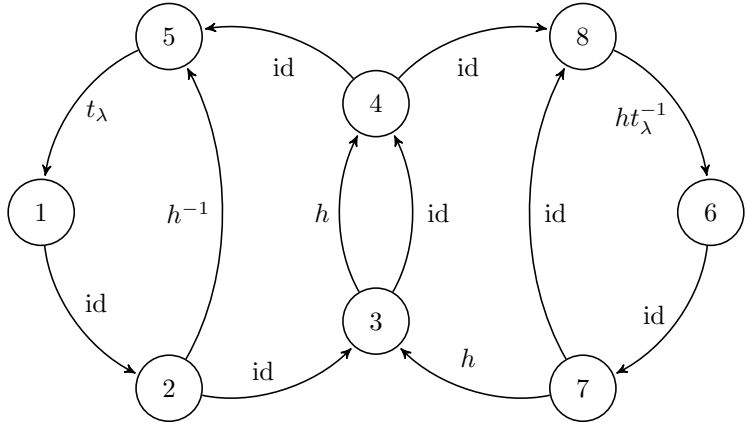
\begin{figure}[h]
\begin{tikzpicture}[->,shorten >=1pt,auto, semithick]
\def \fak {7 em}
\begin{scope}[on grid]
\node[state] (C6) {$6$};
\node[state] (C8) [above left= .951*\fak and (1-.309)*\fak of C6] {$8$};
\node[state] (C4) [above left= .588*\fak and 1.809*\fak of C6] {$4$};
\node[state] (C3) [below left= .588*\fak and 1.809*\fak of C6] {$3$};
\node[state] (C7) [below left= .951*\fak and (1-.309)*\fak of C6] {$7$};
\node[state] (C2) [below left= .951*\fak and (1.809+.809+.309)*\fak of C6] 
{$2$};
\node[state] (C5) [above left= .951*\fak and (1.809+.809+.309)*\fak of C6] 
{$5$};
\node[state] (C1) [left=(1.809+.809+1)*\fak of C6] {$1$};
\def \los {.85}
\path (C1) edge [bend right,looseness=\los] node {$\id$} (C2)
		  (C2) edge [bend right,looseness=\los] node {$\id$} (C3)
		  		 edge [bend right,looseness=\los] node[left=1pt] {$h^{-1}$} (C5)
		  (C3) edge [bend left,looseness=\los] node {$h$} (C4)
		  		 edge [bend right,looseness=\los] node[right=2pt] {$\id$} (C4)
		  (C4) edge [bend right,looseness=\los] node[below=5pt] {$\id$} (C5)
		  		 edge [bend left,looseness=\los] node [below=5pt] {$\id$} (C8)
		  (C5) edge [bend right,looseness=\los] node [pos=.3,below=5pt] 
{$t_{\lambda}$} (C1)
		  (C6) edge [bend left,looseness=\los] node[pos=.7,above=5pt] {$\id$} 
(C7)
		  (C7) edge [bend left,looseness=\los] node[right=2pt] {$\id$} (C8)
		  		 edge [bend left,looseness=\los] node[pos=.5, above=5pt] {$h$} 
(C3)
		  (C8) edge [bend left,looseness=\los] node[pos=.2, below=10pt] 
{$ht_{\lambda}^{-1}$} (C6);
\end{scope}
\end{tikzpicture}
\caption[halfspaces]{The return graph $\ReturnGraph{0}$ for the weakly non-collapsing set of 
branches~$\BrS'$ from Example~\ref{EX:G3wnc} for~$\Gamma_{\lambda}$.}\label{FIG:G3returngraphwnc}
\end{figure}

The transfer operator associated to this set of branches than reads as
\[
\TO{s}'=\begin{pmatrix}
0&				0&				0&			0&\alpha_s(t_{\lambda}^{-1})&0&			
0&				0\\
1&				0&				0&				0&				0&				
0&				0&				0\\
0&				1&				0&				0&				0&				
0&	\alpha_s(h^{-1})&	0\\
0&				0&  1+\alpha_s(h^{-1})&  0&				0&				0&			
	0&				0\\
0&		\alpha_s(h)&		0&				1&				0&				0&		
		0&				0\\
0&				0&				0&				0&				0&		 0&	   
0&\alpha_s(t_{\lambda}h^{-1})\\
0&				0&				0&				0&				0&				
1&				0&				0\\
0&				0&				0&				1&				0&				0&	
1&	0
\end{pmatrix}.
\]
\end{example}

\subsection{Algorithms for branch reduction}\label{SUBSEC:branchred}

We recall that $\BrS = \{\Cs{1},\dots,\Cs{N}\}$ is a (fixed, given) set of 
branches for the geodesic flow~$\GeoFlow$ on~$\Orbi$, and that $\BrU=\bigcup_{j=1}^N\Cs{j}$ denotes its branch union and $\CrSc=\pi(\BrU)$ the arising cross section of~$\GeoFlow$, where $\pi\colon\UTB\H\to\UTB\Orbi$ denotes the canonical projection.
We set 
\[
A_0\coloneqq A = \{1,\dots,N\}\,,\quad \Trans{0}{j}{k} \coloneqq \Trans{}{j}{k}
\]
for all~$j,k\in A_0$, and 
\[
\Heir{0}{\ell}\coloneqq\defset{k\in A_0}{\Trans{0}{\ell}{k}\ne\varnothing}
\]
for all~$\ell\in A_0$. In what follows, we present the \emph{branch reduction algorithm}, split into the two Algorithms~\ref{nodereductionI} and~\ref{nodereductionII}, that reduces the return graph by constructing a (finite) cascade of subsets~$A_r$ of $A_0$ and 
related sets~$\Heir{r}{j}$ and~$\Trans{r}{j}{k}$, for 
$r=1,2,\dots$, until we have achieved that $j\in\Heir{r}{j}$ for all~$j\in 
A_r$. The algorithm includes choices and, depending on the group~$\Gamma$, the 
cardinality of the set of remaining nodes might vary for 
different choices. 
(We refrain to fix these choices in any artificial way and hence slightly abuse the notion of algorithm here.)
This phenomenon potentially leads to different families of 
transfer operators, reflecting the inherent ambiguity of discretization of 
flows 
on quotient spaces. Fast transfer operators with the same spectral parameter arising from different such choices need not be mutually conjugate. 
However, their Fredholm determinants will coincide, as is guaranteed by the 
combination of 
Theorems~\ref{mainthm:F_P_NECM} and~\ref{mainthm} below, and hence the 
$1$-eigenfunctions of the two families of fast transfer operators are closely 
linked. Therefore, we observe independence of these choices in this aspect.

As mentioned, the branch reduction algorithm naturally splits into two stand-alone parts, each of which we present below as separate procedures. The first part, presented as 
Algorithm~\ref{nodereductionI}, removes those nodes of the return graph that 
only have a single outgoing edge and the edge does not loop back to the same 
node. The second part, presented as Algorithm~\ref{nodereductionII},  
successively deletes nodes which are not among their own successors.

For convenience we set $R_0\coloneqq A_0$ and, for any $j\in A_0$, 
\[
 \Trans{0}{j}{.} \coloneqq \bigcup_{k\in A_0} \Trans{0}{j}{k}\,.
\]
Likewise, as soon as, for some $r\in\N$, the sets~$A_r$ and $\Trans{r}{j}{k}$ for $j,k\in A_r$ are defined during the run of Algorithm~\ref{nodereductionI}, we set 
\[
 \Trans{r}{j}{.} \coloneqq \bigcup_{k\in A_r} \Trans{r}{j}{k}\,.
\]

\index[defs]{node reduction}%
\begin{algo}[Branch reduction, part I]\label{nodereductionI}
The index $r$ starts at $1$.
\vspace*{-\baselineskip}
\vspace*{4pt}
\begin{enumerate}
\renewcommand{\labelenumi}{\emph{Step\,r.}}
\item Set 
\[
R_r\coloneqq\defset{j\in 
A_{r-1}}{\#\Heir{r-1}{j}=1\land\Heir{r-1}{j}\ne\{j\}\land\#\Trans{r-1}{j}{.}=1}
\,.
\]
If $R_r=\varnothing$, the algorithm terminates. Otherwise choose $j\in R_r$, 
set
\[
A_r\coloneqq A_{r-1}\setminus\{j\}\,,
\]
and define for all $i,k\in A_r$
\[
\Trans{r}{i}{k}\coloneqq\Trans{r-1}{i}{k}\cup\bigcup_{g_1\in\Trans{r-1}{i}{j}}
\bigcup_{g_2\in\Trans{r-1}{j}{k}}\{g_1g_2\}
\]
and
\[
\Heir{r}{i}\coloneqq\defset{\ell\in A_r}{\Trans{r}{i}{\ell}\ne\varnothing}\,.
\]
Carry out \emph{Step r+1}.\algofin
\end{enumerate}
\end{algo}
\index[symbols]{R@$R_r$}%
\index[symbols]{H@$\Heir{r}{j}$}%
\index[symbols]{Gr@$\Trans{r}{i}{k}$}%

In each step of Algorithm~\ref{nodereductionI} a node, say~$j$, is chosen and 
deleted (in the sense that $j\in A_{r-1}$ but $j\notin A_r$). Subsequently, 
each pair of an incoming edge and an outgoing edge of~$j$ is combined to a new 
edge, thereby ``bridging'' above~$j$. More precisely, suppose that there is an 
edge from~$i$ to~$j$ weighted by~$g_1$ and an edge from~$j$ to~$k$ weighted 
by~$g_2$ for some $i,k\in A_{r-1}\setminus\{j\}$ and $g_1,g_2\in\Gamma$, then we 
combine these to an edge from~$i$ to~$k$ weighted by~$g_1g_2$. We note that if, 
for $j,k\in A_{r-1}$, we have $k\notin\Heir{r-1}{j}$, then 
$\Trans{r-1}{j}{k}=\varnothing$ and hence $\Trans{r}{i}{k}=\Trans{r-1}{i}{k}$ 
for all $i\in A_{r}$. 

Let $\kappa_1\in\N_0$
\index[symbols]{ka@$\kappa_1$}%
be the unique 
number for which $R_{\kappa_1}\not=\varnothing$ but 
$R_{\kappa_1+1}=\varnothing$. In other words, $\kappa_1+1$ is the step in which 
Algorithm~\ref{nodereductionI} terminates. (See 
Proposition~\ref{PROP:algorithmsbehave} for its existence.) Then 
Algorithm~\ref{nodereductionI} constructed the sets $A_r$, $\Heir{r}{j}$, and 
$\Trans{r}{j}{k}$ for all $r\in\{1,\dots,\kappa_1\}$ and $j,k\in A_r$.
(We emphasize that in the case~$\kappa_1=0$, Algorithm~\ref{nodereductionI} is void and does not construct any new sets.)

The second part of the branch reduction algorithm, Algorithm~\ref{nodereductionII} below, 
is almost identical to Algorithm~\ref{nodereductionI}. The only, but very 
important difference is the base set from which the nodes are chosen in each 
step.

\index[defs]{node reduction}%
\begin{algo}[Branch reduction, part II]\label{nodereductionII}
The index $r$ starts at $\kappa_1+1$.
\vspace*{-\baselineskip}
\vspace*{4pt}
\begin{enumerate}
\renewcommand{\labelenumi}{\emph{Step\,r.}}
\item Define $P_r\coloneqq\defset{j\in A_{r-1}}{j\notin\Heir{r-1}{j}}$. If 
$P_r=\varnothing$, the algorithm terminates. Otherwise choose $j\in P_r$, set
\[
A_r\coloneqq A_{r-1}\setminus\{j\},
\]
and define for all $i,k\in A_r$
\[
\Trans{r}{i}{k}\coloneqq\Trans{r-1}{i}{k}\cup\bigcup_{g_1\in\Trans{r-1}{i}{j}}
\bigcup_{g_2\in\Trans{r-1}{j}{k}}\{g_1g_2\}
\]
and
\[
\Heir{r}{i}\coloneqq\defset{\ell\in A_{r}}{\Trans{r}{i}{\ell}\ne\varnothing}\,.
\]
Carry out \emph{Step r+1}.\algofin
\end{enumerate}
\end{algo}
\index[symbols]{P@$P_r$}%

Let $\kappa_2\in\N_0$
\index[symbols]{kappa@$\kappa_2$}%
be defined analogously to~$\kappa_1$ but 
with respect to Algorithm~\ref{nodereductionII}. That is, $\kappa_2+1$ shall be 
the step in which Algorithm~\ref{nodereductionII} terminates. In other words, if 
we set $P_0\coloneqq A_0$, then $\kappa_2$ is the unique number such that 
$P_{\kappa_2}\not=\varnothing$ but $P_{\kappa_2+1} = \varnothing$.  In 
Proposition~\ref{PROP:algorithmsbehave} we will show that $\kappa_2$ is indeed 
well-defined. 

For each $r\in \{1,\ldots, \kappa_2\}$ we define the \emph{return graph of level 
r}, $\ReturnGraph{r}$,
\index[defs]{return graph!level}%
\index[symbols]{RGa@$\ReturnGraph{r}$}%
analogously 
to~$\ReturnGraph{0} = \ReturnGraph{}$, with $A_r$
\index[symbols]{Aa@$A_r$}%
being the set of nodes, and edges and edge 
weights resulting from~$\Trans{r}{.}{.}$. We emphasize that 
it may happen that Algorithm~\ref{nodereductionI} or~\ref{nodereductionII} is 
void, or even both, and consequently $\kappa_1=0$ or $\kappa_2=\kappa_1$, or 
both. See Example~\ref{EX:modularsurface}.

\begin{example}\label{EX:modularsurface}
Consider the modular group~$\Gamma=\mathrm{PSL}_2(\Z)$.
The two elements
\begin{eqnarray*}
s\coloneqq\begin{bmatrix}0&1\\-1&0\end{bmatrix}&\text{and}&t\coloneqq 
\begin{bmatrix}1&1\\0&1\end{bmatrix}
\end{eqnarray*}
form a complete set of generators for~$\Gamma$.
A well-known cross section for the geodesic flow on the modular 
surface~$\quod{\Gamma}{\H}$ is given by the set
\[
\Cs{1}\coloneqq\defset{\nu\in\UTB\H}{\base{\nu}\in(0,\infty)_{\H},\,\gamma_{\nu}
(+\infty)\in(0,+\infty)_{\R}}
\]
(see Figure~\ref{FIG:modsurf}). The set~$\{\Cs{1}\}$ has the structure of a set 
of branches.
From Figure~\ref{FIG:modsurf} we read off that the return 
graph~$\ReturnGraph{0}$ of~$\Gamma$ w.r.t.~$\{\Cs{1}\}$ consists solely of the 
two loops
\begin{eqnarray*}
1\edge{t}1&\text{and}&1\edge{st^{-1}s}1\,.
\end{eqnarray*}
Therefore,~$\Heir{0}{1}=\{1\}$ and the sets~$R_1$ and~$P_1$ from the 
Algorithms~\ref{nodereductionI}~and~\ref{nodereductionII} are empty. 
Consequently, we find~$0=\kappa_1=\kappa_2$.

\begin{figure}[h]
\begin{tikzpicture}[scale=5]
\def \height {1.5}
\tikzmath{\p=1/(2*sqrt(3));}
\foreach \x/\y/\z in {0/$0$/$\Cs{1}$, 1/$1$/$t\act\Cs{1}$}{
		\tikzmath{\q= 50 - 25*\x;}
		\fill[color=gray!\q] (\x,\height) -- (\x,0) -- (\x+.05,0) -- 
(\x+.05,\height) -- cycle;
		\draw (\x,\height) -- (\x,-.03) node [below] {\y};
		\coordinate [label=right:\z] (C\x) at (\x+.05,\height*.6);
		}
\fill[color=gray!25] (1,0) arc (0:180:.5) -- (0.05,0) arc (180:0:0.5 - .05) -- 
cycle;
\draw (0,0) arc (180:0:.5);
\fill[pattern=north west lines, pattern color=lightgray!70] (1,1) arc (90:120:1) 
arc (60:90:1) -- (0,\height) -- (1,\height) -- cycle;
\coordinate [label=below:$\color{gray}\fund$] (F) at (.5,\height-.3);
\coordinate [label=below:$st^{-1}s\act\Cs{1}$] (C2) at (.5,.45);
\draw[style=thick] (-.2,0) -- (1.2,0);
\end{tikzpicture}
\caption[halfspaces]{A fundamental domain for the modular group alongside the 
cross section~$\Cs{1}$ for the geodesic flow on the modular 
surface~$\quod{\PSL_2(\Z)}{\H}$.}\label{FIG:modsurf}
\end{figure}
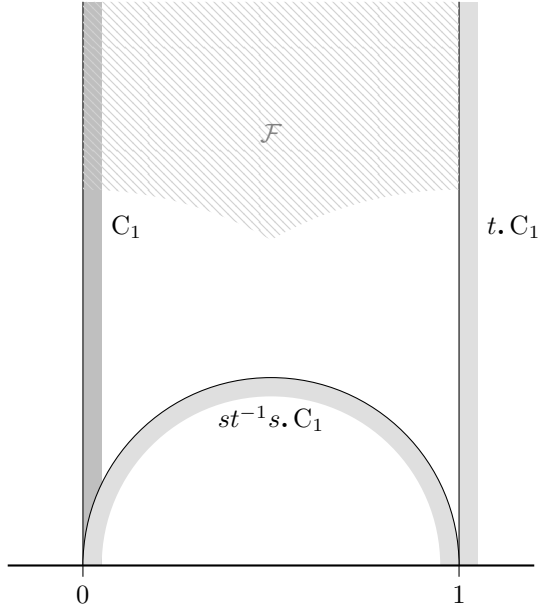
\end{example}

With these preparations we can now show that Algorithms~\ref{nodereductionI} and~\ref{nodereductionII} are indeed correct and provide sets of branches. For any $r\in\{0,\ldots, \kappa_2\}$ we set
\[
 \BrS_r\coloneqq \defset{\Cs{j}}{j\in A_r}
\]
and
\[
\BrU^{(r)} \coloneqq \bigcup_{j\in A_r}\Cs{j}\,.
\]

\begin{prop}\label{PROP:algorithmsbehave}
Algorithms~\ref{nodereductionI} and~\ref{nodereductionII} terminate after 
finitely many steps without reducing the set of nodes to the empty set.  
Further, for each $r\in\{0,\dots,\kappa_2\}$ the family~$\BrS_r$ is a set of 
branches for the geodesic flow on~$\Orbi$. The 
family~$(\Trans{r}{j}{k})_{j,k\in A_r}$ is the family of forward transition sets 
in~\eqref{BP:intervaldecomp}. If $\BrS_0$ is admissible, then $\BrS_r$ is 
admissible as well.
\end{prop}

\begin{proof}
In each step of Algorithms~\ref{nodereductionI} and~\ref{nodereductionII}, 
one element of the set~$A_0$ gets eliminated, resulting in the decreasing 
cascade of subsets
\[
 \ldots \varsubsetneq A_3\varsubsetneq A_2 \varsubsetneq A_1 \varsubsetneq A_0\,.
\]
Thus, the number of steps in both algorithms is bounded from above by $\#A_0 < 
+\infty$. In turn, both algorithms terminate (after finitely many 
steps) and hence $\kappa_1$ and $\kappa_2$ are well-defined. We first show that 
$A_{\kappa_1}\not=\varnothing$. If $\kappa_1=0$, then $A_{\kappa_1} = A_0 
\not=\varnothing$. Thus, suppose now that $\kappa_1\geq 1$. To seek a 
contradiction, we assume that $A_{\kappa_1} = \varnothing$. Thus 
$A_{\kappa_1-1}$ contains exactly one elements, say $j_0$. Then either 
$\#\Heir{\kappa_1-1}{j_0}=1$ and hence $\Heir{\kappa_1-1}{j_0} = \{j_0\}$, or 
$\#\Heir{\kappa_1-1}{j_0}=0$ and hence $\#\Trans{\kappa_1-1}{j_0}{j_0}=0$. In 
either case, $R_{\kappa_1} = \varnothing$. This contradicts the definition 
of~$\kappa_1$, and hence $A_{\kappa_1}\not=\varnothing$. We now show that 
$A_{\kappa_2}\not=\varnothing$. If $\kappa_2=\kappa_1$, then $A_{\kappa_2} = 
A_{\kappa_1} \not=\varnothing$. Thus, we suppose now that $\kappa_2>\kappa_1$. 
As before, to seek a contradiction, we assume that $A_{\kappa_2} = 
\varnothing$. 
Again, $A_{\kappa_2-1}$ contains exactly one element, say~$k_0$. By 
Proposition~\ref{PROP:returngraph} we find a cycle of~$\ReturnGraph{0}$ 
that contains~$k_0$. In each step of the node-elimination-processes of 
Algorithms~\ref{nodereductionI} and~\ref{nodereductionII}, at most one node 
(other than~$k_0$) of this cycle gets eliminated. If an elimination of a node in 
the cycle happens, the two adjacent nodes of the eliminated node in the cycle 
get connected by a new edge that combines then the two old ones. Thus, the cycle is 
``preserved'' but shortened and has changed weights. Thus, after the 
step~$\kappa_2-1$, the node $k_0$ is contained in a cycle, which 
is just a loop at~$k_0$. In turn, $P_{\kappa_2}=\varnothing$. This contradicts 
the definition of~$\kappa_2$. Hence, $A_{\kappa_2}\not=\varnothing$.

We now establish the statement on the sets of branches. For any 
$r\in\{0,\ldots,\kappa_2\}$, the family~$\BrS_r = \defset{\Cs{j}}{j\in 
A_r}$ is a subset of the original set of branches~$\BrS=\BrS_0 = 
\{\Cs{1},\dots,\Cs{N}\}$. 
Hence most of the properties that we impose for a set of branches are immediate 
from those of~$\BrS$. Indeed, the only properties which remain to be proven 
for~$\BrS_r$ are~\eqref{BP:coverlimitset} and~\eqref{BP:intervaldecomp}, where 
for the former we take advantage of Proposition~\ref{PROP:oldB1} and 
prove~\eqref{BP:closedgeodesicsXtoH} instead.
For both properties we proceed by an inductive argument and note 
that they are already known to be valid for $\BrS_0 = \BrS$. 

Let $r\in\{0,\ldots, \kappa_2-1\}$ be such that~\eqref{BP:intervaldecomp} 
is already established for~$\BrS_r$ by using, for all $j,k\in A_r$, the set~ 
$\Trans{r}{j}{k}$ for the forward transition set in~\eqref{BP:intervaldecomp}. 
Let $j\in A_{r+1}$ and suppose that $A_r\setminus A_{r+1} = \{p\}$. If 
$p\notin\Heir{r}{j}$, then 
\[
 \Trans{r+1}{j}{k} = \Trans{r}{j}{k}
\]
for all~$k\in\Heir{r+1}{j}$. In this case, 
(\ref{BP:intervaldecomp}\ref{BP:intervaldecompGdecomp}) for $r$ and $r+1$ are 
identical statements for the considered index~$j$ and hence, (\ref{BP:intervaldecomp}\ref{BP:intervaldecompGdecomp}) is valid for $r+1$. If 
$p\in\Heir{r}{j}$, then 
\[
 \Heir{r+1}{j} = \left(\Heir{r}{j}\setminus\{p\}\right)\cup\Heir{r}{p}\,,
\]
where necessarily $p\notin\Heir{r}{p}$. Taking advantage of the inductive 
hypothesis for the first equality below, we find
\begin{align*}
\Iset{j}&=\bigcup_{k\in\Heir{r}{j}}\bigcup_{g\in\Trans{r}{j}{k}}g\act\Iset{k}
\\
&=\bigcup_{k\in\Heir{r}{j}\setminus\{p\}}\bigcup_{g\in\Trans{r}{j}{k}}g\act\Iset
{k}\cup\bigcup_{h\in\Trans{r}{j}{p}}h\act\Iset{p}
\\
&=\bigcup_{k\in\Heir{r}{j}\setminus\{p\}}\bigcup_{g\in\Trans{r}{j}{k}}g\act\Iset
{k}\cup\bigcup_{h\in\Trans{r}{j}{p}}\bigcup_{q\in\Heir{r}{p}}\bigcup_{w\in\Trans
{r}{p}{q}}hw\act\Iset{q}
\\
&=\bigcup_{k\in\Heir{r}{j}\setminus\{p\}}\bigcup_{g\in\Trans{r}{j}{k}}g\act\Iset
{k}\cup\bigcup_{q\in\Heir{r}{p}}\bigcup_{h\in\Trans{r}{j}{p}}\bigcup_{w\in\Trans
{r}{p}{q}}hw\act\Iset{q}
\\
&=\bigcup_{k\in\Heir{r+1}{j}}\bigcup_{g\in\Trans{r+1}{j}{k}}g\act\Iset{k}\,.
\end{align*}
Since $\BrS_r$ is known to satisfy~\eqref{BP:intervaldecomp}, the unions in 
all steps are disjoint. This establishes the second part of 
(\ref{BP:intervaldecomp}\ref{BP:intervaldecompGdecomp}) for~$\BrS_{r+1}$. The 
first part as well as  (\ref{BP:intervaldecomp}\ref{BP:intervaldecompback}) 
follow analogously. Further, 
(\ref{BP:intervaldecomp}\ref{BP:intervaldecompGgeod}) is immediate by the 
construction of~$\BrS_{r+1}$.

Now let $r\in\{0,\ldots, \kappa_2-1\}$ be such 
that~\eqref{BP:closedgeodesicsXtoH} is established for~$\BrS_r$. Together with 
the previous discussion this then already shows that~$\BrS_r$ is a set of 
branches. Suppose that $\widehat\gamma$ is a periodic geodesic on~$\Orbi$. By 
hypothesis, $\wh\gamma$ has a lift to~$\H$ which intersects~$\BrU^{(r)}$. In order 
to show~\eqref{BP:closedgeodesicsXtoH} for~$\BrS_{r+1}$, it remains 
to show that there is also such a lift that intersects~$\BrU^{(r+1)}$. To seek a 
contradiction, we assume that all lifts of~$\widehat\gamma$ 
intersect~$\Gamma\act\BrU^{(r)}$ only on~$\Gamma\act\Cs{p}$. Then 
Proposition~\ref{PROP:CofSoBnonst} implies that $p\in\Heir{r}{p}$, 
which contradicts $p\in A_r\setminus A_{r+1}$ (note that $\BrS_r$ is already 
known to be a set of branches). In turn, $\BrS_{r+1}$ 
satisfies~\eqref{BP:closedgeodesicsXtoH}. 

Finally we suppose that $\BrS$ is admissible. Hence we find $q\in\widehat{\R}$ 
and an open neighborhood~$\mathcal{U}$ of~$q$ in~$\widehat\R$ such 
that $\mathcal{U}\cap\bigcup_{j\in A_0}\Iset{j}=\varnothing$ and~$q\notin I_j$ for any~$j\in A_0$. For any 
$r\in\{0,\ldots,\kappa_2\}$, the inclusion $A_r\subseteq A_0$ implies 
$\bigcup_{j\in A_r}\Iset{j}\subseteq\bigcup_{j\in A_0}\Iset{j}$. Therefore,  
the set of branches~$\BrS_r$ is admissible.
\end{proof}

\begin{example}\label{EX:G3Graphred}
Recall the group~$\Gamma_{\lambda}$ as well as its sets of branches~$\BrS$ and~$\BrS'$ and 
their respective return graphs from Example~\ref{EX:G3Graph}.
A complete 
reduction procedure for $\Gamma_{\lambda}$ takes $6$ steps in total and leads to 
the return graph $\ReturnGraph{6}$ indicated in Figure~\ref{FIG:G3redreturn}.
In this example, it so happens that every possible sequence of choices for the 
Algorithms~\ref{nodereductionI}~and~\ref{nodereductionII} leads to the same 
return graph~$\ReturnGraph{6}$, regardless of whether starts out with~$\BrS$ or with~$\BrS'$.
The arising set of branches~$\{\Cs{2},\Cs{7}\}$ is easily seen to be non-collapsing, for~$\{2,7\}=D_\ini$.

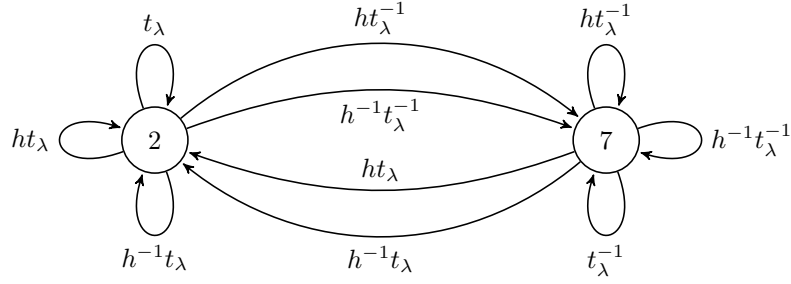
\begin{figure}[h]
\begin{tikzpicture}[->,>=stealth',shorten >=1pt,auto,node distance=6cm, 
semithick]
\node[state] (C2) {$2$};
\node[state] (C7) [right of=C2] {$7$};

\path (C2)	edge [loop above,out=110,in=70,looseness=10] node {$t_{\lambda}$} 
(C2)
					edge [loop left,out=200,in=160,looseness=10] node 
{$ht_{\lambda}$} (C2)
					edge [loop below,out=290,in=250,looseness=10] node 
{$h^{-1}t_{\lambda}$} (C2)
					edge [bend left, out=40, in=140] node {$ht_{\lambda}^{-1}$} 
(C7)
					edge [bend left, out=20, in=160] node [below] 
{$h^{-1}t_{\lambda}^{-1}$} (C7)
		  (C7)	edge [loop above,out=110,in=70,looseness=10] node 
{$ht_{\lambda}^{-1}$} (C7)
		  			edge [loop right,out=20,in=340,looseness=10] node 
{$h^{-1}t_{\lambda}^{-1}$} (C7)
		  			edge [loop below,out=290,in=250,looseness=10] node 
{$t_{\lambda}^{-1}$} (C7)
		  			edge [bend left,out=20, in=160] node [above] 
{$ht_{\lambda}$} (C2)
		  			edge [bend left,out=40, in=140] node {$h^{-1}t_{\lambda}$} 
(C2);
\end{tikzpicture}
\caption[halfspaces]{The maximally reduced return graph $\ReturnGraph{6}$ 
for~$\BrS$.}\label{FIG:G3redreturn}
\end{figure}
\end{example}

As the following proposition establishes, finiteness of ramification is preserved by the branch reduction algorithm.

\begin{prop}\label{PROP:reddontram}
If the set of branches~$\BrS$ is finitely ramified, then $\BrS_r$ is finitely 
ramified for each~$r\in\{0,\ldots,\kappa_2\}$.
\end{prop}

\begin{proof}
If $\Orbi$ has a no cusps, then the statement holds by 
Lemma~\ref{LEM:nocuspnoram}. We now suppose that~$\Orbi$ has cusps, that 
$\BrS=\BrS_0$ is finitely ramified and that $\kappa_2\geq 1$. It suffices to 
show that $\BrS_1$ is finitely ramified as the remaining claims then follow 
immediately by finite induction. By Proposition~\ref{PROP:Attandram} it further 
suffices to show that each cusp of~$\Orbi$ is attached to~$\BrS_1$. 

Let $\wh c$ be a cusp of~$\Orbi$, represented by~$c\in\wh\R$. By hypothesis and 
Proposition~\ref{PROP:Attandram}, $\wh c$ is attached to~$\BrS_0$. Thus, 
\[
\IAtt{\BrS_0}{c}\coloneqq\overline{\bigcup_{(j,h)\in\Att{\BrS_0}{c}} 
h\act I_j}\,,
\]
where $\Att{\BrS_0}{c}\coloneqq\{(j,h)\in A_0\times\Gamma \mid c \in 
h\act\geo\overline{\base{\Cs{j}}}\}$, is a neighborhood of~$c$ in~$\wh\R$. If 
$\Att{\BrS_1}{c} = \Att{\BrS_0}{c}$, then $\wh c$ is also attached to~$\BrS_1$. 
It remains to consider the case that $\Att{\BrS_1}{c} \not=\Att{\BrS_0}{c}$. 
Then there exists $(j_0,h)\in\Att{\BrS_0}{c}$ that is not contained 
in~$\Att{\BrS_1}{c}$. The index~$j_0$ is the (unique) node of~$\ReturnGraph{0}$ 
that gets eliminated in Algorithm~\ref{nodereductionI} or~\ref{nodereductionII}. 
Let $x_0$ be the unique endpoint of~$\overline{\base{\Cs{j_0}}}$, and hence 
of~$I_{j_0}$, such that $h\act x_0 = c$. Since each cusp representative is an 
accumulation point of~$\wh\R_{\st}$, $x_0$ is also an endpoint of~$\Iset{j_0}$. 
By~(\ref{BP:intervaldecomp}\ref{BP:intervaldecompGdecomp}) we find a (unique) 
pair~$(k,g)\in A_0\times \Trans{0}{j_0}{k}$ such that $hg\act \Iset{k}\subseteq 
h\act \Iset{j_0}$, and $c$ is an endpoint of both~$hg\act\Iset{k}$ and~$hg\act\overline{\base{\Cs{k}}}$. It follows that 
\[
 \bigl(\IAtt{\BrS_0}{c} \setminus \overline{h\act I_{j_0}}\bigr) \cup 
\overline{hg\act I_k}
\]
is also a neighborhood of~$c$ in~$\wh\R$. Further, $k\not=j_0$ because 
Algorithms~\ref{nodereductionI} and~\ref{nodereductionII} require $j_0\notin 
H_0(j_0)$ for elimination of~$j_0$. Therefore, $(k,hg)\in \Att{\BrS_1}{c}$. 
Substituting each appearance of~$j_0$ in~$\Att{\BrS_0}{c}$ in this way we 
obtain~$\Att{\BrS_1}{c}$ and find that $\IAtt{\BrS_1}{c}$ is a neighborhood 
of~$c$ in~$\wh\R$. Thus, $\wh c$ is attached to~$\BrS_1$. This completes the 
proof.
\end{proof}

\section{Identity elimination and reduced set of 
branches}\label{SEC:stepredux}

In order to fulfill all requirements of a strict transfer operator approach, it is 
essential to ensure a unique coding of periodic geodesics in terms of the chosen 
set of generators for the underlying Fuchsian group.
This property, which we call the \emph{non-collapsing property}, is codified 
in~\eqref{BP:noidentity}.
In terms of the return graph of the considered set of branches, it states that the weights along edge sequences never combine to the identity.
Sets of branches that initially do not satisfy~\eqref{BP:noidentity} can 
be remodeled via a reduction procedure that removes such identity 
transformations from the system.
In this section we present, discuss and prove such a procedure, which we call \emph{identity elimination}.

Let~$\BrS=\{\Cs{1},\dots,\Cs{N}\}$ be a set of branches for the geodesic flow 
on~$\Orbi$.
Proposition~\ref{PROP:noncollaps} allows us to suppose~$\BrS$ to be weakly non-collapsing 
(see~\eqref{BP:noncollaps}) without shrinking the realm of applicability of the identity elimination algorithm. This additional hypothesis eliminates some technical subleties. It assures that all identity transformations present in the system are 
\emph{visible} to the algorithm as no identities are concealed 
by~$\Gamma$-copies of~$\BrS$.

Let~$\kappa_2$,~$A_r$,~$\Heir{r}{j}$, and~$\Trans{r}{j}{k}$ for~$j,k\in A_r$ 
and~$r\in\{1,\dots,\kappa_2\}$ be as in Section~\ref{SUBSEC:branchred}.
The procedure discussed in this section is uniform in the ``level of 
reduction'', meaning uniform with respect to~$r\in\{0,\dots,\kappa_2\}$. For that reason we omit the subscript~$r$ throughout.
However, we remark that a sufficient level of reduction might already render the 
emerging set of branches non-collapsing, as can be observed, e.g., in the situation of the initial and the fully reduced return graph for the group~$\Gamma_\lambda$ from 
Example~\ref{EX:G3Def}, indicated in Figure~\ref{FIG:G3returngraph} and 
Figure~\ref{FIG:G3redreturn}, respectively.
However, this is not always the case, not even in the weakly non-collapsing 
case.

Recall the subset~$D_\ini$ of~$A$ of indices of initial branches from~\eqref{eq:D_initial}.
Recall further the branch trees~$B_j$,~$j\in A$, and the branch forest~$F_\ini$ for~$\BrS$ from Section~\ref{SUBSEC:reducedbranches}.
For every~$i\in D_\ini$ consider all nodes in~$B_i$ of the form~$(*,\id)$.
Since~$\BrS$ is weakly non-collapsing, these nodes (and their connecting edges) form a sub-tree~$B'_i$ of~$B_i$ of finite depth.
\index[symbols]{Bzjj@$B'_j$}%
Furthermore, for the same reason, for every~$k\in A$ there exists~$j\in D_\ini$ such that~$B'_j$ contains the node~$(k,\id)$.
Denote by~$F'_\ini$ the forest of sub-trees~$B'_i$,~$i\in D_\ini$.
\index[symbols]{Finii@$F'_\ini$}%
The forest~$F'_\ini$ can be seen as a disconnected, directed graph of which each connected component is a tree of nodes of the form~$(*,\id)$. Therefore, each (directed) path in~$F'_\ini$ can be indexed by a tuple consisting of the index (the element in~$A$) of the root node of its super-tree and the index of its end node. (It is necessary to consider the root node in this indexing as well, because end nodes for different paths in~$F'_\ini$ may coincide.)
We denote by~$\Delta_\ini\subseteq D_\ini\times A$ the set of these indices.
\index[symbols]{Delini@$\Delta_\ini$}%
By construction,~$\#\Delta_\ini<+\infty$, and for every~$i\in D_\ini$ the set~$\Delta_\ini$ contains an element of the form~$(i,*)$.
To each path~$(i,k)\in\Delta_\ini$ we assign its length~$\ell_{(i,k)}\in\N_0$, that is the level of~$(k,\id)$ in~$B'_i$.
Then~$\ell_{(i,k)}$ is the unique integer for which there exists~$\nu\in\Cs{i}$ such that
\[
\bigl(\itindex{\BrU,\ell_{(i,k)}}(\nu),\ittrans{\BrU,\ell_{(i,k)}}(\nu)\bigr)=(k,\id)\,,
\]
with~$\itindex{\BrU,*}(\nu)$ as in~\eqref{EQDEF:itindex} and~$\ittrans{\BrU,*}(\nu)$ as in~\eqref{EQDEF:ittrans1} and~\eqref{EQDEF:ittrans2}.
We further assign a sequence~$(a^{\delta}_n)_{n=0}^{\ell_{\delta}}$
\index[symbols]{adn@$a_n^{\delta}$}%
in~$A$ to each~$\delta=(i,k)\in\Delta_\ini$ by imposing that~$(i,k)$ indexes the path
\[
(i,\id)=(a^{(i,k)}_{\ell_{(i,k)}},\id)\edge{\phantom{-.}}(a^{(i,k)}_{\ell_{(i,k)}-1},\id)\edge{\phantom{-.}}\dots\edge{\phantom{-.}}(a^{(i,k)}_{1},\id)\edge{\phantom{-.}}(a^{(i,k)}_{0},\id)=(k,\id)\,.
\]
The following algorithm (Algorithm~\eqref{stepreduction} below) will traverse paths in backwards direction.
That justifies the counter-intuitive numbering of the members of the (finite) sequence~$(a^{\delta}_n)$.
The possibility of multiple occurrences of a single node~$(k,\id)$ in~$F'_\ini$ necessitates an iterative approach, where certain branches and transition sets might be redefined several times.
We therefore initialize the procedure by setting
\[
\ell_{\delta_0}\coloneqq 0\,,\qquad \Transss{\ell_{\delta_0}}{0}{j}{k}\coloneqq\Trans{}{j}{k}\qquad\text{and}\qquad\Css{j}{0,0}\coloneqq\Cs{j}
\]
\index[symbols]{G1@$\Transss{n}{r}{j}{k}$}%
\index[symbols]{Cjr@$\Css{j}{r,s_r}$}%
for all~$j,k\in A$. 
Further we fix an (arbitrary) enumeration of~$\Delta_\ini$. Thus, 
\begin{equation}\label{EQ:numberDelta}
\Delta_\ini=\{\delta_1,\dots,\delta_\eta\}\,,
\end{equation}
with~$\eta\coloneqq\#\Delta_\ini$.

\begin{algo}[Identity elimination]\label{stepreduction}
The index~$r$ runs from~$1$ to~$\eta$.
\vspace*{-\baselineskip}
\vspace*{4pt}
\begin{enumerate}
\renewcommand{\labelenumi}{\emph{Step\,$r$.}}
\item
Set
\[
\Transss{0}{r}{\lambda}{\kappa}\coloneqq\Transss{\ell_{\delta_{r-1}}}{r-1}{\lambda}{\kappa}\qquad\text{and}\qquad\Css{\lambda}{r,0}\coloneqq\Css{\lambda}{r-1,\ell_{\delta_{r-1}}}
\]
for all~$\lambda,\kappa\in A$.
The index~$s_r$ runs from~$1$ to~$\ell_{\delta_r}$.
\vspace*{3pt}
\begin{enumerate}
\renewcommand{\labelenumii}{\emph{Substep\,$s_r$.}}
\item
For all~$j\in A$ define
\[
\Transss{s_r}{r}{j}{a_{s_r-\ell}^{\delta_r}}\coloneqq\Transss{s_r-1}{r}{j}{a_{s_r-\ell}^{\delta_r}}\setminus\{\id\}\cup\Transss{s_r-1}{r}{j}{a_{s_r}^{\delta_r}}\,,
\]
for~$\ell=1,\dots,s_r$, as well as
\[
\Transss{s_r}{r}{j}{\kappa}\coloneqq\Transss{s_r-1}{r}{j}{\kappa}\,,
\]
for all~$\kappa\in A\setminus\{a_0^{\delta_r},\dots,a_{s_r-1}^{\delta_r}\}$.
Further set
\[
V_{s_r}^{(r)}\coloneqq\defset{\nu\in\Css{a_{s_r}^{\delta_r}}{r,s_r-1}}{\bigl(\itindex{\BrU,1}(\nu),\ittrans{\BrU,1}(\nu)\bigr)=\bigl(a_{s_r-1}^{\delta_r},\id\bigr)}
\]
\index[symbols]{Vsr@$V_{s_r}^{(r)}$}%
and define
\[
\Css{j}{r,s_r}\coloneqq\Css{j}{r,s_r-1}\setminus V_{s_r}^{(r)}\,.
\]\algofin
\end{enumerate}
\end{enumerate}
\end{algo}

Let~$r,r'\in\{1,\dots,\eta\}$ and~$s_r\in\{1,\dots,\ell_{\delta_r}\}$,~$s_{r'}\in\{1,\dots,\ell_{\delta_{r'}}\}$.
We define a relation~``$\leq$'' on the set of pairs~$(r,s_r)$ by setting
\begin{equation}\label{EQDEF:orderrsr}
(r,s_r)\leq(r',s_{r'})\quad\logeq\quad r<r'~\lor~\left(r=r'\,\land\,s_r\leq s_{r'}\right)\,.
\end{equation}
Then~``$\leq$'' is a total order.

\begin{lemma}\label{LEM:VsrCsr}
For $j\in A$, $r,r'\in\{1,\dots,\eta\}$ and~$s_r\in\{1,\dots,\ell_{\delta_r}\}$,~$s_{r'}\in\{1,\dots,\ell_{\delta_{r'}}\}$ we have~
\begin{enumerate}[label=$\mathrm{(\roman*)}$, ref=$\mathrm{\roman*}$]
\item\label{VsrCsr:Vsrdisjoint}
$V_{s_r}^{(r)}\cap V_{s_{r'}}^{(r')}\ne\varnothing$ if and only if~$V_{s_r}^{(r)}=V_{s_{r'}}^{(r')}$,
\item\label{VsrCsr:fullunion}
$ \Css{j}{r,s_r}=\Cs{j}\setminus\biggl(\,\bigcup\limits_{p=1}^{r-1}\bigcup\limits_{\ell=1}^{\ell_{\delta_p}}V_{\ell}^{(p)}\cup\bigcup\limits_{i=1}^{s_r}V_{i}^{(r)}\biggr)$,
\item\label{VsrCsr:CsrinC}
$\Css{j}{r',s_{r'}}\subseteq\Css{j}{r,s_r}\subseteq\Cs{j}$ if and only if~$(r,s_r)\leq(r',s_{r'})$.
\end{enumerate}
\end{lemma}

\begin{proof}
Statement~\eqref{VsrCsr:Vsrdisjoint} is immediate from the definition of the sets~$V_{s_r}^{(r)}$ in Algorithm~\ref{stepreduction} and the uniqueness of the system of iterated sequences from~\eqref{eq:itseq} for any given~$\nu\in\BrU$.
Statement~\eqref{VsrCsr:fullunion} follows by straightforward, recursive application of the definition of~$\Css{j}{r,s_r}$ in Algorithm~\ref{stepreduction}. From this presentation of~$\Css{j}{r,s_r}$ we obtain that, if~$r<r'$,
\begin{align*}
\Cs{j}\setminus\Css{j}{r',s_{r'}}&=\bigcup_{p=1}^{r'-1}\bigcup_{\ell=1}^{\ell_{\delta_p}}V_{\ell}^{(p)}\cup\bigcup_{i=1}^{s_{r'}}V_{i}^{(r')}\\
&=\bigcup_{p=1}^{r-1}\bigcup_{\ell=1}^{\ell_{\delta_p}}V_{\ell}^{(p)}\cup\bigcup_{i=1}^{s_{r}}V_{i}^{(r)}\cup\bigcup_{i=s_r+1}^{\ell_{\delta_r}}V_i^{(r)}\cup \bigcup_{p=r+1}^{r'-1}\bigcup_{\ell=1}^{\ell_{\delta_p}}V_{\ell}^{(p)}\cup\bigcup_{i=1}^{s_{r'}}V_{i}^{(r')}\\
&\supseteq\Cs{j}\setminus\Css{j}{r,s_r}\,,
\end{align*}
and, if~$r=r'$ and~$s_r\leq s_{r'}$,
\[
\Cs{j}\setminus\Css{j}{r',s_{r'}}=\bigcup_{p=1}^{r'-1}\bigcup_{\ell=1}^{\ell_{\delta_p}}V_{\ell}^{(p)}\cup\bigcup_{i=1}^{s_{r}}V_{i}^{(r)}\cup\bigcup_{i=s_r+1}^{s_{r'}}V_i^{(r')}\supseteq\Cs{j}\setminus\Css{j}{r,s_r}\,.
\]
This immediately yields~\eqref{VsrCsr:CsrinC}.
\end{proof}

From Lemma~\ref{LEM:VsrCsr}\eqref{VsrCsr:fullunion} we obtain that, for every~$j\in A$, Algorithm~\ref{stepreduction} ultimately defines the set of unit tangent vectors 
\begin{equation}\label{EQDEF:redCS}
\Csr{j}\coloneqq\Css{j}{\eta,\ell_{\delta_\eta}}=\Cs{j}\setminus\bigcup_{r=1}^{\eta}\bigcup_{s=1}^{\ell_{\delta_\eta}}V_{s}^{(r)}\,.
\end{equation}
\index[symbols]{Cjrr@$\Csr{j}$}%
Accordingly, we set
\begin{equation}\label{EQDEF:redInterval}
\Ired{j}\coloneqq I_j\setminus\bigcup_{r=1}^{\eta}\bigcup_{s=1}^{\ell_{\delta_\eta}}\defset{\gamma_\nu(+\infty)}{\nu\in V_s^{(r)}}\,.
\end{equation}
\index[symbols]{Iab@$\Ired{j}$}%
Depending on the initial set of branches~$\BrS$ and the level of reduction, 
Algorithm~\ref{stepreduction} might render single branches essentially empty, in 
the sense that~$\Csr{j,\st}=\varnothing$.
We account for this possibility by updating the index set to be
\begin{equation}\label{EQNDEF:Atilde}
\wt A\coloneqq\defset{k\in A}{\Csr{k,\st}\ne\varnothing}\,.
\end{equation}
\index[symbols]{Abb@$\wt A$}%
From~\eqref{EQDEF:redCS} and the definition of the sets~$V_{*}^{(*)}$ we read off that the branch~$\Cs{j}$ can only be rendered essentially empty by Algorithm~\ref{stepreduction} if~$\bigcup_{k\in A}\Trans{}{j}{k}=\{\id\}$.
This implies in particular that
\[
\defset{j\in A}{(*,j)\in\Delta_\ini}\subseteq\wt A\,.
\]
Hence,~$\wt A\ne\varnothing$.
We further define
\begin{equation}\label{EQDEF:redTranssets}
\Transs{}{j}{k}\coloneqq\Transss{\ell_{\delta_\eta}}{\eta}{j}{k}\,,
\end{equation}
\index[symbols]{G@$\Transs{}{j}{k}$}%
for~$j,k\in\wt A$, as well as
\begin{equation}\label{EQDEF:redBrSredBrU}
\wt\BrS\coloneqq\defset{\Csr{j}}{j\in\wt A}\qquad\text{and}\qquad\wt\BrU\coloneqq\bigcup\wt\BrS = \bigcup_{j\in \wt A} \Csr{j}\,.
\end{equation}
\index[symbols]{C@$\wt\BrS$}%
\index[symbols]{C@$\wt\BrU$}%
From~\eqref{EQDEF:redInterval} it is apparent that the sets~$\Csr{j}$, $j\in A$, are not necessarily branches in the sense of Definition~\ref{DEF:setofbranches} anymore, due to possible violation of~\eqref{BP:allvectors}.
See also Definition~\ref{DEF:redsetofbranches} below.
Nevertheless, we continue to call them \emph{branches} and the sets~$\Transs{}{j}{k}$ \emph{(forward) transition sets} for~$j,k\in A$.

We now analyze the structure and interrelation of the sets of unit tangent vectors and transformations successively defined by Algorithm~\ref{stepreduction}.
To that end, we let $r\in\{0,\dots,\eta\}$, $s_r\in\{0,\dots,\ell_{\delta_r}\}$, and set
\[
\BrS_{s_r}^{(r)}\coloneqq\defset{\Css{j}{r,s_r}}{j\in \wt A}\quad\text{and}\quad\BrU_{s_r}^{(r)}\coloneqq\bigcup\BrS_{s_r}^{(r)} = \bigcup_{M \in \BrS_{s_r}^{(r)}} M\,.
\]
\index[symbols]{Crr@$\BrS_{s_r}^{(r)}$}%
\index[symbols]{Crr@$\BrU_{s_r}^{(r)}$}%
Then, obviously,
\begin{equation}\label{EQ:BrUsr}
\BrU_{s_r}^{(r)}=\BrU_{s_r-1}^{(r)}\setminus V_{s_r}^{(r)}\,.
\end{equation}
For~$\nu\in\BrU_{s_r}^{(r)}$ we define a system of sequences
\begin{equation}\label{EQDEF:algoitseq}
\bigl[\bigl(\ittime{\BrU_{s_r}^{(r)},n}(\nu)\bigr)_n,\bigl(\itindex{\BrU_{s_r}^{(r)},n}(\nu)\bigr)_n,\bigl(\ittrans{\BrU_{s_r}^{(r)},n}(\nu)\bigr)_n\bigr]
\end{equation}
as in~\eqref{eq:itseq}, with~$\BrU_{s_r}^{(r)}$ in place of~$\BrU$.

\begin{lemma}\label{LEM:FSOTS}
Let~$r\in\{0,\dots,\eta\}$ and~$s_r\in\{0,\dots,\ell_{\delta_r}\}$.
For~$\nu\in\BrU_{s_r}^{(r)}$ the system of sequences in~\eqref{EQDEF:algoitseq} is well-defined.
Furthermore, the set~$\defset{\Transss{s_r}{r}{j}{k}}{j,k\in \wt A}$
is a full set of transition sets for~$\BrS_{s_r}^{(r)}$ up to identities, in the sense that
\begin{enumerate}[label=$\mathrm{(\roman*)}$, ref=$\mathrm{\roman*}$]
\item\label{FSOTS:alltrans}
$\forall\, j,k\in A\ \forall\,\nu\in\Css{j}{r,s_r},\itindex{\BrU_{s_r}^{(r)},1}(\nu)=k\colon\ittrans{\BrU_{s_r}^{(r)},1}(\nu)\in\Transss{s_r}{r}{j}{k}\cup\{\id\}\,,$
\item\label{FSOTS:vector}
$\forall\, j,k\in A\ \forall\, g\in\Transss{s_r}{r}{j}{k}\ \exists\,\nu\in\Css{j}{r,s_r}\colon(\itindex{\BrU_{s_r}^{(r)},1}(\nu),\ittrans{\BrU_{s_r}^{(r)},1}(\nu))=(k,g)\,.$
\end{enumerate}
\end{lemma}

\begin{proof}
We argue by induction over the totally ordered set of pairs~$(r,s_r)$,~$r\in\{1,\dots,\eta\}$,~$s_r\in\{1,\dots,\ell_{\delta_r}\}$.
For~$r=s_r=0$ there is nothing to show.
We suppose that all claims have already been proven for~$(r,s_r-1)$ for some $r\geq0$.
Let~$\nu\in\BrU_{s_r}^{(r)}$.
From~\eqref{EQ:BrUsr} we read off that
\[
\ittime{\BrU_{s_r}^{(r)},n}(\nu)=\ittime{\BrU_{s_r-1}^{(r)},m_n}(\nu)\,,\quad\itindex{\BrU_{s_r}^{(r)},n}(\nu)=\itindex{\BrU_{s_r-1}^{(r)},m_n}(\nu)\,,
\]
and
\begin{equation}\label{EQ:ittranssr}
\ittrans{\BrU_{s_r}^{(r)},n}(\nu)=\ittrans{\BrU_{s_r-1}^{(r)},n}(\nu)\cdots\ittrans{\BrU_{s_r-1}^{(r)},m_n}(\nu)\,,
\end{equation}
where, \textit{a priori},~$m_0\coloneqq 0$ and
\[
m_n\coloneqq\left\{
\begin{array}{rl}
\min\defset{m\geq n}{\gamma_\nu'\bigl(\ittime{\BrU_{s_r-1}^{(r)},m}(\nu)\bigr)\notin\Gamma\act V_{s_r}^{(r)}}&\text{for }n>0\,,\\
\rule{0pt}{4ex}\max\defset{m\leq n}{\gamma_\nu'\bigl(\ittime{\BrU_{s_r-1}^{(r)},m}(\nu)\bigr)\notin\Gamma\act V_{s_r}^{(r)}}&\text{for }n<0\,.
\end{array}
\right.
\]
But since each node~$(j,\id)$ appears at most once in the path~$\delta_r$, we have~$a_{s_r}^{\delta_r}\ne a_{s_r-1}^{\delta_r}$ and hence
\begin{equation}\label{EQ:simplemn}
m_n=\left\{
\begin{array}{cl}
n&\text{if }\gamma_\nu'\bigl(\ittime{\BrU_{s_r-1}^{(r)},m}(\nu)\bigr)\notin\Gamma\act V_{s_r}^{(r)}\,,\\
n+1&\text{if }\gamma_\nu'\bigl(\ittime{\BrU_{s_r-1}^{(r)},m}(\nu)\bigr)\in\Gamma\act V_{s_r}^{(r)}\text{ and }n>0\,,\\
n-1&\text{if }\gamma_\nu'\bigl(\ittime{\BrU_{s_r-1}^{(r)},m}(\nu)\bigr)\in\Gamma\act V_{s_r}^{(r)}\text{ and }n<0\,,
\end{array}
\right.
\end{equation}
Since all objects exist by hypothesis, the system of sequences from~\eqref{EQDEF:algoitseq} is well-defined.
We set
\[
j\coloneqq\itindex{\BrU_{s_r}^{(r)},0}(\nu)\,,\quad k\coloneqq\itindex{\BrU_{s_r}^{(r)},1}(\nu)\quad\text{and}\quad g\coloneqq\ittrans{\BrU_{s_r}^{(r)},1}(\nu)\,.
\]
Concerning statement~\eqref{FSOTS:alltrans} we have to show that
\begin{equation}\label{eq:setcorrect}
g\in\Transss{s_r}{r}{j}{k}\cup\{\id\}\,.
\end{equation}
To that end note that
\[
\Transss{s_r-1}{r}{j}{k}\subseteq\Transss{s_r}{r}{j}{k}\cup\{\id\}\,.
\]
If~$k\ne a_{s_r}^{\delta_r}$, then 
\[
\gamma_\nu'\bigl(\ittime{\BrU_{s_r-1}^{(r)},1}(\nu)\bigr)\in\BrU_{s_r-1}^{(r)}\setminus\Css{a_{s_r}^{\delta_r}}{r,s_r-1}\subseteq\BrU_{s_r-1}^{(r)}\setminus V_{s_r}^{(r)}\,.
\]
Hence, by the discussion above and the hypothesis,
\[
g=\ittrans{\BrU_{s_r}^{(r)},1}(\nu)=\ittrans{\BrU_{s_r-1}^{(r)},1}(\nu)\in\Transss{s_r-1}{r}{j}{k}\subseteq\Transss{s_r}{r}{j}{k}\cup\{\id\}\,.
\]
Now let~$k=a_{s_r}^{\delta_r}$.
If~$\gamma_\nu'(\ittime{\BrU_{s_r-1}^{(r)},1}(\nu))\notin \Gamma\act V_{s_r}^{(r)}$, then we may argue as before.
It thus remains to consider the case that 
\[
\gamma_\nu'(\ittime{\BrU_{s_r-1}^{(r)},1}(\nu))\in \Gamma\act V_{s_r}^{(r)}\,.
\]
Then, from the definition of~$V_{s_r}^{(r)}$ we obtain
\[
\ittrans{\BrU_{s_r-1}^{(r)},2}(\nu)=\id\,.
\]
From that,~\eqref{EQ:ittranssr},~\eqref{EQ:simplemn}, and the hypothesis we obtain that
\begin{align*}
g=\ittrans{\BrU_{s_r}^{(r)},1}(\nu)&=\ittrans{\BrU_{s_r-1}^{(r)},1}(\nu)\cdot\ittrans{\BrU_{s_r-1}^{(r)},2}(\nu)\\
&=\ittrans{\BrU_{s_r-1}^{(r)},1}(\nu)\in\Transss{s_r-1}{r}{j}{k}\subseteq\Transss{s_r}{r}{j}{k}\,.
\end{align*}
This yields~\eqref{FSOTS:alltrans}.

Concerning~\eqref{FSOTS:vector} let~$j\in A$ and~$k\in\{a_0^{\delta_r},\dots,a_{s_r-1}^{\delta_r}\}$, for in all other cases the claim is immediate from the hypothesis.
Let~$s\in\{0,\dots,s_r-1\}$ be such that~$k=a_s^{\delta_r}$.
If~$\Transss{s_r}{r}{j}{k}=\varnothing$, then there is nothing to show.
Thus, we consider the case that 
\[
\Transss{s_r}{r}{j}{k}\not=\varnothing\,.
\]
For~$g\in\Transss{s_r-1}{r}{j}{k}$ the claim is immediate from the hypothesis. If~$\Transss{s_r-1}{r}{j}{a_{s_r}^{\delta_r}}=\varnothing$, then we are finished.
Thus, we suppose that 
\[
\Transss{s_r-1}{r}{j}{a_{s_r}^{\delta_r}}\not=\varnothing
\]
and pick $g\in\Transss{s_r-1}{r}{j}{a_{s_r}^{\delta_r}}$.
We show that there exists~$\nu\in\Css{j}{r,s_r}$ such that
\begin{equation}\label{EQ:itintersectinV}
\gamma_\nu'(\ittime{\BrU,s_r-s+1}(\nu))\in\Css{k}{r,s}\quad\text{and}\quad\gamma_\nu'(\ittime{\BrU,\ell}(\nu))\in V_{s_r-\ell+1}^{(r)}
\end{equation}
for all~$\ell\in\{1,\dots,s_r-s\}$.
First note that, by definition of the path~$\delta_r$, we have~$V_{i}^{(r)}\ne\varnothing$ for all~$i\in\{1,\dots,\ell_{\delta_r}\}$, and because of~$k\in\wt A$ and Lemma~\ref{LEM:VsrCsr}\eqref{VsrCsr:CsrinC} we have~$\Css{k,\st}{r,s}\ne\varnothing$.
Because of~$g\in\Transss{s_r-1}{r}{j}{a_{s_r}^{\delta_r}}$, the hypothesis, and the structure of the path~$\delta_r$, we have
\[
g\act\Plussp{a_\ell^{\delta_r}}\varsubsetneq g\act\Plussp{a_{s_r}^{\delta_r}}\varsubsetneq\Plussp{j}\,,
\]
for all~$\ell\in\{0,\dots,s_r-1\}$.
Hence, in particular there exists~$\nu\in\Css{j}{r,s_r-1}$ such that
\[
(\gamma_\nu(+\infty),\gamma_\nu(-\infty))\in g\act\Ired{k,\st}\times\Jset{j}\,.
\]
Then also~$\nu\in\Css{j}{r,s_r}$, for otherwise~$j=a_{s_r}^{\delta_r}$ and~$\nu\in V_{s_r}^{(r)}$.
But then~$g=\id$ and~$a_{s_r-1}^{\delta_r}=a_{s_r}^{\delta_r}$ by the definition of~$V_{s_r}^{(r)}$, which contradicts the structure of the path~$\delta_r$.
Lemma~\ref{LEM:VsrCsr}\eqref{VsrCsr:CsrinC} now implies that
\[
\gamma_\nu(+\infty)\in g\act\Ired{k,\st}\subseteq g\act\defset{\gamma_\mu(+\infty)}{\mu\in\Css{k,\st}{r,s}}\,.
\]
Since~$\Jset{j}\subseteq\Jset{k}$, this together with~\eqref{BP:allvectors} and again Lemma~\ref{LEM:VsrCsr}\eqref{VsrCsr:CsrinC} imply that
\begin{equation}\label{EQ:nuintersectCsskrs}
\gamma_\nu'\bigl(0,+\infty\bigr)\cap\Css{k}{r,s}\ne\varnothing\,,
\end{equation}
and by counting intersections with the initial set of branches we see that~$\nu$ fulfills~\eqref{EQ:itintersectinV}.
From Algorithm~\ref{stepreduction} we now see that, at \emph{Substep~$s_r$}, all sets~$V_{\ell}^{(r)}$ for~$\ell\in\{1,\dots,s_r\}$ have already been removed from their respective branch.
Hence, while~$\gamma_\nu$ does intersect each of the branches~$g\act\Cs{a_{s_r-\ell}^{\delta_r}}$, it does not intersect~$g\act \Css{a_{s_r-\ell}^{\delta_r}}{r,s_r}$, for~$\ell\in\{0,\dots,s_r-s\}$.
From this and~\eqref{EQ:nuintersectCsskrs} we conclude
\[
\itindex{\BrU_{s_r}^{(r)},1}(\nu)=k
\]
and
\[
\ittrans{\BrU_{s_r}^{(r)},1}(\nu)=\ittrans{\BrU,1}(\nu)\cdots\ittrans{\BrU,s_r-s+1}(\nu)=g\cdot\id\cdots\id=g\,.
\]
This shows~\eqref{FSOTS:vector} and thereby finishes the proof.
\end{proof}

\begin{remark}
In part~\eqref{FSOTS:alltrans} of Lemma~\ref{LEM:FSOTS} it is necessary to include the identity transformation, for, depending on the enumeration of~$\Delta_\ini$ and whether or not~$\Transss{s_r-1}{r}{j}{a_{s_r}^{\delta_r}}$ contains the identity,~$\Transss{s_r}{r}{j}{k}$ might end up differing from the actual transition set for~$j,k\in A$ with respect to~$\BrU_{s_r}^{(r)}$ by lacking exactly the identity transformation.
This is due to a slight imprecision in Algorithm~\ref{stepreduction} in the handling of such situations, which we accepted in favor of clarity.
Simply put, Algorithm~\ref{stepreduction} might remove certain identities ``too soon.''
But since \emph{all} identity transformations are removed in the end (see Proposition~\ref{PROP:stepreduxworx} below), this deviation does not matter for the final transition sets.
\end{remark}

\begin{prop}\label{PROP:idelimsucceeds}
The sets~$\Csr{j}$ and~$\Transs{}{j}{k}$,~$j,k\in \wt A$, are independent of the enumeration chosen in~\eqref{EQ:numberDelta}, and we have
\[
\bigcup_{j,k\in \wt A}\Transs{}{j}{k}\subseteq\Gamma^*\,.
\]
\end{prop}

\begin{proof}
Let~$j\in \wt A$.
If~$\id\notin\bigcup_{k\in A}\Trans{}{j}{k}$, then there is nothing to show, since every set~$\Transss{s_r}{r}{j}{*}$ emerges as the union of two sets of the form~$\Transss{s_r-1}{r}{j}{*}$, for all~$r$ and all~$s_r$, and hence cannot introduce identity transformations.
Thus, suppose that $\id\in\bigcup_{k\in A}\Trans{}{j}{k}$.
For every~$k\in \wt A$ for which~$\id\in\Trans{}{j}{k}$ there exists $r\in\{1,\dots,\eta\}$ such that~$(j,\id)\edge{\phantom{-.}}(k,\id)$ is a sub-path of~$\delta_r$.
This means there exists $s\in\{1,\dots,\ell_{\delta_r}\}$ such that~$j=a_s^{\delta_r}$ and~$k=a_{s-1}^{\delta_r}$.
We set~$a\coloneqq a_{\ell_{\delta_r}}^{\delta_r}$.
Recursive application of the definition of the sets~$\Transss{*}{r}{j}{*}$ yields
\[
\Transss{\ell_{\delta_r}}{r}{j}{k}=\Transss{s-1}{r}{j}{k}\setminus\{\id\}\cup\bigcup_{\ell=s-1}^{\ell_{\delta_r}-1}\Transss{\ell}{r}{j}{a_{\ell+1}^{\delta_r}}\setminus\{\id\}\cup\Transss{\ell_{\delta_r}-1}{r}{j}{a}\,.
\]
Hence, in order to conclude that~$\id\notin\Transss{\ell_{\delta_r}}{r}{j}{k}$ it suffices to show that~$\id\notin\Transss{\ell_{\delta_r}-1}{r}{j}{a}$.
Assume for contradiction that 
\[
\id\in\Transss{\ell_{\delta_r}-1}{r}{j}{a}\,.
\]
By Lemma~\ref{LEM:FSOTS}\eqref{FSOTS:vector} we find~$\nu\in\Css{j}{r,\ell_{\delta_r}-1}$ such that
\[
\bigl(\itindex{\BrU_{\ell_{\delta_r}-1}^{(r)},1}(\nu),\ittrans{\BrU_{\ell_{\delta_r}-1}^{(r)},1}(\nu)\bigr)=(a,\id)\,.
\]
By Lemma~\ref{LEM:VsrCsr}\eqref{VsrCsr:CsrinC} also~$\nu\in\Cs{j}$ and thus there exists~$n\in\N$ such that
\[
\bigl(\itindex{\BrU,n}(\nu),\ittrans{\BrU,n}(\nu)\bigr)=(a,\id)\,.
\]
This means that in the return graph~$\ReturnGraph{0}$ contains a non-degenerate path from~$j$ to~$a$ with accumulated weight~$\id$.
By choice of~$j$, the tuple~$(j,\id)$ is a node in the tree~$B'_{a}$, which means either~$j=a$, or $\ReturnGraph{0}$ contains a non-degenerate path from~$a$ to~$j$ with accumulated weight~$\id$.
In either case we obtain a proper loop in~$\ReturnGraph{0}$ with weight~$\id$, which is contradictory.
Hence,~$\id\notin\Transss{\ell_{\delta_r}-1}{r}{j}{a}$, and therefore~$\id\notin\Transss{\ell_{\delta_r}}{r}{j}{k}$.
Since this argument applies for all~$r\in\{1,\dots,\eta\}$ for which~$(j,\id)$ is contained in the path~$\delta_r$, we conclude
\[
\id\notin\bigcup_{k\in \wt A}\Transs{}{j}{k}\,,
\]
which yields the second claim.

From~\eqref{EQDEF:redCS} it is immediately clear that~$\Csr{j}$ does not depend on the enumeration of~$\Delta_\ini$.
(We emphasize that the definition of the sets~$V_{s_r}^{(r)}$ for $r\in\{1,\ldots, \eta\}$ does not depend on the specific enumeration.)
Lemma~\ref{LEM:FSOTS} implies that $\Transs{}{j}{k}=\Transss{\ell_{\delta_\eta}}{\eta}{j}{k}$ is a full transition set for~$j,k\in\wt A$ up to identities, which necessitates that~$\Transs{}{j}{k}\cup\{\id\}$ also does not depend on the enumeration of~$\Delta_\ini$.
Since~$\id\notin\Transs{}{j}{k}$ by the first part of the proof, this implies that~$\Transs{}{j}{k}$ itself is independent of the enumeration of~$\Delta_\ini$.
Hence, the first claim follows and the proof is finished.
\end{proof}

\begin{cor}\label{COR:Csreasyformula}
For all~$j\in\wt A$ we have
\[
\Csr{j}=\Cs{j}\setminus\bigcup_{k\in A}\defset{\nu\in\Cs{j}}{\gamma_\nu'(\ittime{\BrU,1}(\nu))\in\Cs{k}}\,.
\]
\end{cor}

\begin{defi}\label{DEF:redsetofbranches}
A set~$\wt\BrS=\defset{\Csr{j}}{j\in \wt A}$ of subsets of~$\UTB\H$ is 
called a~\emph{reduced set of branches for the geodesic flow on}~$\Orbi$
\index[defs]{set of branches!reduced}%
\index[defs]{reduced set of branches}%
if it 
satisfies the 
properties~\eqref{BP:closedgeodesicsHtoX},~\eqref{BP:completegeodesics}, 
\eqref{BP:pointintohalfspaces}, and \eqref{BP:disjointunion} from Definition~\ref{DEF:setofbranches}, the 
property~\eqref{BP:closedgeodesicsXtoH} from Proposition~\ref{PROP:oldB1}, as 
well as the following three properties:
\begin{enumerate}[label=$\mathrm{(B5_{\red}\Roman*)}$, 
ref=$\mathrm{B5_{\red}\Roman*}$, leftmargin=9.5ex]
\item\label{BP:allvectorsredI}
For each~$j\in \wt A$ and each pair~$(x,y)\in\Ired{j,\st}\times\Jset{j}$ there 
exists a (unique) vector~$\nu\in\Csr{j}$ such that
\[
(x,y)=(\gamma_\nu(+\infty),\gamma_\nu(-\infty))\,.
\]
\index[defs]{0BPred01@(B5$_{\red}$I)}%
\index[symbols]{BPred01@(B5$_{\red}$I)}%
\item\label{BP:allvectorsredII}
For each~$j\in \wt A$ and each pair~$(x,y)\in\Iset{j}\times\Jset{j}$ there 
exist~$k\in\wt A$ and a (unique) vector~$\nu\in\Csr{k}$ such that
\[
(x,y)=(\gamma_\nu(+\infty),\gamma_\nu(-\infty))\,.
\]
\index[defs]{0BPred02@(B5$_{\red}$II)}%
\index[symbols]{BPred02@(B5$_{\red}$II)}%
\end{enumerate}
\begin{enumerate}
[resume*,label=$\mathrm{(B7_{\red})}$, ref=$\mathrm{B7_{\red}}$]
\item\label{BP:intervaldecompred}
For each pair~$(a,b)\in \wt A\times\wt A$ there exists a (possibly empty) subset~$\Transs{}{a}{b}$ of~$\Gamma$ such that
\begin{enumerate}[label=$\mathrm{(\alph*)}$, ref=$\mathrm{\alph*}$]
\item\label{BP:intervaldecompGdecompred}
for all~$j\in\wt A$ we have 
\begin{align*}
\bigcup_{k\in\wt A}\bigcup_{g\in\Transs{}{j}{k}}g\act\Ired{k}&\subseteq 
\Ired{j}
\intertext{and}
\bigcup_{k\in 
\wt A}\bigcup_{g\in\Transs{}{j}{k}}g\act\Ired{k,\st}&=\Ired{j,\st}\,,
\end{align*}
and these unions are disjoint,
\item\label{BP:intervaldecompGgeodred}
for each pair~$(j,k)\in\wt A\times\wt A$, each $g\in\Transs{}{j}{k}$, and each pair of points~$(z,w)\in\base{\Csr{j}}\times g\act\base{\Csr{k}}$, the geodesic segment~$(z,w)_{\H}$ is nonempty, is contained in~$\Plussp{j}$ and does not intersect~$\Gamma\act\wt\BrU$, where
\[
\wt\BrU\coloneqq\bigcup_{j\in\wt A}\Csr{j}\,,
\]
\index[symbols]{Cab@$\wt\BrU$}%
\item\label{BP:intervaldecompbackred}
for all~$j\in\wt A$ we have 
\[
 \Jset{j}\subseteq\bigcup_{k\in\wt A}\bigcup_{h\in\Transs{}{k}{j}}h^{-1}\act\Jset{k}\,.
\]
\end{enumerate}
\end{enumerate}
A reduced set of branches is called~\emph{admissible}
\index[defs]{reduced set of branches!admissible}%
\index[defs]{set of branches! reduced, admissible}%
\index[defs]{admissible}%
if it satisfies 
property~\eqref{BP:leavespaceforflip}, and it is called~\emph{non-collapsing}
\index[defs]{reduced set of branches!non-collapsing}%
\index[defs]{set of branches!reduced, non-collapsing}%
\index[defs]{non-collapsing}%
if 
it satisfies property~\eqref{BP:noidentity} from 
Definition~\ref{DEF:setofbranches}.
\end{defi}

\begin{remark}\label{REM:redSoBrem}
Depending on the underlying Fuchsian group, a non-collapsing behavior and \eqref{BP:allvectors} are often incompatible to each other for the 
explicit algorithmic construction procedures of sets of branches we use.
But non-collapsing reduced sets of branches will suffice for the purpose of all 
the following discussions and constructions.
The approach via sets of branches that get adequately reduced to ensure 
non-collapsing behavior at the cost of the strong 
property~\eqref{BP:allvectors}, instead of an axiomatic approach, was chosen to 
mimic the algorithmic process of constructing these sets.
Consistently, Proposition~\ref{PROP:stepreduxworx} below shows that for every 
Fuchsian group for which a set of branches exists we obtain a non-collapsing 
reduced set of branches via the above procedure.
\end{remark}

\begin{prop}\label{PROP:stepreduxworx}
The set~$\wt\BrS$ is a non-collapsing reduced set of branches for the geodesic 
flow on~$\Orbi$ with associated forward transition sets given 
by~$\Transs{}{j}{k}$ for any choice of~$j,k\in \wt A$.
If~$\BrS$ is admissible, then so is~$\wt\BrS$.
\end{prop}

\begin{proof}
In order to distinguish the application of the defining 
properties~\eqref{BP:closedgeodesicsHtoX}--\eqref{BP:leavespaceforflip} 
for~$\BrS$ from those for~$\wt\BrS$ we aim to prove, we denote the latter ones 
by~(\ref{BP:closedgeodesicsHtoX}$_{\rred}$)--(\ref{BP:pointintohalfspaces}$_{\rred}$) 
and (\ref{BP:disjointunion}$_{\rred}$)--(\ref{BP:leavespaceforflip}$_{\rred}$). We emphasize that $\wt\BrS$ is not required to satisfy~\eqref{BP:allvectors}. 

Property~(\ref{BP:closedgeodesicsHtoX}$_{\rred}$) is immediate from~\eqref{BP:closedgeodesicsHtoX} and the definition of~$\wt A$.
Further, Property~(\ref{BP:completegeodesics}$_{\rred}$) is immediate 
from~\eqref{BP:completegeodesics} and the fact that 
Algorithm~\ref{stepreduction} does not interfere with the sets~$J_j$ 
and~$\Jset{j}$ for~$j\in\wt A$.
Since~$\wt A\subseteq A$ and~$\Csr{j}\subseteq\Cs{j}$ for every~$j\in\wt A$, the 
properties~(\ref{BP:pointintohalfspaces}$_{\rred}$) 
and~(\ref{BP:disjointunion}$_{\rred}$) are direct consequences 
of~\eqref{BP:pointintohalfspaces} and~\eqref{BP:disjointunion}, respectively.
And since~$\Ired{j}\subseteq I_j$ for all~$j\in\wt A$ by virtue of~\eqref{EQDEF:redInterval}, property~(\ref{BP:leavespaceforflip}$_{\rred}$) is immediate from~\eqref{BP:leavespaceforflip}.

Let~$j\in\wt A$ and $\nu\in\Cs{j,\st}$.
Set~$x\coloneqq\gamma_\nu(+\infty)$ and~$y\coloneqq\gamma_\nu(-\infty)$. 
Then we have~$(x,y)\in \Iset{j}\times \Jset{j}$, and~$\nu\in\Csr{j}$ if and only if~$x\in\Ired{j}$.
This together with~\eqref{BP:allvectors} already yields~\eqref{BP:allvectorsredI}.
Property~\eqref{BP:allvectorsredII} follows immediately from~\eqref{BP:allvectorsredI} in this case.
Assume now that~$\nu\notin\Csr{j}$.
Then, by~\eqref{EQDEF:redInterval}, there exists~$k_1\in\Heir{}{j}$ such that~$x\in I_{k_1}$.
If~$x\notin\Ired{k_1}$, then, again by~\eqref{EQDEF:redInterval}, there exists~$k_2\in\Heir{}{k_1}$ such that~$x\in I_{k_2}$.
Iterating this argument is equivalent to traveling down a path in~$\Delta_\ini$.
Or in other words, there exist~$\delta_r\in\Delta_\ini$ and~$p\in\{1,\dots,\ell_{\delta_r}\}$ such that
\[
j=a_p^{\delta_r}\quad\text{and}\quad k_\iota=a_{p-\iota}^{\delta_r}\quad\text{for }\iota=1,2,\dots\,.
\]
Hence, we obtain~$k_p=a_0^{\delta_r}$ by virtue of Algorithm~\ref{stepreduction}, and thus~$\Csr{k_p}=\Cs{k_p}$ and~$x\in I_{k_p}=\Ired{k_p}$.
Since~$\base{\Cs{k_p}}\subseteq\Plussp{\Cs{j}}$, we find~$y\in J_j\subseteq J_{k_p}$.
Since~$(x,y)\in\wh\R_{\st}\times\wh\R_{\st}$, by~\eqref{BP:allvectors} we find a unique~$\nu'\in\Cs{k_p}$ such that~$(\gamma_{\nu'}(+\infty),\gamma_{\nu'}(-\infty))=(x,y)$.
This yields~\eqref{BP:allvectorsredII}.

Let~$\wh\gamma\in\Geo_{\Per}(\Orbi)$.
By Proposition~\ref{PROP:oldB1} there exists~$\gamma\in\Geo(\H)$,~$\pi(\gamma)=\wh\gamma$, such that~$\gamma$ intersects~$\BrU$.
Let~$j\in A$ be such that~$\gamma'(t)\in\Cs{j}$ for some~$t\in\R$.
Because of Corollary~\ref{COR:Csreasyformula} we may assume~$j\in\wt A$.
Then~$(\gamma(+\infty),\gamma(-\infty))\in\Iset{j}\times\Jset{j}$.
If~$\gamma(+\infty)\in\Ired{j,\st}$, then~$\gamma'(t)\in\Csr{j}$ by~\eqref{BP:allvectorsredI}.
Otherwise, by~\eqref{BP:allvectorsredII} we find~$t'\in\R$ and~$k\in\wt A$ such that~$\gamma'(t')\in\Csr{k}$.
Hence, in either case~$\gamma$ intersects~$\wt\BrU$, which implies that~$\wt\BrS$ fulfills~\eqref{BP:closedgeodesicsXtoH}.

Let again~$j\in\wt A$.
Since~$\Csr{j}\ne\varnothing$, we have~$\Ired{j,\st}\ne\varnothing$.
Hence, there exists~$\nu\in\Csr{j}$ such that
\[
(\gamma_\nu(+\infty),\gamma_\nu(-\infty))\in\Ired{j,\st}\times\Jset{j}
\subseteq\wh\R_{\st}\times\wh\R_{\st}
\subseteq\Lambda(\Gamma)\times\Lambda(\Gamma)\,.
\]
In particular,~$(\gamma_\nu(+\infty),\gamma_\nu(-\infty))\in\Ired{j}\times J_j$, 
which is an open set in~$\wh\R\times\wh\R$.
Therefore, there exists~$\eps>0$ such that
\[
\Ball{\wh\R,\eps}{\gamma_\nu(+\infty)}\times\Ball{\wh\R,\eps}{\gamma_\nu(-\infty)}\subseteq\Ired{j}\times J_j\,,
\]
where
\[
\Ball{\wh\R,\eps}{x}\coloneqq\left\{
\begin{array}{cl}
(x-\eps,x+\eps)&\text{if }x\in\R\,,\\
(\frac{1}{\eps},-\frac{1}{\eps})_c&\text{if }x=\infty\,.
\end{array}
\right.
\]
\index[symbols]{BallRR@$\Ball{\wh\R,\eps}{x}$}%
By Proposition~\ref{PROP:EXliesdense} we find a representative~$\gamma$ 
of some periodic geodesic on~$\Orbi$ such that
\[
(\gamma(+\infty),\gamma(-\infty))\in\Ball{\wh\R,\eps}{\gamma_\nu(+\infty)}\times\Ball{\wh\R,\eps}{\gamma_\nu(-\infty)}\,.
\]
By construction,
\[
(\gamma(+\infty),\gamma(-\infty))\in\Ired{j,\st}\times\Jset{j}\,.
\]
Combining this with~\eqref{BP:allvectorsredI} yields~(\ref{BP:closedgeodesicsHtoX}$_{\rred}$).
Finally, all statements of~\eqref{BP:intervaldecompred} follow from the combination of~$\eqref{BP:intervaldecomp}$ with~\eqref{EQDEF:redInterval},~\eqref{BP:allvectorsredI}, and~\eqref{BP:allvectorsredII}.
\end{proof}

Let~$\nu\in\wt\BrU$ and define the \emph{system of 
iterated sequences of~$\nu$ with respect to~$\wt\BrU$} as
\index[defs]{iterated sequences}%
\[
[(\ittime{\wt\BrU,n}(\nu))_n,(\itindex{\wt\BrU,n}(\nu))_n,(\ittrans{\wt\BrU,n}
(\nu))_n]\,,
\]
where
\[
\ittime{\wt\BrU,n}(\nu)\coloneqq\ittime{\BrU_{\ell_{\delta_\eta}}^{(\eta)},n}(\nu)\,,\quad
\itindex{\wt\BrU,n}(\nu)\coloneqq\itindex{\BrU_{\ell_{\delta_\eta}}^{(\eta)},n}(\nu)\,,
\]
and
\[
\ittrans{\wt\BrU,n}(\nu)\coloneqq\ittrans{\BrU_{\ell_{\delta_\eta}}^{(\eta)},n}(\nu)\,,
\]
for all~$n\in\Z$, with~$\bigl[\bigl(\ittime{\BrU_{s_r}^{(r)},n}(\nu)\bigr)_n,\bigl(\itindex{\BrU_{s_r}^{(r)},n}(\nu)\bigr)_n,\bigl(\ittrans{\BrU_{s_r}^{(r)},n}(\nu)\bigr)_n\bigr]$ as in~\eqref{EQDEF:algoitseq} for $r\in\{1,\dots,\eta\}$ and $s_r\in\{1,\dots,\ell_{\delta_r}\}$.
Because of Proposition~\ref{PROP:stepreduxworx} this system of sequences is independent of the enumeration of~$\Delta_\ini$.
We further obtain the following analogue of Proposition~\ref{PROP:allinit}.

\begin{cor}\label{COR:allinitred}
Let $\nu\in\wt\BrU_{\st}$, $k\in\wt A$, $t\in\R$ and $g\in\Gamma$ be such that
\[
\gamma_\nu^{\prime}(t)\in g\act\Csr{k}\,.
\]
Then there exists a unique element~$n\in\Z$ such that~$\sgn(n)=\sgn(t)$ and
\begin{eqnarray*}
k=\itindex{\wt\BrU,n}(\nu),&t=\ittime{\wt\BrU,n}(\nu),\text{ 
and}&g=\ittrans{\wt\BrU,\sgn(t)}(\nu)\ittrans{\wt\BrU,2\sgn(t)}
(\nu)\cdots\ittrans{
\wt\BrU,n}(\nu)\,.
\end{eqnarray*}
\end{cor}

A reduced set of branches is called \emph{finitely ramified} 
\index[defs]{reduced set of branches!finitely ramified}%
\index[defs]{set of branches!reduced, finitely ramified}%
\index[defs]{finitely ramified}%
if~$\#\Transs{}{j}{k}<+\infty$ for all~$j,k\in\wt A$.
The following result shows that Algorithm~\ref{stepreduction} does not negate 
the efforts of Section~\ref{SUBSEC:branchram}.

\begin{prop}\label{PROP:finramandcollaps}
If~$\BrS$ is finitely ramified, then so is~$\wt\BrS$.
\end{prop}

\begin{proof}
Let~$j,k\in\wt A$.
By hypothesis,~$\#\Transss{\ell_{\delta_0}}{0}{j}{k}=\#\Trans{}{j}{k}<+\infty$.
In every step of Algorithm~\ref{stepreduction}, a new transition set~$\Transss{0}{r}{j}{k}$ emerges as the union of at most two sets of the type~$\Transss{*}{r-1}{j}{k}$, which are seen to be of finite cardinality by recursive application of this argument.
Since Algorithm~\ref{stepreduction} terminates after finitely many steps, this yields the set~$\Transs{}{j}{k}$ as a finite union of finite subsets of~$\Gamma$.
\end{proof}

\begin{example}\label{EX:G3idelim}
Recall the group~$\Gamma_\lambda$ as well as its weakly non-collapsing set of branches~$\BrS'$ from Example~\ref{EX:G3Graph}.
As has been seen in Example~\ref{EX:G3Graphred}, the completely reduced set of branches~$\{\Cs{2},\Cs{7}\}$ is already non-collapsing, so there is no need to apply Algorithm~\ref{stepreduction} in this case.
But since branch reduction is optional, we might investigate the outcome of the identity elimination if applied to~$\BrS'$.
We have~$D_\ini=\{1,6\}$ and from the forest~$F'_\ini$, which is depicted in Figure~\ref{FIG:G3forest}, we see that~$\eta=3$.
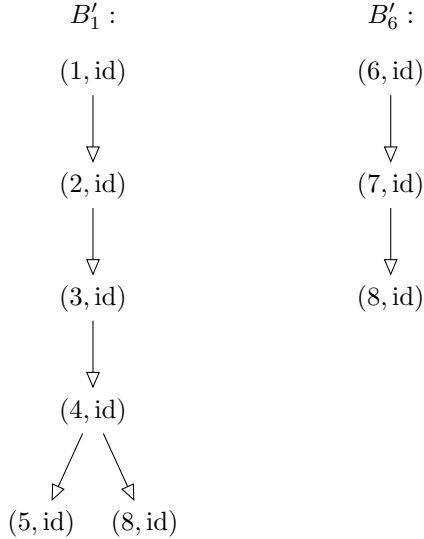
\begin{figure}[h]
\begin{tikzpicture}[level/.style={sibling distance=55mm/#1},edge from parent/.style={draw,-open triangle 45}]
\node (z){$(1,\id)$}
  child {node (a) {$(2,\id)$}
  	child {node (b) {$(3,\id)$}
  		child {node (c) {$(4,\id)$}
  			child {node (d) {$(5,\id)$}} child {node (e) {$(8,\id)$}}
  		}
  	}
   };
\node [above=1ex of z] {$B'_1:$};   
\node (zz) [right=8em of z] {$(6,\id)$}
	child {node (f) {$(7,\id)$}
		child {node (g) {$(8,\id)$}}
	};
\node [above=1ex of zz] {$B'_6:$};
\end{tikzpicture}
\caption[subtreesidentity]{The subtrees~$B_1'$ and $B_6'$ of $F'_\ini$ for the set of branches~$\BrS'$ for~$\Gamma_\lambda$.}\label{FIG:G3forest}
\end{figure}
Furthermore, from Figure~\ref{FIG:G3returngraphwnc} it can be seen that~$\bigcup_{k\in A}\Trans{}{j}{k}=\{\id\}$ for~$j=1,4,6$, wherefore we obtain~$\wt A=\{2,3,5,7,8\}$.
In order to display the emerging reduced set of branches~$\wt\BrS=\defset{\Csr{j}}{j\in\wt A}$ we provide a ``return graph'' (Figure~\ref{FIG:G3altredreturn}), a depiction of the reduced branches themselves (Figure~\ref{FIG:G3altredbranches}), as well as a list of the sets~$\Ired{j}$,~$j\in\wt A$: We have
\begin{figure}[h]
\begin{tikzpicture}[->,shorten >=1pt,auto, semithick]
\def \fak {7.5 em}
\begin{scope}[on grid]
\node[state] (C3) {$3$};
\node[state] (C5) [below right= \fak and \fak of C3] {$5$};
\node[state] (C7) [below left= \fak and \fak of C3] {$7$};
\node[state] (C8) [below= 2*\fak of C3] {$8$};
\node[state] (C2) [below right= .4*\fak and 1.6*\fak of C3] {$2$};
\def \los {1}
\path (C5.north) edge [bend right,looseness=\los] node[pos=.5,above right] {$t_\lambda$} (C3.east)
		  (C5.south) edge [bend left,looseness=\los] node {$t_\lambda$} (C8.east)
		  (C3.300) edge [bend right,looseness=\los] node {$h$} (C5.150)
		  (C5.65) edge [bend left,looseness=\los] node {$t_\lambda$} (C2.205)
		  (C2.245) edge [bend left,looseness=\los] node {$h^{-1}$} (C5.25)
		  (C8.west) edge [bend left,looseness=\los] node {$ht_\lambda^{-1}$} (C7.south)
		  (C7.north) edge [bend left,looseness=\los] node {$h$} (C3.west)
		  (C5) edge [loop right,out=-25,in=-65,looseness=15] node[pos=.5,below right] {$t_\lambda$} (C5)
		  (C8) edge [loop right,out=25,in=65,looseness=15] node[pos=.5,above=1pt] {$ht_\lambda^{-1}$} (C8)
		  (C3.south) edge node[pos=.3,right] {$h$} (C8.north)
		  (C7.330) edge [bend left,looseness=\los] node {$h$} (C8.120)
		  (C7.east) edge node[pos=.3,above] {$h$} (C5.west);
\end{scope}
\end{tikzpicture}
\caption{A return-style graph for the reduced set of branches~$\wt\BrS$ for~$\Gamma_\lambda$.}\label{FIG:G3altredreturn}
\end{figure}
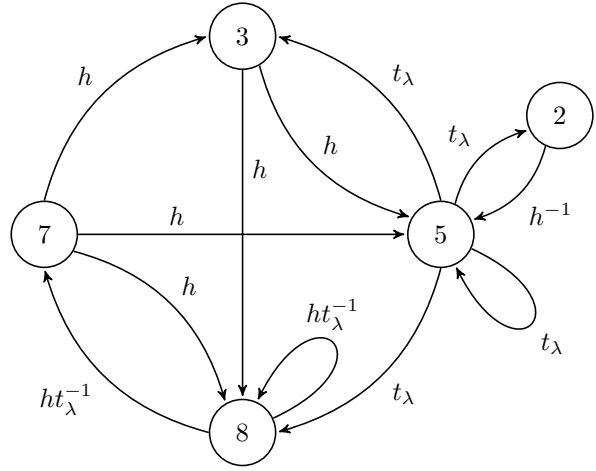
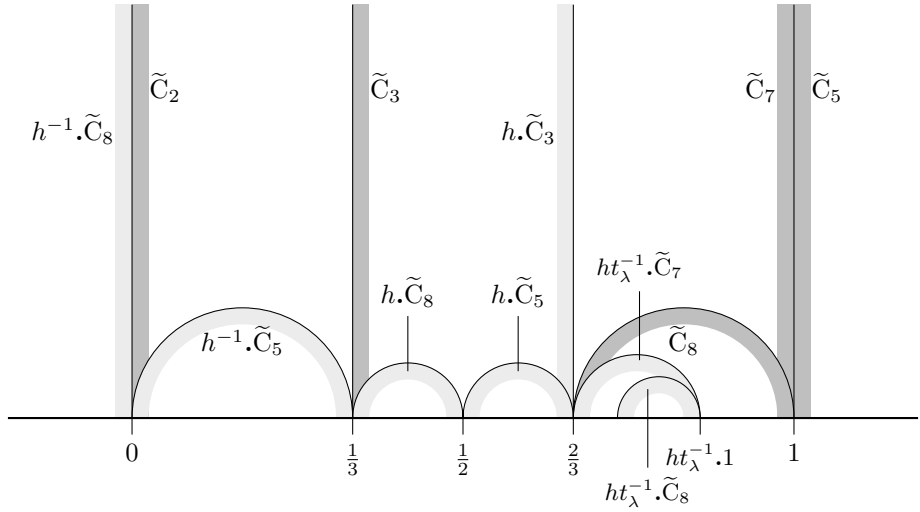
\begin{figure}[h]
\begin{tikzpicture}[scale=11]
\tikzmath{\q=1/(2*sqrt(3));
				  \r=.8/6;
				  \v=.1+.8*1/3;
				  \w=.1+2*.8*1/3;
				  \p=\w+\r+.02;
				  }
\fill[color=gray!50] (.1,0) -- (.1+.02,0) -- (.1+.02,.5) -- (.1,.5) -- cycle;
\fill[color=gray!50] (\v,0) -- (\v+.02,0) -- (\v+.02,.5) -- (\v,.5) -- cycle;
\fill[color=gray!15] (\w-.02,0) -- (\w,0) -- (\w,.5) -- (\w-.02,.5) -- cycle;
\fill[color=gray!15] (.1-.02,0) -- (.1,0) -- (.1,.5) -- (.1-.02,.5) -- cycle;
\fill[color=gray!50] (.9-.02,0) -- (.9+.02,0) -- (.9+.02,.5) -- (.9-.02,.5) -- cycle;
\fill[color=gray!15] (\v,0) arc (0:180:\r) -- (.1+.02,0) arc (180:0:\r-.02) -- cycle;
\fill[color=gray!15] (\w,0) arc (0:180:.5*\r) -- (\w-\r+.02,0) arc (180:0:.5*\r-.02) -- cycle;
\fill[color=gray!15] (\v,0) arc (180:0:.5*\r) -- (\w-\r-.02,0) arc (0:180:.5*\r-.02) -- cycle;
\fill[color=gray!50] (.9,0) arc (0:180:\r) -- (\w+.02,0) arc (180:0:\r-.02) -- cycle;
\fill[color=gray!15] (\p,0) arc (0:180:.5*\r+.01) -- (\w+.02,0) arc (180:0:.5*\r-.01) -- cycle;
\fill[color=gray!15] (\p,0) arc (0:180:.3*\r+.01) -- (\p-.6*\r,0) arc (180:0:.3*\r-.01) -- cycle;
\draw (\v,0) arc (0:180:\r);
\draw (.9,0) arc (0:180:\r);
\draw (\p,0) arc (0:180:.5*\r+.01);
\draw (\v,0) arc (180:0:.5*\r) arc (180:0:.5*\r);
\draw (\p,0) arc (0:180:.3*\r+.01);
\foreach \x/\y in {.1/$0$,\v/$\frac{1}{3}$,\w/$\frac{2}{3}$,.9/$1$}{
	\draw (\x,.5) -- (\x,-0.02) node [below] {\y};
}
\draw (.5,0) -- (.5,-.02) node [below] {$\frac{1}{2}$};
\draw (\p,0) -- (\p,-.02) node [below] {$\small ht_\lambda^{-1}\act 1$};
\draw[style=thick] (-.05,0) -- (1.05,0);
\coordinate[label=right:$\Csr{2}$] (C2) at (.11,.4);
\coordinate[label=right:$\Csr{3}$] (C3) at (\v+.01,.4);
\coordinate[label=right:$\Csr{5}$] (C5) at (.91,.4);
\coordinate[label=left:$\Csr{7}$] (C7) at (.89,.4);
\coordinate[label=below:$\Csr{8}$] (C8) at (\w+\r,\r-.01);
\coordinate[label=below:$h^{-1}\act\Csr{5}$] (h1C5) at (.1+\r,\r-.01);
\coordinate[label=left:$h^{-1}\act\Csr{8}$] (h1C8) at (.09,.35);
\coordinate[label=left:$h\act\Csr{3}$] (hC3) at (\w-.01,.35);
\draw (\v+.5*\r,.5*\r-.01) -- (\v+.5*\r,\r-.01) node[above] {$h\act\Csr{8}$};
\draw (\w-.5*\r,.5*\r-.01) -- (\w-.5*\r,\r-.01) node[above] {$h\act\Csr{5}$};
\draw (\w+.6*\r,.6*\r-.01) -- (\w+.6*\r,\r+.02) node[above] {$\small ht_\lambda^{-1}\act\Csr{7}$};
\draw (\w+.6*\r+.01,.3*\r-.005) -- (\w+.6*\r+.01,-.6*\r+.02) node[below] {$\small ht_\lambda^{-1}\act\Csr{8}$};
\end{tikzpicture}
\caption{The reduced set of branches~$\wt\BrS$ for~$\Gamma_\lambda$ emerging from~$\BrS'$ via Algorithm~\ref{stepreduction}.}\label{FIG:G3altredbranches}
\end{figure}
\begin{align*}
\Ired{2}&=\bigl(0,\tfrac{1}{3}\bigr)\,,\quad\Ired{3}=\bigl(\tfrac{1}{3},\tfrac{2}{3}\bigr)\,,\quad\Ired{5}=\bigl(1,+\infty\bigr)\,,
\\
\Ired{7}&=\bigl(-\infty,\tfrac{2}{3}\bigr)\,,\quad\text{and}\quad\Ired{8}=\bigl(\tfrac{2}{3},1\bigr)\,.
\end{align*}
\end{example}

\section{Cuspidal acceleration}\label{SEC:cuspacc}

In this section, we present the \emph{cuspidal acceleration algorithm}, which is the main step in our constructions of strict transfer operator approaches. This algorithm ultimately yields the passage from a non-uniformly expanding discrete dynamical system to a uniformly expanding one, which then guarantees nuclearity of the arising fast transfer operators. As the naming suggests, this algorithm only affects hyperbolic orbisurfaces with cusps. 

As the algorithms of branch reduction and identity elimination, also the cuspidal acceleration algorithm works by modifying a given set of representatives for a cross section. Here, we start with a cross section and a set of representatives for which the induced discrete dynamical system is typically not uniformly expanding. (If the induced system is already uniformly expanding, then the cuspidal acceleration algorithm is void and does not modify the cross section.) The non-uniformity in the expansion rate originates from the property of the initial cross section to encode each single winding of a geodesic around a cusp as a separate intersection event. To achieve uniformity, successive windings around a cusp should be merged into one (somewhat collective) event. The cuspidal acceleration algorithm achieves exactly this by a careful elimination of certain unit tangent vectors in the set of representatives for the initial cross section. We refer to Remark~\ref{REM:howtoacc} below for a more detailed explanation. The algorithm itself consists indeed of a single (simultaneous) elimination step, for which reason it is stated as a definition, namely Definition~\ref{DEF:cuspacc}, in which the \emph{accelerated} set of representatives of the \emph{accelerated} cross section is defined. The remaining section is then dedicated to the proof that this set is indeed a cross section. The following sections mostly discuss how this cross section and the set of representatives give rise (in a natural way) to a strict transfer operator approach.

Throughout this section let 
\[
\BrS^{(i)}\coloneqq\defset{\Cs{j}^{(i)}}{j\in A^{(i)}}
\]
be a set of branches with $A^{(i)}\coloneqq \{1,\ldots, N\}$ (with $(i)$ indicating \emph{initial}). We emphasize that the considerations in what follows do \emph{not} require that the set of branches~$\BrS^{(i)}$ is branch reduced. I.e., it is not required that the branch reduction algorithm from Section~\ref{SEC:branchred} has been applied to~$\BrS^{(i)}$. We further let~$\wt A$,~$\wt\BrS$,~$\wt\BrU$,~$\Ired{j}$ and~$\Transs{}{j}{k}$ for~$j,k\in\wt A$, and
\[
[(\ittime{\wt\BrU,n}(\nu))_n,(\itindex{\wt\BrU,n}(\nu))_n,(\ittrans{\wt\BrU,n}
(\nu))_n]
\]
for~$\nu\in\Csr{j}$,~$j\in\wt A$, be as in Section~\ref{SEC:stepredux}.
That is, $\wt\BrS$ is a reduced set of branches (see Definition~\ref{DEF:redsetofbranches}).
For~$j\in\wt A$ we further define
\[
\wt H(j)\coloneqq\defset{k\in\wt A}{\Transs{}{j}{k}\ne\varnothing}\,.
\]
\index[symbols]{H@$\wt{H}_{j}$}%

\begin{center}
\framebox{
\begin{minipage}{.83\textwidth}
From now on we omit all tildes (\,$\wt{\ }$\,) from the notation. Thus, throughout this section, $\BrS$ denotes a \emph{reduced} set of branches in the sense of Definition~\ref{DEF:redsetofbranches}. We caution that this notation is not fully consistent with the one of the previous sections but here preferably in favor of avoiding overloaded notation. We further assume that~$\CrSc$ is a \emph{strong} cross section for~$\GeoFlow$ (see Section~\ref{SUBSEC:cross}).
\end{minipage}
}
\end{center}

\medskip

We will take advantage of a certain cyclic behavior of $\Gamma$-translates of~$\BrS$ at cusps.
By this we refer to the following property: Let $j\in A$ and recall the endpoints~$\point{j}{\eX}$ and $\point{j}{\eY}$ of~$\overline{\base{\Cs{j}}}$ from 
Remark~\ref{REM:SoBrem}\eqref{SoBrem:descrHalfspaces}.
Recall further that~$\Ired{j}\subseteq I_j=(\eX_j,\eY_j)_c$.
There are two possibilities for~$\eZ_j\in\{\eX_j,\eY_j\}$ in regard to~$\Ired{j}$:
\begin{enumerate}[label=$\mathrm{(\alph*)}$, ref=$\mathrm{\alph*}$]
\item either~$\eZ_j$ is a boundary point of~$\Ired{j}$ in the~$\wh\R$-topology, or
\item there exists~$\eps>0$ such that
\begin{align*}
(\eZ_j-\eps,\eZ_j+\eps)\cap\Ired{j}=\varnothing & \qquad\text{if~$\eZ_j\ne\infty$},
\intertext{and}
(-\frac{1}{\eps},\frac{1}{\eps})\cap\Ired{j}=\varnothing & \qquad\text{if~$\eZ_j=\infty$}.
\end{align*}
\end{enumerate}
We suppose that, say, $\point{j}{\eX}$ is a boundary point of~$\Ired{j}$ and represents a cusp of $\Orbi$, say $\wt c$ (see~\eqref{BP:completegeodesics}).
Those two assumptions are not mutually exclusive, because, given the latter, if the former were not the case, by Corollary~\ref{COR:Csreasyformula} we would find~$j'\in A$ such that~$\Ired{j'}\subseteq I_{j'}\subseteq I_j$ and~$\eX_j$ is a boundary point of~$\Ired{j'}$.
Then~$\eX_{j'}=\eX_j$, and we thus may proceed with~$j'$ instead of~$j$.
Now, due to~\eqref{BP:intervaldecompred} we find~$k\in\Heir{}{j}$ and a
transformation $g\in\Trans{}{j}{k}$ such that $\point{j}{\eX}=g\act\point{k}{\eX}$ and~$\eX_k$ is an endpoint of~$\Ired{k}$.
The tuple~$(k,g)$ is uniquely determined.
Clearly,~$\eX_k$ is again a representative of~$\wt c$.
Iterating this argument we are led back, after finitely many steps, to some~$\Gamma$-translate of~$\eX_j$, from where on the cycle repeats (see Lemma~\ref{LEM:periodiccycles} below).
This yields the notion of~\emph{$\eX$-cycles}, which is made more rigorous by the following definition.
We may argue analogously if~$\eY_j$ is a boundary point of~$\Ired{j}$ and represents a cusp of~$\Orbi$.

\begin{defi}\label{DEF:xytuple}
Let $\eZ\in\{\eX,\eY\}$ and let 
\[
 A_\eZ \coloneqq \defset{ j\in A }{ \text{$\point{j}{\eZ}$ cuspidal boundary point of~$\Ired{j}$} }
\]
\index[symbols]{Az@$A_\eZ$}%
\index[symbols]{Aza@$A_\eX$}%
\index[symbols]{Azb@$A_\eY$}%
be the subset of elements $j\in A$ for which the endpoint~$\point{j}{\eZ}$ of 
the geodesic segment~$\overline{\base{\BrU_j}}$ represents a cusp of $\Orbi$ and coincides with a boundary point of~$\Ired{j}$ in the~$\wh\R$-topology.
By the discussion right before this definition, for each $j\in A_{\eZ}$ there 
exists a (unique) pair~$(k_j,g_j) \in \Heir{}{j} \times \Trans{}{j}{k_j}$ that 
satisfies $\point{j}{\eZ}=g_j\act\point{k_j}{\eZ}$ and $k_j\in A_\eZ$. We call 
the pair~$(k_j,g_j)$ the \emph{$\eZ$-tuple of~$j$}.
\index[defs]{Z-tuple}%
\index[defs]{X-tuple}%
\index[defs]{Y-tuple}%
Further, we define the maps
\begin{eqnarray*}
\psi_{\eZ}\colon A_\eZ\to A_\eZ\,, \qquad \psi_{\eZ}(j)\coloneqq k_j
\end{eqnarray*}
\index[symbols]{pZ@$\psi_\eZ$}%
and 
\[
\cycnext{\eZ}\colon A_\eZ\to \Gamma\,,\qquad  \cycnext{\eZ}(j) \coloneqq g_j\,.
\]
\index[symbols]{gZ@$\cycnext{\eZ}$}%
For each $j\in A_\eZ$, iterated application of~$\psi_\eZ$ leads to the sequence 
\[
\cyc{\eZ}{j}\coloneqq(\psi_{\eZ}^{r}(j))_{r\in\N_0}\,,
\] 
\index[symbols]{cyc@$\cyc{\eZ}{j}$}%
which we call the \emph{$\eZ$-cycle of~$j$}.
\index[defs]{Z-cycle}%
\index[defs]{X-cycle}%
\index[defs]{Y-cycle}%
\end{defi}

\begin{lemma}\label{LEM:periodiccycles}
Let $\eZ\in\{\eX,\eY\}$ and $j\in A_{\eZ}$. Then the sequence~$\cyc{\eZ}{j}$ is 
periodic with (minimal) period length 
\begin{align*}
\min\defset{n\in\N}{\exists\,\nu\in\Cs{j}\colon\itindex{\BrU,n}
(\nu)=j\land\ittrans{\BrU,1}(\nu)\cdots\ittrans{\BrU,n}(\nu)\act\point{j}{\eZ}
=\point{j}{\eZ}}\,.
\end{align*}
\end{lemma}

\begin{proof}
Let $\fund$ be a Ford fundamental domain for the action of~$\Gamma$ on~$\H$. 
Then each cusp~$\wh c$ of~$\Orbi$ has at least one 
representative~$c\in\widehat{\R}$ such that $c$ is an infinite vertex of~$\fund$ 
and each sufficiently small geodesic segment on~$\Orbi$ with endpoint~$\wh c$ 
(i.e., contained in a sufficiently small horoball centered at~$\wh c\,$) has a 
representing geodesic segment on~$\H$ with endpoint~$c$ that is contained 
in~$\overline\fund$ (see~\cite[IV.7F]{Lehner}). Consequently, there 
exists~$h\in\Gamma$ such that $\point{j}{\eZ}$ is an infinite vertex 
of~$h\act\fund$ and the geodesic segment~$\overline{\base{\Cs{j}}}$ 
intersects~$h\act\overline\fund$ in any small horoball centered 
at~$\point{j}{\eZ}$. By the Poincar\'e theorem on fundamental polyhedra, the 
(conjugate) primitive vertex cycle transformation of~$\point{j}{\eZ}$, say $p$, 
is parabolic, fixes $\point{j}{\eZ}$ and is a generator of the stabilizer group 
of~$\point{j}{\eZ}$ in~$\Gamma$. Thus, either $p\act\Ired{j,\st}\subseteq\Ired{j,\st}$ 
or $p^{-1}\act\Ired{j,\st}\subseteq\Ired{j,\st}$, where we may suppose the former 
without loss of generality. Since representatives of cusps cannot coincide with 
endpoints of representative intervals of funnels (see~\cite[IV.7E and 
IV.7G]{Lehner}),
\begin{align}\label{EQN:IsetjdenseinIj}
\forall\varepsilon>0\colon\Ired{j,\st}\cap\Ball{\wh\R,\varepsilon}{\point{j}{\eZ}}
\ne\varnothing\,.
\end{align}
Let $\nu\in\Cs{j}$ be such that 
$(\gamma_{\nu}(+\infty),\gamma_{\nu}(-\infty))\in p\act\Ired{j,\st}\times\Jset{j}$. 
The combination of 
Lemma~\ref{LEM:halfspaceFuture}\eqref{halfspaceFutureintersect} and 
Proposition~\ref{PROP:allinit} yields a unique element~$n\in\N$ only depending 
on~$j$ such that for all~$0\leq m\leq n$ we have 
\[
\gamma_{\nu}^{\prime}(\ittime{\BrU,m}(\nu))\in\ittrans{\BrU,1}
(\nu)\cdots\ittrans{\BrU,m}(\nu)\act\Cs{\itindex{\BrU,m}(\nu)}
\]
with $\ittrans{\BrU,1}(\nu)\cdots\ittrans{\BrU,n}(\nu)=p$ and 
$\itindex{\BrU,n}(\nu)=j$. Since $p$ fixes $\point{j}{\eZ}$ and 
\[
\ittrans{\BrU,m+1}(\nu)\act\Ired{\itindex{\BrU,m+1}(\nu),\st}\subseteq 
\Ired{\itindex{\BrU,m}(\nu),\st}
\]
for all $0\leq m < n$, \eqref{EQN:IsetjdenseinIj} implies 
\[
\ittrans{\BrU,1}(\nu)\cdots\ittrans{\BrU,m}(\nu)\act\eZ_{\itindex{\BrU,m}
(\nu)}=\point{j}{\eZ}
\]
for all such~$m$. Moreover, since $\overline{\base{\Cs{j}}}$ 
intersects~$h\act\overline\fund$ in any small horoball centered 
at~$\point{j}{\eZ}$, the part of~$\overline{\base{\Cs{j}}}$ sufficiently 
near~$\point{j}{\eZ}$ is contained in the fundamental domain~$h\act\fund$ or in 
its boundary. The fundamental domains neighboring~$h\act\fund$ 
at~$\point{j}{\eZ}$ are $ph\act\fund$ and $p^{-1}h\act\fund$. In turn, the 
indices $\itindex{\BrU,m}$ of the iterated intersection branches of~$\nu$ are 
not equal to~$j$ for $m\in\{1,\ldots, n-1\}$. Further, since 
$\ittrans{\BrU,m}(\nu)\in\Trans{}{\itindex{\BrU,m}(\nu)}{\itindex{\BrU,m+1}
(\nu)}$ for all $m=0,\dots,n$, we obtain
\[
\cyc{\eZ}{j}_{m}=\psi_{\eZ}^{m}(j)=\itindex{\BrU,m}(\nu)\,.
\]
Set $\nu_r\coloneqq p^{r-1}\act\nu$ for $r\geq 1$. Then we find 
$\itindex{\BrU,m}(\nu_r)=\itindex{\BrU,m}(\nu)$ for all~$m\in\N_0$, and the 
translates $\ittrans{\BrU,m}(\nu_r)\act\Ired{\itindex{\BrU,m}(\nu_r),\st}$ fulfill 
the same conditions as before. This yields the periodicity of $\cyc{\eZ}{j}$ 
with period length~$n$, which is indeed the minimal period length as seen from 
the generator properties of~$p$. This completes the proof.
\end{proof}

\begin{figure}[h]
\begin{tikzpicture}[scale=10]
\foreach \y in {.09,.06,.035}{
	\tikzmath{\z = 2.2*\y;
	}
	\fill[color=lightgray!50] (.9,0) -- +(\z,0) arc (0:180:1.2*\y);
	\fill[color=white] (.9,0) --+(2*\y,0) arc (0:180:\y);
}
\foreach \y in {.09,.06,.035}{
	\tikzmath{\m = .1-\y;
	}
	\fill[color=lightgray!50] (\m,0) -- +(1.2*\y,0) arc (0:180:1.2*\y);
	\fill[color=white] (\m,0) --+(\y,0) arc (0:180:\y);
}
\fill[color=lightgray] (.5,0) -- +(0:.4) arc (0:180:.4);
\fill[color=white] (.5,0)-- +(0:.38) arc (0:180:.38);
\draw[style=thick] (0:.9) arc (0:180:.4);
 \foreach \x/\y in {.1/$\point{j}{\eX}$,.9/$\point{j}{\eY}$}
    \draw (\x,0.00) -- (\x,-0.02) node [below] {\y};
\coordinate [label=below:$\color{gray}\Cs{j}$] (C) at (.5,.38);
\foreach \y in {.09,.06,.035}{
	\tikzmath{\z = 2*\y;
	}
	\fill[color=lightgray!50] (.1,0) -- +(\z,0) arc (0:180:\y);
	\fill[color=white] (.1,0) --+(1.8*\y,0) arc (0:180:.8*\y);
}
\foreach \y in {.09,.06,.035}{
	\tikzmath{\m = .9-\y;
	}
	\fill[color=lightgray!50] (\m,0) -- +(\y,0) arc (0:180:\y);
	\fill[color=white] (\m,0) --+(.8*\y,0) arc (0:180:.8*\y);
}
\foreach \x in {.1,.9}{
	\foreach \y in {.09,.06,.035}{
		\tikzmath{\z = \x + 2*\y;
		}
		\draw (\z,0) arc (0:180:\y);
		\draw (\x,0) arc (0:180:\y);
	}
}
\draw[style=thick] (-.1,0) -- (1.1,0);
\coordinate [label=$\Plussp{j}$] (H+) at (.5,.17);
\coordinate [label=$\Minussp{j}$] (H-) at (.9,.32);
\draw[->,dashed] (.1,0) arc (270:-90:.1);
\draw[->,dashed] (.9,0) arc (-90:270:.1);
\end{tikzpicture}
\caption[disc model]{Branch cycles for a single branch $\Cs{j}$. The 
$\eX$-cycles circle clockwise, $\eY$-cycles circle 
counterclockwise.}\label{FIG:diskmodel}
\end{figure}
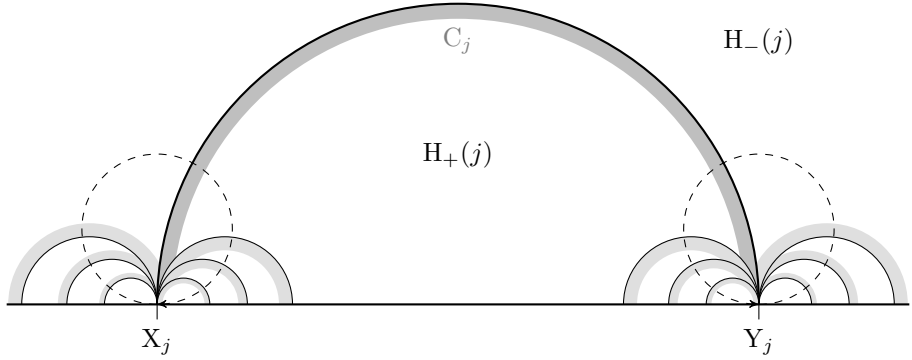

For any $\eZ\in\{\eX,\eY\}$, the set~$A_\eZ$ from Definition~\ref{DEF:xytuple} 
decomposes into finite cycles under the map~$\psi_\eZ$, as shown in 
Lemma~\ref{LEM:periodiccycles}. For~$j\in A_\eZ$, we denote the (minimal) period 
length of the $\psi_\eZ$-cycle of~$j$ by~$\sigma_{\eZ}(j)$, thus 
\begin{align*}
\sigma_{\eZ}(j) & = \min\defset{r\in\N}{\psi_{\eZ}^r(j)=j}
\\
& = \min\defset{n\in\N}{\exists\nu\in\Cs{j}\colon\itindex{\BrU,n}
(\nu)=j\land\ittrans{\BrU,1}(\nu)\cdots\ittrans{\BrU,n}(\nu)\act\point{j}{\eZ}
=\point{j}{\eZ}}\,,
\end{align*}
\index[symbols]{sZ@$\sigma_\eZ$}%
and we set 
\begin{equation}\label{eq:def_uZ}
\cyctrans{j}{\eZ}\coloneqq\cycnext{\eZ}(j)\cycnext{\eZ}(\psi_{\eZ}(j))
\cdots\cycnext{\eZ}(\psi_{\eZ}^{\sigma_{\eZ}(j)-1}(j))\,, 
\end{equation}
\index[symbols]{uZ@$\cyctrans{j}{\eZ}$}%
with $\cycnext{\eZ}$ being the map from Definition~\ref{DEF:xytuple}. As seen in 
the proof of Lemma~\ref{LEM:periodiccycles}, the element~$\cyctrans{j}{\eZ}$ is 
a generator of the stabilizer group of~$\point{j}{\eZ}$ in~$\Gamma$ and hence 
parabolic. In particular, $\cyctrans{j}{\eZ}\act\point{j}{\eZ} = 
\point{j}{\eZ}$. The latter can also be deduced immediately from the property 
that $\cycnext{\eZ}(k)^{-1}\act\point{k}{\eZ}=\point{\psi_{\eZ}(k)}{\eZ}$ for 
all $k\in A_{\eZ}$ by observing that
\begin{align*}
\cyctrans{j}{\eZ}^{-1}\act\point{j}{\eZ}&=\cycnext{\eZ}(\psi_{\eZ}^{\sigma_{\eZ}
(j)-1}(j))^{-1}\cdots\cycnext{\eZ}(j)^{-1}\act\point{j}{\eZ}=\cycnext{\eZ}(\psi_
{\eZ}
^{\sigma_{\eZ}(j)-1}(j))^{-1}\act\point{\psi_{\eZ}^{\sigma_{\eZ}(j)-1}(j)}{
\eZ}\\
&=\point{\psi_{\eZ}\left(\psi_{\eZ}^{\sigma_{\eZ}(j)-1}(j)\right)}{\eZ}=\point{
\psi_{\eZ}^{\sigma_{\eZ}(j)}(j)}{\eZ}=\point{j}{\eZ}\,.
\end{align*}
For any~$j\in A$ we set
\begin{align*}
\cycset{j}{\eZ}&\coloneqq\left\{
\begin{array}{cl}
\defset{\psi_{\eZ}^r(j)}{ r\in\{0,\ldots,\sigma_{\eZ}(j)-1\}} &\qquad \text{if 
$j\in 
A_{\eZ}$},
\\
\varnothing&\qquad \text{otherwise}.
\end{array}
\right.
\end{align*}
\index[symbols]{cyca@$\cycset{j}{\eZ}$}%
Thus, for $j\in A_\eZ$, the set~$\cycset{j}{\eZ}$ contains exactly the elements 
of the $\eZ$-cycle of~$j$. Clearly, for any $k\in\cycset{j}{\eZ}$, the 
sequence~$\cyc{\eZ}{k}$ is a shift of the sequence~$\cyc{\eZ}{j}$, the 
transformation~$\cyctrans{k}{\eZ}$ is conjugate to~$\cyctrans{j}{\eZ}$, and 
$\cycset{k}{\eZ}=\cycset{j}{\eZ}$. From another point of view, the sets 
$\cycset{j}{\eZ}$, $j\in A_\eZ$, are the equivalence classes for the equivalence 
relation
\[
j\sim k\quad\logeq\quad\exists r\in\N\colon\psi_{\eZ}^r(j)=k
\]
\index[symbols]{1005@$\sim$}%
on $A_{\eZ}$. The number of equivalence classes depends on the number of cusps 
of~$\Orbi$. 

\begin{example}\label{EX:G3acc}
Recall the group~$\Gamma_{\lambda}$ and its return graph~$\ReturnGraph{0}$ for its weakly non-collapsing set of branches~$\BrS'$ (Figure~\ref{FIG:G3returngraphwnc}) from 
Example~\ref{EX:G3Graph} as well as its reduced return graph~$\ReturnGraph{6}$ for the set of branches~$\{\Cs{2},\Cs{7}\}$
from Example~\ref{EX:G3Graphred}.
As we have seen, both sets are non-collapsing.
There is one $\eX$-cycle and one $\eY$-cycle in~$\ReturnGraph{6}$ (Figure~\ref{FIG:G3redreturn}) given by
\[
7\edge{t_\lambda^{-1}}7\qquad\text{and}\qquad 2\edge{t_\lambda}2
\]
respectively.
For the reduced set of branches~$\wt\BrS$ (Figure~\ref{FIG:G3altredreturn}) we retrieve these cycles as
\[
7\edge{h}3\edge{h}8\edge{ht_\lambda^{-1}}7\qquad\text{and}\qquad 5\edge{t_\lambda}5
\]
respectively.
Hence, in the former setting we obtain the sets
\[
A_{\eX}=\{7\}\qquad\text{and}\qquad A_{\eY}=\{2\}
\]
as well as the transformations
\[
\cyctrans{7}{\eX}=\cycnext{\eX}(7)=t_\lambda^{-1}\qquad\text{and}\qquad\cyctrans{2}{\eY}=\cycnext{\eY}(2)=t_\lambda\,.
\]
On the other hand, in the latter setting we obtain
\[
A_{\eX}=\{3,7,8\}\qquad\text{and}\qquad A_{\eY}=\{5\}
\]
and the transformations
\[
\cycnext{\eX}(3)=\cycnext{\eX}(7)=h\,,\quad\cycnext{\eX}(8)=ht_{\lambda}^{-1}\,,\qquad\text{and}\qquad\cyctrans{5}{\eY}=\cycnext{\eY}(5)=t_\lambda\,.
\]
But since~$h^3=\id$ we again find
\[
\cyctrans{7}{\eX}=\cycnext{\eX}(7)\cdot\cycnext{\eX}(3)\cdot\cycnext{\eX}(8)=h\cdot h\cdot ht_\lambda^{-1}=t_\lambda^{-1}\,,
\]
but
\[
\cyctrans{3}{\eX}=h^{-1}t_\lambda^{-1}h\qquad\text{and}\qquad\cyctrans{8}{\eX}=ht_\lambda^{-1}h^{-1}\,.
\]
\end{example}

We now introduce the acceleration procedure mentioned above.
Again, the process is presented in geometric terms, by a deletion of certain subsets of unit tangent vectors from the reduced branches.
The emerging system gives rise to a ``faster'' symbolic 
dynamics arising from a new cross section for the geodesic flow (see 
Proposition~\ref{PROP:accsystemcrosssection} below).
The emphasis lies on branches that, even after the identity elimination, are still attached to cuspidal points in the sense of Definition~\ref{DEF:xytuple}.
We therefore call the procedure \emph{cuspidal acceleration} or \emph{cuspidal acceleration algorithm}.
\index[defs]{cuspidal acceleration}%

\begin{defi}[Cuspidal acceleration]\label{DEF:cuspacc}
For $\eZ\in\{\eX,\eY\}$ define the sets
\index[symbols]{Kz@$K_{\eZ}(j)$}%
\index[symbols]{Mz@$M_{\eZ}(j)$}%
\begin{align*}
K_{\eZ}(j)&\coloneqq\left\{
\begin{array}{cl}
\defset{\nu\in\Cs{j,\st}}{\gamma_\nu(+\infty)\in\cycnext{\eZ}(j)\act\,\Ired{\psi_{
\eZ}(j),\st}}&\text{if $j\in A_{\eZ}$},\\
\varnothing&\text{otherwise,}
\end{array}\right.
\intertext{and}
M_{\eZ}(j)&\coloneqq\left\{
\begin{array}{cl}
\defset{\nu\in\Cs{j,\st}}{\gamma_\nu(-\infty)\in\cycnext{\eZ}(\psi_{\eZ}^{-1}
(j))^{
-1}\act\Jset{\psi_{\eZ}^{-1}(j)}}&\text{if }j\in A_{\eZ},\\
\varnothing&\text{otherwise.}
\end{array}\right.
\end{align*}
We call~$K_{\eZ}(j)$ the \emph{forward} and~$M_{\eZ}(j)$ the 
\emph{backward}~$\eZ$-\emph{elimination set} of~$j$.
\index[defs]{forward elimination set}%
\index[defs]{backward elimination set}%
\index[defs]{elimination set!forward}%
\index[defs]{elimination set!backward}%
For $j\in A$ we call
\[
\Cs{j,\acc}\coloneqq\bigcap_{\eZ\in\{\eX,\eY\}}\Cs{j,\st}\!\setminus\bigl(K_{\eZ
}
(j)\cap M_{\eZ}(j)\bigr)
\]
the \emph{acceleration} of $\Cs{j}$
\index[defs]{acceleration}%
\index[symbols]{Caccj@$\Cs{j,\acc}$}%
and
\[
\BrU_{\acc}\coloneqq\bigcup_{j\in A}\Cs{j,\acc}
\]
the \emph{acceleration} of $\BrU$.
\index[defs]{acceleration}%
\index[symbols]{Cacc@$\BrU_{\acc}$}%
\end{defi}

\begin{remark}\label{REM:howtoacc}
We comment on the motivation for Definition~\ref{DEF:cuspacc}. Let $\wh\gamma$ 
be a geodesic on~$\Orbi$ and $\wh c$ a cusp of~$\Orbi$. We say that the set of 
branches~$\BrS$ \emph{detects that $\wh\gamma$ travels towards~$\wh c$} if 
there 
exists a representing geodesic~$\gamma$ of~$\wh\gamma$ on~$\H$ and 
$\eZ\in\{\eX,\eY\}$ and $j\in A_\eZ$ such that $\eZ_j$ represents~$\wh c$, the 
geodesic~$\gamma$ intersects~$\BrU_j$ at some time, say~$t_0$, and 
$\gamma(+\infty) \in \cycnext{\eZ}(j)\act\Ired{\psi_{\eZ}(j)}$. In such a case, 
the next intersection (after time~$t_0$) between~$\gamma$ and~$\Gamma\act\BrU$ 
is on~$\cycnext{\eZ}(j)\act\BrU_{\psi_\eZ(j)}$ at, say, time~$t_1$. Further 
next and previous intersections of~$\gamma$ and~$\Gamma\act\BrU$ might be 
``near''~$\eZ_j$, thus given by the $\eZ$-cycle of~$j$. More precisely, it might 
happen that $\gamma$ intersects 
$\cycnext{\eZ}(\psi_{\eZ}^{-1}(j))^{-1}\act\BrU_{\psi_{\eZ}^{-1}(j)}$, in which 
case the previous intersection of~$\gamma$ and~$\Gamma\act\BrU$ is indeed 
on~$\cycnext{\eZ}(\psi_{\eZ}^{-1}(j))^{-1}\act\BrU_{\psi_{\eZ}^{-1}(j)}$, as can 
easily be seen from the definition of the  $\eZ$-cycle of~$j$. Likewise, the 
next intersection after time~$t_1$ might be 
on~$\cycnext{\eZ}(j)\cycnext{\eZ}(\psi_\eZ(j))\act\BrU_{\psi_\eZ^2(j)}$. Let us 
suppose that $\gamma$ intersects 
\begin{align}\label{eq:intersect_sequence}
 & 
\cycnext{\eZ}(\psi_\eZ^{-1}(j))^{-1}\cycnext{\eZ}(\psi_\eZ^{-2}(j))^{-1}\cdots 
\cycnext{\eZ}(\psi_\eZ^{-k_1}(j))^{-1}\act\BrU_{\psi_\eZ^{-k_1}(j)}\,,\  
\ldots\,, 
 \\ \nonumber
 &  \cycnext{\eZ}(\psi_\eZ^{-1}(j))^{-1}\act \BrU_{\psi_\eZ^{-1}(j)}\,,  
\BrU_j\,, \cycnext{\eZ}(j)\act \BrU_{\psi_\eZ(j)}\,,\  \ldots\,, 
 \\ \nonumber 
 & \cycnext{\eZ}(j) \cycnext{\eZ}(\psi_\eZ(j))\cdots 
\cycnext{\eZ}(\psi_\eZ^{k_2-1}(j))\act\BrU_{\psi_\eZ^{k_2}(j)}
\end{align}
with $k_1,k_2\in\N_0$ maximal. Large values for~$k_1,k_2$ indicate 
that~$\wh\gamma$ stays ``near'' the cusp~$\wh c$ for a rather long time, and 
larger values for~$k_1,k_2$ translate to deeper cusp excursions. We call the 
part of~$\wh\gamma$ corresponding to~\eqref{eq:intersect_sequence} a 
\emph{maximal cusp excursion} into the cusp region of~$\wh c$ or a 
\emph{sojourn} of~$\wh\gamma$ near~$\wh c$. We emphasize that $\wh\gamma$ can 
experience several disjoint sojourns at the same cusp, each one separated from 
the others by some time spend ``far away'' from the cusp. 

Each sojourn of~$\wh\gamma$ near~$\wh c$ typically contains several windings 
around the cusp (region of)~$\wh c$, expressed by a high power of the parabolic 
element~$\cyctrans{j}{\eZ}$ from~\eqref{eq:def_uZ}, as explained further below. 
Using the (reduced) set of branches~$\BrS$ for the coding of the geodesic flow on~$\Orbi$, 
as done for the development of slow transfer operators, leads to separate coding 
of each single cusp winding (and also of all the intermediate intersections). It 
is exactly this detailed (``slow'') coding of cusp windings that cause slow 
transfer operators typically to be non-nuclear. To enforce nuclearity, the idea 
is to encode each sojourn by a single step in the coding (``fast'' coding) or, 
in other words, to induce on the cusp excursions, at the expense of constructing 
an infinitely branched discrete dynamical system (and a symbolic dynamics with 
an infinite alphabet). Technically, this acceleration will be achieved by 
omitting those tangent vectors from the branches in~$\BrS$ that are internal to 
a $\eZ$-cycle. 

To be more precise, we now express the intersection properties 
in~\eqref{eq:intersect_sequence} in terms of the endpoints of~$\gamma$. To 
facilitate notation, we set, for any $k\in\cycset{j}{\eZ}$, 
\begin{equation}\label{eq:def_intergroup}
\cycnext{\eZ}(j,k)\coloneqq\cycnext{\eZ}(j)\cycnext{\eZ}(\psi_{\eZ}(j))
\cdots\cycnext{\eZ}(\psi_{\eZ}^{r-1}(j))
\end{equation}
\index[symbols]{gz@$\cycnext{\eZ}(j,k)$}%
with $r\coloneqq\min\defset{\ell\in\N}{\psi_{\eZ}^\ell(j)=k}$. Then 
\begin{eqnarray*}
\cycnext{\eZ}(j,\psi_{\eZ}(j))=\cycnext{\eZ}(j)&\text{and}&\cycnext{\eZ}({j,j})
=\cyctrans{j}{\eZ}.
\end{eqnarray*}
If $\eZ=\eX$, then the interval~$\cycnext{\eX}({j})\act 
I_{\psi_{\eX}(j)}$ decomposes as
\[
\bigcup_{n\in\N_0}\bigcup_{k\in\cycset{j}{\eX}}\left(\cyctrans{j}{
\eX}^n\cycnext{\eX}({j,k})\cycnext{\eX}(k)\act\point{\psi_{\eX}(k)}{\eY},\ 
\cyctrans
{j}{\eX}^n\cycnext{\eX}({j,k})\act\point{k}{\eY}\right)
\]
and thus the set~$\cycnext{\eX}({j})\act\,
\Ired{\psi_{\eX}(j)}$ decomposes as
\[
\bigcup_{n\in\N_0}\bigcup_{k\in\cycset{j}{\eX}}\left(\cyctrans{j}{
\eX}^n\cycnext{\eX}({j,k})\cycnext{\eX}(k)\act\point{\psi_{\eX}(k)}{\eY},\ 
\cyctrans
{j}{\eX}^n\cycnext{\eX}({j,k})\act\point{k}{\eY}\right)\cap\cyctrans
{j}{\eX}^n\cycnext{\eX}({j,k})\act\,\Ired{k}\,.
\]
On the other hand, the interval~$\cycnext{\eX}(\psi_{\eX}^{-1}(j))^{-1}\act 
J_{\psi_\eX^{-1}(j)}$ decomposes as 
\[
\bigcup_{n\in\N_0}\bigcup_{k\in\cycset{j}{\eX}}\left( 
\cyctrans{j}{\eX}^{-n}\cycnext{\eX}(k,j)^{-1}\act\point{k}{\eY}, 
\cyctrans{j}{\eX}^{-n}\cycnext{\eX}({k,j})^{-1}\cycnext{\eX}(\psi_\eX^{-1}(k))^{
-1}\act\point{\psi_{\eX}^{-1}(k)}{\eY} \right)\,.
\]
If $\eZ=\eY$, then the decomposition is analogous, with the roles of $\eX$ and 
$\eY$ interchanged and the order of the interval boundaries switched. See 
Figure~\ref{FIG:XYcyc}. For any $n\in\N_0$, $k\in \cycset{j}{\eZ}$, 
$\eZ\in\{\eX,\eY\}$ we set 
\index[symbols]{DnX@$D_{n,\eX}^+(j,k)$}%
\index[symbols]{DnY@$D_{n,\eY}^+(j,k)$}%
\index[symbols]{DnX@$D_{n,\eX}^-(j,k)$}%
\index[symbols]{DnY@$D_{n,\eY}^-(j,k)$}%
\begin{align*}
D_{n,\eX}^+(j,k) & \coloneqq \left(\cyctrans{j}{
\eX}^n\cycnext{\eX}({j,k})\cycnext{\eX}(k)\act\point{\psi_{\eX}(k)}{\eY},\ 
\cyctrans
{j}{\eX}^n\cycnext{\eX}({j,k})\act\point{k}{\eY}\right)
\intertext{and}
D_{n,\eY}^+(j,k) & \coloneqq 
\left(\cyctrans{j}{\eY}^n\cycnext{\eY}({j,k})\act\point{k}{\eX},\ \cyctrans{j}{
\eY}^n\cycnext{\eY}({j,k})\cycnext{\eY}(k)\act\point{\psi_{\eY}(k)}{\eX}\right)
\intertext{as well as}
D_{n,\eX}^-(j,k) & \coloneqq \left( 
\cyctrans{j}{\eX}^{-n}\cycnext{\eX}(k,j)^{-1}\act\point{k}{\eY}, 
\cyctrans{j}{\eX}^{-n}\cycnext{\eX}({k,j})^{-1}\cycnext{\eX}(\psi_\eX^{-1}(k))^{
-1}\act\point{\psi_{\eX}^{-1}(k)}{\eY} \right)
\intertext{and}
D_{n,\eY}^-(j,k) & \coloneqq \left( 
\cyctrans{j}{\eY}^{-n}\cycnext{\eY}({k,j})^{-1}\cycnext{\eY}(\psi_\eY^{-1}(k))^{
-1}\act\point{\psi_{\eY}^{-1}(k)}{\eX}, 
\cyctrans{j}{\eY}^{-n}\cycnext{\eY}(k,j)^{-1}\act\point{k}{\eX} \right)\,.
\end{align*}
Based on these intervals we further set
\begin{align*}
\mc D_{n,\eX}^+(j,k) & \coloneqq D_{n,\eX}^+(j,k) \cap\cyctrans
{j}{\eX}^n\cycnext{\eX}({j,k})\act\,\Ired{k}
\intertext{and}
\mc D_{n,\eY}^+(j,k) & \coloneqq D_{n,\eY}^+(j,k)\cap\cyctrans{j}{\eY}^n\cycnext{\eY}({j,k})\act\,\Ired{k}\,.
\end{align*}
\index[symbols]{DnX@$\mc{D}_{n,\eX}^+(j,k)$}%
\index[symbols]{DnY@$\mc{D}_{n,\eY}^+(j,k)$}%
Then the geodesic~$\gamma$ (with the properties as 
in~\eqref{eq:intersect_sequence}) satisfies $\gamma(+\infty)\in 
\mc D_{n,\eZ}^+(j,k)$ if and only if $k_2 = n + r$ with $r$ as 
in~\eqref{eq:def_intergroup}. It satisfies $\gamma(-\infty) \in 
D_{m,\eZ}^-(j,k)$ if and only if $k_1 = m + s$ with $s = 
\min\defset{\ell\in\N}{\psi_{\eZ}^\ell(k) = j}$. The sum of the values of~$n$ 
for~$k_1$ and~$k_2$ is the number of full windings around the cusp (region 
of)~$\wh c$ of this sojourn of~$\wh\gamma$ near~$\wh c$.

For the acceleration we now want to eliminate from the branches all those 
tangent vectors that cause intersections within a sojourn, or, in other words, 
within a $\eZ$-cycle. This elimination process is of a local nature; we only 
need to ask for the nature of the next and the previous intersection, and not of 
any further intersections. For the branch~$\BrU_j$ it means that we need to 
eliminate all those vectors~$\nu\in\BrU_j$ for which $\gamma_\nu(+\infty) \in 
\cycnext{\eZ}(j)\act\,\Ired{\psi_\eZ(j),\st}$ and $\gamma_\nu(-\infty)\in 
\cycnext{\eZ}(\psi_\eZ^{-1}(j))^{-1}\act\Jset{\psi_\eZ^{-1}(j)}$.
Thus, we need to eliminate from~$\BrU_j$ exactly the set $K_\eZ(j)\cap M_\eZ(j)$ 
for the acceleration. 

In Proposition~\ref{PROP:accsystem}\eqref{accsystem:disjointremoves} below we 
show that the sets $K_{\eX}(j)$ and $K_{\eY}(j)$ as well as $M_{\eX}(j)$ and 
$M_{\eY}(j)$ do not intersect, thus, there is no interference 
between different cycles during the elimination or acceleration procedure. In 
the remainder of this section we show that this heuristics on the necessary 
modifications of the set of branches indeed leads to the desired results.
\end{remark}

\begin{figure}[h]
\begin{tikzpicture}[scale=.955]
\begin{scope}
\clip (-6.2,-.06) rectangle (6.2,4.5);
\fill[color=lightgray!90] (0,0) arc (180:0:30) -- (59.85,0) arc (0:180:29.85) -- 
cycle;  
\foreach \x in {6,3.3,2}{
	\fill[color=lightgray!45] (0,0) arc (180:0:\x/2) -- (\x - .15,0) arc 
(0:180:\x/2 - .15) -- cycle;
	\fill[color=lightgray!45] (0,0) arc (0:180:\x/2) -- (-\x - .15,0) arc 
(180:0:\x/2 + .15) -- cycle;
}
\foreach \x in {6,3.3,2}{
\draw (0,0) arc (180:0:\x/2);
\draw (0,0) arc (0:180:\x/2);
}
\draw[style=thick] (0,0) arc (180:160:30);
\foreach \x in {1.87,1.89,1.91,1.76,1.74,1.72,1.68,1.66,1.64,1.62,1.6} {
	\pgfmathsetmacro{\Cos}{cos(\x*pi r)}%
	\pgfmathsetmacro{\Sin}{sin(\x*pi r)}%
	\node at (3*\Cos,3*\Sin+3)[circle,fill,inner sep=.5pt]{};
	}
\foreach \x in {1.95,1.93,1.91,1.78,1.76,1.74,1.68,1.66,1.64,1.62,1.6} {	
	\pgfmathsetmacro{\Cos}{cos(\x*pi r)}%
	\pgfmathsetmacro{\Sin}{sin(\x*pi r)}%
	\node at (-3*\Cos,3*\Sin+3)[circle,fill,inner sep=.5pt]{};
	}
\end{scope}
\foreach \x/\y/\c in
{3/2.9/$\cycnext{\eX}(j)\act\Cs{\psi_{\eX}(j)}$,
1.61/1.51/$\cycnext{\eX}({j,k})\act\Cs{k}$,
1/.9/$\cyctrans{j}{\eX}\act\Cs{j}$,
-1.4/1.6/$\cyctrans{j}{\eX}^{-1}\act\Cs{j}$,
-1.81/2.26/$\cycnext{\eX}({\ell,j})^{-1}\act\Cs{\ell}$,
-3/3.65/$\cycnext{\eX}(\psi_{\eX}^{-1}(j))^{-1}\act\Cs{\psi_{\eX}^{-1}(j)}$}{
	\coordinate [label=below:\small{\c}] (C2\x) at (\x,\y);
}
\draw[style=thick] (0,0) -- (0,-.1) node [below] {\small{$\eX_j$}};
\foreach \x/\z/\y in
{-3.3/.6/$\cycnext{\eX}({\ell,j})^{-1}\act\eY_{\ell}$,
-2/.1/$\cyctrans{j}{\eX}^{-1}\act\eY_{j}$,
2/.1/$\cyctrans{j}{\eX}\act\eY_{j}$,
3.3/.6/$\cycnext{\eX}({j,k})\act\eY_{k}$}{
	\draw (\x,0) -- (\x,-\z) node [below] {\small{\y}};
}
\foreach \x/\y in
{-5/$\cycnext{\eX}(\psi_{\eX}^{-1}(j))^{-1}\act\eY_{\psi_{\eX}^{-1}(j)}$,
5.4/$\cycnext{\eX}(j)\act\eY_{\psi_{\eX}(j)}$}{
	\draw[color=white] (\x,0) -- (\x,-.1) node [below] 
{\color{black}\small{\y}};
}
\draw (-6,0) -- (-6,-.1);
\draw (6,0) -- (6,-.1);
\draw[style=thick] (-6.2,0) -- (6.2,0);
\coordinate [label=right:$\Cs{j}$] (CS) at (.31,4);
\end{tikzpicture}

\begin{tikzpicture}[scale=.955]
\begin{scope}
\clip (-6.2,-.06) rectangle (6.2,4.5);
\fill[color=lightgray!90] (0,0) arc (180:0:30) -- (60.15,0) arc (0:180:30.15) -- 
cycle;  
\foreach \x in {6,3.3,2}{
	\fill[color=lightgray!45] (0,0) arc (180:0:\x/2) -- (\x + .15,0) arc 
(0:180:\x/2 + .15) -- cycle;
	\fill[color=lightgray!45] (0,0) arc (0:180:\x/2) -- (-\x + .15,0) arc 
(180:0:\x/2 - .15) -- cycle;
}
\foreach \x in {6,3.3,2}{
	\draw (0,0) arc (180:0:\x/2);
	\draw (0,0) arc (0:180:\x/2);
}
\draw[style=thick] (0,0) arc (180:160:30);
\foreach \x in {1.95,1.93,1.91,1.78,1.76,1.74,1.68,1.66,1.64,1.62,1.6} {
	\pgfmathsetmacro{\Cos}{cos(\x*pi r)}%
	\pgfmathsetmacro{\Sin}{sin(\x*pi r)}%
	\node at (3*\Cos,3*\Sin+3)[circle,fill,inner sep=.5pt]{};
	}
\foreach \x in {1.87,1.89,1.91,1.76,1.74,1.72,1.68,1.66,1.64,1.62,1.6} {	
	\pgfmathsetmacro{\Cos}{cos(\x*pi r)}%
	\pgfmathsetmacro{\Sin}{sin(\x*pi r)}%
	\node at (-3*\Cos,3*\Sin+3)[circle,fill,inner sep=.5pt]{};
	}
\end{scope}
\foreach \x/\y/\c in
{-3/2.9/$\cycnext{\eY}(j)\act\Cs{\psi_{\eY}(j)}$,
-1.61/1.51/$\cycnext{\eY}({j,k})\act\Cs{k}$,
-1/.9/$\cyctrans{j}{\eY}\act\Cs{j}$,
1.4/1.6/$\cyctrans{j}{\eY}^{-1}\act\Cs{j}$,
1.81/2.26/$\cycnext{\eY}({\ell,j})^{-1}\act\Cs{\ell}$,
3/3.65/$\cycnext{\eY}(\psi_{\eY}^{-1}(j))^{-1}\act\Cs{\psi_{\eY}^{-1}(j)}$}{
	\coordinate [label=below:\small{\c}] (C2\x) at (\x,\y);
}
\draw[style=thick] (0,0) -- (0,-.1) node [below] {\small{$\eY_j$}};
\foreach \x/\z/\y in
{3.3/.6/$\cycnext{\eY}({\ell,j})^{-1}\act\eX_{\ell}$,
2/.1/$\cyctrans{j}{\eY}^{-1}\act\eX_{j}$,
-2/.1/$\cyctrans{j}{\eY}\act\eX_{j}$,
-3.3/.6/$\cycnext{\eY}({j,k})\act\eX_{k}$}{
	\draw (\x,0) -- (\x,-\z) node [below] {\small{\y}};
}
\foreach \x/\y in
{5/$\cycnext{\eY}(\psi_{\eY}^{-1}(j))^{-1}\act\eX_{\psi_{\eY}^{-1}(j)}$,
-5.4/$\cycnext{\eY}(j)\act\eX_{\psi_{\eY}(j)}$}{
	\draw[color=white] (\x,0) -- (\x,-.1) node [below] 
{\color{black}\small{\y}};
}
\draw (-6,0) -- (-6,-.1);
\draw (6,0) -- (6,-.1);
\draw[style=thick] (-6.2,0) -- (6.2,0);
\coordinate [label=right:$\Cs{j}$] (CS) at (-.4,4);
\end{tikzpicture}
\caption[branch cycle]{The situation for an~$\eX$-cycle (above) and 
a~$\eY$-cycle (below).}\label{FIG:XYcyc}
\end{figure}
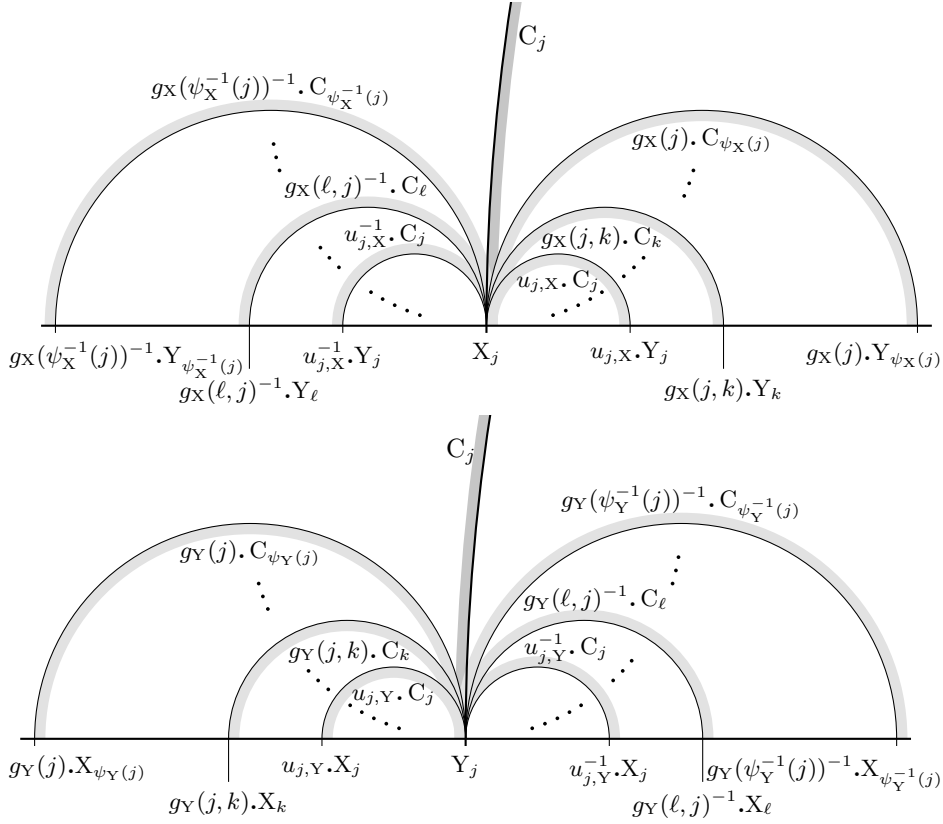

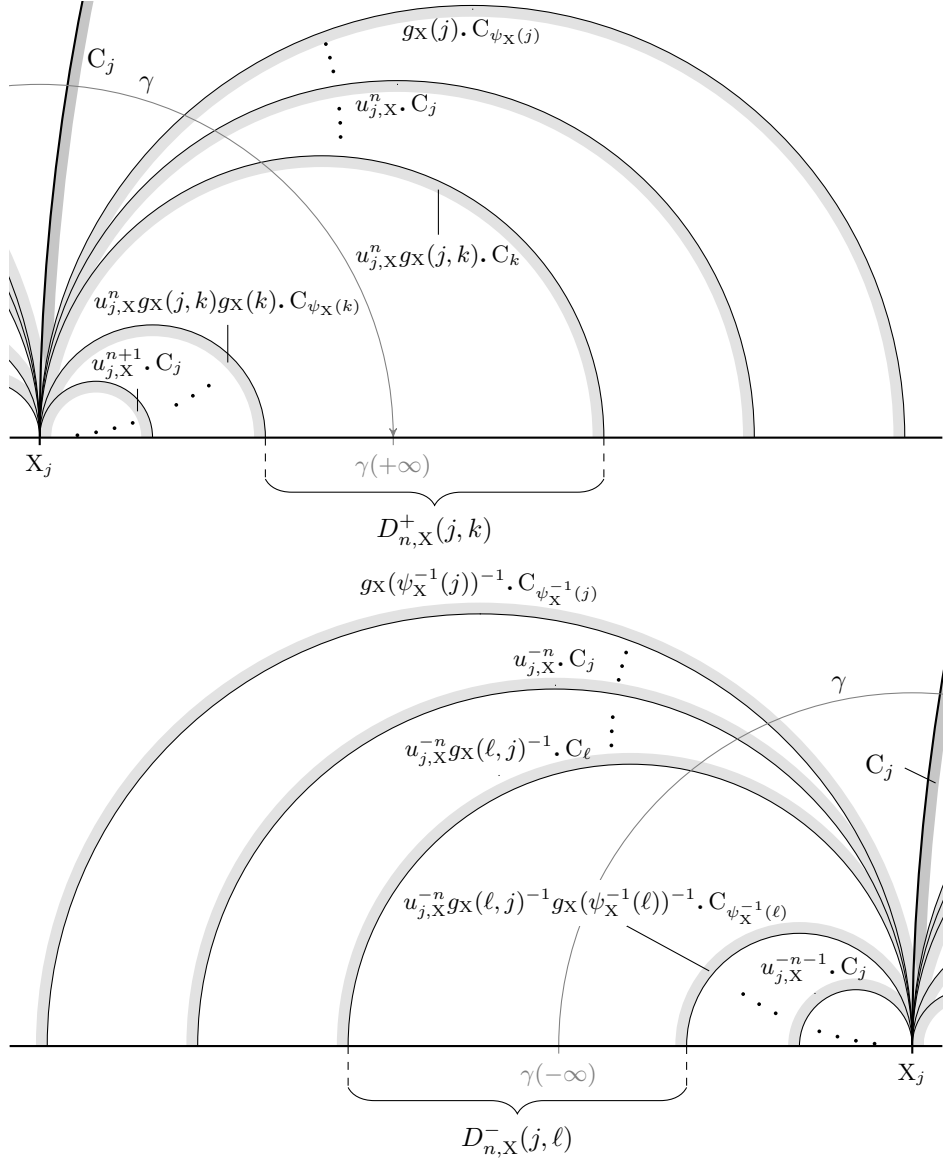
\begin{figure}[h]
\begin{tikzpicture}
\def \shi {-5.8}
\begin{scope}
\clip (-6.2,-.06) rectangle (6.2,5.8);
\fill[color=lightgray!90] (0+\shi,0) arc (180:0:30) -- (59.85+\shi,0) arc 
(0:180:29.85) -- cycle;  
\foreach \x in {11.5,9.5,7.5,3,1.5}{
	\fill[color=lightgray!45] (0+\shi,0) arc (180:0:\x/2) -- (\x - .15+\shi,0) 
arc (0:180:\x/2 - .15) -- cycle;
	\fill[color=lightgray!45] (0+\shi,0) arc (0:180:\x/2) -- (-\x - .15+\shi,0) 
arc (180:0:\x/2 + .15) -- cycle;
}
\foreach \x in {11.5,9.5,7.5,3,1.5}{
\draw (0+\shi,0) arc (180:0:\x/2);
\draw (0+\shi,0) arc (0:180:\x/2);
}
\draw[style=thick] (0+\shi,0) arc (180:160:30);
\foreach \x in {2.1,2.085,2.07,2.03,2.015,2,1.69,1.67,1.65,1.6,1.58,1.56,1.54} {
	\pgfmathsetmacro{\Cos}{cos(\x*pi r)}%
	\pgfmathsetmacro{\Sin}{sin(\x*pi r)}%
	\node at (4*\Cos+\shi,4*\Sin+4)[circle,fill,inner sep=.5pt]{};
	}
\draw[color=gray,->] (-4.7+\shi,0) arc (180:0:4.7);	
\end{scope}
\foreach \x/\y/\w/\c in
{5.75/5.7/5.7/$\cycnext{\eX}(j)\act\Cs{\psi_{\eX}(j)}$,
4.75/4.7/4.7/$\cyctrans{j}{\eX}^n\act\Cs{j}$,
5.3/3.3/2.7/$\cyctrans{j}{\eX}^n\cycnext{\eX}({j,k})\act\Cs{k}$}{
	\draw (\x+\shi,\y) -- (\x+\shi,\w) node [below] {\small{\c}};
}
\foreach \x/\y/\w/\c in
{2.5/1/1.5/$\cyctrans{j}{\eX}^n\cycnext{\eX}({j,k})\cycnext{\eX}(k)\act\Cs{\psi_
{\eX}(k)}$,
1.3/.35/.65/$\cyctrans{j}{\eX}^{n+1}\act\Cs{j}$}{
	\draw (\x+\shi,\y) -- (\x+\shi,\w) node [above] {\small{\c}};
}
\draw[dashed] (3+\shi,0) -- (3+\shi,-.55);
\draw[dashed] (7.5+\shi,0) -- (7.5+\shi,-.55);
\draw [decorate,decoration={brace,amplitude=10pt,mirror}]
(3+\shi,-.55) -- (7.5+\shi,-.55) node [below,midway,yshift=-10pt]
{$D_{n,\eX}^+(j,k)$};
\draw[style=thick] (0+\shi,0) -- (0+\shi,-.1) node [below] {\small{$\eX_j$}};
\draw[color=gray] (4.7+\shi,0) -- (4.7+\shi,-.1) node [below] 
{\small{$\gamma(+\infty)$}};
\draw[style=thick] (-6.2,0) -- (6.2,0);
\coordinate [label=right:$\Cs{j}$] (CS) at (.5+\shi,5);
\coordinate [label=right:$\gamma$] (gam) at (1.2+\shi,4.7);
\end{tikzpicture}

\begin{tikzpicture}
\def \shi {5.8}
\begin{scope}
\clip (-6.2,-.06) rectangle (6.2,6);
\fill[color=lightgray!90] (0+\shi,0) arc (180:0:30) -- (59.85+\shi,0) arc 
(0:180:29.85) -- cycle;  
\foreach \x in {11.5,9.5,7.5,3,1.5}{
	\fill[color=lightgray!45] (0+\shi,0) arc (180:0:\x/2) -- (\x - .15+\shi,0) 
arc (0:180:\x/2 - .15) -- cycle;
	\fill[color=lightgray!45] (0+\shi,0) arc (0:180:\x/2) -- (-\x - .15+\shi,0) 
arc (180:0:\x/2 + .15) -- cycle;
}
\foreach \x in {11.5,9.5,7.5,3,1.5}{
\draw (0+\shi,0) arc (180:0:\x/2);
\draw (0+\shi,0) arc (0:180:\x/2);
}
\draw[style=thick] (0+\shi,0) arc (180:160:30);
\foreach \x in {2.1,2.085,2.07,2.03,2.015,2,1.69,1.67,1.65,1.6,1.58,1.56,1.54} 
{	
	\pgfmathsetmacro{\Cos}{cos(\x*pi r)}%
	\pgfmathsetmacro{\Sin}{sin(\x*pi r)}%
	\node at (-4*\Cos+\shi,4*\Sin+4)[circle,fill,inner sep=.5pt]{};
	}
\draw[color=gray] (-4.7+\shi,0) arc (180:0:4.7);
\end{scope}
\foreach \x/\y/\w/\c in
{}{
	\draw (-\x+\shi,\y) -- (-\x+\shi,\w) node [below] {\small{\c}};
}
\foreach \x/\y/\w/\c in
{5.75/5.75/5.75/$\cycnext{\eX}(\psi_{\eX}^{-1}(j))^{-1}\act\Cs{\psi_{\eX}^{-1}
(j)}$,
4.75/4.8/4.8/$\cyctrans{j}{\eX}^{-n}\act\Cs{j}$,
1.3/.7/.7/$\cyctrans{j}{\eX}^{-n-1}\act\Cs{j}$,
5.5/3.6/3.6/$\cyctrans{j}{\eX}^{-n}\cycnext{\eX}({\ell,j})^{-1}\act\Cs{\ell}$}{
	\draw (-\x+\shi,\y) -- (-\x+\shi,\w) node [above] {\small{\c}};
}
\draw (-2.7+\shi,1) -- (-3.8+\shi,1.6);
\fill[color=white] (-3.8+\shi,1.6) rectangle (-5+\shi,2.2);
\coordinate 
[label=above:\small{$\cyctrans{j}{\eX}^{-n}\cycnext{\eX}({\ell,j})^{-1}\cycnext{
\eX}(\psi_{\eX}^{-1}(\ell))^{-1}\act\Cs{\psi_{\eX}^{-1}(\ell)}$}] (LongCS) at 
(-4.2+\shi,1.5);
\draw[dashed] (-3+\shi,0) -- (-3+\shi,-.55);
\draw[dashed] (-7.5+\shi,0) -- (-7.5+\shi,-.55);
\draw [decorate,decoration={brace,amplitude=10pt}]
(-3+\shi,-.55) -- (-7.5+\shi,-.55) node [below,midway,yshift=-10pt]
{$D_{n,\eX}^-(j,\ell)$};
\draw[color=gray] (-4.7+\shi,0) -- (-4.7+\shi,-.1) node [below] 
{\small{$\gamma(-\infty)$}};
\coordinate [label=right:$\gamma$] (gam) at (-1.2+\shi,4.8);
\draw (.31+\shi,3.5) -- (-.1+\shi,3.7) node [left] {$\Cs{j}$};
\draw[style=thick] (0+\shi,0) -- (0+\shi,-.1) node [below] {\small{$\eX_j$}};
\draw[style=thick] (-6.2,0) -- (6.2,0);
\end{tikzpicture}
\caption[branch cycle]{A representative~$\gamma$ of a geodesic with sojourn near 
a cusp detected by an~$\eX$-cycle in forward (above) and backward (below) 
time.}\label{FIG:sojourn}
\end{figure}
We recall that $\defset{\Cs{j}}{j\in A}$ is a reduced set of branches for the geodesic 
flow on~$\Orbi$.
In the case that $\Orbi$ does not have cusps, it is consistent with 
Definition~\ref{DEF:cuspacc} to set $\Cs{j,\acc}\coloneqq\Cs{j,\st}$ for all 
$j\in A$. For $M\subseteq\UTB\H$ we define
\begin{equation}\label{EQDEF:IJof}
I(M)\coloneqq\defset{\gamma_\nu(+\infty)}{\nu\in 
M}\quad\text{and}\quad J(M)\coloneqq\defset{\gamma_\nu(-\infty)}{\nu\in M}.
\end{equation}
We set
\begin{align}\label{EQDEF:Astar}
A^*\coloneqq\defset{j\in A}{\Cs{j,\acc}\ne\varnothing}.
\end{align}
\index[symbols]{As@$A^*$}%
We further set $\CrSc_{\acc}\coloneqq\pi(\BrU_{\acc})$.
\index[symbols]{Caccs@$\CrSc_{\acc}$}%

\begin{prop}\label{PROP:accsystem}
For any $j\in A$ and $\eZ\in\{\eX,\eY\}$, the elimination sets~$K_\eZ(j)$ 
and $M_\eZ(j)$ and the acceleration~$\Cs{j,\acc}$ of~$\BrU_j$ satisfy the 
following properties:
\begin{enumerate}[label=$\mathrm{(\roman*)}$, ref=$\mathrm{\roman*}$]
\item\label{accsystem:disjointremoves}
We have $K_{\eX}(j)\cap K_{\eY}(j)=\varnothing$ and 
$M_{\eX}(j)\cap M_{\eY}(j)=\varnothing$. 
\item\label{accsystem:empty}
The set~$\Cs{j,\acc}$ is empty if and only if there exists~$\eZ\in\{\eX,\eY\}$ 
such that $j\in A_{\eZ}$ and 
\[
\Ired{j,\st}=\cycnext{\eZ}(j)\act\,\Ired{\psi_{\eZ}(j),\st}\quad\text{and}\quad\Jset{j}=\cycnext{
\eZ}({\psi_{\eZ}^{-1}(j)})^{-1}\act\Jset{\psi_{\eZ}^{-1}(j)}\,.
\]
\item\label{accsystem:nonempty} If $j\in A_{\eX}\cap A_{\eY}$, then 
$\Cs{j,\acc}\ne\varnothing$.
\end{enumerate}
\end{prop}

\begin{proof}
For the proof of~\eqref{accsystem:disjointremoves} we suppose that $j\in 
A_{\eX}\cap A_{\eY}$ (because otherwise there is nothing to show) and 
assume, in order to seek a contradiction, that the sets~$K_{\eX}(j)$ 
and~$K_{\eY}(j)$ are \emph{not} disjoint. Then there exists $\nu\in\Cs{j,\st}$ 
such that 
\begin{equation}\label{eq:nu_intersect}
\gamma_\nu(+\infty)\in\cycnext{\eX}({j})\act\,\Ired{\psi_{\eX}(j),\st}\cap\cycnext{\eY
}({j})\act\,\Ired{\psi_{\eY}(j),\st}\subseteq\cycnext{\eX}({j})\act\Iset{\psi_{\eX}(j)}\cap\cycnext{\eY
}({j})\act\Iset{\psi_{\eY}(j)}\,.
\end{equation}
For any $\eZ \in\{\eX,\eY\}$ we have 
\begin{equation}\label{eq:next_general}
\cycnext{\eZ}(j)\act\point{\psi_{\eZ}(j)}{\eZ}=\point{j}{\eZ}\qquad\text{and}
\qquad\cycnext{\eZ}(j)\act\overline{\base{\Cs{\psi_{\eZ}(j)}}}\ne\overline{\base
{\Cs{j}}}\,. 
\end{equation}
Thus, \eqref{eq:nu_intersect} implies 
\[
\cycnext{\eX}({j})\act\overline{\base{\Cs{\psi_{\eX}(j)}}}\cap\cycnext{\eY}(j)
\act\overline{\base{\Cs{\psi_{\eY}(j)}}}\ne\varnothing\,.
\]
From~\eqref{BP:disjointunion} it now follows that $\cycnext{\eX}({j}) = 
\cycnext{\eY}(j) \eqqcolon g$ and $\psi_{\eX}(j) = \psi_{\eY}(j) \eqqcolon k$.  
With~\eqref{eq:next_general} we obtain $g\act\eZ_k = \eZ_j$, and hence 
$g\act\overline{\base{\Cs{k}}} = \overline{\base{\Cs{j}}}$, which 
contradicts~\eqref{eq:next_general}. In turn, $K_{\eX}(j)\cap 
K_{\eY}(j)=\varnothing$. The proof of $M_{\eX}(j)\cap M_{\eY}(j)=\varnothing$ is 
analogous.

For the proof of~\eqref{accsystem:empty} we let $\eZ\in\{\eX,\eY\}$ and note 
that the equalities $\Ired{j,\st}=\cycnext{\eZ}(j)\act\,\Ired{\psi_{\eZ}(j),\st}$ and 
$\Jset{j}=\cycnext{\eZ}(\psi_{\eZ}^{-1}(j))^{-1}\act\Jset{\psi_{\eZ}^{-1}(j)}$ 
are equivalent to $K_{\eZ}(j)=\Cs{j,\st}$ and $M_{\eZ}(j)=\Cs{j,\st}$,  
respectively. Hence, if these equalities are satisfied for the considered 
index~$j\in A_\eZ$, then 
\[
\Cs{j,\acc}\subseteq\Cs{j,\st}\!\setminus\bigl(K_{\eZ}(j)\cap 
M_{\eZ}(j)\bigr)=\Cs{j,\st}\!\setminus\Cs{j,\st}=\varnothing\,.
\]
In order to prove the converse implication, we suppose that 
$\Cs{j,\acc}=\varnothing$. Then we find $\eZ\in\{\eX,\eY\}$ such that $j\in 
A_{\eZ}$ because otherwise $\Cs{j,\acc}=\Cs{j,\st}\ne\varnothing$. Thus, we know 
already (see the discussion before this proposition) that 
$K_\eZ(j)\not=\varnothing\not= M_\eZ(j)$ and will now show that $K_\eZ(j) = 
\Cs{j,\st} = M_\eZ(j)$. From the definition of~$\Cs{j,\acc}$ it follows 
immediately that 
\[
\Cs{j,\st}=\bigl(K_{\eX}(j)\cap M_{\eX}(j)\bigr)\cup\bigl(K_{\eY}(j)\cap 
M_{\eY}(j)\bigr)\,.
\]
Let $\eZ'\in\{\eX,\eY\}$ be such that $\{\eZ,\eZ'\} = \{\eX,\eY\}$. 
From~\eqref{accsystem:disjointremoves} we obtain the inclusions
$K_{\eZ'}(j)\subseteq\Cs{j,\st}\!\setminus K_{\eZ}(j)$ and 
$M_{\eZ'}(j)\subseteq\Cs{j,\st}\!\setminus M_{\eZ}(j)$. Thus, $K_{\eZ'}(j)\cap 
M_{\eZ'}(j)\subseteq\Cs{j,\st}\!\setminus\left(K_{\eZ}(j)\cup 
M_{\eZ}(j)\right)$ and hence
\[
\Cs{j,\st}=\bigl(K_{\eZ}(j)\cap 
M_{\eZ}(j)\bigr)\cup\Cs{j,\st}\!\setminus\bigl(K_{\eZ}(j)\cup 
M_{\eZ}(j)\bigr)\,.
\]
It follows that
\begin{align*}
\varnothing&=\Cs{j,\st}\!\setminus\Bigl(\bigl(K_{\eZ}(j)\cap 
M_{\eZ}(j)\bigr)\cup\Cs{j,\st}\!\setminus\bigl(K_{\eZ}(j)\cup 
M_{\eZ}(j)\bigr)\Bigr)\\
&=\bigl((\Cs{j,\st}\!\setminus K_{\eZ}(j))\cup (\Cs{j,\st}\!\setminus 
M_{\eZ}(j))\bigr)\cap\bigl(K_{\eZ}(j)\cup M_{\eZ}(j)\bigr)\\
&=\bigl((\Cs{j,\st}\!\setminus K_{\eZ}(j))\cap 
M_{\eZ}(j)\bigr)\cup\bigl((\Cs{j,\st}\!\setminus M_{\eZ}(j))\cap 
K_{\eZ}(j)\bigr)\,.
\end{align*}
Therefore
\begin{eqnarray}\label{eq:Csjst_K}
(\Cs{j,\st}\!\setminus K_{\eZ}(j))\cap 
M_{\eZ}(j)=\varnothing&\text{and}&(\Cs{j,\st}\!\setminus M_{\eZ}(j))\cap 
K_{\eZ}(j)=\varnothing\,.
\end{eqnarray}
In order to seek a contradiction we assume that $\Cs{j,\st}\!\setminus 
K_{\eZ}(j)\ne\varnothing$. Then we find 
$x\in\Ired{j,\st}\setminus\cycnext{\eZ}(j)\act\,\Ired{\psi_{\eZ}(j),\st}$ and $y\in 
\cycnext{\eZ}({\psi_{\eZ}^{-1}(j)})^{-1}\act\Jset{\psi_{\eZ}^{-1}(j)}$. 
By~\eqref{BP:allvectorsredI} there exists~$\nu\in\Cs{j,\st}$ with 
$(\gamma_\nu(+\infty),\gamma_\nu(-\infty))=(x,y)$. Thus, 
$\nu\in  M_{\eZ}(j) \cap (\Cs{j,\st}\!\setminus K_{\eZ}(j))$, which yields a 
contradiction to~\eqref{eq:Csjst_K}. In turn, $K_{\eZ}(j)=\Cs{j,\st}$ and, by an 
analogous argument, also $M_{\eZ}(j)=\Cs{j,\st}$.

For the proof of~\eqref{accsystem:nonempty} we suppose that $j\in A_{\eX}\cap 
A_{\eY}$. Then we find $\nu\in K_{\eX}(j)\cap M_{\eY}(j)$, due to~\eqref{BP:allvectorsredI} and the nonemptiness of~$K_{\eX}(j)$ and 
$M_{\eY}(j)$. Using that $K_{\eX}(j)\subseteq\Cs{j,\st}\!\setminus K_{\eY}(j)$ 
and $M_{\eY}(j)\subseteq\Cs{j,\st}\!\setminus M_{\eX}(j)$ 
by~\eqref{accsystem:disjointremoves}, we obtain 
\begin{align*}
\varnothing&\ne K_{\eX}(j)\cap M_{\eY}(j)\\
&\subseteq\Cs{j,\st}\!\setminus\bigl(K_{\eY}(j)\cup M_{\eX}(j)\bigr)\\
&\subseteq\Cs{j,\st}\!\setminus\bigl(K_{\eY}(j)\cup 
M_{\eX}(j)\bigr)\cup\Cs{j,\st}\!\setminus\bigl(K_{\eY}(j)\cup K_{\eX}(j)\bigr)\\
&\phantom{\subseteq}\,\cup\Cs{j,\st}\!\setminus\bigl(K_{\eX}(j)\cup 
M_{\eY}(j)\bigr)\cup\Cs{j,\st}\!\setminus\bigl(M_{\eX}(j)\cup M_{\eY}(j)\bigr)\\
&=\Cs{j,\st}\!\setminus\left(\bigl(K_{\eY}(j)\cup 
M_{\eX}(j)\bigr)\cap\bigl(K_{\eY}(j)\cup K_{\eX}(j)\bigr)\right.\\
&\phantom{\Cs{j,\st}\!\setminus(((}\left.\cap\,\bigl(K_{\eX}(j)\cup 
M_{\eY}(j)\bigr)\cap\bigl(M_{\eX}(j)\cup M_{\eY}(j)\bigr)\right)\\
&=\Cs{j,\st}\!\setminus\bigcup_{\eZ\in\{\eX,\eY\}}\bigl(K_{\eZ}(j)\cap 
M_{\eZ}(j)\bigr)\\
&=\Cs{j,\acc}.
\end{align*}
This completes the proof.
\end{proof}

\begin{remark}\label{REM:Itoldyoutoreduce}
Let $\eZ,\eZ'\in\{\eX,\eY\}$, $\eZ\ne\eZ'$, and $j\in A_{\eZ}$. The conditions
\[
\Ired{j,\st}=\cycnext{\eZ}(j)\act\,\Ired{\psi_{\eZ}(j),\st}\quad\text{and}\quad\Jset{j}=
\cycnext{\eZ}(\psi_{\eZ}^{-1}(j))^{-1}\act\Jset{\psi_{\eZ}^{-1}(j)}
\]
in Proposition~\ref{PROP:accsystem}\eqref{accsystem:empty} imply that $\eZ'_j$ 
is an inner point of a representative interval of 
some funnel of~$\Orbi$. In particular, because 
of~(\ref{BP:intervaldecompred}\ref{BP:intervaldecompGdecompred}), the structure 
of~$\Ired{j,\st}$ implies
\[
\Heir{}{j}=\{\psi_{\eZ}(j)\}\quad\text{and}\quad\Trans{}{j}{\psi_{\eZ}(j)}=\{
\cycnext{\eZ}(j)\}\,.
\]
In this case $j\ne\psi_{\eZ}(j)$, because otherwise $\Ired{j,\st}$ would be empty. 
Algorithm~\ref{nodereductionI} removes all branches of that type from the set 
of branches. Hence, if the level of reduction before applying the cuspidal 
acceleration is at least $\kappa_1$, with $\kappa_1$ as in 
Section~\ref{SUBSEC:branchred}, then we have $A^*=A$.
\end{remark}

\begin{prop}\label{PROP:accsystem:geodintersect}
Let $\gamma$ be a geodesic on~$\H$ that intersects~$\Gamma\act\BrU_{\st}$. Then 
$\gamma$ intersects~$\Gamma\act\BrU_{\acc}$. More precisely, if $\gamma$ 
intersects~$\Gamma\act\BrU_{\st}$ at time~$t^*$, then there exist $t^*_+, 
t^*_-\in\R$ with $t^*_+\geq t^* \geq t^*_-$ such that $\gamma$ 
intersects~$\Gamma\act\BrU_{\acc}$ at time~$t^*_+$ and at time~$t^*_-$.
\end{prop}

\begin{proof}
Without loss of generality we may suppose that the intersection between the 
geodesic~$\gamma$ and $\Gamma\act\BrU_{\st}$ at time~$t^*$ is on~$\BrU_{\st}$, 
say at $\nu =\gamma^{\prime}(t^*) \in \Cs{j,\st}$ with $j\in A$. We may further 
suppose that 
$\nu$ is an element of
\begin{align*}
\Cs{j,\st}\!\setminus\Cs{j,\acc}=\bigcup_{\eZ\in\{\eX,\eY\}}\bigl(K_{\eZ}
(j)\cap 
M_{\eZ}(j)\bigr)\,,
\end{align*}
as otherwise there is nothing to prove. Thus, there is $\eZ\in\{\eX,\eY\}$ such 
that $\nu\in K_{\eZ}(j) \cap M_\eZ(j)$ and $j\in A_\eZ$, and $\eZ$ is unique by 
Proposition~\ref{PROP:accsystem}\eqref{accsystem:disjointremoves}. By the 
discussion in Remark~\ref{REM:howtoacc}, we find $n\in\N_0$ and 
$k\in\cycset{j}{\eZ}$ such that $\gamma(+\infty)\in D_{n,\eZ}^+(j,k)_{\st}$, 
with the set~$D_{n,\eZ}^+(j,k)$ as defined in Remark~\ref{REM:howtoacc}. Since
\begin{align}\label{eq:shifted_Dnjk}
\cycnext{\eZ}(j,k)^{-1}\cyctrans{j}{\eZ}^{-n}\act 
D_{n,\eZ}^+(j,k)_{\st}&=\Iset{k}\setminus\cycnext{\eZ}({k})\act\Iset{\psi_{\eZ}
(k)}\,,
\end{align}
Lemma~\ref{LEM:halfspaceFuture} shows that $\gamma$ intersects 
$\cyctrans{j}{\eZ}^n\cycnext{\eZ}(j,k)\act\Cs{k}$ at some time~$t^*_+>t^*$. More 
precisely, using Remark~\ref{REM:inst} and the full extent of the 
equality~\eqref{eq:shifted_Dnjk}, we obtain that
\[
\cycnext{\eZ}(j,k)^{-1}\cyctrans{j}{\eZ}^{-n}\act\gamma^{\prime}(t^*_+)\in\Cs{k,
\st}\setminus K_{\eZ}(k)\,.
\]
By construction, 
$\cycnext{\eZ}(j,k)^{-1}\cyctrans{j}{\eZ}^{-n}\act\gamma^{\prime}(t^*_+)\in 
M_\eZ(k)$. Since $M_\eX(k)\cap M_\eY(k) = \varnothing$ by 
Proposition~\ref{PROP:accsystem}\eqref{accsystem:disjointremoves}, it follows 
that 
\[
\cycnext{\eZ}(j,k)^{-1}\cyctrans{j}{\eZ}^{-n}\act\gamma^{\prime}(t^*_+)\in 
\Cs{k,\st}\setminus \bigcup_{\eW\in\{\eX,\eY\}} \left( K_\eW(k) \cap 
M_\eW(k)\right) = \Cs{k,\acc}\,.
\]
Thus, $\gamma$ intersects~$\Gamma\act\BrU_{\acc}$ at time~$t^*_+$. The proof of 
the existence of an intersection time $t^*_- \leq  t^*$ is analogous, using the 
set~$D_{n,\eZ}^-(j,k)$ (from Remark~\ref{REM:howtoacc}) for suitable $n\in\N_0$, 
$k\in\cycset{j}{\eZ}$ instead of~$D_{n,\eZ}^-(j,k)$ and 
Lemma~\ref{LEM:halfspacePast} instead of Lemma~\ref{LEM:halfspaceFuture}.
\end{proof}

\begin{prop}\label{PROP:accsystemcrosssection}
For each $\mu\in\mathcal{M}_{\Vanish(\Orbi)}$ the set~$\CrSc_{\acc}$ is a 
strong cross section for the geodesic flow on $\Orbi$ with respect to~$\mu$. 
Each geodesic in~$\Geo(\Orbi)\setminus\Vanish(\Orbi)$ intersects~$\CrSc_{\acc}$ 
infinitely often in past and future.
\end{prop}

\begin{proof}
We recall from Corollary~\ref{CofSoB:strongcrosssection} that $\CrSc_{\st}$ is a 
strong cross section with respect to~$\mu$ and that each geodesic 
in~$\Geo(\Orbi)\setminus\Vanish(\Orbi)$ intersects~$\CrSc_{\st}$ infinitely 
often in past and future. Therefore, as $\CrSc_{\acc} \subseteq \CrSc_{\st}$, 
the validity of~\eqref{CS:discreteintime} for~$\CrSc_{\acc}$ is immediate from 
its validity for~$\CrSc_{\st}$. In order to establish~\eqref{CS:infinitelyoften} 
and~\eqref{CS:strong} for~$\CrSc_{\acc}$ we let $\wh\gamma$ be a geodesic 
on~$\Orbi$ that intersects~$\CrSc_{\st}$ (which is true for all geodesics 
in~$\Geo(\Orbi)\setminus\Vanish(\Orbi)$ and hence for $\mu$-almost geodesics 
on~$\Orbi$) and let~$(t_n)_{n\in\Z}$ be the bi-infinite sequence of intersection 
times. From Proposition~\ref{PROP:accsystem:geodintersect} it follows that the 
sequence of intersection times of~$\wh\gamma$ with~$\CrSc_{\acc}$ is a 
\emph{bi-infinite} subsequence of~$(t_n)_n$, showing that $\wh\gamma$ 
intersects~$\CrSc_{\acc}$ infinitely often in past and future. This completes 
the proof. 
\end{proof}

\section{Structure of accelerated system}\label{SEC:structacc}

In this section we discuss the structure of the acceleration of a reduced set 
of branches. In particular, we provide a partition of the set of representatives that is better suited for a coding of geodesics (or, equivalently, the passage to a discrete dynamical system) than the family immediate from Definition~\ref{DEF:cuspacc}. These results will be crucial for the discussion in the next section, where we establish the strict transfer operator approach. We resume the notation from Section~\ref{SEC:cuspacc}. Thus, $\BrS$ denotes a \emph{reduced} set of branches for the geodesic flow on~$\Orbi$. 

The cuspidal acceleration procedure effectively dissects each branch into up to three 
mutually disjoint pieces, as the following lemma shows.

\begin{lemma}\label{LEM:accbranches}
For every~$j\in A$ we have
\[
\Cs{j,\acc}=\bigl(\Cs{j,\st}\setminus\left(K_\eX(j)\cup 
K_\eY(j)\right)\bigr)\cup\bigl(K_\eX(j)\setminus 
M_\eX(j)\bigr)\cup\bigl(K_\eY(j)\setminus M_\eY(j)\bigr)
\]
and the union on the right hand side is disjoint.
\end{lemma}

\begin{proof}
Let~$j\in A$. For the sake of improved readability we use the abbreviations
\[
\Cs{}\coloneqq\Cs{j,\st}\,,\qquad K_\eZ\coloneqq K_\eZ(j)\,,\qquad 
M_\eZ\coloneqq M_\eZ(j)
\]
for~$\eZ\in\{\eX,\eY\}$. Since $K_\eX\cap K_\eY = \varnothing$ by 
Proposition~\ref{PROP:accsystem}\eqref{accsystem:disjointremoves}, the claimed 
union is disjoint. In order to show the claimed equality, we recall from 
Definition~\ref{DEF:cuspacc} that 
\[
 \Cs{j,\acc} = \Cs{}\setminus \bigl( (K_\eX\cap M_\eX) \cup (K_\eY\cap M_\eY) 
\bigr)\,, 
\]
which clearly contains the set~$\Cs{}\setminus (K_\eX\cup K_\eY)$. Using again 
that $K_\eX\cap K_\eY = \varnothing$, we see that also $K_\eX\setminus M_\eX$ 
and $K_\eY\setminus M_\eY$ are subsets of~$\Cs{j,\acc}$. Thus, 
\begin{equation}\label{eq:Csaccunion}
 \bigl(\Cs{}\setminus ( K_\eX\cup K_\eY)\bigr) \cup \bigl(K_\eX\setminus 
M_\eX\bigr) \cup  \bigl(K_\eY\setminus M_\eY\bigr) \subseteq \Cs{j,\acc}\,.
\end{equation}
It remains to establish the converse inclusion relation. To that end we consider 
any element~$v\in\Cs{j,\acc}$. If $\nu\notin K_\eX\cup K_\eY$, then $\nu$ is 
obviously contained in the union on the left hand side of~\eqref{eq:Csaccunion}. 
If $\nu\in K_\eX$, then $\nu\notin M_\eX$ since otherwise we would have $\nu\in 
K_\eX\cap M_\eX$, which contradicts to $\nu\in\Cs{j,\acc}$. Analogously, for 
$\nu\in K_\eY$. In turn, the inclusion relation in~\eqref{eq:Csaccunion} is 
indeed an equality. 
\end{proof}

For~$\eZ\in\{\eX,\eY\}$ we set
\begin{equation}\label{eq:def_AZ}
A_\eZ^*\coloneqq A^*\cap A_\eZ\,.
\end{equation}
\index[symbols]{Azs@$A_\eZ^*$}%
The following definition is motivated by Lemma~\ref{LEM:accbranches}.
We recall from Definition~\ref{DEF:cuspacc} that~$K_\eZ(j)=M_\eZ(j)=\varnothing$ 
whenever~$j\in A^*\setminus A^*_\eZ$.

\begin{defi}\label{DEF:accbranches}
For~$j\in A^*$ we set
\[
\Csacc{(j,\eR)}\coloneqq\Cs{j,\st}\setminus\bigl(K_\eX(j)\cup K_\eY(j)\bigr)\,,
\]
and, for~$\eZ\in\{\eX,\eY\}$, 
\[
\Csacc{(j,\eZ)}\coloneqq K_\eZ(j)\setminus M_\eZ(j)\,.
\]
\index[symbols]{Cacca@$\Csacc{(j,\eR)}$}%
We further set 
\[
 \Index \coloneqq \defset{ (j,V)\in A^*\times\{\eX,\eR,\eY\} }{ 
\Csacc{(j,V)}\not=\varnothing }\,. 
\]
\index[symbols]{Ab@$\Index$}%
For~$a\in\Index$ we call~$\Csacc{a}$ an \emph{accelerated} (or \emph{induced}) 
\emph{branch}
\index[defs]{induced branch}%
\index[defs]{accelerated branch}%
\index[defs]{branch!induced}%
\index[defs]{branch!accelerated}%
and denote by
\[
\BrS_{\acc}\coloneqq\defset{\Csacc{a}}{a\in\Index}
\]
the \emph{set of all accelerated branches} or \emph{accelerated system}. 
\index[defs]{set of accelerated branches}%
\index[defs]{accelerated system}%
\index[symbols]{Cacc@$\BrS_{\acc}$}%
For any~$j\in A^*$ we define 
\[
I_{(j,V)}\coloneqq\left\{\!
\begin{array}{cl}
\cycnext{V}(j)\act\,\Ired{\psi_V(j)} &\text{if $V\in\{\eX,\eY\}$}\,,\\
\Ired{j}\setminus\left(\cycnext{\eX}(j)\act\,\Ired{\psi_\eX(j)}\cup\cycnext{\eY}(j)\act 
\,\Ired{\psi_\eY(j)}\right) &\text{if $V=\eR$}\,,
\end{array}
\right.
\]
\index[symbols]{Ij@$I_{(j,V)}$}%
as well as
\[
J_{(j,V)}\coloneqq\left\{\!
\begin{array}{cl}
J_{j}\setminus\cycnext{V}(\psi_V^{-1}(j))^{-1}\act 
J_{\psi_V^{-1}(j)}\,,&\text{if $V\in\{\eX,\eY\}$}\,,
\\
J_{j}\,,&\text{if $V=\eR$}\,,
\end{array}
\right.
\]
\index[symbols]{Jj@$J_{(j,V)}$}%
where, for~$j\in A^*\setminus A^*_\eZ$, we set 
~$I_{\psi_\eZ(j)}\coloneqq\varnothing$ 
and~$J_{\psi_\eZ^{-1}(j)}\coloneqq\varnothing$ for $\eZ\in\{\eX,\eY\}$.
\end{defi}

We note that for each~$j\in A^*$, the set~$\Cs{j,\acc}$ decomposes into the 
disjoint union
\[
 \Cs{j,\acc} = \Csacc{(j,\eR)} \cup \Csacc{(j,\eX)} \cup \Csacc{(j,\eY)}
\]
by Lemma~\ref{LEM:accbranches}, an observation that is needed for the proof of 
the following lemma.

\begin{lemma}\label{LEM:accprops}
The accelerated system~$\BrS_{\acc}$ satisfies the following properties.
\begin{enumerate}[label=$\mathrm{(\roman*)}$, ref=$\mathrm{\roman*}$]
\item\label{LEM:accprops:union}
We have
\[
\bigcup_{a\in\Index}\Csacc{a}=\bigcup_{j\in A^*}\Cs{j,\acc}=\BrU_{\acc}\,.
\]

\item\label{LEM:accprops:Isetunion}
For all~$j\in A^*$ we have
$\Ired{j}=I_{(j,\eX)}\cup I_{(j,\eR)}\cup I_{(j,\eY)}$. This union is disjoint.

\item\label{LEM:accprops:Iset}
For all~$a\in\Index$ we have
$\Iset{a}=I(\Csacc{a})$.

\item\label{LEM:accprops:endpoints}
For all~$a\in\Index$, none of the boundary points of~$I_a$ and~$J_a$ is a hyperbolic fixed 
point.
\end{enumerate}
\end{lemma}

\begin{proof}
The statement of~\eqref{LEM:accprops:union} follows directly from 
Lemma~\ref{LEM:accbranches}. Further, \eqref{LEM:accprops:Isetunion} is 
obviously true because, for~$V\in\{\eX,\eY\}$, we 
have~$I_{(j,V)}=\cycnext{V}(j)\act\,\Ired{\psi_V(j)}$. In order to 
establish~\eqref{LEM:accprops:Iset} we let $a\coloneqq(j,V)\in\Index$ and 
suppose first that $V\in\{\eX,\eY\}$. The definition of~$\Csacc{a}$ immediately 
shows that
\[
 I(\Csacc{a}) \subseteq I(K_V(j)) = \Iset{(j,V)} = \Iset{a}\,.
\]
For the converse inclusion relation we pick any point
\[
y\in \Jset{j}\setminus\cycnext{V}(\psi_V^{-1}(j))^{-1}\act J_{\psi_V^{-1}(j)}\,.
\]
Its existence follows from $K_V(j)\setminus M_V(j) =\Csacc{a}\ne\varnothing$.
By~\eqref{BP:allvectorsredI}, we find for each $x\in I(K_V(j))$ an 
element~$\nu\in\Cs{j}$ such that
\[
(\gamma_\nu(+\infty),\gamma_\nu(-\infty))=(x,y)\,.
\]
Each such element~$\nu$ is in~$K_V(j)\setminus M_V(j)=\Csacc{a}$. Thus, 
$I(K_V(j))\subseteq I(\Csacc{a})$. We suppose now that~$V=\eR$. Then
\begin{align*}
I(\Csacc{a})&=\Ired{j,\st}\setminus\bigl(I(K_{\eX}(j))\cup I(K_{\eY}(j))\bigr)\\
&=\Ired{j,\st}\setminus\bigl(\cycnext{\eX}(j)\act\,\Ired{\psi_{\eX}(j),\st}\cup\cycnext{
\eY}(j)\act\,\Ired{\psi_{\eY}(j),\st}\bigr)\\
&=\Iset{a}\,,
\end{align*}
where we set~$\Ired{\psi_{\eZ}(j)}\coloneqq\varnothing$ 
whenever~$(j,\eZ)\notin\Index$, for~$\eZ\in\{\eX,\eY\}$.
This shows~\eqref{LEM:accprops:Iset}.

Finally, we observe that, for every~$a\in\Index$, the boundary points of~$I_a$ 
and~$J_a$ in the~$\wh\R$-topology emerge as~$\Gamma$-translates of the endpoints of the geodesic 
segments~$\overline{\base{\Cs{j}}}$ for~$j\in A$.
In other words, they are 
contained in the set
\[
\Gamma\act\defset{\eX_j,\eY_j}{j\in A}\,.
\]
Hence,~\eqref{LEM:accprops:endpoints} is a consequence 
of~\eqref{BP:completegeodesics}.
\end{proof}

Our next goal is to ``update'' the transition sets according to the accelerated branches.
In other words, for~$a,b\in\Index$ we search for a characterization of the set
\begin{equation}\label{EQN:acctransset}
\wh{\mathcal{G}}(a,b)\coloneqq\defset{g\in\Gamma}{\exists\, 
\nu\in\Csacc{a}\colon\gamma'(\ittime{\acc,1}(\nu))\in g\act\Csacc{b}}
\end{equation}
\index[symbols]{Gbb@$\wh{\mathcal{G}}(a,b)$}%
in terms of the transition sets for~$\BrS$, where~$\ittime{\acc,1}(\nu)$
\index[symbols]{tacc@$\ittime{\acc,1}(\nu)$}%
denotes 
the first return time w.r.t.~$\BrU_{\acc}$.
This will enable us to prove that the first return 
time~$\ittime{\acc,1}(\nu)$ does indeed exist for all~$\nu\in\BrU_{\acc}$, and 
thus the sets~$\Trans{\acc}{a}{b}$ belonging to the accelerated system are well-defined (see~\eqref{EQDEF:Transacc1} below).
Further, it is a crucial step in the determination of the transformations that 
constitute the structure tuple for the strict transfer operator approach.

We start with updating the cycles and cycle transformations and sets from 
Section~\ref{SEC:cuspacc} to the situation of the accelerated branches.

\begin{defi}\label{DEF:starcycs}
Let~$\eZ\in\{\eX,\eY\}$ and let~$j\in A$ be such that~$(j,\eZ)\in\Index$.
We call the subsequence~$\cycstar{\eZ}(j)$ of~$\cyc{\eZ}{j}$ of all~$k$ that are 
elements of~$A^*$ the \emph{induced}~$\eZ$-\emph{cycle} of~$j$.
We further define the map~$\psi^*_{\eZ}\colon A^*_{\eZ}\to A^*_{\eZ}$ by
\[
\psi^*_{\eZ}(\cycstar{\eZ}(j)_n)\coloneqq\cycstar{\eZ}(j)_{n+1}
\]
\index[symbols]{pZ@$\psi^*_{\eZ}$}%
\index[defs]{induced Z-cycle}%
\index[defs]{Z-cycle!induced}%
for all $n\in\N$.
We also define the transformations
\[
\cycstarnext{(j,\eZ)}\coloneqq\cycnext{\eZ}(j,\psi^*_{\eZ}(j))
\]
\index[symbols]{gs@$\cycstarnext{(j,\eZ)}$}%
and set as before
\[
\cycstarnext{(j,\eZ),(k,\eZ)}\coloneqq\cycstarnext{(j,\eZ)}\cdot\cycstarnext{
(\psi
^*_{\eZ}(j),\eZ)}\cdots\cycstarnext{((\psi^*_\eZ)^{r_{j,k}-1}(j),\eZ)}\,,
\]
\index[symbols]{gsa@$\cycstarnext{(j,\eZ),(k,\eZ)}$}%
where $r_{j,k}\coloneqq\min\defset{\ell\in\N}{(\psi^*_\eZ)^\ell(j)=k}$,
\index[symbols]{r@$r_{j,k}$}%
as well 
as
\[
\cycstartrans{(j,\eZ)}\coloneqq\cycstarnext{(j,\eZ),(j,\eZ)}\,.
\]
\index[symbols]{u@$\cycstartrans{(j,\eZ)}$}%
Finally, we define the \emph{induced cycle set} of~$(j,\eZ)$ to be
\[
\cycstarset{(j,\eZ)}\coloneqq\{(j,\eZ),(\psi^*_\eZ(j),\eZ),\dots,((\psi^*_\eZ)^{
r_{j,j}-1}(j),\eZ)\}\,.
\]
\index[symbols]{Cycs@$\cycstarset{(j,\eZ)}$}%
\index[defs]{induced cycle set}%
\index[defs]{cycle set!induced}%
\end{defi}

\begin{remark}\label{REM:subtleties}
Note that the updated  transformations~$\cycstarnext{(j,V)}$ already make up for 
the loss of branches during the acceleration procedure.
Thus, the cycle transformation remains unaltered, meaning that
\[
\cycstartrans{(j,V)}=\cyctrans{j}{V}\,,
\]
for every~$(j,V)\in A^*\times\{\eX,\eY\}$.
Note further that the induced $\eZ$-cycles are allowed to contain members~$k$ 
for which~$(k,\eZ)$ is not an element of~$\Index$.
This is due to the following eventuality: Let~$j\in A$ be such that~$\Cs{j,\st}= 
M_\eZ(j)$ for some~$\eZ\in\{\eX,\eY\}$, but~$\Cs{j,\st}\ne K_\eZ(j)$.
Then~$j\in A^*_\eZ$, but~$\Csacc{(j,\eZ)}=\varnothing$, 
thus~$(j,\eZ)\notin\Index$.
But~$\Index$ will include the index~$(j,\eR)$.
This distinction proves necessary in the following construction of induced 
transition sets:~$j$ must be included in the induced cycle in order to allow the 
other induced branches of that cycle to ``see''~$(j,\eR)$. But~$(j,\eZ)$ is 
excluded from~$\Index$ and hence we do not construct transition sets for it.
If the level of reduction is at least~$\kappa_1$, with~$\kappa_1$ as in 
Section~\ref{SUBSEC:branchred}, then this eventuality does not occur and, 
consequentially,~$(k,\eZ)\in\Index$ for every member~$k$ of any 
induced~$\eZ$-cycle.
\end{remark}

With that we are now prepared to determine the \emph{induced transition 
sets}~$\Trans{\acc}{a}{b}$ for~$a,b\in\Index$.
\index[symbols]{Gacc@$\Trans{\acc}{a}{b}$}%
\index[defs]{induced transformation set}%
\index[defs]{transformation set!induced}%
Let~$a=(j,V)\in\Index$ and suppose first that~$V\in\{\eX,\eY\}$.
In Remark~\ref{REM:howtoacc} we have derived the decomposition
\[
I_{(j,V)}=\cycnext{V}(j)\act\,
\Ired{\psi_V(j)}=\bigcup_{n\in\N_0}\bigcup_{k\in\cycset{j}{V}}\mc D_{n,V}^+(j,k)\,.
\]
We can now rewrite
\begin{align*}
D_{n,V}^+(j,k)&=\cyctrans{j}{V}^n\cycnext{V}(j,k)\act 
I_k\setminus\cyctrans{j}{V}^n\cycnext{V}(j,k)\cycnext{V}(k)\act I_{\psi_V(k)}\\
&=\cyctrans{j}{V}^n\cycnext{V}(j,k)\act\bigl(I_k\setminus I_{(k,V)}\bigr)\,.
\end{align*}
Hence,
\[
\mc D_{n,V}^+(j,k)=D_{n,V}^+(j,k)\cap \cyctrans{j}{V}^n\cycnext{V}(j,k)\act\,\Ired{k}=\cyctrans{j}{V}^n\cycnext{V}(j,k)\act\bigl(\,\Ired{k}\setminus I_{(k,V)}\bigr)\,.
\]
Therefore, by passing to~``$\st$"-sets we obtain
\begin{equation}\label{EQ:Iadecomp}
\begin{aligned}
\Iset{(j,V)}&=\bigcup_{n\in\N_0}\bigcup_{k\in\cycset{j}{V}}\cyctrans{j}{V}
^n\cycnext{V}(j,k)\act\bigl(\,\Ired{k,\st}\setminus I(K_V(k))\bigr)\\
&=\bigcup_{n\in\N_0}\bigcup_{k\in\cycset{j}{V}}\cyctrans{j}{V}^n\cycnext{V}(j,
k)\act\bigl(\Iset{(k,\eR)}\cup\Iset{(k,V')}\bigr)\,,
\end{aligned}
\end{equation}
with~$V'$ such that~$\{V,V'\}=\{\eX,\eY\}$ and all unions being disjoint. 
By taking Remark~\ref{REM:subtleties} into consideration, 
Proposition~\ref{PROP:accsystem}\eqref{accsystem:empty} allows us to pass 
to~$(k,V)\in\cycstarset{(j,V)}$ in the second union. The 
transformations~$\cycnext{V}(j,k)$ then need to be substituted 
by~$\cycstarnext{(j,V),(k,V)}$. Hence we obtain the disjoint union
\begin{equation}\label{EQ:Iadec_star}
 \Iset{(j,V)} = 
\bigcup_{n\in\N_0}\bigcup_{(k,V)\in\cycstarset{(j,V)}}\cycstartrans{(j,V)}
^n\cycstarnext{(j,V),(k,V)}\act\bigl(\Iset{(k,\eR)}\cup\Iset{(k,V')}\bigr)\,.
\end{equation}
For any~$b\in\Index$,~$b=(k,W)$ (and $a=(j,V)$, $V\in\{\eX,\eY\}$) we therefore 
define 
\begin{equation}\label{EQDEF:Transacc1}
\Trans{\acc}{a}{b}\coloneqq\left\{\!
\begin{array}{cl}
\bigcup_{n\in\N_0}\{\cycstartrans{a}^n\cycstarnext{a,\wt b}\} &\text{if 
$W\ne V$  and $\wt b\coloneqq(k,V)\in\cycstarset{a}$}\,,\\
\varnothing &\text{otherwise}.
\end{array}
\right.
\end{equation}

We suppose now that~$V=\eR$, thus $a=(j,\eR)$. We consider $(k,g)\in 
A\times\Gamma$ such that $g\in\Trans{}{j}{k}$, pick $\nu\in\Csacc{a}$ such that 
\[
 \gamma_\nu(+\infty) \notin 
\cycnext{\eX}(j)\act\,\Ired{\psi_{\eX}(j),\st}\cup\cycnext{\eY}(j)\act\,\Ired{\psi_{\eY}
(j),\st}\,,
\]
and let $\eta$ denote the intersection vector of~$\gamma_\nu$ with 
$g\act\Cs{k}$. In what follows we argue that $\eta$ is contained in 
$g\act\Csacc{b}$ for some~$b\in\Index$ of the form~$b=(k,W)$. To that end, we 
first note that 
\begin{align*}
 \Iset{(j,\eR)} & = \bigcup_{\substack{ k\in A \\ k\notin\{\psi_{\eX}(j), 
\psi_{\eY}(j)  \} }}\bigcup_{g\in\Trans{}{j}{k}}g\act\,\Ired{k,\st} 
  \quad \cup \bigcup_{\substack{ g\in\Trans{}{j}{\psi_\eX(j)} \\ g\not= 
\cycnext{\eX}(j) }}g\act\,\Ired{\psi_\eX(j),\st} 
  \\
  & 
  \qquad
  \cup \bigcup_{\substack{ g\in\Trans{}{j}{\psi_\eY(j)} \\ g\not= 
\cycnext{\eY}(j) }}g\act\,\Ired{\psi_\eY(j),\st}
\end{align*}
by~(\ref{BP:intervaldecompred}\ref{BP:intervaldecompGdecompred}), 
Definition~\ref{DEF:accbranches} and 
Lemma~\ref{LEM:accprops}\eqref{LEM:accprops:Isetunion}. Therefore the hypotheses 
on~$\nu$ imply that $g\act\,\Ired{k} \subseteq I_{(j,\eR)}$. Then, for 
any~$\eZ\in\{\eX,\eY\}$, we have $J_j\ne g\act J(M_{\eZ}(k))$, as follows 
immediately from the definition of $M_\eZ(k)$. Thus, 
\[
J_j\cap g\act J(M_{\eZ}(k))=\varnothing
\]
by Proposition~\ref{PROP:BPintervaldecompV}\eqref{BP:intervaldecompVdecomp}. It 
follows that $\eta\in \Cs{k,\acc}$ and, by Lemma~\ref{LEM:accbranches}, $\eta\in 
\Csacc{b}$ for some~$b=(k,W)\in\Index$. Therefore, for any~$b=(k,W)\in\Index$ 
(and $a=(j,\eR)$) we define
\begin{equation}\label{EQDEF:Transacc2}
\Trans{\acc}{a}{b}\coloneqq\left\{\!
\begin{array}{cl}
\Trans{}{j}{k}\setminus\{\cycnext{\eX}(j)\} &\text{if $k=\psi_\eX(j)$}\,,
\\
\Trans{}{j}{k}\setminus\{\cycnext{\eY}(j)\} &\text{if $k=\psi_\eY(j)$}\,,
\\
\Trans{}{j}{k} &\text{otherwise}.
\end{array}
\right.
\end{equation}

In the following proposition we collect the adapted ``set of branches''-style 
properties satisfied by~$\BrS_{\acc}$.
Each of the statements follows, in a straightforward way, from its respective 
counterpart in Definition~\ref{DEF:redsetofbranches} and the constructions 
above, for which reason we omit a detailed proof.

\begin{prop}\label{PROP:accsystemnewstuff}
The accelerated system~$\BrS_{\acc}$ satisfies the following properties:
\begin{enumerate}[label=$\mathrm{(B\arabic*_{\acc})}$, 
ref=$\mathrm{B\arabic*_{\acc}}$, itemsep=2ex]
\item\label{BPacc:closedgeodesicsHtoX}
\index[defs]{0BPacc01@(B1$_{\acc}$)}%
\index[symbols]{BPacc01@(B1$_{\acc}$)}%
For each~$a\in\Index$ there exists~$\nu\in\Csacc{a}$ such 
that~$\pi(\gamma_\nu)\in\Geo_{\Per}(\Orbi)$.
\item\label{BPacc:completegeodesics}
\index[defs]{0BPacc02@(B2$_{\acc}$)}%
\index[symbols]{BPacc02@(B2$_{\acc}$)}%
For each~$a=(j,V)\in\Index$ the set~$\base{\Csacc{a}}$ is contained in a 
complete geodesic segment~$\mathrm{b}_a$ in~$\H$ with endpoints 
in~$\wh\R\setminus\wh\R_{\st}$.
This segment is given by
\[
\mathrm{b}_a=\overline{\base{\Cs{j}}}\,.
\]
\item\label{BPacc:pointintohalfspaces}
\index[defs]{0BPacc03@(B3$_{\acc}$)}%
\index[symbols]{BPacc03@(B3$_{\acc}$)}%
For each~$a\in\Index$ all elements of~$\Csacc{a}$ point into the same open 
half-space relative to~$\mathrm{b}_a$.
\item\label{BPacc:coverlimitset}
\index[defs]{0BPacc04@(B4$_{\acc}$)}%
\index[symbols]{BPacc04@(B4$_{\acc}$)}%
The~$\Gamma$-translates of~$\defset{I_a}{a\in\Index}$ cover~$\wh\R_{\st}$.
\item\label{BPacc:allvectors}
\index[defs]{0BPacc05@(B5$_{\acc}$)}%
\index[symbols]{BPacc05@(B5$_{\acc}$)}%
For each~$a\in\Index$ and each pair~$(x,y)\in\Iset{a}\times\Jset{a}$ there 
exists a unique element~$\nu\in\Csacc{a}$ such that
\[
(x,y)=\left(\gamma_\nu(+\infty),\gamma_\nu(-\infty)\right)\,.
\]
\item\label{BPacc:disjointunion}
\index[defs]{0BPacc06@(B6$_{\acc}$)}%
\index[symbols]{BPacc06@(B6$_{\acc}$)}%
If~$\mathrm{b}_a\cap g\act\mathrm{b}_b\ne\varnothing$ for 
some~$a=(j,V),b=(k,W)\in\Index$ and some~$g\in\Gamma$, then either~$j=k$ 
and~$g=\id$, or~$\mathrm{H}_{\pm}(j)=g\act\mathrm{H}_{\mp}(k)$.
\item\label{BPacc:intervaldecomp}
\index[defs]{0BPacc07@(B7$_{\acc}$)}%
\index[symbols]{BPacc07@(B7$_{\acc}$)}%
Let~$\Trans{\acc}{a}{b}$ be defined as in~\eqref{EQDEF:Transacc1} 
and~\eqref{EQDEF:Transacc2} and let~$\wh{\mathcal{G}}(a,b)$ be defined as in 
\eqref{EQN:acctransset} for~$a,b\in\Index$.
\begin{enumerate}[label=$\mathrm{(\roman*)}$, ref=$\mathrm{\roman*}$]
\item\label{BPacc:decomp}
For every~$a\in\Index$ we have
\[
\Iset{a}=\bigcup_{b\in\Index}\bigcup_{g\in\Trans{\acc}{a}{b}}g\act\Iset{b}\,.
\]
This union is disjoint.
\item\label{BPacc:frt}
For every~$a\in\Index$ and every~$\nu\in\Csacc{a}$ there 
exists~$t\in(0,+\infty)$ such that~$\gamma_\nu(t)\in\Gamma\act\BrU_{\acc}$.
\item\label{BPacc:transsetsequal}
For every pair~$(a,b)\in\Index\times\Index$ we have
\[
\Trans{\acc}{a}{b}=\wh{\mathcal{G}}(a,b)\,.
\]
\end{enumerate}
\item\label{BPacc:leavespaceforflip}
\index[defs]{0BPacc08@(B8$_{\acc}$)}%
\index[symbols]{BPacc08@(B8$_{\acc}$)}%
If~$\BrS$ is admissible, then~$\BrS_\acc$ is admissible in the sense that there 
exist a point~$q\in\wh\R$ and an open neighborhood~$\mathcal{U}$ of~$q$ 
in~$\wh\R$ such that
\[
\mathcal{U}\cap\bigcup_{a\in\Index}\Iset{a}=\varnothing\qquad\text{and}\qquad q\notin I_a\,,
\]
for every~$a\in\Index$.
\end{enumerate}
\end{prop}

\begin{remark}\label{REM:accgroupelements}
Let~$a,b\in\Index$.
By combining part~\eqref{BPacc:transsetsequal} of~\eqref{BPacc:intervaldecomp} 
with~\eqref{EQN:acctransset}, the definition of~$\Csacc{a}$, and 
Algorithm~\ref{stepreduction} we see that every element~$g\in\Trans{\acc}{a}{b}$ 
emerges as the product of transition set elements.
This means we find~$n\in\N$,~$k_1,\dots,k_{n+1}\in A$, 
and~$h_i\in\Trans{}{k_i}{k_{i+1}}$ 
such that~$g=h_1h_2\cdots h_{n+1}$.
\end{remark}

\begin{cor}\label{COR:accsystemnewit}
Let~$\nu\in\BrU_{\acc}$.
Then there exist uniquely determined sequences
\[
(\ittime{\acc,n}(\nu))_{n\in\Z}\text{ in 
}\R,\quad(\itindex{\acc,n}(\nu))_{n\in\Z}\text{ in 
}\Index,\quad\text{and}\quad(\ittrans{\acc,n}(\nu))_{n\in\Z}\text{ in }\Gamma\,,
\]
\index[symbols]{tacc@$\ittime{\acc,n}(\nu)$}%
\index[symbols]{kacc@$\itindex{\acc,n}(\nu)$}%
\index[symbols]{gacc@$\ittrans{\acc,n}(\nu)$}%
which satisfy the following properties:
\begin{enumerate}[label=$\mathrm{(\roman*)}$, ref=$\mathrm{\roman*}$]
\item\label{accsystemnewit:firstreturn}
The sequence~$(\ittime{\acc,n}(\nu))_{n\in\Z}$ is a subsequence of 
$(\ittime{\BrU,n}(\nu))_{n\in\Z}$. It satisfies
\begin{align*}
\ittime{\acc,0}(\nu)&=0
\intertext{and}
\ittime{\acc,n}(\nu)&=\begin{cases}\min\defset{t>\ittime{\acc,n-1}(\nu)}{
\gamma_\nu^{\prime}(t)\in\Gamma\act\BrU_{\acc}}&\text{for $n\geq 
1$}\,,
\\
\max\defset{t<\ittime{\acc,n+1}(\nu)}{\gamma_\nu^{\prime}
(t)\in\Gamma\act\BrU_{\acc}}&\text{for $n\leq-1$}\,.
\end{cases}
\end{align*}

\item\label{accsystemnewit:transforms}
For all $n\in\Z$ we have
\[
\ittrans{\acc,n}(\nu)\in\Trans{\acc}{\itindex{\acc,n-1}(\nu)}{\itindex{\acc,n}
(\nu)}\,.
\]

\item\label{accsystemnewit:intersect}
Let $a\in\Index$, $t\in\R$, and $g\in\Gamma$ be such that 
$\gamma_\nu^{\prime}(t)\in g\act\Csacc{a}$. Then there exists exactly one 
index~$n\in\Z$ such that
\begin{align*}
a & =\itindex{\acc,n}(\nu)\,,\qquad t=\ittime{\acc,n}(\nu)
\intertext{and}
g & =\ittrans{\acc,\sgn(t)}(\nu)\cdot\ittrans{\acc,2\sgn(t)}(\nu)\cdots\ittrans{
\acc,n}(\nu)\,.
\end{align*}
\end{enumerate}
\end{cor}

The following result on the relation between elements of the set of representatives of the accelerated cross section and finite sequences of transition elements should be considered as an ``accelerated variant'' of Lemma~\ref{LEM:paths_gen}.

\begin{cor}\label{COR:seq_tangvec}
Let $m\in\N$ and suppose that $a_0, a_1, \ldots, a_m\in\Index$ and $g_1, \ldots, 
g_m\in\Gamma$ are such that 
\[
 g_j \in \Trans{\acc}{a_{j-1}}{a_j} \qquad\text{for $j\in\{1,\ldots, m\}$}\,.
\]
Then there exists~$\nu\in\Csacc{a_0}$ such that 
\begin{align*}
 a_j & = \itindex{\acc,j}(\nu) \qquad\text{for $j\in\{0,\ldots, m\}$}
 \intertext{and}
 g_j & = \ittrans{\acc,j}(\nu) \qquad\text{for $j\in\{1,\ldots, m\}$}\,.
\end{align*}
Furthermore, the subset of~$\Csacc{a_0}$ of all vectors with that property is 
given by
\begin{equation}\label{EQN:Csacc_initialrestrict}
\Csacc{a_0}\vert_{h_m\act\Csacc{a_m}}\coloneqq\defset{\nu\in\Csacc{a_0}}{\exists 
t^*>0:\gamma_\nu(t^*)\in h_m\act\Csacc{a_m}}\,,
\end{equation}
where~$h_m\coloneqq g_1\cdots g_m$.
\end{cor}

\begin{proof}
The assumptions imply that
\[
g_j\act I_{a_j}\subseteq I_{a_{j-1}}\quad\text{and}\quad J_{a_{j-1}}\subseteq 
g_j\act J_{a_j}\,,
\]
for all~$j\in\{1,\dots,n\}$.
The combination of~\eqref{BPacc:allvectors} and~\eqref{BPacc:intervaldecomp} 
therefore implies
\[
\varnothing\ne\Csacc{a_0}\vert_{g_1\cdots 
g_m\act\Csacc{a_m}}\subseteq\Csacc{a_0}\vert_{g_1\cdots 
g_{m-1}\act\Csacc{a_{m-1}}}\subseteq\ldots\subseteq\Csacc{a_0}\vert_{
g_1\act\Csacc{a_1}}\,.
\]
Hence, we may choose~$\nu_0\in\Csacc{a_0}\vert_{h_m\act\Csacc{a_m}}$, with~$h_m$ 
as above, and consider the system of 
sequences~$[(\ittime{\acc,n}(\nu_0))_n,(\itindex{\acc,n}(\nu_0))_n,(\ittrans{
\acc,n}(\nu_0))_n]$ associated to it by Corollary~\ref{COR:accsystemnewit}.
Part~\eqref{BPacc:transsetsequal} of Property~\eqref{BPacc:intervaldecomp} 
together with~\eqref{EQN:acctransset} and the uniqueness of the associated 
sequences immediately implies
\[
a_1=\itindex{\acc,1}(\nu_0)\quad\text{and}\quad g_1=\ittrans{\acc,1}(\nu_0)\,.
\]
The first assertion now follows inductively by setting~$\nu_j\coloneqq 
\gamma_{\nu_{j-1}}'(\ittime{\acc,1}(\nu_{j-1}))$ for~$j=1,\dots,m$ and repeating 
this argument.

The second assertion is immediate from Corollary~\ref{COR:accsystemnewit}.
\end{proof}

In the next section we will establish the structure tuple for the strict 
transfer operator approach.
For that goal, several sets of transformations are needed, all of which being required 
to be finite.
The following result, which is a straightforward consequence 
of~(\ref{BPacc:intervaldecomp}\ref{BPacc:transsetsequal}) 
and~\eqref{EQDEF:Transacc2}, plays a crucial part in assuring finiteness for 
many of these sets.

\begin{cor}\label{COR:finramrest}
Suppose that the set of branches~$\BrS$ is finitely ramified.
Then for all~$a,b\in\Index$ with~$a=(j,\eR)$ for some~$j\in A^*$ we have
\[
\#\Trans{\acc}{a}{b}<+\infty\,.
\]
\end{cor}

\section{Existence of strict transfer operator 
approaches}\label{SEC:strictTOAexist}

Let $\Gamma$ be a geometrically finite Fuchsian group with at least one hyperbolic element that admits a set of branches. In this section we state and prove the detailed variant of Theorem~\ref{thmA}. See Theorem~\ref{mainthm}. Thus, we show that $\Gamma$ admits a strict transfer operator approach. In addition, we discuss how any given set of branches defines a structure tuple (in the sense of Section~\ref{SUBSEC:stricttrans}). For a precise statement, we need a few preparations. 

By Proposition~\ref{PROP:noncollaps}, every set of branches for~$\Gamma$ (in the sense of Definition~\ref{DEF:setofbranches}) can be turned into one that is admissible. By Proposition~\ref{PROP:allwithfiniteram}, it can further be extended to one that is finitely ramified, preserving the property of admissibility. Moreover, Proposition~\ref{PROP:stepreduxworx} guarantees that it can then be reduced to a reduced set of branches (in the sense of Definition~\ref{DEF:redsetofbranches}) that is non-collapsing. Proposition~\ref{PROP:finramandcollaps} shows that the finite ramification property remains preserved. Therefore, we may and shall suppose that $\BrS=\defset{\Cs{j}}{j\in A}$ is a reduced, admissible, non-collapsing and finitely ramified set of branches for~$\Gamma$.

We resume the notations from Section~\ref{SEC:structacc} and recall, in particular, the index set~$\Index$ as well as the sets~$\cycstarset{a}$,~$I_a$, and~$J_a$, and the transformations~$\cycstarnext{a}$,~$u_a$, and~$\cycstarnext{a,b}$ for~$a,b\in\Index$.
We denote by~$\prA\colon\Index\to A$ the projection onto the first component and 
by~$\prZ\colon\Index\to\{\eX,\eR,\eY\}$ the projection onto the second 
component.
Hence, for every~$a\in\Index$ we have~$a=(\prA(a),\prZ(a))$.
For each $a,b\in\Index$, we define
\[
P_{a,b}\coloneqq\left\{
\begin{array}{cl}
\left\{u_b^{-1}\right\}&\text{if $\prZ(a)\ne\prZ(b)$ and 
$(\prA(a),\prZ(b))\in\cycstarset{b}$}\,,\\
\varnothing&\text{otherwise},
\end{array}
\right.
\]
\index[symbols]{Pab@$P_{a,b}$}%
where we set~$\cycstarset{(\cdot,\eR)}\coloneqq\varnothing$.
For~$p\in P_{a,b}$ we set
\begin{equation}\label{eq:def_gp}
g_p\coloneqq\cycstarnext{b,\wt a}^{-1}
\end{equation}
\index[symbols]{gp@$g_p$}%
with 
\[
 \wt a \coloneqq (\prA(a),\prZ(b))\,.
\]
Furthermore, we define
\[
C_{a,b}\coloneqq\left\{
\begin{array}{cl}
\Trans{\acc}{b}{a}^{-1}&\text{if }\prZ(b)=\eR\,,\\
\left\{\cycstarnext{b,\wt a}^{-1}\right\}&\text{if 
$\prZ(b)\notin\{\prZ(a),\eR\}$ and 
$\wt a \coloneqq (\prA(a),\prZ(b))\in\cycstarset{b}$}\,,\\
\varnothing&\text{otherwise},
\end{array}
\right.
\]
\index[symbols]{Cab@$C_{a,b}$}%
where for~$G\subseteq\Gamma$ we denote~$G^{-1}\coloneqq\defset{g^{-1}}{g\in G}$. 
We emphasize that the requirement ``$(\prA(a),\prZ(b))\in\cycstarset{b}$'' in 
the conditions of these definitions is indeed correct and should not read 
``$(\prA(a),\prZ(a))\in\cycstarset{b}$.'' Further we note that for 
$a,b\in\Index$ such that $\pi_A(a) = \pi_A(b)$ and $P_{a,b}\not=\varnothing$ we 
could define $g_p$ in~\eqref{eq:def_gp} (for the unique $p\in P_{a,b}$) to be 
the identity element in~$\Gamma$ and then also define $C_{a,b}$ to be $\{\id\}$. 
However, the chosen definition has a slight advantage in the proofs of what 
follows.
Finally, since~$\BrS$ is admissible,~$\wh\R\setminus\bigcup_{j\in A}I_{j}$ has inner points.
Let~$\xi$ be such an inner point and let~$q_\xi\in\PSLR$ be such that~$q_\xi\act\xi=\infty$.
Hence, for instance,
\[
q_\xi\coloneqq\bmat{\xi}{-1-\xi^2}{1}{-\xi}\,.
\]
Then~$\infty$ is an inner point of~$q_\xi\act\bigl(\wh\R\setminus\bigcup_{j\in A}I_{j}\bigr)$, meaning that each of the sets~$q_\xi\act I_j$ is an interval in~$\R$.
Since~$I_a\subseteq \Ired{\pi_A(a)}\subseteq I_{\pi_A(a)}$ for every~$a\in\Index$, the convex hull of~$q_\xi\act\overline{\Iset{a}}$ in~$\R$, which we may denote by~$\conv(q_\xi\act\overline{\Iset{a}})$, is an interval in~$\R$ as well.
Note that~$\conv(q_\xi\act\overline{\Iset{a}})=\overline{\conv(q_\xi\act\Iset{a})}$.
We define for every~$a\in\Index$,
\[
\wh{I}_a\coloneqq q_\xi^{-1}\act\conv(q_\xi\act\overline{\Iset{a}})\,.
\]
This definition is independent of the choice of~$\xi$ and~$q_\xi$.

With these preparations we are now ready to formulate our main result.

\begin{thm}\label{mainthm}
Let~$\Gamma$ be a Fuchsian group that admits the construction of a set of 
branches for the geodesic flow on~$\Orbi$.
Then~$\Gamma$ admits a strict transfer operator approach with structure tuple 
given by
\[
\mathcal{S}\coloneqq\left(\Index,(\wh{I}_{a})_{a\in\Index},(P_{a,b})_{a,b\in\Index},
(C_{a,b})_{a,b\in\Index},((g_p)_{p\in P_{a,b}})_{a,b\in\Index}\right).
\]
\end{thm}
\index[symbols]{S@$\mathcal S$}%

Before we discuss a proof of Theorem~\ref{mainthm}, we provide a brief example 
of a structure tuple and the associated transfer operator family.

\begin{example}\label{EX:G3structup}
Recall the group~$\Gamma_{\lambda}$ from Example~\ref{EX:G3Def} and recall 
further its reduced set of branches~$\{\Cs{2},\Cs{7}\}$ from 
Example~\ref{EX:G3Graphred}.
Its easy to see that~$\{\Cs{2},\Cs{7}\}$ is already non-collapsing.
In Example~\ref{EX:G3acc} we identified the~$\eX$- and~$\eY$-cycle
\[
7\edge{t_\lambda^{-1}}7\qquad\text{and}\qquad 2\edge{t_\lambda}2\,,
\]
hence we obtain the index set
\[
\Index=\{(2,\eR),(2,\eY),(7,\eX),(7,\eR)\}\,,
\]
with
\[
\cycstarset{(2,\eY)}=\{(2,\eY)\}\,,\quad\cycstarset{(7,\eX)}=\{(7,\eX)\}\,,\quad\text{and}\quad\cycstarset{(2,\eR)}=\cycstarset{(7,\eR)}=\varnothing\,.
\]
From Figure~\ref{FIG:G3redreturn} we read off the transition sets
\begin{align*}
\Trans{\acc}{(2,\eR)}{(2,\eR)}&=\Trans{\acc}{(2,\eR)}{(2,\eY)}=\Trans{\acc}{(7,\eR)}{(2,\eR)}\\
&=\Trans{\acc}{(7,\eR)}{(2,\eY)}=\{ht_\lambda,h^{-1}t_\lambda\}\,,\\
\Trans{\acc}{(2,\eR)}{(7,\eX)}&=\Trans{\acc}{(2,\eR)}{(7,\eR)}=\Trans{\acc}{(7,\eR)}{(7,\eX)}\\
&=\Trans{\acc}{(7,\eR)}{(7,\eR)}=\{ht_\lambda^{-1},h^{-1}t_\lambda^{-1}\}\,,\\
\Trans{\acc}{(2,\eY)}{(2,\eR)}&=\defset{t_\lambda^n}{n\in\N}\,,\\
\Trans{\acc}{(7,\eX)}{(7,\eR)}&=\defset{t_\lambda^{-n}}{n\in\N}\,,
\end{align*}
and~$\Trans{\acc}{a}{b}=\varnothing$ for all other choices of~$(a,b)\in\Index\times\Index$.
From that we deduce the structure tuple for~$\Gamma_\lambda$, consisting of the index set~$\Index$ and the following quantities:
\begin{itemize}
\item the family of intervals~$(\wh{I}_a)_{a\in\Index}$ consisting of
\[
\wh{I}_{(2,\eR)}=[0,\lambda]\,,\quad \wh{I}_{(2,\eY)}=[\lambda,\infty]\,,\quad 
\wh{I}_{(7,\eX)}=[-\infty,1-\lambda]\,,\quad \wh{I}_{(7,\eR)}=[1-\lambda,1]\,,
\]

\item the families of transformation sets consisting of
\[
P_{(2,\eR),(2,\eY)}=\{t_{\lambda}^{-1}\}\,,\qquad 
P_{(7,\eR),(7,\eX)}=\{t_{\lambda}\}\,,
\]
and~$P_{a,b}=\varnothing$ for every other choice of~$(a,b)\in\Index\times\Index$, and
\begin{align*}
C_{(2,\eR),(2,\eR)}&=C_{(2,\eR),(7,\eR)}=C_{(2,\eY),(2,\eR)}=C_{(2,\eY),(7,\eR)}=\{t_\lambda^{-1}h^{-1},t_\lambda^{-1}h\}\,,\\
C_{(7,\eX),(2,\eR)}&=C_{(7,\eX),(7,\eR)}=C_{(7,\eR),(2,\eR)}=C_{(7,\eR),(7,\eR)}=\{t_\lambda h^{-1},t_\lambda h\}\,,
\end{align*}
as well as
\[
C_{(2,\eR),(2,\eY)}=\{t_\lambda^{-1}\}\,,\qquad C_{(7,\eR),(7,\eX)}=\{t_\lambda\}\,,
\]
and~$C_{a,b}=\varnothing$ for every other choice of~$(a,b)\in\Index\times\Index$,

\item and the transformations
\[
\begin{array}{ll}
g_p=t_\lambda^{-1}\,,&\text{for }\ p\in P_{(2,\eR),(2,\eY)}\,,\\
g_p=t_\lambda\,,&\text{for }\ p\in P_{(7,\eR),(7,\eX)}\,.
\end{array}
\]
\end{itemize}
The associated fast transfer operator in matrix representation takes the form
\[
\begin{small}
\fTO{s}=\begin{pmatrix}
\alpha_s(t_{\lambda}^{-1}h^{-1})+\alpha_s(t_{\lambda}^{-1}h)&
\sum_n\alpha_s(t_{\lambda}^{-n})&
0&
\alpha_s(t_{\lambda}^{-1}h^{-1})+\alpha_s(t_{\lambda}^{-1}h)
\\
\alpha_s(t_{\lambda}^{-1}h^{-1})+\alpha_s(t_{\lambda}^{-1}h)&
0&
0&
\alpha_s(t_{\lambda}^{-1}h^{-1})+\alpha_s(t_{\lambda}^{-1}h)
\\
\alpha_s(t_{\lambda}h^{-1})+\alpha_s(t_{\lambda}h)&
0&
0&
\alpha_s(t_{\lambda}h^{-1})+\alpha_s(t_{\lambda}h)
\\
\alpha_s(t_{\lambda}h^{-1})+\alpha_s(t_{\lambda}h)&
0&
\sum_n\alpha_s(t_{\lambda}^{n})&
\alpha_s(t_{\lambda}h^{-1})+\alpha_s(t_{\lambda}h)
\end{pmatrix}.
\end{small}
\]
\end{example}

\begin{remark}
Let~$a\in\Index$.
Then
\[
\wh{I}_{a,\st}=\wh{I}_a\cap\wh\R_{\st}=q_\xi^{-1}\act\bigl(\conv(q_\xi\act\Iset{a})\cap\wh\R_{\st}\bigr)=q_\xi^{-1}\act\bigl(q_\xi\act\Iset{a}\bigr)=\Iset{a}\,.
\]
Therefore, we may drop the~$~\wh{}~$~whenever~``$\st$''-sets are concerned.
The intervals~$\wh{I}_a$ were introduced solely to fit the needs of a strict transfer operator approach \emph{verbatim}.
\end{remark}

The proof of Theorem~\ref{mainthm} is split into the following six 
subsections.

\subsection{Step~0}\label{SUBSEC:Step0}

In this step we check the initial requirements for a strict transfer 
operator approach. The finiteness of the set~$\Index$ is obvious from its construction.
Let~$a,b\in\Index$.
From Definition~\ref{DEF:starcycs} and Remark~\ref{REM:accgroupelements} we see immediately that each of the sets~$P_{a,b}$ and~$C_{a,b}$ consists completely of elements 
of~$\Gamma$.
Consequentially,~$g_p\in\Gamma$ for every~$p\in P_{a,b}$.
The finiteness of~$P_{a,b}$ is obvious from its definition.
The same is true for~$C_{a,b}$ whenever~$\prZ(b)\ne\eR$.
In the case~$\prZ(b)=\eR$, since~$\BrS$ is finitely ramified, we obtain from 
Corollary~\ref{COR:finramrest} that
\[
\#C_{a,b}=\#\Trans{\acc}{b}{a}<+\infty\,.
\]
Finally, since~$P_{a,b}=\{u_b^{-1}\}$ whenever it is non-empty,~$u_b=u_{\prA(b),\prZ(b)}$, and~$u_{j,\eZ}$ is parabolic for every choice of~$(j,\eZ)\in A_{\eZ}\times\{\eX,\eY\}$ by the discussion right after~\eqref{eq:def_uZ}, every set~$P_{a,b}$ consists solely of parabolic elements.

\subsection{Property~\ref{staPROP1}}\label{SUBSEC:Proofprop1}

By a straightforward inspection we observe that 
\[
C_{a,b}^{-1} \cup \defset{ p^{-n}g_p^{-1} }{ p\in P_{a,b},\ n\in\N } 
= \Trans{\acc}{b}{a}
\]
for all~$a,b\in\Index$. This union is disjoint. In the second set of the 
union on the left hand side no element gets constructed twice. Therefore, the 
first part of Property~\ref{staPROP1}\eqref{staPROP1:Pinclude} and all of~\eqref{staPROP1:Cinclude} and 
of~\eqref{staPROP1:decompofIsets} are immediate 
by~\eqref{BPacc:intervaldecomp}\eqref{BPacc:decomp}. For the second part 
of~\eqref{staPROP1:Pinclude} we let $a,b\in\Index$ be such that 
$P_{a,b}\ne\varnothing$. Thus, $P_{a,b}=\{\cycstartrans{b}^{-1}\}$. Since the 
element~$\cycstartrans{b}^{-1}$ is parabolic (as we showed in 
Section~\ref{SEC:cuspacc}), we find $q\in\PSLR$ such that 
$\cycstartrans{b}^{-1}=qtq^{-1}$ with $t$ acting either as $z\mapsto z+1$ or as 
$z\mapsto z-1$. Then $\cycstartrans{b}^{-n}=\cycstartrans{b}^{-1}$ for 
some~$n\in\N$ implies $t^{n-1}=\id$ and hence $n=1$. In turn, 
$\cycstartrans{b}^{-n}\ne\cycstartrans{b}^{-1}$ for $n\geq 2$, which establishes 
the second part of~\eqref{staPROP1:Pinclude}.

\subsection{Property~\ref{staPROP2}}\label{SUBSEC:prop2}

For each~$n\in\N$, the set~$\Per_n$ with respect to $\mathcal{S}$ consists of 
all~$g\in\Gamma$ for which there exist an~$a\in\Index$ such that
\[
g^{-1}\act\Iset{a}\times\{a\}\ni(x,a)\mapsto(g\act 
x,a)\in\Iset{a}\times\{a\}
\]
is a submap of~$F^n$.
It is clear right away that~$\id\notin\Per_n$.
Property~\ref{staPROP2} asserts the 
union~$\Per\coloneqq\bigcup_{n=1}^\infty\Per_n$ to be disjoint.
In order to prove this assertion we decompose, for each~$n\in\N$, the set~$\Per_n$
into the sets
\[
\Per_{a,n}\coloneqq\defset{g\in\Gamma}{\left\{\begin{array}{clc}g^{-1}\act\Iset{
a}\times\{a\}&\to&\Iset{a}\times\{a\}\\(x,a)&\mapsto&(g\act 
x,a)\end{array}\right.\text{ is a submap of }F^n}.
\]
\index[symbols]{Peran@$\Per_{a,n}$}%
for~$a\in\Index$.
We emphasize that the union~$\Per_n=\bigcup_{a\in\Index}\Per_{a,n}$ is not necessarily 
disjoint.
Further, for any $n,m\in\N$, we have $\Per_n\cap\Per_m\ne\varnothing$ if and only if there are~$a,b\in\Index$ such that~$\Per_{a,n}\cap\Per_{b,m}\ne\varnothing$.
Therefore, the assertion of Property~\ref{staPROP2} is equivalent 
to~$\Per_{a,n}\cap\Per_{b,m}\ne\varnothing$ implying~$n=m$ for all (not 
necessarily distinct)~$a,b\in\Index$.

The following result equips us with all the necessary information regarding the 
elements of the sets~$\Per_{a,n}$.

\begin{prop}\label{PROP:Peranelements}
Let~$a\in\Index$,~$n\in\N$, and~$g\in\Per_{a,n}$.
Then~$g$ is hyperbolic. Its repelling fixed point~$\fixp{-}{g}$ is an inner 
point of~$\wh{I}_a$. Its attracting fixed point~$\fixp{+}{g}$ is an element 
of~$J_{\prA(a)}$.
\end{prop}

\begin{proof}
The definitions of the sets~$P_{a,b}$,~$C_{a,b}$, and~$\Per_{a,n}$ for 
any~$a,b\in\Index$ and~$n\in\N$, together with Corollary~\ref{COR:seq_tangvec} 
imply the existence of~$\nu\in\Csacc{a}$ such that
\begin{equation}\label{EQ:g_by_nu}
g^{-1}=\ittrans{\acc,1}(\nu)\cdot\ittrans{\acc,2}(\nu)\cdots\ittrans{\acc,n}
(\nu)\,.
\end{equation}
Therefore
\begin{equation}\label{EQ:intervalmatryoshka}
g^{-1}\act I_a\subseteq g^{-1}\act I_{\prA(a)}\subseteq I_a\subseteq 
I_{\prA(a)}\,.
\end{equation}
Hence, for all~$k\in\N$, we have
\[
g^{-k-1}\act\overline{\base{\Cs{\prA(a)}}}\subseteq 
g^{-k}\act\Plussp{\prA(a)}\,.
\]
Since $\Gamma$-translates of branches do not accumulate in~$\H$ (see 
Proposition~\ref{PROP:branches_locfinite}), the ``limits'' of the set sequences~$(g^{-k}\act\overline{I_a})_{k\in\N}$ and~$( g^{-k}\act \overline{I_{\prA(a)}} 
)_{k\in\N}$ are equal and a singleton in~$\wh\R$, that is, the intersection 
\begin{equation}\label{EQ:shrink_inter}
\bigcap_{k\in\N} g^{-k}\act\overline{I_a} = \bigcap_{k\in\N} g^{-k}\act 
\overline{I_{\prA(a)}}
\end{equation}
is a singleton in~$\wh\R$, consisting of a fixed point of~$g^{-1}$.
Therefore~$g$ is either hyperbolic or parabolic.

We will now show that $g$ is not parabolic, by means of a proof by 
contradiction. To that end we assume that~$g$ is parabolic. Then the singleton 
from~\eqref{EQ:shrink_inter} consists of the unique fixed point~$\fixp{}{g}$ 
of~$g$. We recall that in \emph{small} neighborhoods of~$\fixp{}{g}$ in~$\wh\R$, 
the action of~$g$, as being a parabolic element, is attracting to~$\fixp{}{g}$ 
on one of the sides of~$\fixp{}{g}$ and repelling on the other. Thus, for any 
interval~$I$ in~$\wh\R$ with $\fixp{}{g}$ in the interior of~$I$ and $I$ not 
being all of~$\wh\R$, we have $g^{-1}\act I\nsubseteq I$. Therefore 
\eqref{EQ:intervalmatryoshka} implies that $\fixp{}{g}$ is a boundary point 
of~$I_a$ and also of~$I_{\prA(a)}$. This implies that~$\prZ(a)\in\{\eX,\eY\}$ 
and hence $\fixp{}{g} = \prZ(a)_{\prA(a)}$ (i.e., if $a=(j,\eZ)$, then 
$\fixp{}{g} = \eZ_j$). Further we see that for the vector~$\nu$ 
from~\eqref{EQ:g_by_nu} we find exactly one 
pair~$(b,r)\in\cycstarset{a}\times\N_0$ with $b$ of the form 
$(k,\prZ(a))\in\Index$ such that
\[
\gamma_\nu(+\infty)\in\cycstartrans{a}^r\cycstarnext{a,b}\act\bigl(\Iset{
(\prA(b),\eR)}\cup\Iset{(\prA(b),\prZ(a)')}\bigr)
\]
(see~\eqref{EQ:Iadec_star}), where $\prZ(a)'$ is such that $\{\prZ(a),\prZ(a)'\} 
= \{\eX,\eY\}$.
Hence,
\[
\ittrans{\acc,1}=\cycstartrans{a}^r\cycstarnext{a,b}
\]
and
\begin{equation}\label{EQ:gI_shrink}
g^{-1}\act I_a\subseteq\ittrans{\acc,1}^{-1}g^{-1}\act I_a\subseteq 
I_{(\prA(b),\eR)}\cup I_{(\prA(b),\prZ(a)')}\,.
\end{equation}
Since~$b\in\cycstarset{a}$, the set~$I_b = I_{(\prA(b),\prZ(a))}$ does not vanish and, moreover, $\fixp{}{g}$ is a boundary point of~$I_b$. 
Comparing with~\eqref{EQ:gI_shrink}, we see that $\fixp{}{g}$ cannot be 
contained in~$g^{-1}\act\overline{I_a}$, which is a contradiction 
to~\eqref{EQ:shrink_inter} being the singleton~$\{\fixp{}{g}\}$. In turn, the 
element~$g$ is not parabolic.

We obtain that the element~$g$ is hyperbolic. The 
singleton~\eqref{EQ:shrink_inter} consists of the fixed 
point~$\fixp{+}{g^{-1}}=\fixp{-}{g}$ (attracting for~$g^{-1}$, repelling 
for~$g$), as follows immediately from the definition of~\eqref{EQ:shrink_inter}. 
Further, we clearly have
\[
 \fixp{-}{g}\in \overline{I_a} \subseteq \overline{I_{\prA(a)}}\,.
\]
Since, by construction, the boundary points of~$I_a$ in~$\wh\R$ are elements 
of~$\wh\R\setminus\wh\R_{\st}$, it follows that~$\fixp{-}{g}$ is an inner point 
of~$I_a$.
Finally, we can apply the same argument with the roles of~$g$ and~$g^{-1}$ 
interchanged to show that~$\fixp{+}{g}=\fixp{-}{g^{-1}}$ is an element 
of~$J_{\prA(a)}$.
\end{proof}

In what follows we extend on the initial argument in the proof of 
Proposition~\ref{PROP:Peranelements} to find an equivalent description of the 
set~$\Per_{a,n}$ in terms of iterated intersections of induced geodesics 
with~$\Gamma\act\BrU_{\acc}$.
Together with an inspection of the relationship between~$\Csacc{a}$ 
and~$\Csacc{b}$ in view of Proposition~\ref{PROP:Peranelements}, this will 
enable us to effectively compare the quantities~$m$ and~$n$.

For~$v,w\in\Index$ and~$h\in\Gamma$ we define~$\Csacc{v}\vert_{h\act\Csacc{w}}$ 
to be the subset of~$\Csacc{v}$ of all vectors~$\nu$ for which~$\gamma_\nu$ 
eventually intersects~$h\act\Csacc{w}$.
To be more precise, we set
\begin{equation}\label{EQ:Csacc_restrict}
\Csacc{v}\vert_{h\act\Csacc{w}}\coloneqq\defset{\nu\in\Csacc{v}}{\exists\,
t^*>0:\gamma_\nu'(t^*)\in h\act\Csacc{v}}
\end{equation}
(see also~\eqref{EQN:Csacc_initialrestrict}).
As for every~$\nu\in\Csacc{v}\vert_{h\act\Csacc{w}}$ the intersection time~$t^*$ 
is uniquely determined, the value 
\begin{equation}\label{EQ:def_intnr}
\varphi(v,w,h,\nu)\coloneqq\#\defset{t\in(0,t^*]}{
\gamma_\nu'(t)\in\Gamma\act\BrU_{\acc}}
\end{equation}
is well-defined. 
\index[symbols]{phi@$\varphi(v,w,h,\nu)$}%

The following result is an immediate consequence of 
Corollary~\ref{COR:seq_tangvec}.
\begin{lemma}\label{LEM:invintercount}
Let~$v,w\in\Index$ and~$h\in\Gamma$ be such 
that~$\Csacc{v}\vert_{h\act\Csacc{w}}\ne\varnothing$.
Then for every choice of~$\nu,\eta\in\Csacc{v}\vert_{h\act\Csacc{w}}$ we have
\[
\varphi(v,w,h,\nu)=\varphi(v,w,h,\eta)\,.
\]
\end{lemma}

Because of Lemma~\ref{LEM:invintercount}, for every two 
pairs~$(v,p),(w,q)\in\Index\times\Gamma$ we can define the \emph{intersection 
count}
\begin{equation}\label{DEF:countit}
\varphi((v,p),(w,q))\coloneqq\left\{
\begin{array}{cl}
\varphi(v,w,p^{-1}q,\nu)&\text{if 
$\Csacc{v}\vert_{p^{-1}q\act\Csacc{w}}\ne\varnothing$}\,,\\
0&\text{otherwise,}
\end{array}
\right.
\end{equation}
\index[symbols]{phia@$\varphi((v,p),(w,q))$}%
\index[defs]{intersection count}%
\index[defs]{count!intersection}%
with an arbitrary choice of~$\nu\in\Csacc{v}\vert_{p^{-1}q\act\Csacc{w}}$ in the 
first case.

\begin{lemma}\label{LEM:intercountprops}
The intersection count has the following properties:
\begin{enumerate}[label=$\mathrm{(\roman*)}$, ref=$\mathrm{\roman*}$]
\item\label{intercountprops:invariant}
For all~$(v,p),(w,q)\in\Index\times\Gamma$ the intersection 
count~$\varphi((v,p),(w,q))$ is invariant under~$\Gamma$ in the sense that
\[
\forall h\in\Gamma:\varphi((v,hp),(w,hq))=\varphi((v,p),(w,q))\,.
\]

\item\label{intercountprops:additive}
Let~$(v,p),(w,q),(u,h)\in\Index\times\Gamma$ be such that
\[
\Csacc{v}\vert_{p^{-1}q\act\Csacc{w}}\ne\varnothing\qquad\text{and}\qquad\Csacc{
w}\vert_{q^{-1}h\act\Csacc{u}}\ne\varnothing\,.
\]
Then
\[
\varphi((v,p),(u,h))=\varphi((v,p),(w,q))+\varphi((w,q),(u,h))\,.
\]

\item\label{intercountprops:Per}
For all~$v\in\Index$,~$n\in\N$, and~$h\in\Per_{v,n}$ we have
\[
\varphi((v,\id),(v,h^{-1}))=n\,.
\]
\end{enumerate}
\end{lemma}

\begin{proof}
The statements~\eqref{intercountprops:invariant} and~\eqref{intercountprops:Per} 
are immediately clear from the definitions of the intersection count and the 
sets involved. Regarding statement~\eqref{intercountprops:additive}, the 
hypotheses imply that
\[
p\act\Csacc{v}\vert_{q\act\Csacc{w}}\ne\varnothing\qquad\text{and}\qquad 
q\act\Csacc{w}\vert_{h\act\Csacc{u}}\ne\varnothing\,.
\]
Therefore, we have
\[
\Csacc{v}\vert_{p^{-1}h\act\Csacc{u}}\ne\varnothing\,.
\]
Hence all intersection counts involved are nonzero. We pick 
$\nu\in\Csacc{v}\vert_{p^{-1}h\act\Csacc{u}}$. Then $\nu$ is also an element 
of~$\Csacc{v}\vert_{p^{-1}q\act\Csacc{w}}$ and hence there exist uniquely 
determined~$t^*_1,t^*_2\in(0,+\infty)$,~$t^*_1<t^*_2$, such that
\[
\gamma_\nu'(t^*_1)\in 
p^{-1}q\act\Csacc{w}\qquad\text{and}\qquad\gamma_\nu'(t^*_2)\in 
p^{-1}h\act\Csacc{u}\,.
\]
With that we calculate
\begin{align*}
\varphi((v,p),(u,h))&=\varphi(v,u,p^{-1}h,\nu)=\#\defset{t\in(0,t^*_2]}{
\gamma_\nu'(t)\in\Gamma\act\BrU_{\acc}}\\
&=\#\defset{t\in(0,t_1^*]}{\gamma_\nu'(t)\in\Gamma\act\BrU_{\acc}}+\#\defset{
t\in(t_1^*,t_2^*]}{\gamma_\nu'(t)\in\Gamma\act\BrU_{\acc}}\\
&=\varphi(v,w,p^{-1}q,\nu)+\varphi(w,u,q^{-1}pp^{-1}h,\nu)\\
&=\varphi((v,p),(w,q))+\varphi((w,q),(u,h))\,.\qedhere
\end{align*}
\end{proof}

Let~$a,b\in\Index$ and~$m,n\in\N$ be such 
that~$\Per_{a,n}\cap\Per_{b,m}\ne\varnothing$.
Denote~$j\coloneqq\prA(a)$ and~$k\coloneqq\prA(b)$.
Further let~$g\in\Per_{a,n}\cap\Per_{b,m}$.
Because of Proposition~\ref{PROP:Peranelements} the transformation~$g$ is 
hyperbolic with
\[
(\fixp{-}{g},\fixp{+}{g})\in \bigl(I_a\cap I_b\bigr)\times\bigl(J_j\cap 
J_k\bigr)\,.
\]
This implies in particular that
\[
\Iset{a}\cap\Iset{b}\ne\varnothing\qquad\text{and}\qquad\Jset{j}\cap\Jset{k}
\ne\varnothing\,.
\]
The combination of those two shows that either
\[
\Plussp{j}\subseteq 
\Plussp{k}\qquad\text{or}\qquad\Plussp{k}\subseteq\Plussp{j}\,,
\]
where we suppose the former without loss of generality. Furthermore, there 
exists~$N\in\N_0$ such that
\[
g^{-N}\act\Plussp{j}\subseteq\Plussp{k}\subseteq g^{N+2}\act\Plussp{j}
\]
and
\begin{equation}\label{EQ:Csacc_nonempty}
\Csacc{b}\vert_{g^{N+2}\act\Csacc{a}}\ne\varnothing\qquad\text{and}\qquad\Csacc{
a}\vert_{g^{-N}\act\Csacc{b}}\ne\varnothing\,.
\end{equation}
This is due to the fact (see Proposition~\ref{PROP:Peranelements}) that the set 
sequences~$(g^r\act\overline{\base{\Cs{j}}})_{r\in\Z}$ 
and~$(g^r\act\overline{\base{\Cs{k}}})_{r\in\Z}$ converge to the 
singleton~$\{\fixp{\pm}{g}\}$ for~$k\to\pm\infty$ and that the fixed 
points~$\fixp{-}{g}$ and~$\fixp{+}{g}$ are inner points of~$I_j\cap I_k$ 
and~$J_j\cap J_k$, respectively. The~$2$ in the exponent of~$g$ instead of a~$1$ 
accounts for the possibility that $k$ equals~$\psi^*_\eZ(j)$ or vice versa for 
either~$\eZ\in\{\eX,\eY\}$. We note that~\eqref{EQ:Csacc_nonempty} in 
combination with 
\[
 \Csacc{a}\vert_{g^{-1}\act\Csacc{a}}\not=\varnothing
\]
(because $g\in\Per_{a,n}$) yields that 
\[
 \Csacc{a}\vert_{g^{-(N+2)}\act\Csacc{b}} \not=\varnothing\qquad\text{and}\qquad 
\Csacc{b}\vert_{g^N\act\Csacc{a}} \not=\varnothing\,.
\]
With Lemma~\ref{LEM:intercountprops} we now obtain
\begin{align*}
n&=\varphi((a,\id),(a,g^{-1}))
=\varphi((a,g^{N+1}),(a,g^{N}))
\\
& = \frac12\Bigl(\varphi((a,g^{N+2}),(a,g^{N+1})) +  
\varphi((a,g^{N+1}),(a,g^{N}))\Bigr)
\\
& =\frac{1}{2}\varphi((a,g^{N+2}),(a,g^{N}))
\\
&=\frac{1}{2}\Bigl(\varphi((a,g^{N+2}),(b,\id))+\varphi((b,\id),(a,g^{N}
))\Bigr)\\
&=\frac{1}{2}\Bigl(\varphi((a,g^{N+2}),(b,\id))+\varphi((b,g^2),(a,g^{N+2}
))\Bigr)
\\
&=\frac{1}{2}\varphi((b,g^2),(b,\id))
=\varphi((b,\id),(b,g^{-1}))=m\,.
\end{align*}
This completes the proof of Property~\ref{staPROP2}.

\subsection{Property~\ref{staPROP3}}\label{SUBSEC:prop3}

Part \eqref{staPROP3:hyperbolic} of Property~\ref{staPROP3} has already been 
shown to hold in Proposition~\ref{PROP:Peranelements}.
For parts~\eqref{staPROP3:primitive} and~\eqref{staPROP3:representatives} we 
start with some preparations.

For the first step we recall the set~$\Per_{a,n}$ for $a\in\Index$ and $n\in\N$ 
from Section~\ref{SUBSEC:prop2} (proof of Property~\ref{staPROP2} for~$\mc S$) 
and further that each element in~$\Per_{a,n}$ is hyperbolic by 
Proposition~\ref{PROP:Peranelements}. Therefore, for any $g\in\Per_{a,n}$, its 
two fixed points~$\fixp{\pm}{g^{-1}} = \fixp{\mp}{g}$ are in~$\wh\R_{\st}$ and 
each geodesic~$\gamma$ on~$\h$ with $\gamma(\pm\infty) = \fixp{\pm}{g^{-1}}$ is 
a representative of the axis~$\alpha(g^{-1})$ of~$g^{-1}$. The geodesic 
on~$\Orbi$ associated to~$\gamma$ is periodic with period length~$\ell(g^{-1})$. 
See Section~\ref{SUBSEC:geodsurfaces} and, in particular, 
Lemma~\ref{LEM:periodicaxes}.

\begin{lemma}\label{LEM:Perandaxes}
Let $a\in\Index$, $n\in\N$ and $g\in\Per_{a,n}$. Let $\gamma$ be a 
representative of~$\alpha(g^{-1})$. Then $\gamma$ intersects $\Csacc{a}$.
\end{lemma}

\begin{proof}
By Proposition~\ref{PROP:Peranelements}, $\fixp{+}{g^{-1}}\in \Iset{a}$ and 
$\fixp{-}{g^{-1}}\in \Jset{\prA(a)}$. From~\eqref{BP:allvectors} and 
Remark~\ref{REM:inst} we obtain that $\gamma$ intersects~$\Cs{\prA(a),\st}$, and 
hence $\gamma$ intersects~$\Cs{\prA(a),\acc}$. Combining 
Lemmas~\ref{LEM:accbranches}, \ref{LEM:accprops}\eqref{LEM:accprops:Isetunion} 
and~\ref{LEM:accprops}\eqref{LEM:accprops:Iset}, we see that $\gamma$ 
intersects~$\Csacc{a}$.
\end{proof}

\begin{defi}\label{DEF:uniquevec}
Let $a\in\Index$, $g\in\Gamma$ hyperbolic and suppose that $\gamma$ is a 
geodesic on~$\h$ satisfying $\gamma(\pm\infty) = \fixp{\pm}{g}$ and 
$\gamma'(0)\in\Csacc{a}$. 
We set 
\[
\nu(a,g)\coloneqq\gamma^{\prime}(0)\,.
\]
\end{defi}
The vector $\nu(a,g)$ is well-defined 
whenever~$(\fixp{+}{g},\fixp{-}{g})\in\Iset{a}\times\Jset{\prA(a)}$, hence, in 
particular in the case that~$g^{-1}\in\Per_{a,n}$ for any~$n\in\N$.
The geodesic~$\gamma$ in Definition~\ref{DEF:uniquevec} is then a representative 
of~$\alpha(g)$ and equals~$\gamma_{\nu(a,g)}$.

\begin{lemma}\label{LEM:uniquevec}
Let $a\in\Index$, $g\in\Gamma$ be hyperbolic such 
that
\[
(\fixp{+}{g^{-1}},\fixp{-}{g^{-1}})\in\Iset{a}\times\Jset{\prA(a)}\,,
\]
and set $\nu\coloneqq \nu(a,g^{-1})$. 
\begin{enumerate}[label=$\mathrm{(\roman*)}$, ref=$\mathrm{\roman*}$]
\item\label{uniquevec:conj}
For each~$m\in\Z$ set 
\[
 h_m \coloneqq \ittrans{\acc,\sgn(m)}(\nu)\cdots\ittrans{\acc,m}(\nu)\,.
\]
Then we have, for each~$m\in\Z$, 
\[
\nu(\itindex{\acc,m}(\nu), h_m g 
h_m^{-1})=h_m^{-1}\act\gamma_{\nu}^{\prime}(\ittime{\acc,m}(\nu))\,.
\]

\item\label{uniquevec:powers}
For all $m\in\N$ we have $\nu(a,g)=\nu(a,g^m)$.

\item\label{uniquevec:itseq}
The 
sequences~$\bigl((\itindex{\acc,n}(\nu),\ittrans{\acc,n}(\nu))\bigr)_{n\in\N}$ 
and~$\bigl((\itindex{\acc,-n}(\nu),\ittrans{\acc,-n}(\nu))\bigr)_{n\in\N}$ 
in~$\Index\times\Gamma$ are periodic with period length
\[
\ell\coloneqq\countit{(a,\id),(a,g^{-1})}\,.
\]
For all $m\in\N_0$ we have
\begin{align*}
\ittrans{\acc,m\ell+1}(\nu)\cdots\ittrans{\acc,(m+1)\ell}(\nu)&=g^{-1}
\intertext{and}
\ittrans{\acc,-m\ell-1}(\nu)\cdots\ittrans{\acc,-(m+1)\ell}(\nu)&=g\,.
\end{align*}
In particular we obtain
\[
 \ittrans{\acc,-n}(\nu) = \ittrans{\acc,n\ell - n +1}(\nu)^{-1}
\]
for all~$n\in\N$.
\end{enumerate}
\end{lemma}

\begin{proof}
Statements~\eqref{uniquevec:conj} and~\eqref{uniquevec:powers} are 
straightforward consequences of 
Lemma~\ref{LEM:periodicaxes}\eqref{periodicaxes:conjugacy} and 
\eqref{periodicaxes:powers}, respectively.
In order to verify \eqref{uniquevec:itseq} consider the 
geodesic~$\gamma\coloneqq\gamma_{\nu(a,g)}$.
Then~$\gamma$ is a representative of~$\alpha(g^{-1})$.
Thus, $\gamma(\R)$ is fixed by~$g$ and~$g^{-1}$.
By Lemma~\ref{LEM:Perandaxes},~$\gamma$ intersects~$g^m\act\Csacc{a}$ for 
every~$m\in\Z$.
Let~$t_1>0$ be such that $\gamma'(t_1)=g^{-1}\act\nu$.
By~\eqref{EQ:def_intnr} and~\eqref{DEF:countit}, the number of intersections 
of~$\gamma$ with~$\Gamma\act\BrU_{\acc}$ at times~$t\in(0,t_1]$ is given 
by~$\countit{(a,\id),(a,g^{-1})}$.
Applying~\eqref{uniquevec:conj} with~$m=1$ now shows the periodicity of the 
first 
sequence~$\bigl((\itindex{\acc,n}(\nu),\ittrans{\acc,n}(\nu))\bigr)_{n\in\N}$ 
with the claimed period length, as well as
\[
\ittrans{\acc,m\ell+1}(\nu)\cdots\ittrans{\acc,(m+1)\ell}(\nu)=g^{-1}
\]
for every~$m\in\N$.
Because of Lemma~\ref{LEM:intercountprops}\eqref{intercountprops:invariant} the 
remaining statements follow analogously, by considering~$g$ instead of~$g^{-1}$ 
and~$t_2<0$ such that~$\gamma'(t_2)=g\act\nu$.
\end{proof}

\begin{prop}\label{PROP:exacthitonj}
Let $a\in\Index$, $n\in\N$ and $g\in\Gamma$. Then $g\in\Per_{a,n}$ if and only 
if there exists $\nu\in\Csacc{a}$ and $t^*>0$ such that 
\[
\#\defset{ t\in (0,t^*] }{\gamma_\nu^{\prime}(t)\in\Gamma\act\BrU_{\acc}}=n
\]
and
\[
\gamma_\nu^{\prime}(t^*)=g^{-1}\act\nu\,.
\]
In this case, $t^*$ is the displacement length~$\ell(g^{-1})$ of~$g^{-1}$.
\end{prop}

\begin{proof}
We suppose first that $g\in\Per_{a,n}$. By 
Lemma~\ref{LEM:intercountprops}\eqref{intercountprops:Per}, 
\begin{equation}\label{eq:countisn}
 \countit{(a,\id),(a,g^{-1})} = n\,.
\end{equation}
By~\eqref{DEF:countit} this implies (since $n\not=0$) that 
$\Csacc{a}\vert_{g^{-1}\act\Csacc{a}}$ is not empty, and further that 
\[
\fixp{-}{g^{-1}} \in \Jset{\prA(a)} \subseteq g^{-1}\act\Jset{\prA(a)}
\]
and
\[
\fixp{+}{g^{-1}} \in \Iset{a}\cap g^{-1}\act \Iset{a}\,.
\]
Let $\nu$ denote the unique element of~$\Csacc{a}$ satisfying 
\[
 \bigl(\gamma_\nu(+\infty), \gamma_\nu(-\infty)\bigr) = 
\bigl(\fixp{+}{g^{-1}},\fixp{-}{g^{-1}}\bigr)
\]
and let $t^*$ denote the intersection time of~$\gamma_\nu$ with $g^{-1}\act 
\Csacc{a}$. Then 
\[
 \#\defset{ t\in (0,t^*] }{\gamma_\nu^{\prime}(t)\in\Gamma\act\BrU_{\acc}}=n
\]
by~\eqref{DEF:countit} and~\eqref{eq:countisn}. Further let $\xi\coloneqq 
\gamma_\nu^\prime(t^*)$ denote the intersection vector of~$\gamma_\nu$ 
with~$g^{-1}\act\Csacc{a}$. Then $g\act\xi\in\Csacc{a}$ and 
\[
 \gamma_{g\act\xi}(\pm\infty) = g\act\gamma_\nu(\pm\infty) = 
\fixp{\pm}{g^{-1}}\,.
\]
Hence, 
\[
 \gamma_\nu^\prime(t^*) = \xi = g^{-1}\act\nu\,.
\]
This proves the one claimed implication of the proposition. For the converse 
implication we suppose that $\nu\in\Csacc{a}$ and $t^*>0$ are chosen such that 
\begin{equation}\label{eq:knowisn}
 \#\defset{ t\in (0,t^*] }{\gamma_\nu^{\prime}(t)\in\Gamma\act\BrU_{\acc}}=n
\end{equation}
and $\gamma_\nu^\prime(t^*) = g^{-1}\act\nu$. From 
Corollary~\ref{COR:accsystemnewit}\eqref{accsystemnewit:intersect} 
and~\eqref{eq:knowisn} we obtain that 
\[
 g^{-1} = \ittrans{\acc,1}(\nu)\cdots\ittrans{\acc,n}(\nu)\,.
\]
The definition of the map~$F$ now immediately implies that $g\in\Per_{a,n}$. 
This completes the proof of the converse implication.
Subsequently, the equality 
$t^*=\ell(g^{-1})$ follows immediately from the definition of the 
displacement length.
\end{proof}

The following result implies Property~\ref{staPROP3}\eqref{staPROP3:primitive}.

\begin{prop}\label{PROP:primitiveofPer}
Let $a\in\Index$, $n\in\N$, and $h\in\Per_{a,n}$. Let $h_0\in\Gamma$ be primitive such 
that $h_0^m=h$. Then $h_0\in\Per_{a,\tfrac{n}{m}}$.
\end{prop}

\begin{proof}
Since $h\in\Per_{a,n}$, we find $\nu\in\Csacc{a}$ such that, with $t^*\coloneqq 
\ell(h^{-1})$ (the displacement length of~$h^{-1}$), we have
\begin{equation}\label{eq:part1}
 \gamma_\nu^\prime(t^*) = h^{-1}\act\nu
\end{equation}
and 
\begin{equation}\label{eq:part2}
 \#\defset{t\in (0,t^*]}{\gamma_\nu^\prime(t)\in\Gamma\act\Cs{\acc}} = n
\end{equation}
by Proposition~\ref{PROP:exacthitonj}. From~\eqref{eq:part1} it follows that 
$\gamma_\nu$ is a representative of the axis of~$h^{-1}$ and hence also 
of~$h_0^{-1}$ (cf.\@ Section~\ref{SUBSEC:geodsurfaces}), and further that 
$\nu=\nu(a,h^{-1})$ and that $\nu(a,h_0^{-1})$ exists and equals~$\nu$, the 
latter by Lemma~\ref{LEM:uniquevec}\eqref{uniquevec:powers}. Therefore, for 
$\wt t\coloneqq \ell(h_0^{-1})$ we obtain 
\[
 \gamma_\nu^\prime(\wt t) = h_0^{-1}\act\nu\,.
\]
The system of accelerated iterated sequences of any element in~$\Cs{\acc}$ depends only on this considered element (and of course the choice of~$\Cs{\acc}$) and 
it is equivariant under the action of elements in~$\Gamma$, as can be observed 
from Corollary~\ref{COR:accsystemnewit}. See also 
Lemmas~\ref{LEM:intercountprops}\eqref{intercountprops:invariant} 
and~\ref{LEM:uniquevec}\eqref{uniquevec:itseq}. Thus, the accelerated sequence 
of intersection times of~$h_0^{-1}\act\nu$ is only shifted against the 
accelerated sequence of intersection times of~$\nu$, and in particular, the 
sequence of differences 
\begin{equation}\label{eq:diffseq}
 \bigl(\ittime{\acc,n+1}(\nu) - \ittime{\acc,n}(\nu)\bigr)_n = 
\bigl(\ittime{\acc,n+1}(h_0^{-1}\act\nu) - 
\ittime{\acc,n}(h_0^{-1}\act\nu)\bigr)_n
\end{equation}
is periodic. The relation $\ell(h^{-1}) = m\ell(h_0^{-1})$ between the 
displacement lengths given by 
Lemma~\ref{LEM:periodicaxes}\eqref{periodicaxes:powers} implies that $m$ (from 
$h_0^m=h$) is a period length of the sequence~\eqref{eq:diffseq}. From this and 
\eqref{eq:part2} it follows that 
\[
 \#\defset{t\in (0,\wt t]}{\gamma_\nu^\prime(t) \in \Gamma\act\Cs{\acc}} = 
\frac1m \#\defset{t\in (0,t^*]}{\gamma_\nu^\prime(t) \in \Gamma\act\Cs{\acc}}
  = \frac{n}{m}\,.
\]
By Proposition~\ref{PROP:exacthitonj}, $h_0\in\Per_{a,\frac{n}{m}}$.
\end{proof}

In order to establish~\eqref{staPROP3:representatives} of 
Property~\ref{staPROP3} we let $[g]\in[\Gamma]_{\mathrm{h}}$ and consider the 
equivalence class of periodic geodesics~$\varrho([g^{-1}])$ on~$\Orbi$, 
where~$\varrho$ is the map from~\eqref{EQDEF:mapaxistoX}. The combination 
of~\eqref{BP:closedgeodesicsXtoH} (see Proposition~\ref{PROP:oldB1}) with 
Proposition~\ref{PROP:accsystemcrosssection} (or 
Proposition~\ref{PROP:accsystem:geodintersect}) imply that we find $a\in\Index$
and a geodesic~$\gamma$ on~$\H$ such that $\gamma$ intersects~$\Csacc{a}$. By 
Lemma~\ref{LEM:periodicaxes}\eqref{periodicaxes:conjugacy} there exists a 
representative $h$ of $[g]$ such that $\gamma$ represents the axis of~$h$, thus, 
$[\gamma]=\alpha(h^{-1})$. Applying Proposition~\ref{PROP:exacthitonj} with 
$t^*=\ell(g^{-1}) = \ell(h^{-1})$ and
$\nu$ being the intersection vector of~$\gamma$ with~$\Csacc{a}$ yields that 
$h\in\Per_{a,n}$ for some $n\in\N$. Property~\ref{staPROP2} now shows uniqueness 
of~$n$, which completes the proof of~\eqref{staPROP3:representatives}.

\subsection{Property~\ref{staPROP4}}\label{SUBSEC:prop4}

Let $g\in\Per$. Recall the word length $\omega(g)$ of $g$ from 
\eqref{EQDEF:wordlengthofh}. Let $a\in\Index$ and $n\in\N$ be such that 
$g\in\Per_{a,n}$.
Let $g_0\in\Gamma$ be the primitive of $g$ and recall the objects $m(g)$ and 
$p(g)$ from \eqref{EQDEF:ratioofh} and before.
By applying Proposition~\ref{PROP:primitiveofPer} we observe
\begin{align}\label{EQN:peqomega0}
\omega(g_0)=\frac{m(g)\omega(g_0)}{m(g)}=\frac{\omega(g)}{m(g)}=p(g)\,.
\end{align}

Recall further the vector~$\nu(b,h)$ from Definition~\ref{DEF:uniquevec}, which 
is uniquely determined for all~$b\in\Index$ and~$h^{-1}\in\Per_{b,m}$,~$m\in\N$.

We will make extensive use of the following abbreviation: 
For~$b\in\Index$,~$h^{-1}\in\bigcup_{k\in\N}\Per_{b,k}$, and~$m\in\Z$ we write
\[
\mathrm{g}_{m}(b,h)\coloneqq\left(\ittrans{\acc,\sgn(m)}(\nu(b,h))\cdots\ittrans
{\acc,m}(\nu(b,h))\right)^{\sgn(m)}\,.
\]
Note that~$\mathrm{g}_{k\cdot\omega(h_0)}(b,h)=h_0^k$ for all~$k\in\N_0$ 
and~$h_0$ the primitive of~$h$ in~$\Gamma$, by Lemma~\ref{LEM:uniquevec} and 
Proposition~\ref{PROP:primitiveofPer}.
It is further in line with this definition to 
set~$\mathrm{g}_0(b,h)\coloneqq\id$ for every possible choice of~$b$ and~$h$.

We recall that the initial set of branches~$\BrS$ has been supposed to be 
non-collapsing.

\begin{lemma}\label{LEM:noncollapspayoff}
Let~$b$ and~$h$ be as before and let~$k,\ell\in\N$,~$k\not=\ell$.
Then
\[
\mathrm{g}_k(b,h)^{-1}\mathrm{g}_\ell(b,h)\ne\id\,.
\]
\end{lemma}

\begin{proof}
Without loss of generality, we may suppose that $k<\ell$. 
By Remark~\ref{REM:accgroupelements}, for every~$j\in\{k+1,\dots,\ell\}$ we 
find~$m_j\in\N$,~$k_{j,0},\dots,k_{j,m_j}\in A$, 
and~$h_{j,i}\in\Trans{}{k_{j,i-1}}{k_{j,i}}$ ($i\in\{1,\ldots, m_j\}$) such that
\[
\ittrans{\acc,j}(\nu(b,h))=h_{j,1}\cdots h_{j,m_j}\,.
\]
By construction,~$\Csacc{\itindex{\acc,j}(\nu(b,h))}\subseteq\Cs{k_{j,m_j}}$, 
for all~$j$.
Thereby,~$k_{j,m_j}=k_{j+1,0}$.
From this we obtain
\begin{align*}
\mathrm{g}_k(b,h)^{-1}\mathrm{g}_\ell(b,h)&=\ittrans{\acc,k+1}(\nu(b,
h))\cdots\ittrans{\acc,\ell}(\nu(b,h))\\
&=h_{k+1,1}\cdots h_{k+1,m_{k+1}}h_{k+2,1}\cdots h_{k+2,m_{k+2}}\cdots\cdots 
h_{\ell,1}\cdots h_{\ell,m_\ell}\\
&\ne\id\,,
\end{align*}
where the final relation is due to~\eqref{BP:noidentity}.
\end{proof}

\begin{lemma}\label{LEM:conjinPer}
Let~$g$,~$a$, and~$n$ be as before.
Let $q\in\Gamma$ be such that $qgq^{-1}\in\Per_n$.
Then there exists exactly one $\kappa\in\N$, $\kappa\leq\omega(g_0)$, such that
\begin{equation}\label{eq:commute}
\mathrm{g}_{\kappa}(a,g^{-1})\cdot q\cdot 
g=g\cdot\mathrm{g}_{\kappa}(a,g^{-1})\cdot q\,.
\end{equation}
Or, equivalently, there exists a unique~$\kappa \in\{1,\ldots, \omega(g_0)\}$ 
such that
\begin{equation}\label{eq:representq}
q\in Q_\kappa\coloneqq\defset{\mathrm{g}_{\kappa}(a,g^{-1})^{-1}\cdot 
g_0^\ell}{\ell\in\Z}\,,
\end{equation}
where~$g_0$ denotes the primitive of~$g$ in~$\Gamma$.
\end{lemma}

\begin{proof}
The equivalence of~\eqref{eq:commute} and~\eqref{eq:representq} is an immediate 
consequence of the fact that non-identity elements of~$\PSLR$ commute if and 
only if they have the same fixed point set 
(see~\cite[Theorem~2.3.2]{Katok_fuchsian}).

Let $b\in\Index$ be such that $h\coloneqq qgq^{-1}\in\Per_{b,n}$.
In the case~$q=\id$ we have 
\begin{equation*}
\mathrm{g}_{\omega(g_0)}(a,g^{-1})\cdot 
g=g\cdot\mathrm{g}_{\omega(g_0)}(a,g^{-1})
\end{equation*}
by Lemma~\ref{LEM:uniquevec} and Proposition~\ref{PROP:primitiveofPer} because 
$\mathrm{g}_{\omega(g_0)}(a,g^{-1})=g_0^{-1}$.
See also the argumentation right before Lemma~\ref{LEM:noncollapspayoff}. Thus, 
the 
identity~\eqref{eq:commute} is valid for~$\kappa=\omega(g_0)$. It remains to 
show that \eqref{eq:commute} is not valid for any smaller value for~$\kappa$ 
in~$\N$. 
To that end we consider~\eqref{eq:commute} as the quest for elements commuting 
with~$g$ (recall that $q=\id$ for the moment).
Applying \cite[Theorem~2.3.2]{Katok_fuchsian} we see that any solution 
$\mathrm{g}_{\kappa}(a,g^{-1})$ 
of~\eqref{eq:commute} must be a non-trivial power of~$g_0$.
Within $\{1,\ldots, \omega(g_0)\}$ only $\kappa=\omega(g_0)$ yields 
such an element in~$\Gamma$.

We consider now the case that~$q\ne\id$.
By Lemma~\ref{LEM:periodicaxes}\eqref{periodicaxes:conjugacy} the axis 
of $h^{-1}=qg^{-1}q^{-1}$ is given by $q\act\alpha(g^{-1})$. Any of the 
representatives of~$\alpha(h^{-1})$ intersects~$\Csacc{b}$ by 
Lemma~\ref{LEM:Perandaxes}.
Hence,~$\gamma_{\nu(a,g^{-1})}$, which is a representative 
of~$\alpha(g^{-1})$, intersects $q^{-1}\act\Csacc{b}$.
This implies the existence of $r\in\Z\setminus\{0\}$ such that
\[
q=\mathrm{g}_r(a,g^{-1})^{-\sgn(r)}\,.
\]
There is a unique way to write
\[
r=\lambda\cdot\omega(g_0)+\kappa\,,
\]
with $\lambda\in\Z$ and $\kappa\in(0,\omega(g_0)]\cap\Z$.
Using Lemma~\ref{LEM:uniquevec}\eqref{uniquevec:itseq} we calculate for~$r>0$
\begin{align*}
qgq^{-1}&=\mathrm{g}_{\lambda\omega(g_0)+\kappa}(a,g^{-1})^{-1}\cdot 
g_0^{m(g)}\cdot 
\mathrm{g}_{\lambda\omega(g_0)+\kappa}(a,g^{-1})\\
&=\mathrm{g}_{\kappa}(a,g^{-1})^{-1}\cdot g_0^{-\lambda}\cdot g_0^{m(g)}\cdot 
g_0^\lambda\cdot \mathrm{g}_{\kappa}(a,g^{-1})\\
&=\mathrm{g}_{\kappa}(a,g^{-1})^{-1}\cdot g\cdot 
\mathrm{g}_{\kappa}(a,g^{-1})\,.
\end{align*}
For~$r<0$, we calculate, again by applying 
Lemma~\ref{LEM:uniquevec}\eqref{uniquevec:itseq},
\begin{align*}
\mathrm{g}_r(a,g^{-1})^{-1}&= \ittrans{\acc,-1}\cdots\ittrans{\acc,r}
\\
& = \ittrans{\acc,\omega(g_0)}^{-1}\ittrans{\acc,2\omega(g_0)-1}^{-1}\cdots 
\ittrans{\acc,-r\omega(g_0) - (r-1)}^{-1}
\\
& = \bigl( 
\ittrans{\acc,\omega(g_0)}^{-1}\ittrans{\acc,2\omega(g_0)-1}^{-1}\cdots 
\ittrans{\acc,1}^{-1}  \bigr)^{-\lambda} 
\ittrans{\acc,2\omega(g_0)}^{-1}\cdots 
\ittrans{\acc,\kappa+1}^{-1}
\\
& \qquad \cdot 
\ittrans{\acc,\kappa}^{-1}\cdots 
\ittrans{\acc,1}^{-1}\ittrans{\acc,1}\cdots 
\ittrans{\acc,\kappa}
\\
& = \bigl( \ittrans{\acc,\omega(g_0)}^{-1}\ittrans{\acc,2\omega(g_0)-1}^{-1} 
\cdots \ittrans{\acc,1}^{-1}  \bigr)^{-\lambda + 1} \mathrm{g}_\kappa(a,g^{-1})
\\
& = g_0^{-\lambda+1}\mathrm{g}_\kappa(a,g^{-1})\,,
\end{align*}
where every transformation~$\ittrans{\acc,j}$,~$j\in\Z$, is to be understood 
with respect to the vector~$\nu(a,g^{-1})$.
Thus, we can proceed as in the case~$r>0$ to obtain
\[
qgq^{-1}=\mathrm{g}_{\kappa}(a,g^{-1})^{-1}\cdot 
g\cdot\mathrm{g}_{\kappa}(a,g^{-1})\,.
\]
It remains to show that $\kappa$ is unique in $\{1,\ldots,\omega(g_0)\}$ with 
\[
 q = \mathrm{g}_{\kappa}(a,g^{-1})^{-1}\cdot g_0^{\ell}
\]
for some~$\ell\in\Z$. To that end, we suppose that there 
exist~$\vartheta\in\{1,\dots,\omega(g_0)\}$ and $m\in\Z$ such that 
\[
 q = \mathrm{g}_{\vartheta}(a,g^{-1})^{-1}\cdot g_0^{m}\,.
\]
Then 
\begin{equation}\label{eq:foundp}
 \mathrm{g}_{\kappa}(a,g^{-1})\mathrm{g}_{\vartheta}(a,g^{-1})^{-1} = g_0^p
\end{equation}
for some~$p\in\Z$. For $p\in\N_0$, this is equivalent to 
\begin{equation}\label{eq:isid1}
 g_0^{-p}\mathrm{g}_{\kappa}(a,g^{-1})\mathrm{g}_{\vartheta}(a,g^{-1})^{-1} = 
\id\,.
\end{equation}
From Lemma~\ref{LEM:uniquevec}\eqref{uniquevec:itseq} we obtain 
\[
 g_0^{-p}\mathrm{g}_{\kappa}(a,g^{-1}) = 
\mathrm{g}_{p\omega(g_0)+\kappa}(a,g^{-1})\,.
\]
Using this identity in~\eqref{eq:isid1} and combining with 
Lemma~\ref{LEM:noncollapspayoff} we find
\[
 p\omega(g_0)+\kappa = \vartheta\,,
\]
thus $p=0$ and 
\[
 \kappa = \vartheta\,.
\]
For the case that in~\eqref{eq:foundp}, $p$ is a nonpositive integer, we 
convert~\eqref{eq:foundp} into 
\[
 \mathrm{g}_{\kappa}(a,g^{-1})\mathrm{g}_{\vartheta}(a,g^{-1})^{-1}g_0^{-p} = 
\id\,.
\]
Again using Lemma~\ref{LEM:uniquevec}\eqref{uniquevec:itseq} we obtain 
\[
 \mathrm{g}_{\vartheta}(a,g^{-1})^{-1}g_0^{-p} = 
\bigl(g_0^p\mathrm{g}_{\vartheta}(a,g^{-1})  \bigr)^{-1}  = 
\mathrm{g}_{-p\omega(g_0) + \vartheta}(a,g^{-1})^{-1}
\]
With Lemma~\ref{LEM:noncollapspayoff} we find
\[
 \kappa = -p\omega(g_0) + \vartheta\,.
\]
Therefore, $p=0$ and 
\[
 \kappa = \vartheta\,.
\]
This completes the proof. 
\end{proof}

We now establish Property~\ref{staPROP4}. 
Consider the sets~$Q_\kappa$ from~\eqref{eq:representq}.
Since~$g_0^{m(g)}=g$, we have
\[
qgq^{-1}=\mathrm{g}_\kappa(a,g^{-1})^{-1}\cdot 
g\cdot\mathrm{g}_\kappa(a,g^{-1})\,,
\]
for every~$q\in Q_\kappa$.
Therefore, because of the uniqueness of~$\kappa$ from Lemma~\ref{LEM:conjinPer}, 
the number~$\#([g]\cap\Per_n)$ is bounded from above by
\[
\#\{1,\dots,\omega(g_0)\}=\omega(g_0)\stackrel{\eqref{EQN:peqomega0}}{=}p(g)\,.
\]
Hence, it remains to show that
\begin{equation}\label{eq:qkinPer}
h_\kappa\coloneqq\mathrm{g}_\kappa(a,g^{-1})^{-1}\cdot 
g\cdot\mathrm{g}_\kappa(a,g^{-1})\in\Per_n
\end{equation}
for all~$\kappa\in\{1,\dots,\omega(g_0)\}$.
This is achieved via the same argument as in the proof of 
Lemma~\ref{LEM:conjinPer}:
By Lemma~\ref{LEM:periodicaxes}\eqref{periodicaxes:conjugacy} the axis 
of~$h_\kappa^{-1}$ is given 
by~$\mathrm{g}_\kappa(a,g^{-1})^{-1}\act\alpha(g^{-1})$ and because of 
Lemma~\ref{LEM:Perandaxes} it intersects $\Csacc{b}$, 
where~$b\coloneqq\itindex{\acc,\kappa}(\nu(a,g^{-1}))$.
This implies
\[
(\fixp{+}{h_\kappa^{-1}},\fixp{-}{h_\kappa^{-1}})\in\Iset{b}\times\Jset{b}
\subseteq I_b^\circ\times J_b\,.
\]
Thus,~$\mathrm{g}_\kappa(a,g^{-1})^{-1}\act\alpha(g^{-1})$ 
intersects~$h_\kappa^{-1}\act\Csacc{b}$ as well, 
implying~$h_\kappa\in\Per_{b,m}$ for some~$m\in\N$.
Since~$g\in\Per_n$ and~$h_\kappa\in[g]$, part~\eqref{staPROP3:representatives} 
of Property~\ref{staPROP3} now yields~$m=n$.
This in turn yields~\eqref{eq:qkinPer} and thereby finishes the proof of 
Property~\ref{staPROP4}.

\subsection{Property~\ref{staPROP5}}\label{SUBSEC:prop5}

With the following proposition we can now establishes 
Property~\ref{staPROP5} and finish the proof of Theorem~\ref{mainthm}. We 
remark that the proof of Proposition~\ref{prop:prop5} is constructive.

\begin{prop}\label{prop:prop5}
There exists a family~$(\mc E_a)_{a\in \Index}$ of open disks in~$\wh\C$ which 
satisfies all properties of Property~\ref{staPROP5}. 
\end{prop}

Instead of establishing directly the existence of such a family~$(\mc E_a)_{a\in \Index}$
of open disks in~$\C$, we will first provide a family of intervals of~$\R$ which satisfies a set of properties corresponding to those requested in Property~\ref{staPROP5}. Working on the real axis, results in a less involved discussion. A first part is provided by Lemma~\ref{LEM:prep_prop5} below, the remaining steps are done in the proof of Proposition~\ref{prop:prop5}. We then expand the real intervals to complex disks with centers in the real line (``complex disk hull''). Since all considered actions are by linear fractional transformations, all inclusion properties are inherited by the complex disks. 

We start with a few preparations.
Taking advantage of the admissibility of the set of branches~$\BrS$, we may 
suppose that for each $a\in\Index$, the interval~$\wh{I}_a$ is a bounded subset 
of~$\R$.
For that, we possibly need to conjugate the group~$\Gamma$ by some element in~$\PSL_2(\R)$, which, however, does not affect the validity of any results.
Alternatively, we may interpret this step as using a non-standard chart for the relevant part of~$\wh\R$.
See Remark~\ref{REM:SoBrem}\eqref{SoBrem:leavespaceforflip}.
For each $a\in\Index$ let $x_a,y_a\in\R$, $x_a<y_a$, denote the boundary points of~$\wh{I}_a$, thus
\begin{equation}\label{eq:La_bdry}
\wh{I}_a = [x_a, y_a]\,.
\end{equation}

\begin{lemma}\label{LEM:prep_prop5}
There exists a family $\bigl( (\eps_a^x,\eps_a^y) \bigr)_{a\in\Index}$ in~$\R^2$ 
such that for each~$a\in\Index$
\begin{enumerate}[label=$\mathrm{(\roman*)}$, ref=$\mathrm{\roman*}$]
\item\label{eq:prepprop5_large} $\eps_a^x,\eps_a^y>0$\,,
\end{enumerate}
and with 
\[
 E_a \coloneqq (x_a-\eps_a^x, y_a+\eps_a^y)
\]
we have
\begin{enumerate}[resume, label=$\mathrm{(\roman*)}$, ref=$\mathrm{\roman*}$]
\item\label{eq:prepprop5_away} for each~$b\in\Index$ and each~$g\in C_{a,b}$ we 
have
\[
  g\act\infty\notin \overline{E_a}\,,
\]
\item\label{eq:prepprop5_space} for each $b\in\Index$ and each $g\in C_{a,b}$ 
such that $g^{-1}\act \wh{I}_a$ is contained in the interior of~$\wh{I}_b$, we have
\[
 g^{-1}\act \overline E_a \subseteq E_b\,,
\]
\item\label{eq:prepprop5_compact} for all~$b\in\Index$ and all~$p\in P_{a,b}$ 
such that $g_p^{-1}\act \wh{I}_a$ is contained in the interior of~$\wh{I}_b$, we find a 
compact interval~$K_{a,b,p}$ of~$\R$ such that 
\[
 p^{-n}g_p^{-1}\act\overline{E_a} \subseteq K_{a,b,p} \subseteq E_b
\]
for all~$n\in\N$,
\item\label{eq:prepprop5_fixed} for all~$b\in\Index$ and all~$p\in P_{a,b}$, the 
fixed point of~$p$ is not contained in~$g_p^{-1}\act\overline{E_a}$.
\end{enumerate}
Moreover, for any $a\in\Index$, there exist thresholds~$\eta_a^x>0$ and 
$\eta_a^y>0$ such that any family $\bigl( 
(\eps_a^x,\eps_a^y) \bigr)_{a\in\Index}$ that satisfies $\eps_a^x\in 
(0,\eta_a^x)$ and $\eps_a^y\in (0,\eta_a^y)$ for all $a\in\Index$ 
also satisfies~\eqref{eq:prepprop5_large}--\eqref{eq:prepprop5_fixed}.
\end{lemma}

\begin{proof}
In what follows we will show the existence of the thresholds~$\eta_a^x$ 
and~$\eta_a^y$ for all~$a\in\Index$ by showing that they only need to obey 
a finite number of positive upper bounds. These bounds depend on~$\Index$ and a 
finite number of elements in~$\Gamma$, but they do not have any 
interdependencies among each other, i.e., the values for the thresholds are 
independent of each other. 

We start by considering~$a,b\in\Index$ and $g\in C_{a,b}$. Since $g^{-1}\in 
\Trans{\acc}{b}{a}$ we have
\[
 g^{-1}\act\Iset{a} \subseteq \Iset{b}
\]
by~\eqref{BPacc:intervaldecomp}\eqref{BPacc:decomp}. Thus, $g^{-1}\act \wh{I}_a 
\subseteq \wh{I}_b$. By hypothesis, $\infty\notin \wh{I}_b$ and hence
\[
g\act\infty \notin \wh{I}_a\,.
\]
Since $\wh{I}_a$ is compact (see~\eqref{eq:La_bdry}), we can find an open 
neighborhood~$U$ (a ``thickening'') of~$\wh{I}_a$ in~$\wh\R$ that does not 
contain~$g\act\infty$. The open neighborhood~$U$ can be chosen uniformly for 
all~$b\in\Index$ and all~$g\in C_{a,b}$ as $\Index$ and $C_{a,b}$ are finite 
sets. The first condition on the thresholds~$\eta_a^x$ and $\eta_a^y$ is that 
$(x_a-\eta_a^x, y_a + \eta_a^y)$ is contained in~$U$. This implies a positive 
upper bound for each of~$\eta_a^x$ and~$\eta_a^y$, which can obviously be 
optimized (of which we will not take care here).

We now suppose in addition that $g^{-1}\act \wh{I}_a\subseteq \wh{I}_b^\circ$. 
Again using that $\wh{I}_a$, and hence $g^{-1}\act 
\wh{I}_a$, is compact, we find an open neighborhood~$\wt U$ of~$\wh{I}_a$ in~$\wh\R$ 
such that 
\[
g^{-1}\act\wt U \subseteq \wh{I}_b^\circ\,.
\]
Again taking advantage of the finiteness of~$\Index$ and~$C_{a,b}$ for 
each~$b\in\Index$, we can chose~$\wt U$ uniformly for all~$b\in\Index$ and 
all~$g\in C_{a,b}$. Our second condition on the thresholds~$\eta_a^x$ and 
$\eta_a^y$ is that $(x_a-\eta_a^x, y_a + \eta_a^y)$ is contained in~$\wt U$. 
Also this requirement can obviously be satisfied. 

We now consider $a,b=(k,V)\in\Index$ and $p\in P_{a,b}$. The fixed point of~$p$ 
is 
\[
\prZ(b)_{\prA(b)} = V_k\,,
\]
which is the unique point contained in
\[
\bigcap_{n\in\N_0} p^{-n}\act \wh{I}_b\,.
\]
From~\eqref{EQDEF:Transacc1},~\eqref{BPacc:intervaldecomp}\eqref{BPacc:decomp} 
and~\eqref{EQ:Iadec_star} it follows that 
\[
 g_p^{-1}\act \Iset{a} \subseteq \Iset{b}\setminus p^{-1}\act\Iset{b}
\]
and hence
\[
 g_p^{-1}\act \wh{I}_a \subseteq \wh{I}_b\setminus p^{-2}\act \wh{I}_b\,.
\]
Therefore~$V_k$ is not contained in $g_p^{-1}\act \wh{I}_a$. Our third condition on 
the 
thresholds~$\eta_a^x$ and $\eta_a^y$ is that 
\[
 g_p^{-1}\act (x_a-\eta_a^x, y_a+\eta_a^y) 
\]
does not contain~$V_k$ for all~$b\in\Index$ and all~$p\in P_{a,b}$. 
Analogously to above, using the compactness of~$\wh{I}_a$ and the finiteness 
of~$\Index$ and~$P_{a,b}$, we deduce that such choices of~$\eta_a^x, 
\eta_a^y>0$ are possible. 

We now suppose in addition that $g_p^{-1}\act \wh{I}_a\subseteq \wh{I}_b^\circ$. By the 
compactness of~$\wh{I}_a$ we find an open neighborhood~$W$ of~$\wh{I}_a$ such that 
\[
 g_p^{-1}\act W\subseteq \wh{I}_b^\circ\,.
\]
From~\eqref{EQ:Iadec_star} we obtain that for each subset~$M$ of~$\wh{I}_b^\circ$ 
also
\[
 p^{-n}\act M \subseteq \wh{I}_b^\circ 
\]
for all~$n\in\N_0$, and hence in particular
\[
 p^{-n}g_p^{-1}\act W \subseteq \wh{I}_b^\circ\,.
\]
Thus, for the requested compact set~$K_{a,b,p}$ we may pick~$\wh{I}_b$. As before, 
using the finiteness of~$\Index$ and~$P_{a,b}$, we can choose the open 
neighborhood~$W$ uniformly for all~$b\in\Index$ and all~$p\in P_{a,b}$. The 
fourth, and final condition on the thresholds~$\eta_a^x$ and $\eta_a^y$ is 
\[
 (x_a - \eta_a^x, y_a + \eta_a^y) \subseteq W\,,
\]
which can obviously be satisfied. 

We immediately check that the upper bounds for any of the 
thresholds~$\eta_a^x$, $\eta_a^y$ for any~$a\in\Index$ are independent among 
each other. This completes the proof.
\end{proof}

\begin{proof}[Proof of Proposition~\ref{prop:prop5}]
We pick a family~$\bigl( (\eps_a^x,\eps_a^y) \bigr)_{a\in\Index}$ satisfying 
all properties stated in Lemma~\ref{LEM:prep_prop5} and set, as in 
Lemma~\ref{LEM:prep_prop5}, 
\[
 E_a \coloneqq (x_a-\eps_a^x, y_a + \eps_a^y)
\]
for all $a\in\Index$. In what follows we will show that by possibly 
shrinking $\eps_a^x$ and~$\eps_a^y$ for some~$a\in\Index$ and allowing 
interdependencies among the elements of the family~$\bigl( (\eps_a^x,\eps_a^y) 
\bigr)_{a\in\Index}$, we can guarantee that 
Lemma~\ref{LEM:prep_prop5}\eqref{eq:prepprop5_space} is also valid in the case 
that $g^{-1}\act \wh{I}_a\not\subseteq \wh{I}_b^\circ$ and 
Lemma~\ref{LEM:prep_prop5}\eqref{eq:prepprop5_compact} is also valid if 
$g_p^{-1}\act \wh{I}_a\not\subseteq \wh{I}_b^\circ$. Taking advantage of the discussion 
in the proof of Lemma~\ref{LEM:prep_prop5} we see that both cases can be 
subsumed to the situation that there exist $a,b\in\Index$ and 
$g\in\Trans{acc}{b}{a}^{-1}$ such that 
\[
 g^{-1}\act \wh{I}_a \subseteq \wh{I}_b
\]
and the two sets $g^{-1}\act \wh{I}_a$ and~$\wh{I}_b$ have a common boundary point. It 
suffices to show that we can fix $\eps_a^x, \eps_a^y, \eps_b^x, \eps_b^y>0$ 
such that 
\[
 g^{-1}\act (x_a-\eps_a^x, y_a+\eps_a^y) \subseteq (x_b-\eps_b^x, 
y_b+\eps_b^y)\,.
\]
This is clearly possible for any ``local'' consideration, i.e., for fixed 
$a,b\in \Index$ and $g\in\Trans{acc}{b}{a}^{-1}$. We need to show that global 
choices are possible. 

To that end we note that any such pair of sets~$g^{-1}\act \wh{I}_a$ and~$\wh{I}_b$ has 
a single common boundary point, not two common boundary points, 
by~\eqref{BPacc:disjointunion}. For the boundary point of~$\wh{I}_a$, say~$z_a$, for 
which $g^{-1}\act z_a$ is contained in the interior of~$\wh{I}_b$ we may and shall 
suppose that the threshold~$\eta_a^z$ is chosen sufficiently small such that 
$g^{-1}\act (z_a\pm\eta_a^z)$ is also contained in~$\wh{I}_b$. (We may restrict here 
to either $+$ or $-$ as needed. However, we may also require this property for 
both signs.) We now consider the joint boundary point of~$\wh{I}_b$ and~$g^{-1}\act 
\wh{I}_a$. Without loss of generality, we may suppose that it is~$y_b$. 
Since~$\Gamma$-action preserves orientation, the corresponding boundary point 
of~$\wh{I}_a$ is then~$y_a$, hence 
\[
 g^{-1}\act y_a = y_b\,.
\]
Then we need to pick~$\eps_a^y>0$ such that 
\[
 g^{-1}\act (y_a + \eps_a^y) < y_b + \eps_b^y\,.
\]
Thus, the threshold for admissible choices for~$\eps_a^y$ depends on the value 
of~$\eps_b^y$. It might happen that there is $c\in\Index$ and $h\in 
\Trans{acc}{a}{c}^{-1}$ such that 
\[
 h^{-1}\act \wh{I}_c \subseteq \wh{I}_a\qquad\text{and}\qquad h^{-1}\act y_c = y_a\,.
\]
In other words, the dependency situation reappears for~$\wh{I}_a$ and we observe  
a tower of dependency situations: As the threshold for~$\eps_c^y$ depends on 
the choice of~$\eps_a^y$, and the threshold for~$\eps_a^y$ depends on the 
choice of~$\eps_b^y$, the threshold for~$\eps_c^y$ ultimately depends on the 
choice of~$\eps_b^y$. As long as $a$, $b$ and~$c$ are three distinct elements 
of~$\Index$, we can easily garantee admissible choices for~$\eps_b^y$, 
$\eps_a^y$ and~$\eps_c^y$ by picking them in this very order. 

We may organize all occuring dependencies (i.e., considering all elements 
of~$\Index$ and all situations of coinciding boundary points simultaneously) as 
a directed graph with 
\[
\defset{\eps_a^x, \eps_a^y}{a\in\Index}
\]
as set of vertices. As soon as this dependency graph has loops, unsolvable 
situations may occur. We will now show that the graph is loop-free. 

To that end we assume, in order to seek a contradiction, that there exists a 
finite sequence 
\[
 a_1\,,\ \ldots\,,\ a_{n+1} \in \Index
\]
with $a_{n+1} = a_1$, and 
\[
 g_j\in \Trans{\acc}{a_j}{a_{j+1}} \qquad\text{for $j\in\{1,\ldots, n\}$}
\]
such that 
\[
 y_{a_j} = g_j\act y_{a_{j+1}} \qquad\text{for $j\in\{1,\ldots, n\}$}\,.
\]
Then
\[
 y_{a_1} = g_1\cdots g_n\act y_{a_{n+1}} = q\act y_{a_1}
\]
with 
\[
 q \coloneqq g_1 \cdots g_n\,.
\]
As the boundary points of~$\wh{I}_{a_1}$ are not hyperbolic fixed points (see 
Lemma~\ref{LEM:accprops}\eqref{LEM:accprops:endpoints}), 
the element~$q$ is parabolic and $y_{a_1}$ a parabolic fixed point. In turn, 
$P_{a_2,a_1}\not=\varnothing$, say $P_{a_2,a_1} = \{p\}$, and $y_{a_1}$ is the 
fixed point of~$p$ and $g_p^{-1} = g_1$. However, $g_p^{-1}\act \wh{I}_{a_2} = 
g_1\act \wh{I}_{a_2}$ contains~$y_{a_1}$, which contradicts 
Lemma~\ref{LEM:prep_prop5}\eqref{eq:prepprop5_fixed}. Thus, the dependency 
graph has no loops. 

Finally, for each $a\in\Index$, we let $\mc E_a$ be the complex euclidean circle spanned by the real interval~$E_a$, i.e., $\mc E_a$ is the unique euclidean circle in~$\C$ with center in~$\R$ such that $\R\cap \mc E_a = E_a$. The family~$(\mc E_a)_{a\in \Index}$ has all requested properties.
\end{proof}

\section{A detailed example}\label{SEC:longexample}

In this section we carry out the algorithms of branch reduction, identity elimination and cuspidal acceleration for the hyperbolic orbisurface~$\Orbi\coloneqq\quod{\Gamma}{\H}$, where $\Gamma$ is the projective variant of the Hecke congruence subgroup~$\Gamma_0(7)$. Thus, 
\begin{align*}
 \Gamma & = \defset{ g\in\PSL_2(\Z) }{ g \equiv \bmat{1}{*}{0}{1} \mod 7 }
 \\
 & = \defset{  \bmat{a}{b}{c}{d} \in \PSL_2(\Z) }{ \mat{a}{b}{c}{d} \equiv \mat{1}{*}{0}{1} \mod 7 }\,.
\end{align*}
Besides many details we present also a choice of the family~$(\mc E_a)_{a\in \Index}$ of open, bounded, connected and simply connected sets in~$\widehat{\C}$ as requested in Property~\ref{staPROP5} of a structure tuple. 

At first glance, this example might seem to be rather overly complicated. However, it is one of the first examples (``first'' with respect to complexity) in which the full involvement of the cuspidal acceleration algorithm can be observed.

\subsection{Setup}

In this section we collect some useful background information on the hyperbolic orbisurface~$\Orbi = \quod{\Gamma}{\H}$. Further, we present the initial set of branches for the geodesic flow on~$\Orbi$ from which we will proceed and carry out all algorithms. 

A Ford fundamental domain~$\funddom$ for~$\Orbi$ or, equivalently, for the action of~$\Gamma$ on~$\h$ is given by 
\begin{equation}\label{eq:funddom07}
 \funddom \coloneqq \defset{ z\in\h }{ \Rea z\in (0,1),\ \left| z-\tfrac{k}{7}\right| > 1\ \text{for $k\in\{1,2,\ldots, 6\}$} }\,.
\end{equation}
See Figure~\ref{fig:funddomG07}. 
\bigskip
\begin{figure}[h]
\centering
\includegraphics[width=.95\textwidth]{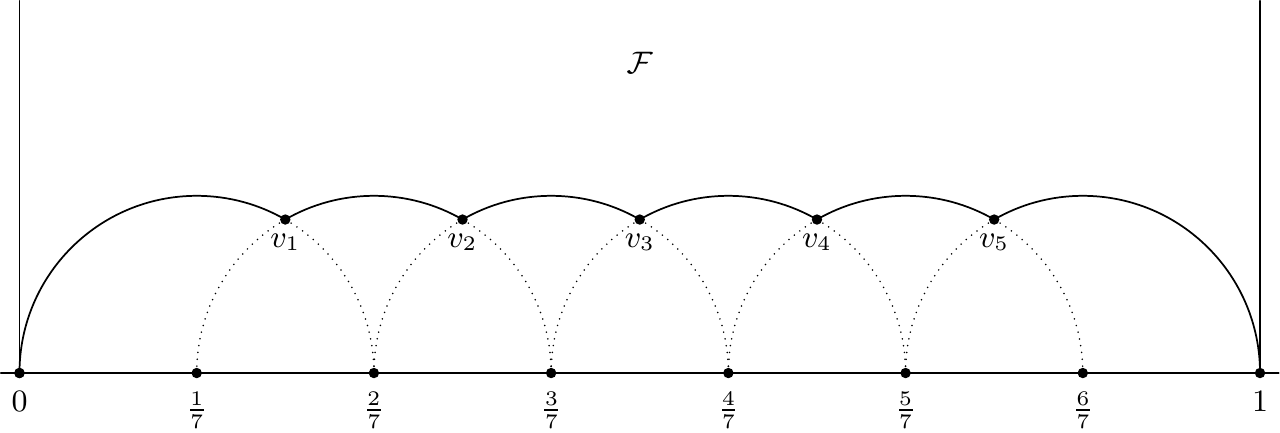}
\caption{The fundamental domain~$\funddom$ for~$\Gamma$.}\label{fig:funddomG07}
\end{figure}
\bigskip
The side-pairing elements are
\begin{align*}
 T & \coloneqq \bmat{1}{1}{0}{1}\,,\qquad & h_1 &\coloneqq \bmat{-6}{1}{-7}{1}\,,
 \\[3mm]
 h_2 & \coloneqq \bmat{-3}{1}{-7}{2}\,, & h_3 & \coloneqq \bmat{-5}{3}{-7}{4}\,.
\end{align*}
The side-pairing is as indicated in Figure~\ref{fig:sidepairingG07}: 
\begin{itemize}
\item The hyperbolic interval~$[0,\infty]_{\H}$ is mapped to~$[1,\infty]_{\H}$ by~$T$ (with $T\act 0 = 1$ and $T\act\infty=\infty$). 
\item The hyperbolic interval~$[0,v_1]_{\H}$ is mapped to~$[1,v_5]_{\H}$ by~$h_1$ (with $h_1\act 0=1$ and $h_1\act v_1 = v_5$).
\item The hyperbolic interval~$[v_1,v_2]_{\H}$ is mapped to~$[v_3,v_2]_{\H}$ by~$h_2$ (with $h_2\act v_1 = v_3$ and $h_2\act v_2 = v_2$).
\item The hyperbolic interval~$[v_3,v_4]_{\H}$ is mapped to~$[v_5,v_4]_{\H}$ by~$h_3$ (with $h_3\act v_3 = v_5$ and $h_3\act v_4 = v_4$).
\end{itemize}
\bigskip
\begin{figure}[h]
\centering
\includegraphics[width=.95\textwidth]{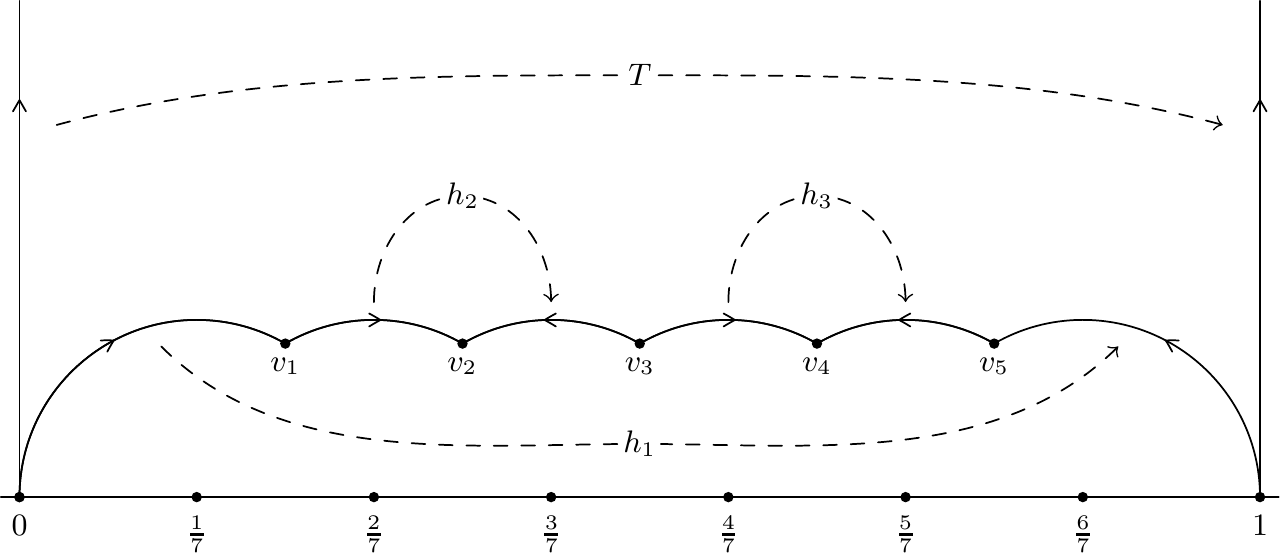}
\caption{The fundamental domain~$\funddom$ for~$\Gamma$.}\label{fig:sidepairingG07}
\end{figure}
\bigskip
From the side-pairings it can immediately be deduced that~$\Orbi$ has two cusps (represented by~$\infty$ and~$0$) and two elliptic points (represented by~$v_2$ and~$v_4$), both of order~$3$. 

The discretization algorithm in~\cite{Pohl_Symdyn2d} (see also~\cite{Pohl_diss}) provides several choices of sets of branches for the geodesic flow on~$\Orbi$. Taking advantage of the fundamental domain~$\funddom$ as given in~\eqref{eq:funddom07}, this algorithm initially constructs the set of branches
\[
\wt\BrS \coloneqq\{ \wt{\mathrm{C}}_1, \ldots, \wt{\mathrm{C}}_8 \}
\]
as indicated in Figure~\ref{fig:cs1_next}. (The seemingly ``chaotic'' enumeration is due to the algorithm. It is also convenient for our purposes.)
This set of branches is weakly non-collapsing, but not admissible. Therefore, for the purpose of our applications, we switch to an admissible, weakly non-collapsing set of branches by shifting the branch~$\wt{\mathrm{C}}_2$ by~$T^{-2}$. Our initial set of branches becomes
\[
\BrS \coloneqq \{\Cs 1,\ldots, \Cs 8\}
\]
as indicated in Figure~\ref{fig:cs3_next}. Thus, 
\[
\Cs 1  \coloneqq \defset{ v\in\UTB\H }{ \Rea\base v = 0\,,\ (\gamma_v(+\infty),\gamma_v(-\infty)) \in I_1 \times J_1  }
\]
with 
\[
I_1 \coloneqq (0,+\infty) \qquad\text{and}\qquad J_1 \coloneqq (-\infty,0)\,, 
\]
and
\[
\Cs 2  \coloneqq \defset{ v\in\UTB\H }{ \Rea\base v = -1\,,\ (\gamma_v(+\infty),\gamma_v(-\infty)) \in I_2 \times J_2  }
\]
with 
\[
I_2 \coloneqq (-\infty, -1) \qquad\text{and}\qquad J_2 \coloneqq (-1,+\infty)\,, 
\]
and 
\[
\Cs 3  \coloneqq \defset{ v\in\UTB\H }{ \Rea\base v = \tfrac67\,,\ (\gamma_v(+\infty),\gamma_v(-\infty)) \in I_3 \times J_3  }
\]
with 
\[
I_3 \coloneqq \left(\tfrac67,+\infty\right) \qquad\text{and}\qquad J_3 \coloneqq \left(-\infty,\tfrac67\right)\,, 
\]
and 
\[
\Cs 4  \coloneqq \defset{ v\in\UTB\H }{ \Rea\base v = \tfrac17\,,\ (\gamma_v(+\infty),\gamma_v(-\infty)) \in I_4 \times J_4  }
\]
with 
\[
I_4 \coloneqq \left(\tfrac17,+\infty\right) \qquad\text{and}\qquad J_4 \coloneqq \left(-\infty,\tfrac17\right)\,, 
\]
and 
\[
\Cs 5  \coloneqq \defset{ v\in\UTB\H }{ \Rea\base v = \tfrac37\,,\ (\gamma_v(+\infty),\gamma_v(-\infty)) \in I_5 \times J_5  }
\]
with 
\[
I_5 \coloneqq \left(\tfrac37,+\infty\right) \qquad\text{and}\qquad J_5 \coloneqq \left(-\infty,\tfrac37\right)\,, 
\]
and 
\[
\Cs 6  \coloneqq \defset{ v\in\UTB\H }{ \Rea\base v = \tfrac27\,,\ (\gamma_v(+\infty),\gamma_v(-\infty)) \in I_6 \times J_6  }
\]
with 
\[
I_6 \coloneqq \left(\tfrac27,+\infty\right) \qquad\text{and}\qquad J_6 \coloneqq \left(-\infty,\tfrac27\right)\,, 
\]
and 
\[
\Cs 7  \coloneqq \defset{ v\in\UTB\H }{ \Rea\base v = \tfrac57\,,\ (\gamma_v(+\infty),\gamma_v(-\infty)) \in I_7 \times J_7  }
\]
with 
\[
I_7 \coloneqq \left(\tfrac57,+\infty\right) \qquad\text{and}\qquad J_7 \coloneqq \left(-\infty,\tfrac57\right)\,,
\]
and 
\[
\Cs 8  \coloneqq \defset{ v\in\UTB\H }{ \Rea\base v = \tfrac47\,,\ (\gamma_v(+\infty),\gamma_v(-\infty)) \in I_8 \times J_8  }
\]
with 
\[
I_8 \coloneqq \left(\tfrac47,+\infty\right) \qquad\text{and}\qquad J_8 \coloneqq \left(-\infty,\tfrac47\right)\,.
\]
We set $A\coloneqq \{1,\ldots, 8\}$. The forward transition sets can be read off from Figure~\ref{fig:cs3_next}. E.g., 
\[
\Trans{}{1}{1}=\{h_1^{-1}T\}\,,\quad \Trans{}{1}{4} = \{\id\}\,,\quad \Trans{}{1}{6} = \varnothing\,.
\]
One easily checks that~$\BrS$ is indeed admissible and weakly non-collapsing. Further, we immediately see that~$\BrS$ is finitely ramified.

\subsection{Application of branch reduction algorithm}\label{SUBSEC:applbr}

The return graph~$\ReturnGraph{0}$ of level~$0$ is indicated in Figure~\ref{fig:RG0}. It can easily be read off from Figure~\ref{fig:cs3_next}.
\begin{figure}[h]
\begin{tikzpicture}[->,shorten >=1pt,auto, semithick]
\def \fak {7.5 em}
\begin{scope}[on grid]
\node[state] (C1) {$1$};
\node[state] (C3) [above left= .78183*\fak and (1-.62349)*\fak of C1] {$3$};
\node[state] (C7) [above left= .97493*\fak and 1.22252*\fak of C1] {$7$};
\node[state] (C8) [above left= .43388*\fak and 1.90097*\fak of C1] {$8$};
\node[state] (C5) [below left= .43388*\fak and 1.90097*\fak of C1] {$5$};
\node[state] (C6) [below left= .97493*\fak and 1.22252*\fak of C1] {$6$};
\node[state] (C4) [below left=.78183*\fak and (1-.62349)*\fak of C1] {$4$};
\node[state] (C2) [left=(-1-.24698)*\fak of C1] {$2$};
\def \los {.7}
\path (C1) edge [bend left,looseness=\los] node {$\id$} (C4)
		  (C1) edge [loop right,out=30,in=330,looseness=10] node[pos=.3,above=3pt] {$h_1^{-1}T$} (C1)
		  (C4) edge [bend left,looseness=\los] node {$\id$} (C6)
		  (C4) edge [bend right,looseness=\los] node[pos=.3,below=5pt] {$h_2^{-1}$} (C8)
		  (C6) edge [bend left,looseness=\los] node {$\id$} (C5)
		  (C6) edge [bend right,looseness=\los] node[pos=.3,above=5pt] {$h_2$} (C5)
		  (C5) edge [bend left,looseness=\los] node {$\id$} (C8)
		  (C5) edge [bend right,looseness=\los] node[pos=.3,above=1pt] {$h_3^{-1}$} (C3)
		  (C8) edge [bend left,looseness=\los] node {$\id$} (C7)
		  (C8) edge [bend right,looseness=\los] node[pos=.7,below=5pt] {$h_3$} (C7)
		  (C7) edge [bend left,looseness=\los] node {$\id$} (C3)
		  (C7) edge [bend left,looseness=\los] node[pos=.3,right=1pt] {$h_1$} (C6)
		  (C3) edge [bend left,looseness=\los] node {$T$} (C1)
		  (C2) edge [loop right,out=30,in=330,looseness=10] node[pos=.3,above=3pt] {$T^{-2}h_1T$} (C2);
\draw[->] (C3.north east) arc [radius=\fak, start angle=120, end angle= 7.5] node[pos=.5,below=7pt] {$h_1T$};
\draw[<-] (C4.south east) arc [radius=\fak, start angle=-120, end angle= -7.5] node[pos=.5,above=7pt] {$T^{-2}h_1$};
\end{scope}
\end{tikzpicture}
\caption{The return graph~$\ReturnGraph{0}$.}\label{fig:RG0}
\end{figure}
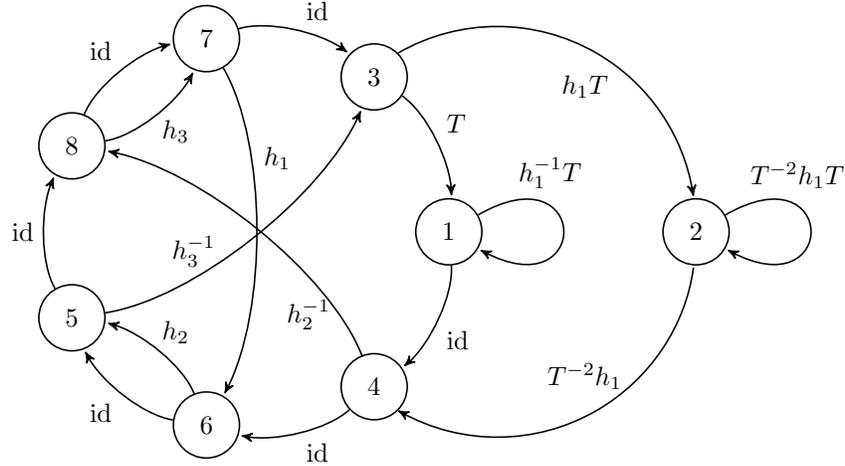
Algorithm~\ref{nodereductionI} is void in our situation. Thus, $\kappa_1=0$.
For Algorithm~\ref{nodereductionII} we pick (in this order) $5, 8, 6, 3, 4$. Then $\kappa_2=5$, 
\[
A_5 = \{1,2,7\}
\]
and 
\begin{align*}
\Trans{5}{1}{1} & = \{ h_1^{-1}T,\ h_3^{-1}T,\ h_2h_3^{-1}T \}  
\\
\Trans{5}{1}{2} & = \{ h_3^{-1}h_1T,\ h_2h_3^{-1}h_1T \}
\\
\Trans{5}{1}{7} & = \{\id,\ h_2,\ h_3,\ h_2^{-1},\ h_2h_3,\ h_2^{-1}h_3 \} 
\\
\Trans{5}{2}{1} & = \{ T^{-2}h_1h_3^{-1}T,\ T^{-2}h_1h_2h_3^{-1}T \}
\\
\Trans{5}{2}{2} & = \{ T^{-2}h_1T,\ T^{-2}h_1h_3^{-1}h_1T\, T^{-2}h_1h_2h_3^{-1}h_1T \}
\\
\Trans{5}{2}{7} & = \{ T^{-2}h_1,\ T^{-2}h_1h_2^{-1},\ T^{-2}h_1h_2,\ T^{-2}h_1h_3,\ T^{-2}h_1h_2h_3,\ T^{-2}h_1h_2^{-1}h_3 \}
\\
\Trans{5}{7}{1} & = \{ T,\ h_1h_3^{-1}T,\ h_1h_2h_3^{-1}T \}
\\
\Trans{5}{7}{2} & = \{ h_1T,\ h_1h_3^{-1}h_1T,\ h_1h_2h_3^{-1}h_1T \}
\\
\Trans{5}{7}{7} & = \{ h_1,\ h_1h_2,\ h_1h_3,\ h_1h_2h_3 \}\,.
\end{align*}
The return graph~$\ReturnGraph{5}$ is indicated in Figure~\ref{fig:RG5}. 
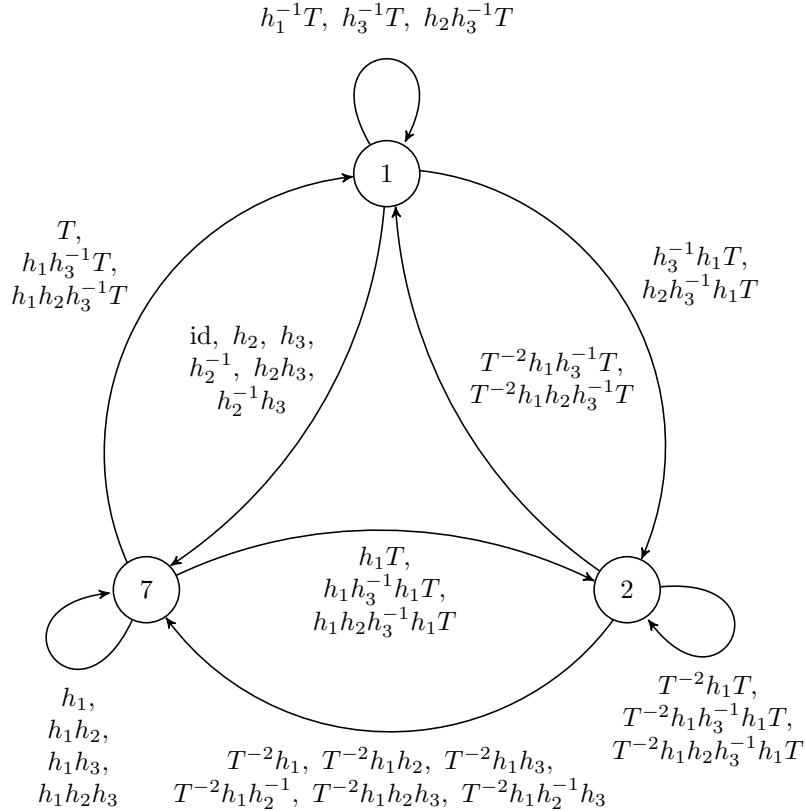
\begin{figure}[h]
\begin{tikzpicture}[->,shorten >=1pt,auto, semithick]
\def \fak {10.5 em}
\def \fakk {\fak*sqrt(3)}
\begin{scope}[on grid]
\node[state] (C1) {$1$};
\node[state] (C7) [below left= 1.5*\fak and .86603*\fak of C1] {$7$};
\node[state] (C2) [below right= 1.5*\fak and .86603*\fak of C1] {$2$};
\draw[<-] (C2.65) arc [radius=\fak, start angle=-24.02084, end angle=84.02084] node[pos=.5,above right=-5pt] {$\begin{array}{c}h_3^{-1}h_1T,\\h_2h_3^{-1}h_1T\end{array}$};
\draw[<-] (C7.305) arc [radius=\fak, start angle=215.97916, end angle=324.02084] node[pos=.5,below=2pt] {$\begin{array}{c}T^{-2}h_1,~T^{-2}h_1h_2,~T^{-2}h_1h_3,\\T^{-2}h_1h_2^{-1},~T^{-2}h_1h_2h_3,~T^{-2}h_1h_2^{-1}h_3\end{array}$};
\draw[<-] (C1.185) arc [radius=\fak, start angle=95.97916, end angle=204.02084] node[pos=.5,above left=-5pt] {$\begin{array}{c}T,\\h_1h_3^{-1}T,\\h_1h_2h_3^{-1}T\end{array}$};
\draw[<-] (C2.165) arc [radius=\fakk, start angle=63.44789, end angle=115.33211] node[pos=.5, below=1pt] {$\begin{array}{c}h_1T,\\h_1h_3^{-1}h_1T,\\h_1h_2h_3^{-1}h_1T\end{array}$};
\draw[<-] (C1.285) arc [radius=\fakk, start angle=183.44789, end angle=235.33211] node[pos=.5, above right=-5pt] {$\begin{array}{c}T^{-2}h_1h_3^{-1}T,\\T^{-2}h_1h_2h_3^{-1}T\end{array}$};
\draw[<-] (C7.45) arc [radius=\fakk, start angle=303.44789, end angle=355.33211] node[pos=.5, above left=-10pt] {$\begin{array}{c}\id,~h_2,~h_3,\\h_2^{-1},~h_2h_3,\\h_2^{-1}h_3\end{array}$};
\path  (C2) edge [loop right,out=5,in=-55,looseness=10] node[pos=.7,below=2pt] {$\begin{array}{c}T^{-2}h_1T,\\T^{-2}h_1h_3^{-1}h_1T,\\T^{-2}h_1h_2h_3^{-1}h_1T\end{array}$} (C2)
(C7) edge [loop right,out=245,in=185,looseness=10] node[pos=.3,below=2pt] {$\begin{array}{c}h_1,\\h_1h_2,\\h_1h_3,\\h_1h_2h_3\end{array}$} (C7)
(C1) edge [loop left,out=120,in=60,looseness=10] node[pos=.5,above=5pt] {$\begin{array}{c}h_1^{-1}T,~h_3^{-1}T,~h_2h_3^{-1}T\end{array}$} (C1);
\end{scope}
\end{tikzpicture}
\caption{The return graph~$\ReturnGraph{5}$ of level~$5$. Multiple edges are indicated by multiple weights.}\label{fig:RG5}
\end{figure}
The associated discrete dynamical system~$(D^{(br)},F^{(br)})$ is defined on the domain (only ``strict'' part; $(br)$ stands for branch reduction) 
\[
 D^{(br)} \coloneqq \Bigl( \bigl(0,\infty\bigr)_{\st}\times\{1\} \Bigr) \cup \Bigl( \bigl(\tfrac57,\infty\bigr)_{\st} \times\{7\} \Bigr) \cup \Bigl( \bigl(-\infty,-1\bigr)_{\st}\times\{2\}\Bigr)
\]
and given by the following submaps:
\begin{align*}
 \bigl(0,\tfrac17\bigr)_{\st}\times\{1\} & \stackrel{\sim}{\longrightarrow} \bigl(0,\infty\bigr)_{\st}\times \{1\}\,, \quad &\bigl(x,1\bigr) & \mapsto \bigl( T^{-1}h_1.x,1\bigr)\,,
 \\[2mm]
 \bigl(\tfrac17,\tfrac{3}{14}\bigr)_{\st}\times\{1\} & \stackrel{\sim}{\longrightarrow} \bigl(\tfrac57,\infty\bigr)_{\st}\times\{7\}\,, & \bigl(x,1\bigr) &\mapsto \bigl(h_3^{-1}h_2.x,7\bigr)\,,
 \\[2mm]
 \bigl(\tfrac{3}{14},\tfrac27\bigr)_{\st}\times\{1\} & \stackrel{\sim}{\longrightarrow} \bigl(\tfrac57,\infty\bigr)_{\st}\times\{7\}\,, & \bigl(x,1\bigr) &\mapsto \bigl(h_2.x,7\bigr)\,,
 \\[2mm]
 \bigl(\tfrac27,\tfrac13\bigr)_{\st}\times\{1\} & \stackrel{\sim}{\longrightarrow} \bigl(-\infty,-1\bigr)_{\st}\times\{2\}\,, & \bigl(x,1\bigr) &\mapsto \bigl(T^{-1}h_1^{-1}h_3h_2^{-1}.x,2\bigr)\,,
 \\[2mm]
 \bigl(\tfrac13,\tfrac{5}{14}\bigr)_{\st}\times\{1\} & \stackrel{\sim}{\longrightarrow} \bigl(0,\infty\bigr)_{\st}\times\{1\}\,, & \bigl(x,1\bigr) &\mapsto \bigl(T^{-1}h_3h_2^{-1}.x,1\bigr)\,,
 \\[2mm]
 \bigl(\tfrac{5}{14},\tfrac{8}{21}\bigr)_{\st}\times\{1\} & \stackrel{\sim}{\longrightarrow} \bigl(\tfrac57,\infty\bigr)_{\st}\times\{7\}\,, & \bigl(x,1\bigr) &\mapsto \bigl(h_3^{-1}h_2^{-1}.x,7\bigr)\,,
 \\[2mm]
 \bigl(\tfrac{8}{21},\tfrac37\bigr)_{\st}\times\{1\} & \stackrel{\sim}{\longrightarrow} \bigl(\tfrac57,\infty\bigr)_{\st}\times\{7\}\,, & \bigl(x,1\bigr) &\mapsto \bigl(h_2^{-1}.x,7\bigr)\,,
 \\[2mm]
 \bigl(\tfrac37,\tfrac12\bigr)_{\st}\times\{1\} & \stackrel{\sim}{\longrightarrow} \bigl(-\infty,-1\bigr)_{\st}\times\{2\}\,, & \bigl(x,1\bigr) &\mapsto \bigl(T^{-1}h_1^{-1}h_3.x,2\bigr)\,,
 \\[2mm]
 \bigl(\tfrac12,\tfrac47\bigr)_{\st}\times\{1\} & \stackrel{\sim}{\longrightarrow} \bigl(0,\infty\bigr)_{\st}\times\{1\}\,, & \bigl(x,1\bigr) &\mapsto \bigl(T^{-1}h_3.x,1\bigr)\,,
 \\[2mm]
 \bigl(\tfrac47,\tfrac57\bigr)_{\st}\times\{1\} & \stackrel{\sim}{\longrightarrow} \bigl(\tfrac57,\infty\bigr)_{\st}\times\{7\}\,, & \bigl(x,1\bigr) &\mapsto \bigl(h_3^{-1}.x,7\bigr)\,,
 \\[2mm]
 \bigl(\tfrac57,\infty\bigr)_{\st}\times\{1\} & \stackrel{\sim}{\longrightarrow} \bigl(\tfrac57,\infty\bigr)_{\st}\times\{7\}\,, & \bigl(x,1\bigr) &\mapsto \bigl(x,7\bigr)\,,
 \\[2mm]
 \bigl(\tfrac57,\tfrac34\bigr)_{\st}\times\{7\} & \stackrel{\sim}{\longrightarrow} \bigl( -\infty, -1\bigr)_{\st}\times\{ 2\}\,, & \bigl(x,7\bigr) &\mapsto \bigl(T^{-1}h_1^{-1}h_3h_2^{-1}h_1^{-1}.x,2\bigr)\,,
 \\[2mm]
 \bigl(\tfrac34,\tfrac{16}{21}\bigr)_{\st}\times\{7\} & \stackrel{\sim}{\longrightarrow} \bigl(0,\infty\bigr)_{\st}\times\{1\}\,, & \bigl(x,7\bigr) &\mapsto \bigl(T^{-1}h_3h_2^{-1}h_1^{-1}.x,1\bigr)\,,
 \\[2mm]
 \bigl(\tfrac{16}{21},\tfrac{27}{35}\bigr)_{\st}\times\{7\} & \stackrel{\sim}{\longrightarrow} \bigl(\tfrac57,\infty\bigr)_{\st}\times\{7\}\,, & \bigl(x,7\bigr) &\mapsto \bigl(h_3^{-1}h_2^{-1}h_1^{-1}.x,7\bigr)\,,
 \\[2mm]
 \bigl(\tfrac{27}{35},\tfrac{11}{14}\bigr)_{\st}\times\{7\} & \stackrel{\sim}{\longrightarrow} \bigl(\tfrac57,\infty\bigr)_{\st}\times\{7\}\,, & \bigl(x,7\bigr) &\mapsto \bigl(h_2^{-1}h_1^{-1}.x,7\bigr)\,,
 \\[2mm]
 \bigl(\tfrac{11}{14},\tfrac45\bigr)_{\st}\times\{7\} & \stackrel{\sim}{\longrightarrow} \bigl(-\infty,-1\bigr)_{\st}\times\{2\}\,, & \bigl(x,7\bigr) &\mapsto \bigl(T^{-1}h_1^{-1}h_3h_1^{-1},2\bigr)\,,
 \\[2mm]
 \bigl(\tfrac45,\tfrac{17}{21}\bigr)_{\st}\times\{7\} & \stackrel{\sim}{\longrightarrow} \bigl(0,\infty\bigr)_{\st}\times\{1\}\,, & \bigl(x,7\bigr) &\mapsto \bigl(T^{-1}h_3h_1^{-1},1\bigr)\,,
 \\[2mm]
 \bigl(\tfrac{17}{21},\tfrac{23}{28}\bigr)_{\st}\times\{7\} & \stackrel{\sim}{\longrightarrow} \bigl(\tfrac57,\infty\bigr)_{\st}\times\{7\}\,, & \bigl(x,7\bigr) &\mapsto \bigl(h_3^{-1}h_1^{-1}.x,7\bigr)\,,
 \\[2mm]
 \bigl(\tfrac{23}{28},\tfrac67\bigr)_{\st}\times\{7\} & \stackrel{\sim}{\longrightarrow} \bigl(\tfrac57,\infty\bigr)_{\st}\times\{7\}\,, & \bigl(x,7\bigr) &\mapsto \bigl(h_1^{-1}.x,7\bigr)\,,
 \\[2mm]
 \bigl(\tfrac{6}{7},1\bigr)_{\st}\times\{7\} & \stackrel{\sim}{\longrightarrow} \bigl(-\infty,-1\bigr)_{\st}\times\{2\}\,, & \bigl(x,7\bigr) &\mapsto \bigl(T^{-1}h_1^{-1}.x,2\bigr)\,,
 \\[2mm]
 \bigl(1,\infty\bigr)_{\st}\times\{7\} & \stackrel{\sim}{\longrightarrow} \bigl(0,\infty\bigr)_{\st}\times\{1\}\,, & \bigl(x,7\bigr) &\mapsto \bigl(T^{-1}.x,1\bigr)\,,
 \\[2mm]
 \bigl(-\infty,-\tfrac{10}{7}\bigr)_{\st}\times\{2\} & \stackrel{\sim}{\longrightarrow} \bigl(\tfrac57,\infty\bigr)_{\st}\times\{7\}\,, & \bigl(x,2\bigr) &\mapsto \bigl(h_3^{-1}h_2h_1^{-1}T^2.x,7\bigr)\,,
 \\[2mm]
 \bigl(-\tfrac{10}{7},-\tfrac{9}{7}\bigr)_{\st}\times\{2\} & \stackrel{\sim}{\longrightarrow} \bigl(\tfrac57,\infty\bigr)_{\st}\times\{7\}\,, & \bigl(x,2\bigr) &\mapsto \bigl(h_2h_1^{-1}T^2.x,7\bigr)\,,
 \\[2mm]
 \bigl(-\tfrac97,-\tfrac54\bigr)_{\st}\times\{2\} & \stackrel{\sim}{\longrightarrow} \bigl(-\infty,-1\bigr)_{\st}\times\{2\}\,, & \bigl(x,2\bigr) &\mapsto \bigl(T^{-1}h_1^{-1}h_3h_2^{-1}h_1^{-1}T^2.x,2\bigr)\,,
 \\[2mm]
 \bigl(-\tfrac54,-\tfrac{26}{21}\bigr)_{\st}\times\{2\} & \stackrel{\sim}{\longrightarrow} \bigl(0,\infty\bigr)_{\st}\times\{1\}\,, & \bigl(x,2\bigr) &\mapsto \bigl(T^{-1}h_3h_2^{-1}h_1^{-1}T^2.x,1\bigr)\,,
 \\[2mm]
 \bigl(-\tfrac{26}{21},-\tfrac{43}{35}\bigr)_{\st}\times\{2\} & \stackrel{\sim}{\longrightarrow} \bigl(\tfrac57,\infty\bigr)_{\st}\times\{7\}\,, & \bigl(x,2\bigr) &\mapsto \bigl(h_3^{-1}h_2^{-1}h_1^{-1}T^2.x,7\bigr)\,,
 \\[2mm]
 \bigl(-\tfrac{43}{35},-\tfrac{17}{14}\bigr)_{\st}\times\{2\} & \stackrel{\sim}{\longrightarrow} \bigl(\tfrac57,\infty\bigr)_{\st}\times\{7\}\,, & \bigl(x,2\bigr) &\mapsto \bigl(h_2^{-1}h_1^{-1}T^2.x,7\bigr)\,,
 \\[2mm]
 \bigl(-\tfrac{17}{14},-\tfrac65\bigr)_{\st}\times\{2\} & \stackrel{\sim}{\longrightarrow} \bigl(-\infty,-1\bigr)_{\st}\times\{2\}\,, & \bigl(x,2\bigr) &\mapsto \bigl(T^{-1}h_1^{-1}h_3h_1^{-1}T^2.x,2\bigr)\,,
 \\[2mm]
 \bigl(-\tfrac65,-\tfrac{25}{21}\bigr)_{\st}\times\{2\} & \stackrel{\sim}{\longrightarrow} \bigl(0,\infty\bigr)_{\st}\times\{1\}\,, & \bigl(x,2\bigr) &\mapsto \bigl(T^{-1}h_3h_1^{-1}T^2.x,1\bigr)\,,
 \\[2mm]
 \bigl(-\tfrac{25}{21},-\tfrac{33}{28}\bigr)_{\st}\times\{2\} & \stackrel{\sim}{\longrightarrow} \bigl(\tfrac57,\infty\bigr)_{\st}\times\{7\}\,, & \bigl(x,2\bigr) &\mapsto \bigl(h_3^{-1}h_1^{-1}T^2.x,7\bigr)\,,
 \\[2mm]
 \bigl(-\tfrac{33}{28},-\tfrac87\bigr)_{\st}\times\{2\} & \stackrel{\sim}{\longrightarrow} \bigl(\tfrac57,\infty\bigr)_{\st}\times\{7\}\,, & \bigl(x,2\bigr) &\mapsto \bigl(h_1^{-1}T^2.x,7\bigr)\,,
 \\[2mm]
 \bigl(-\tfrac87,-1\bigr)_{\st}\times\{2\} & \stackrel{\sim}{\longrightarrow} \bigl(-\infty,-1\bigr)_{\st}\times\{2\}\,, & \bigl(x,2\bigr) &\mapsto \bigl(T^{-1}h_1^{-1}T^2.x,2\bigr)\,.
\end{align*}
The associated transfer operator ${\TO s}^{\!\!\!(br)}$ acts on the function vectors
\[
\begin{pmatrix}
f_a \colon (0,\infty)_{\st} \to \C
\\
f_b \colon \bigl(\tfrac57,\infty\bigr)_{\st}\to\C
\\
f_c \colon (-\infty,-1)_{\st}\to\C
\end{pmatrix}
\]
with $a=1$, $b=7$, $c=2$. It is given by the ``matrix-like'' operator $\bigl( \TO {ij,s}\bigr)_{i,j\in A_5}$ with 
\begin{align*}
\TO {aa,s} & = \tau_s\big(T^{-1}h_1\big) + \tau_s\big(T^{-1}h_3h_2^{-1}\big) + \tau_s\big(T^{-1}h_3\big)\,,
\\
\TO {ab,s} & = \tau_s\big(T^{-1}\big) + \tau_s\big( T^{-1}h_3h_2^{-1}h_1^{-1}\big) + \tau_s\big(T^{-1}h_3h_1^{-1}\big)\,,
\\
\TO {ac,s} & = \tau_s\big(T^{-1}h_3h_1^{-1}T^2\big) + \tau_s\big(T^{-1}h_3h_2^{-1}h_1^{-1}T^2\big)\,,
\\
\TO {ba,s} & = \tau_s\big(\id\big) + \tau_s\big(h_3^{-1}\big) + \tau_s\big(h_2^{-1}\big) + \tau_s\big(h_3^{-1}h_2^{-1}\big) + \tau_s\big(h_2\big) + \tau_s\big(h_3^{-1}h_2\big)\,,
\\
\TO {bb,s} & = \tau_s\big(h_1^{-1}\big) + \tau_s\big( h_3^{-1}h_1^{-1} \big) + \tau_s\big( h_2^{-1}h_1^{-1} \big) + \tau_s\big( h_3^{-1}h_2^{-1}h_1^{-1}\big)\,,
\\
\TO {bc,s} & = \tau_s\big(h_1^{-1}T^2\big) + \tau_s\big( h_3^{-1}h_2^{-1}T^2 \big) + \tau_s\big( h_2^{-1}h_1^{-1}T^2 \big) + \tau_s\big( h_3^{-1}h_2^{-1}h_1^{-1}T^2 \big)
\\
& \quad + \tau_s\big( h_2h_1^{-1}T^2 \big) + \tau_s\big( h_3^{-1}h_2h_1^{-1}T^2 \big)\,,
\\
\TO {ca,s} & = \tau_s\big( T^{-1}h_1^{-1}h_3 \big) + \tau_s\big( T^{-1}h_1^{-1}h_3h_2^{-1} \big)\,,
\\
\TO {cb,s} & = \tau_s\big(T^{-1}h_1^{-1}h_3h_1^{-1} \big) + \tau_s\big( T^{-1}h_1^{-1}h_3h_2^{-1}h_1^{-1}\big) + \tau_s\big(T^{-1}h_1^{-1}\big)\,,
\\
\TO {cc,s} & = \tau_s\big(T^{-1}h_1^{-1}h_3h_1^{-1}T^2 \big) + \tau_s\big(T^{-1}h_1^{-1}h_3h_2^{-1}h_1^{-1}T^2\big) + \tau_s\big(T^{-1}h_1^{-1}T^2\big)\,.
\end{align*}

\subsection{Application of identity elimination}\label{SUBSEC:applie}

As in Section~\ref{SEC:stepredux} we now omit the index~``$5$'' from the notation above and write $\BrS$ instead of~$\BrS_5$, etc. We emphasize that in this section $\BrS$ denotes a different set of branches than the same notation in the Section~\ref{SUBSEC:applbr}. We restrict here to the discussion of the strict cross section and the ``strict'' parts of the intervals to simplify the exposition. 

We have $D_\ini = \{ 1, 2\}$.
The two sub-trees of the branch trees which collect the ``identity chains'' are indicated in Figure~\ref{FIG:subtreesidentity}.
\bigskip
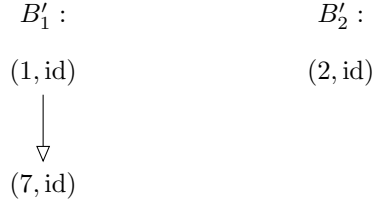
\begin{figure}[h]
\begin{tikzpicture}[level/.style={sibling distance=55mm/#1},edge from parent/.style={draw,-open triangle 45}]
\node (z){$(1,\id)$}
  child {node (a) {$(7,\id)$}
  };
\node [above=1ex of z] {$B'_1:$};   
\node (zz) [right=8em of z] {$(2,\id)$};
\node [above=1ex of zz] {$B'_2:$};
\end{tikzpicture}
\caption[subtreesidentity]{The subtrees~$B_1'$ and $B_2'$ which collect the ``identity chains.''}\label{FIG:subtreesidentity}
\end{figure}
\bigskip 

Thus, 
\[
F_\ini = \{ B_1', B_2' \}\,,\quad \Delta_\ini = \{ (1,7), (2,2) \}\,,\quad
\ell_{(1,7)} = 1\,,\quad \ell_{(2,2)} = 0
\]
and 
\[
a_0^{(1,7)} = 7\,,\quad a_1^{(1,7)} = 7\,,\quad a_0^{(2,2)} = 2\,.
\]
We obtain
\begin{align*}
 \Csr {1,\st} & = \defset{ \nu\in \Cs{1,\st} }{ \gamma_\nu(+\infty) \in \left(0,\tfrac57\right)_{\st} }\,,
 \\
 \Ired{1,\st} & = \Iset{1}\setminus \Iset{7} = \left(0,\tfrac57\right)_{\st}\,,\qquad \Jset{1} = \left(-\infty, 0\right)_{\st}
\end{align*}
and 
\begin{align*}
 \Csr {2,\st} & = \Cs{2,\st}\,,  & \Ired{2,\st} & = \Iset{2} = (-\infty, -1)_{\st},\,, & \Jset{2} & = (-1, +\infty)_{\st}\,,
 \\
 \Csr {7,\st} & = \Cs{7,\st}\,, & \Ired{7,\st} & = \Iset{7} = \left(\tfrac57, +\infty\right)_{\st}\,, & \Jset{7} & = \left(-\infty, \tfrac57\right)_{\st}\,.
\end{align*}
Further, 
\[
 \Transs{}{k}{7} = 
 \begin{cases}
  \Trans{}{1}{1} \cup \Trans{}{1}{7}\setminus\{\id\} & \text{for~$k=1$}
  \\
  \Trans{}{k}{1} \cup \Trans{}{k}{7} & \text{for~$k\not=1$}
 \end{cases}
\]
and, for $i,j\in A$, $j\not=7$, 
\[
 \Transs{}{i}{j} = \Trans{}{i}{j}\,.
\]
Thus, 
\begin{align*}
\Transs{}{1}{1} & = \{ h_1^{-1}T,\ h_3^{-1}T,\ h_2h_3^{-1}T \}  
\\
\Transs{}{1}{2} & = \{ h_3^{-1}h_1T,\ h_2h_3^{-1}h_1T \}
\\
\Transs{}{1}{7} & = \{ h_1^{-1}T,\ h_3^{-1}T,\ h_2h_3^{-1}T,\ h_2,\ h_3,\ h_2^{-1},\ h_2h_3,\ h_2^{-1}h_3 \} 
\\
\Transs{}{2}{1} & = \{ T^{-2}h_1h_3^{-1}T,\ T^{-2}h_1h_2h_3^{-1}T \}
\\
\Transs{}{2}{2} & = \{ T^{-2}h_1T,\ T^{-2}h_1h_3^{-1}h_1T\, T^{-2}h_1h_2h_3^{-1}h_1T \}
\\
\Transs{}{2}{7} & = \{ T^{-2}h_1h_3^{-1}T,\ T^{-2}h_1h_2h_3^{-1}T,\ T^{-2}h_1,\ T^{-2}h_1h_2^{-1},\ T^{-2}h_1h_2,\
\\ & \qquad T^{-2}h_1h_3,\ T^{-2}h_1h_2h_3,\ T^{-2}h_1h_2^{-1}h_3 \}
\\
\Transs{}{7}{1} & = \{ T,\ h_1h_3^{-1}T,\ h_1h_2h_3^{-1}T \}
\\
\Transs{}{7}{2} & = \{ h_1T,\ h_1h_3^{-1}h_1T,\ h_1h_2h_3^{-1}h_1T \}
\\
\Transs{}{7}{7} & = \{ T,\ h_1h_3^{-1}T,\ h_1h_2h_3^{-1}T,\ h_1,\ h_1h_2,\ h_1h_3,\ h_1h_2h_3 \}\,.
\end{align*}

\subsection{Application of cuspidal acceleration}\label{SUBSEC:applca}

As in Section~\ref{SEC:cuspacc} we now omit the tildes from the notation and only consider the underlying strong cross section. 
We have 
\[
 A_X = A = \{1,2,7\}\,,\qquad A_Y = \{2,7\}
\]
and 
\[
\eX_2 = \eY_7 = \infty\,,\quad \eY_2 = -1\,,\quad \eX_1 = 0\,,\quad \eX_7 = \frac57\,.
\]
Further, 
\begin{align*}
\psi_{\eX}(1) &= 1\,, &\psi_{\eX}(2) & = 7\,,  &\psi_{\eX}(7)  &= 2\,,
\\
\cycnext{\eX}(1) & = h_1^{-1}T\,, & \cycnext{\eX}(2) & = T^{-2}h_1h_2^{-1}h_3\,, & \cycnext{\eX}(7) & = h_1h_2h_3^{-1}h_1T\,,
\\
\sigma_{\eX}(1) &= 1\,, &  \sigma_{\eX}(2) & = 2\,, & \sigma_{\eX}(7) &= 2
\\\\
\psi_{\eY}(2) & = 2\,, & \psi_{\eY}(7) &= 7\,,
\\
\cycnext{\eY}(2) & = T^{-2}h_1T\,, & \cycnext{\eY}(7) &= T\,,
\\
\sigma_{\eY}(2) &= 1\,, & \sigma_{\eY}(7) &=1\,.
\end{align*}
We obtain
\begin{align*}
K_{\eX}(1) & = \defset{ \nu\in\Cs{1,\st} }{ \gamma_\nu(+\infty) \in \left(0,\tfrac{5}{42}\right)_{\st} }\,,
\\
K_{\eX}(2) & = \defset{ \nu\in\Cs{2,\st} }{ \gamma_\nu(+\infty) \in \left(\infty,-\tfrac{10}7\right)_{c,\st} }\,,
\\
K_{\eX}(7) & = \defset{ \nu\in\Cs{7,\st} }{ \gamma_\nu(+\infty) \in \left(\tfrac57,\tfrac34\right)_{\st} }\,,
\\
K_{\eY}(1) & = \varnothing\,,
\\
K_{\eY}(2) & = \defset{ \nu\in\Cs{2,\st} }{ \gamma_\nu(+\infty) \in \left(-\tfrac87,-1\right)_{\st} }\,,
\\
K_{\eY}(7) & = \defset{ \nu\in\Cs{7,\st} }{ \gamma_\nu(+\infty) \in \left(\tfrac{12}7,\infty\right)_{\st} }
\end{align*}
and
\begin{align*}
M_{\eX}(1) & = \defset{ \nu\in \Cs{1,\st} }{ \gamma_\nu(-\infty) \in \left(-\tfrac17,0\right)_{\st} }\,,
\\
M_{\eX}(2) & = \defset{ \nu\in \Cs{2,\st} }{ \gamma_\nu(-\infty) \in \left(-\tfrac37,\infty\right)_{\st} }\,,
\\
M_{\eX}(7) & = \defset{ \nu\in \Cs{7,\st} }{ \gamma_\nu(-\infty) \in \left(\tfrac23,\tfrac57\right)_{\st} }\,,
\\
M_{\eY}(1) & = \varnothing\,,
\\
M_{\eY}(2) & = \defset{ \nu\in \Cs{2,\st} }{ \gamma_\nu(-\infty) \in \left(-1,-\tfrac67\right)_{\st} }\,,
\\
M_{\eY}(7) & = \defset{ \nu\in \Cs{7,\st} }{ \gamma_\nu(-\infty) \in \left(\infty,-\tfrac27\right)_{c,\st} }\,.
\end{align*}
Moreover we obtain
\[
 \Index = \{ (1,\eX), (1,\eR), (2,\eX), (2,\eR), (2,\eY), (7,\eX), (7,\eR), (7,\eY) \}\,. 
\]
For convenience we set 
\[
 \UTB\H_{\st} \coloneqq \defset{ \nu\in\UTB\H }{ \gamma_\nu(\pm\infty)\in\wh\R_{\st} }\,.
\]
We have
\begin{align*}
\Csacc{(1,\eX)} & = \defset{ \nu\in\UTB\H_{\st} }{ \Rea\base{\nu}=0\,,\ \gamma_\nu(+\infty)\in \left(0,\tfrac{5}{42}\right)_{\st}\,,\ \gamma_\nu(-\infty)\notin\left(-\tfrac17,0\right)_{\st} }\,,
\\
\Csacc{(1,\eR)} & = \defset{ \nu\in\UTB\H_{\st} }{ \Rea\base{\nu} = 0\,,\ \gamma_\nu(+\infty)\in \left(\tfrac{5}{42},\tfrac57\right)_{\st} }\,,
\\
\Csacc{(2,\eX)} &  = \left\{ \nu\in\UTB\H_{\st} \ \left\vert\ \Rea\base{\nu} = -1\,, \vphantom{\gamma_\nu(+\infty)\in \left(\infty, -\tfrac{10}7\right)_{c,\st}} \right.\right.
\\
& \hphantom{=\{\nu\in\UTB\H_{\st}}\qquad \left. \gamma_\nu(+\infty)\in \left(\infty, -\tfrac{10}7\right)_{c,\st}\,,\ \gamma_\nu(-\infty)\notin \left(-\tfrac37,\infty\right)_{\st} \right\}\,,
\\
\Csacc{(2,\eR)} & = \defset{ \nu\in\UTB\H_{\st} }{ \Rea\base{\nu} = -1\,,\ \gamma_\nu(+\infty)\in \left(-\tfrac{10}7, -\tfrac87\right)_{\st} }\,,
\\
\Csacc{(2,\eY)} & = \left\{ \nu\in\UTB\H_{\st} \ \left\vert\  \Rea\base{\nu} = -1\,, \vphantom{\gamma_\nu(+\infty)\in \left(-\tfrac87, -1\right)_{\st}} \right.\right.
\\
& \hphantom{=\{\nu\in\UTB\H_{\st}}\qquad \left. \gamma_\nu(+\infty)\in \left(-\tfrac87, -1\right)_{\st}\,,\ \gamma_\nu(-\infty)\notin \left(-1,-\tfrac67\right)_{\st} \right\}\,,
\\
\Csacc{(7,\eX)} & = \defset{ \nu\in\UTB\H_{\st} }{ \Rea\base{\nu} = \tfrac57\,,\ \gamma_\nu(+\infty)\in \left(\tfrac57, \tfrac34\right)_{\st}\,,\ \gamma_\nu(-\infty)\notin \left(\tfrac23,\tfrac57\right)_{\st} }\,,
\\
\Csacc{(7,\eR)} & = \defset{ \nu\in\UTB\H_{\st} }{ \Rea\base{\nu} = \tfrac57\,,\ \gamma_\nu(+\infty)\in \left(\tfrac34,\tfrac{12}7\right)_{\st} }\,,
\\
\Csacc{(7,\eY)} & = \left\{ \nu\in\UTB\H_{\st} \ \left\vert\  \Rea\base{\nu} = \tfrac57\,, \vphantom{\gamma_\nu(+\infty)\in \left(\tfrac{12}7,\infty\right)_{\st}}\right.\right.
\\
& \hphantom{=\{\nu\in\UTB\H_{\st}}\qquad \left.  \gamma_\nu(+\infty)\in \left(\tfrac{12}7,\infty\right)_{\st}\,,\ \gamma_\nu(-\infty)\notin \left(\infty,-\tfrac27\right)_{c,\st} \right\}\,.
\end{align*}

In order to provide the accelerated forward transition sets in detail, we first note that $A^* = A$ and hence all induced $\eZ$-cycles, induced cycle sets and transformations are equal to the $\eZ$-cycles, cycle sets and transformations, respectively. Therefore, we will omit the star ($*$) from the notation. We have
\begin{align*}
& \Trans{\acc}{(1,\eR)}{(1,\eX)}  = \Trans{\acc}{(1,\eR)}{(1,\eR)} 
\\
& \qquad = \{ h_3^{-1}T,\ h_2h_3^{-1}T \} 
\\
& \Trans{\acc}{(1,\eR)}{(2,\eX)}  = \Trans{\acc}{(1,\eR)}{(2,\eR)} = \Trans{\acc}{(1,\eR)}{(2,\eY)} 
\\
& \qquad = \{ h_3^{-1}h_1T,\ h_2h_3^{-1}h_1T \}
\\
& \Trans{\acc}{(1,\eR)}{(7,\eX)}  = \Trans{\acc}{(1,\eR)}{(7,\eR)} = \Trans{\acc}{(1,\eR)}{(7,\eY)} 
\\
& \qquad = \{ h_1^{-1}T,\ h_3^{-1}T,\ h_2h_3^{-1}T,\ h_2,\ h_3,\ h_2^{-1},\ h_2h_3,\ h_2^{-1}h_3 \}
\\
& \Trans{\acc}{(2,\eR)}{(1,\eX)}  = \Trans{\acc}{(2,\eR)}{(1,\eR)}
\\
& \qquad = \{ T^{-2}h_1h_3^{-1}T,\ T^{-2}h_1h_2h_3^{-1}T \}
\\
& \Trans{\acc}{(2,\eR)}{(2,\eX)}  = \Trans{\acc}{(2,\eR)}{(2,\eR)} = \Trans{\acc}{(2,\eR)}{(2,\eY)}
\\
& \qquad = \{ T^{-2}h_1h_3^{-1}h_1T,\ T^{-2}h_1h_2h_3^{-1}h_1T \}
\\
& \Trans{\acc}{(2,\eR)}{(7,\eX)}  = \Trans{\acc}{(2,\eR)}{(7,\eR)} = \Trans{\acc}{(2,\eR)}{(7,\eY)}
\\
& \qquad = \{ T^{-2}h_1h_3^{-1}T,\ T^{-2}h_1h_2h_3^{-1}T,\ T^{-2}h_1,\ T^{-2}h_1h_2^{-1},
\\
& \hphantom{ T^{-2}h_1h_3^{-1}T,\ T^{-2}h_1h_2h_3^{-1}T,\ T^{-2}h_1 }  T^{-2}h_1h_2,\ T^{-2}h_1h_3,\ T^{-2}h_1h_2h_3 \}
\\
& \Trans{\acc}{(7,\eR)}{(1,\eX)}  = \Trans{\acc}{(7,\eR)}{(1,\eR)}
\\
& \qquad = \{ T,\ h_1h_3^{-1}T,\ h_1h_2h_3^{-1}T \}
\\
& \Trans{\acc}{(7,\eR)}{(2,\eX)}  = \Trans{\acc}{(7,\eR)}{(2,\eR)} = \Trans{\acc}{(7,\eR)}{(2,\eY)}
\\
& \qquad = \{ h_1T,\ h_1h_3^{-1}h_1T \}
\\
& \Trans{\acc}{(7,\eR)}{(7,\eX)}  = \Trans{\acc}{(7,\eR)}{(7,\eR)} = \Trans{\acc}{(7,\eR)}{(7,\eY)}
\\
& \qquad = \{ h_1h_3^{-1}T,\ h_1h_2h_3^{-1}T,\ h_1,\ h_1h_2,\ h_1h_3,\ h_1h_2h_3 \}
\end{align*}
For the other accelerated transition sets we first collect the induced cycle sets. We have
\begin{align*}
\cycnostarset{(1,\eX)} & = \{ (1,\eX) \}
\\
\cycnostarset{(2,\eX)} & = \{ (2,\eX),\ (7,\eX) \}
\\
\cycnostarset{(2,\eY)} & = \{ (2,\eY) \}
\\
\cycnostarset{(7,\eX)} & = \{ (7,\eX),\ (2,\eX) \}
\\
\cycnostarset{(7,\eY)} & = \{ (7,\eY) \}\,.
\end{align*}
Further
\begin{align*}
\cycstartrans{(1,\eX)} & = \cycnext{\eX}(1) = h_1^{-1}T = \bmat{1}{0}{7}{1}\,,
\\
\cycstartrans{(2,\eX)} & = \cycnext{\eX}(2)\cycnext{\eX}(7) = T^{-2}h_1h_2^{-1}h_3h_1h_2h_3^{-1}h_1T = \bmat{1}{-1}{0}{1}\,,
\\
\cycstartrans{(2,\eY)} & = \cycnext{\eY}(2) = T^{-2}h_1 T = \bmat{8}{7}{-7}{-6}\,,
\\
\cycstartrans{(7,\eX)} & = \cycnext{\eX}(7)\cycnext{\eX}(2) = h_1h_2h_3^{-1}h_1T^{-1}h_1h_2^{-1}h_3 = \bmat{36}{-25}{49}{-34}\,,
\\
\cycstartrans{(7,\eY)} & = \cycnext{\eY}(7) = T = \bmat{1}{1}{0}{1}\,.
\end{align*}
Therefore
\begin{align*}
& \Trans{\acc}{(1,\eX)}{(1,\eR)} = \defset{ \cycstartrans{(1,\eX)}^n }{ n\in\N }\,,
\\
& \Trans{\acc}{(2,\eX)}{(2,\eR)} = \Trans{\acc}{(2,\eX)}{(2,\eY)} = \defset{ \cycstartrans{(2,\eX)}^n }{ n\in\N }\,,
\\
& \Trans{\acc}{(2,\eX)}{(7,\eR)} = \Trans{\acc}{(2,\eX)}{(7,\eY)} = \defset{ \cycstartrans{(2,\eX)}^n\cycnext{\eX}(2) }{ n\in\N_0 }\,,
\\
& \Trans{\acc}{(7,\eX)}{(2,\eR)} = \Trans{\acc}{(7,\eX)}{(2,\eY)} = \defset{ \cycstartrans{(7,\eX)}^n\cycnext{\eX}(7) }{ n\in\N_0 }\,,
\\
& \Trans{\acc}{(7,\eX)}{(7,\eR)} = \Trans{\acc}{(7,\eX)}{(7,\eY)} = \defset{ \cycstartrans{(7,\eX)}^n }{ n\in\N }\,,
\\
& \Trans{\acc}{(1,\eX)}{(1,\eX)} = \Trans{\acc}{(1,\eX)}{(2,\eX)}
\\
& = \Trans{\acc}{(1,\eX)}{(7,\eX)} = \Trans{\acc}{(2,\eX)}{(1,\eX)}
\\
& =\Trans{\acc}{(2,\eX)}{(2,\eX)} = \Trans{\acc}{(2,\eX)}{(7,\eX)} 
\\
& = \Trans{\acc}{(7,\eX)}{(1,\eX)} = \Trans{\acc}{(7,\eX)}{(2,\eX)}
\\
& = \Trans{\acc}{(7,\eX)}{(7,\eX)} = \Trans{\acc}{(2,\eY)}{(1,\eX)} 
\\
& = \Trans{\acc}{(2,\eY)}{(2,\eY)} = \Trans{\acc}{(2,\eY)}{(7,\eY)} 
\\
& = \Trans{\acc}{(7,\eY)}{(1,\eX)} = \Trans{\acc}{(7,\eY)}{(2,\eY)}
\\
& = \Trans{\acc}{(7,\eY)}{(7,\eY)} = \Trans{\acc}{(1,\eX)}{(2,\eR)}
\\
& = \Trans{\acc}{(1,\eX)}{(2,\eY)} = \Trans{\acc}{(1,\eX)}{(7,\eR)}
\\
& = \Trans{\acc}{(1,\eX)}{(7,\eY)} = \Trans{\acc}{(2,\eX)}{(1,\eR)}
\\
& = \Trans{\acc}{(2,\eY)}{(1,\eR)} = \Trans{\acc}{(2,\eY)}{(7,\eR)}
\\
& = \Trans{\acc}{(2,\eY)}{(7,\eX)} = \Trans{\acc}{(7,\eX)}{(1,\eR)}
\\
& = \Trans{\acc}{(7,\eY)}{(1,\eR)} = \Trans{\acc}{(7,\eY)}{(2,\eX)}
\\
& = \Trans{\acc}{(7,\eY)}{(2,\eR)} = \Trans{\acc}{(7,\eY)}{(7,\eX)}
\\
& = \Trans{\acc}{(7,\eY)}{(7,\eR)} = \varnothing\,.
\end{align*}

\subsection{Structure tuple and family of neighborhoods}\label{SUBSEC:structlongex}

With the preparations in the previous sections we can now provide a structure tuple
\[
\mathcal{S} \coloneqq \bigl( \Index, (I_{a})_{a \in \Index}, (P_{a,b})_{a,b \in \Index}, (C_{a,b})_{a,b \in \Index},((g_p)_{p\in P_{a,b}})_{a,b\in \Index} \bigr)
\]
for~$\Gamma$. Following the discussion in Section~\ref{SEC:strictTOAexist}, we obtain
\begin{align*}
I_{(1,\eX)} & = \left(0,\tfrac{5}{42}\right)\,, 
&
I_{(1,\eR)} & = \left(\tfrac{5}{42}, \tfrac57\right)\,,
\\
I_{(2,\eX)} & = \left(-\infty, -\tfrac{10}{7}\right)\,,
&
I_{(2,\eR)} & = \left(-\tfrac{10}{7}, -\tfrac87\right)\,,
&
I_{(2,\eY)} & = \left(-\tfrac87, -1\right)\,,
\\
I_{(7,\eX)} & = \left(\tfrac57,\tfrac34\right)\,,
&
I_{(7,\eR)} & = \left(\tfrac34,\tfrac{12}{7}\right)\,,
&
I_{(7,\eY)} & = \left(\tfrac{12}{7}, +\infty\right)\,.
\end{align*}
Further, 
\begin{align*}
P_{(1,\eR), (1,\eX)} & = \left\{ u_{(1,\eX)}^{-1} \right\} = \left\{ \bmat{1}{0}{-7}{1}\right\}\,,
\\
P_{(2,\eR), (2,\eX)} & = P_{(2,\eY), (2,\eX)} = P_{(7,\eR), (2,\eX)} = P_{(7,\eY), (2,\eX)} = \left\{ u_{(2,\eX)}^{-1} \right\} = \left\{ \bmat{1}{1}{0}{1}\right\}\,,
\\
P_{(2,\eX), (2,\eY)} & = P_{(2,\eR), (2,\eY)} = \left\{ u_{(2,\eY)}^{-1} \right\} = \left\{ \bmat{-6}{-7}{7}{8}\right\}\,,
\\
P_{(2,\eR), (7,\eX)} & = P_{(2,\eY), (7,\eX)} = P_{(7,\eR), (7,\eX)} = P_{(7,\eY), (7,\eX)} = \left\{ u_{(7,\eX)}^{-1} \right\} = \left\{\bmat{-34}{25}{-49}{36}\right\}\,,
\\
P_{(7,\eX), (7,\eY)} & = P_{(7,\eR), (7,\eY)} = \left\{ u_{(7,\eY)}^{-1} \right\} = \left\{ \bmat{1}{-1}{0}{1}\right\}\,,
\end{align*}
and $P_{a,b} = \varnothing$ for all other choices of~$a,b\in\Index$. The fixed points of the parabolic elements in the sets above are as follows:
\begin{center}
\begin{tabular}{l||c|c|c|c|c}
element & $u_{(1,\eX)}^{-1}$ & $u_{(2,\eX)}^{-1}$ & $u_{(2,\eY)}^{-1}$ & $u_{(7,\eX)}^{-1}$ & $u_{(7,\eY)}^{-1}$
\\ \hline
fixed point & $0$ & $\infty$ & $-1$ & $\tfrac57$ & $\infty$
\end{tabular}\,.
\end{center}
For $p=u_{(1,\eX)}^{-1} \in P_{(1,\eR), (1,\eX)}$ we have
\[
g_p = g_{u_{(1,\eX)}^{-1}} = \cycnostarnext{(1,\eX), (1,\eX)}^{-1} = u_{(1,\eX)}^{-1}\,.
\]
For $p=u_{(2,\eX)}^{-1}$ considered as an element of~$P_{(2,\eR), (2,\eX)}$ or of~$P_{(2,\eY), (2,\eX)}$ we have
\[
g_p = g_{u_{(2,\eX)}^{-1}} = \cycnostarnext{(2,\eX), (2,\eX)}^{-1} = u_{(2,\eX)}^{-1}\,.
\]
For $p=u_{(2,\eX)}^{-1}$ considered as an element of~$P_{(7,\eR), (2,\eX)}$ or of~$P_{(7,\eY), (2,\eX)}$ we have
\[
g_p = g_{u_{(2,\eX)}^{-1}} = \cycnostarnext{(2,\eX), (7,\eX)}^{-1} = \cycnext{\eX}(2)^{-1} = \bmat{5}{7}{7}{10}\,.
\]
For $p=u_{(2,\eY)}^{-1} \in P_{(2,\eX), (2,\eY)} = P_{(2,\eR), (2,\eY)}$ we have
\[
g_p = g_{u_{(2,\eY)}^{-1} } = \cycnostarnext{(2,\eY), (2,\eY)}^{-1} = u_{(2,\eY)}^{-1}\,.
\]
For $p=u_{(7,\eX)}^{-1}$ considered as an element of~$P_{(2,\eR), (7,\eX)}$ or of~$P_{(2,\eY), (7,\eX)}$ we have
\[
g_p = g_{u_{(7,\eX)}^{-1}} = \cycnostarnext{(7,\eX), (2,\eX)}^{-1} = \cycnext{\eX}(7)^{-1} = \bmat{-3}{2}{7}{-5}\,.
\]
For $p=u_{(7,\eX)}^{-1}$ considered as an element of~$P_{(7,\eR), (7,\eX)}$ or of~$P_{(7,\eY), (7,\eX)}$ we have
\[
g_p = g_{u_{(7,\eX)}^{-1}} = \cycnostarnext{(7,\eX), (7,\eX)}^{-1} = u_{(7,\eX)}^{-1}\,.
\]
For $p=u_{(7,\eY)}^{-1} \in P_{(7,\eX), (7,\eY)} = P_{(7,\eR), (7,\eY)}$ we have
\[
g_p = g_{u_{(7,\eY)}^{-1}} = \cycnostarnext{(7,\eY), (7,\eY)}^{-1} = u_{(7,\eY)}^{-1}\,.
\]
For all $a\in\Index$ and all $j\in\{1,2,7\}$,
\[
C_{a,(j,\eR)} = \Trans{\acc}{(j,\eR)}{a}^{-1}\,.
\]
Further, 
\begin{align*}
C_{(1,\eR), (1,\eX)} & = \{ \cycnostarnext{(1,\eX), (1,\eX)}^{-1} \} = \{ u_{(1,\eX)}^{-1} \}\,,
\\
C_{(2,\eR), (2,\eX)} & = C_{(2,\eY), (2,\eX)} = \{ \cycnostarnext{(2,\eX), (2,\eX)}^{-1} \} = \{ u_{(2,\eX)}^{-1} \}\,,
\\
C_{(7,\eR), (2,\eX)} & = C_{(7,\eY), (2,\eX)} = \{ \cycnostarnext{(2,\eX), (7,\eX)}^{-1} \} = \{ \cycnext{\eX}(2)^{-1} \}\,,
\\
C_{(2,\eX), (2,\eY)} & = C_{(2,\eR), (2,\eY)} = \{ \cycnostarnext{(2,\eY), (2,\eY)}^{-1} \} = \{ u_{(2,\eY)}^{-1} \}\,,
\\
C_{(2,\eR), (7,\eX)} & = C_{(2,\eY), (7,\eX)} = \{ \cycnostarnext{(7,\eX), (2,\eX)}^{-1} \} = \{ \cycnext{\eX}(7)^{-1} \}\,,
\\
C_{(7,\eR), (7,\eX)} & = C_{(7,\eY), (7,\eX)} = \{ \cycnostarnext{(7,\eX), (7,\eX)}^{-1} \} = \{ u_{(7,\eX)}^{-1} \}\,,
\\
C_{(7,\eX), (7,\eY)} & = C_{(7,\eR), (7,\eY)} = \{ \cycnostarnext{(7,\eY), (7,\eY)}^{-1} \} = \{ u_{(7,\eY)}^{-1} \}\,.
\end{align*}
For all other $a,b\in\Index$, we have $C_{a,b} = \varnothing$.

To end the discussion of this extended example, we now present a family~$(\mc E_a)_{a\in\Index}$ of open, bounded, connected and simply connected sets in~$\wh\C$ for~$\mc S$ as requested in Property~\ref{staPROP5}. As in Section~\ref{SUBSEC:prop5}, for each~$a\in\Index$, we provide an open interval~$E_a$ in~$\wh\R$. The requested set~$\mc E_a$ for~$a\in\Index$ is then the open complex ball in~$\wh\C$ spanned by~$E_a$. More precisely, $\mc E_a$ is the open ball in~$\wh\C$ with center in~$\wh\R$ such that $\mc E_a \cap \wh\R = E_a$. In this context, a subset~$B$ of~$\wh\C$ is called a (complex) ball if there exists $g\in\PSL_2(\R)$ such that $g\act B$ is a euclidean ball in~$\C$. A straightforward, but tedious calculation shows that the following intervals serve our goal:
\begin{align*}
 E_{(1,\eX)} & \coloneqq \left(-\tfrac1{30}, \tfrac{23}{105}\right)\,, & E_{(1,\eR)} & \coloneqq \left( \tfrac{1}{42}, \tfrac67 \right)\,,
 \\
 E_{(2,\eX)} & \coloneqq \left(10^{14}, -\tfrac{19}{14}\right)_c\,, & E_{(2,\eR)} & \coloneqq \left(-\tfrac{17}{7}, -\tfrac{57}{56}\right)\,, & E_{(2,\eY)} & \coloneqq \left(-\tfrac{17}{14}, -\tfrac{99}{100}\right)\,,
 \\
 E_{(7,\eX)} & \coloneqq \left(\frac{38}{63}, \tfrac{17}{20}\right)\,, & E_{(7,\eR)} & \coloneqq \left(\tfrac{37}{50}, \tfrac{27}{14}\right)\,, & E_{(7,\eY)} & \coloneqq \left(\tfrac{11}{7}, -10^{15}\right)_c\,.
\end{align*}


\begin{landscape}

\vspace*{\fill}
\begin{figure}[h]
\begin{center}
\includegraphics{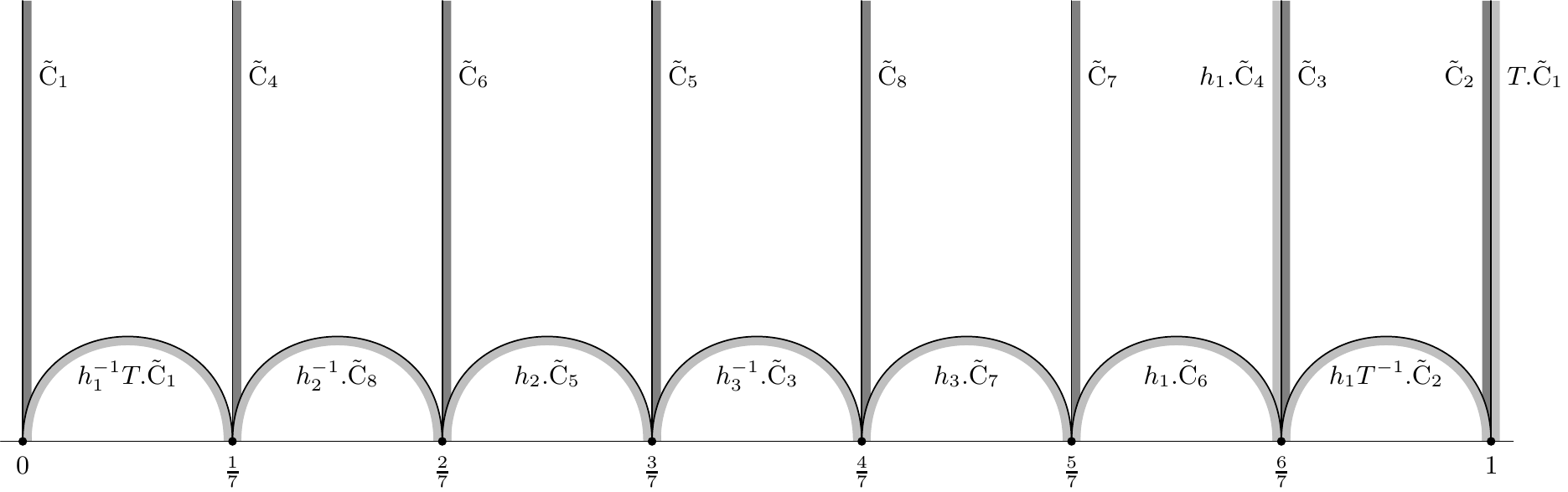}
\end{center}
\caption{A first set of branches and some next intersections.}\label{fig:cs1_next}
\end{figure}
\vspace*{\fill}

\newpage

\vspace*{\fill}
\begin{figure}[h]
\begin{center}
\includegraphics{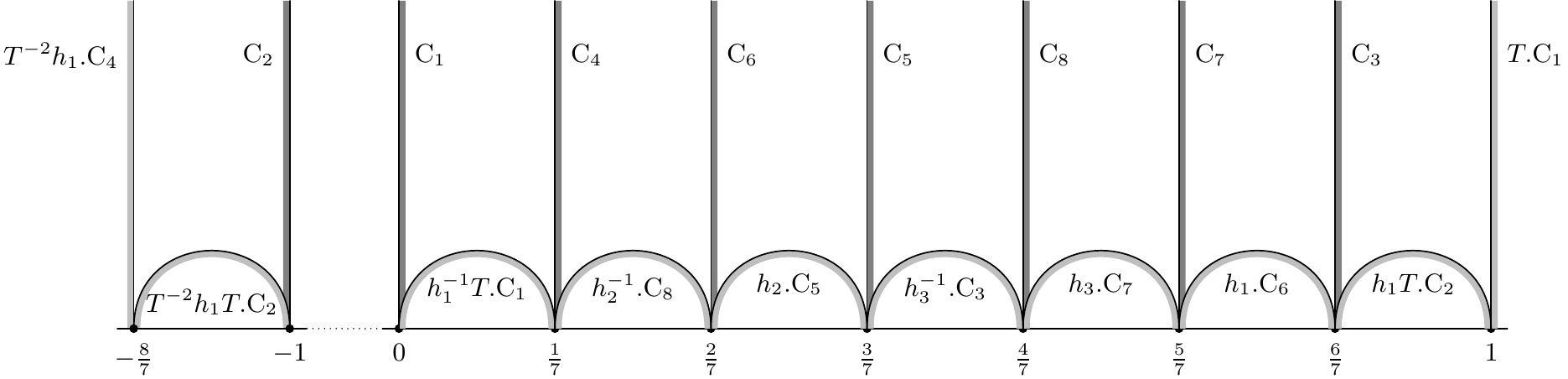}
\end{center}
\caption{The initial set of branches and some next intersections.}\label{fig:cs3_next}
\end{figure}
\vspace*{\fill}

\end{landscape}


\bibliographystyle{amsplain}
\bibliography{pw_TObib}

\printindex[defs]
\printindex[symbols]

\setlength{\parindent}{0pt}

\end{document}